\documentclass[reqno,10pt]{amsart}
\usepackage[margin=1in]{geometry}

\usepackage{graphicx}
\usepackage{amsmath,amssymb}
\usepackage{xcolor}

\usepackage{hyperref}

\usepackage{wrapfig}

\newtheorem{theorem}{Theorem} 
\newtheorem{definition}{Definition}
\newtheorem{lemma}{Lemma}
\newtheorem{proposition}{Proposition}

\newtheorem{remark}{Remark}

\usepackage{enumerate}
\usepackage{amssymb }
\usepackage{cancel}

\makeatletter
\renewcommand*\env@matrix[1][*\c@MaxMatrixCols c]{%
  \hskip -\arraycolsep
  \let\@ifnextchar\new@ifnextchar
  \array{#1}}
\makeatother

\let\e=\varepsilon

\let\p=\partial

\let\O=\Omega

\let\o=\omega

\let\f=\varphi

\numberwithin{equation}{section}

\let\hide\iffalse
\let\unhide\fi

\newcommand{\R}{\mathbb{R}}

\newcommand{\be}{\begin{equation}}
\newcommand{\bm}{\begin{multline}}
\newcommand{\ee}{\end{equation}}
\newcommand{\dd}{\mathrm{d}}

\newcommand{\xb}{x_{\mathbf{b}}}

\newcommand{\tb}{t_{\mathbf{b}}}
\newcommand{\vb}{v_{\mathbf{b}}}

\newcommand{\tf}{t_{\mathbf{f}}}

\newcommand{\Bes}{\begin{eqnarray*}}
\newcommand{\Ees}{\end{eqnarray*}}
\newcommand{\Be}{\begin{equation} }
\newcommand{\Ee}{\end{equation}}
\newcommand{\Bs}{\begin{split}}
\newcommand{\Es}{\end{split}}

\newcommand{\Bex}{B_{\text{ext}}}

\def\p{\partial}

\def\O{\Omega}
\def\R{\mathbb{R}}

\def\B{\begin{equation}}
\def\E{\end{equation}}
\def\BN{\begin{eqnarray*}}
\def\EN{\end{eqnarray*}}

\begin{document}

\title{Lipschitz continuous solutions of the Vlasov-Maxwell systems with a conductor boundary condition} 
 
\author{Yunbai Cao}
\address{Department of Mathematics, Rutgers University, Piscataway, NJ 08854; email: yc1157@math.rutgers.edu}
\author{Chanwoo Kim}
\address{Department of Mathematics, University of Wisconsin-Madison, Madison, WI 53706; email: ckim.pde@gmail.com}

\begin{abstract}
We consider relativistic plasma particles subjected to an external gravitation force in a $3$D half space whose boundary is a perfect conductor. When the mean free path is much bigger than the variation of electromagnetic fields, the collision effect is negligible. As an effective PDE, we study the relativistic Vlasov-Maxwell system and its local-in-time unique solvability in the space-time locally Lipschitz space, for several basic mesoscopic (kinetic) boundary conditions: the inflow, diffuse, and specular reflection boundary conditions. We construct weak solutions to these initial-boundary value problems and study their locally Lipschitz continuity with the aid of a weight function depending on the solutions themselves. Finally, we prove the uniqueness of a solution, by using regularity estimate and realizing the Gauss's law at the boundary within Lipschitz continuous space.
\end{abstract}

\maketitle

 \subsubsection*{\textbf{\large{Introduction}}}

Plasma is the most abundant form of ordinary matter in universe, being mostly associated with stars. The Sun, our nearest star, is composed of 92.1$\%$ hydrogen and 7.8$\%$ helium by number, and 0.1$\%$ of heavier elements. At the central core, hydrogen burns into helium (so-called the p-p chain of reactions starting from the fusion of two protons into a nucleus of deuterium), which is the major reaction that drives the sun’s radiance (see the famous B$^2$FH paper \cite{B2FH} for details). 

 \begin{wrapfigure}{r}{0.25\textwidth}
	\centering
	\includegraphics[width=0.90\linewidth]{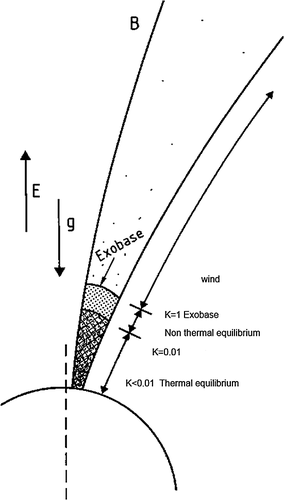} 
	\tiny{Figure 1. The different transition regions in a stellar (\cite{PP})} 
	\label{fig:wrapfig}$^1$
\end{wrapfigure}

At the upper atmosphere of the Sun (solar corona), electrons and protons escape from the solar corona (upper atmosphere), while traces of heavier elements have been identified (\cite{PLL}). This emission of plasma particles is called solar wind. The solar corona can be decomposed according to the Knudsen number of plasma. At low altitude the collision is dominant (Knudsen number $\ll $ density scale height), and hence the particles are assumed to be in hydrostatic/hydrodynamic equilibrium of MHD-type systems. Above this regime (exosphere), the collision rate between particles is assumed to be negligibly small: Knudsen number is about the density scale height of the Sun, which is an order of 100km. These two extreme Knudsen number regimes are separated by a narrow transition regime which is called the exobase (See Figure 1 adopted from \cite{PP}\footnote{permission to use the figure granted by \url{https://www.agu.org/Publish-with-AGU/Publish/Author-Resources/Policies/Permission-policy\#repository}}
). Above the exobase, there have been extensive research activities on the solar wind using collisionless Boltzmann equation (e.g. linear steady Vlasov model), which has been called the exospheric solar wind models. In the early 60s, Chamberlain suggested the ``solar breeze model" that the radial expansion of the solar corona results from the thermal evaporation of the hot coronal protons out of the gravitational field of the Sun \cite{Chamberlain}. In this model the ambient polarization electric field is implemented as a well-known Pannekoek-Rosseland (PR) electric field \cite{Pannekoek, Rosseland}, which will be discussed at \eqref{PR}. \textit{In this paper, we are interested in a kinetic description of the exospheric solar wind using the initial-boundary value problem of the full relativistic Vlasov-Maxwell system subjected to ambient polarization electric field, geomagnetic field, and gravitation.}


When the collision effect is negligible, the master equation describing dynamics of  two species plasma (an average of 95\% of the solar wind ions are protons \cite{PLL}) is the relativistic Vlasov-Maxwell system (RVM)
\Be \label{VMfrakF_I}
\begin{split}
	\p_t f_\pm + \hat v_\pm \cdot \nabla_x f_\pm +   \mathfrak F_\pm \cdot \nabla_v f_\pm = 0,  & \ \  \text{ in } \R_+ \times \O \times \mathbb R^3 ,
	\\ f_\pm(0,x,v) =  f_{0,\pm}(x,v) ,  & \ \  \text{ in }     \O \times \mathbb R^3.
\end{split}
\Ee
Here $f_\pm = f_\pm(t,x,v) \ge 0 $ represents the density distribution functions for the proton $(+)$ and electron $(-)$ respectively. The relativistic velocity is
\Be\label{rel_v_I}
\hat v_\pm = \frac{v}{\sqrt{m_\pm^2 + |v|^2 / c^2}},
\Ee
where $m_\pm$ is the magnitude of the masses of protons and electrons, and $c$ is the speed of light.

 The Lorentz force $\mathfrak F_\pm$ consists of self-consistent field electromagnetic fields plus given polarization electric field $E_{\text{ext}}$, geomagnetic field $B_{\text{ext}}$, and gravitation:
\Be \label{frakF_I}
\mathfrak F_\pm  =  e_\pm \left(E+ E_{\text{ext}} +  \frac{\hat v_\pm}{c}  \times (B + B_{\text{ext}}) \right) - m_\pm g \mathbf e_3,
\Ee 
The self-consistent fields $E(t,x)$, $B(t,x)$ are coupled with $f_\pm$ through the inhomogeneous Maxwell equations
\Be \label{Maxwell_I}
\begin{split}
	\p_t E & = c  \nabla_x \times B - 4 \pi J, \, \ \  \  \nabla_x \cdot E = 4\pi \rho, \ \ \text{ in }  \R_+ \times  \O,
	\\ \p_t B & = - c  \nabla_x \times E, \, \ \  \  \nabla_x \cdot B = 0, \ \ \text{ in }  \R_+ \times  \O,
\end{split}
\Ee
with initial conditions
\Be \label{BEinitdata_I}
E(0,x) = E_0(x),  \ \ B(0,x) = B_0(x), \ \ \text{ in } \O.
\Ee
Here, the electric density and current are defined as
\Be \label{rhoJ_I}
\rho = \int_{\mathbb R^3} ( e_+f_+ + e_- f_-) \dd v,  \ \ \  J = \int_{\mathbb R^3}  (\hat v_+ e_+ f_+  +  \hat v_- e_- f_-) \dd v.
\Ee

Due to its importance, there have been extensive studies on the global regularity of the Cauchy problem of RVM. Here, we only overview papers relevant to our approach, and we refer to \cite{LS, LSt2, DL} for a more complete list of references. In a classical solution context, Glassey and Strauss first studied a continuation criterion of the relativistic Vlasov-Maxwell system in the whole space $\mathbb R^3$ in \cite{GS}, using so-called the Glassey-Strauss representation, which is a crucial tool in our analysis of this paper. It was shown that classical solution exists for all time as long as the velocity support of the particle density function $f$ is compact. Later Klainerman and Staffilani prove the result using a different method in \cite{KS}. The work of \cite{GS} leads to substantial developments in \cite{GSc1, GSc2, GSc3, GS4, GS2, GS3}. Notably in \cite{GSc1, GSc2, GSc3}, Glassey and Schaeffer proved that in the two-dimensional and two-and-a-half dimensional case, for regular initial data with compact velocity support, the system has unique global in time solution. More recently, in \cite{LS} Luk and Strain proved a new continuation criterion for the system by showing that the classical solution exists for all time if the velocity support of $f$ is bounded after projecting to any two-dimensional plane. Then in \cite{LSt2}, they improve the result of \cite{GSc1, GSc2, GSc3} the two-dimensional and two-and-a-half dimensional case by only requiring the initial data to have polynomial decay in velocity space. In addition, they showed that in the three-dimensional case, a regular solution can be extended by assuming a bound on a certain moment of $f$. In a weak solution context, the global weak solutions of the RVM system were obtained in \cite{DL} using a velocity averaging lemma, and the questions of its uniqueness and global regularity are still open.  

In many applications of plasma models, the particles are in contact with a different phase through a sharp interface, which can be considered as a (either solid or moving) boundary. In the solar wind model, under the top of exobase, the space is filled with fully ionized plasma particles with a very short mean free path, which can be considered as a perfect conductor. We set the altitude of the exobase $x_3=0$ and consider the upper half space
\Be\label{domain}
\O  = \mathbb R^3_+ := \{ (x_1, x_2, x_3) \in \mathbb R^3: x_3 > 0 \} . \Ee
At the top of the exobase, we assume a perfect conductor boundary condition for the self-consistent electromagnetic fields. Denote by $n$ the outward unit normal of $\O$ (which is $n=-\mathbf e_3$ for our case); $[V]$ the jump of $V$ across $\p\O$: $[V](x_1,x_2)=  \lim_{x_3 \downarrow 0}V(x_1,x_2,x_3)- \lim_{x_3 \uparrow 0}V(x_1,x_2,x_3)$. Then from $\p_t B = - \nabla_x \times E$ and $\nabla_x \cdot B = 0$, we derive the jump conditions (see \cite{CKKRM} for the details)
\[ \label{jump_EB}
n \times  [E]  = 0 , \ \ \
n \cdot [B]    =0.
\]
In other words, the tangential electric fields $E_1$, $E_2$, and the normal magnetic field $B_3$ are continuous across the interface $\p\O$. Therefore, we obtain boundary conditions for a perfect conductor of the solutions $(E,B)$ to \eqref{VMfrakF}-\eqref{rhoJ}. 
\Be \label{E12B3bc_I}
E_1 = E_2 = 0, \,  B_3 = 0 , \text{ on }  \  \R_+ \times  \p \O.
\Ee
The initial-boundary value problem of the Vlasov-Maxwell system with the perfect conductor boundary condition has been studied by Guo in \cite{Guo93} for general domains with boundary. By approximating the phase space via a sequence of domains and linear systems and using the compactness result of \cite{DL}, he establishes a global existence of weak solutions for RVM with the perfect conductor boundary condition for various boundary conditions of $f$. The regularity question is highly nontrivial since the stability of the ballistic trajectory depends on the sign of the normal component of the field at the boundary. As a matter of fact, in \cite{Guo96}, he constructs an example of the RVM system such that the solution immediately does not belong to $C^1$. Under the favorite sign condition of the field at the boundary, in \cite{Guo95}, he constructs regular solutions for a 1D model of the Vlasov-Maxwell system on a half line. In the proof, he introduces an important weight function $\alpha$ and establishes a crucial velocity lemma. This technique motivates us to define the kinetic weight function in \eqref{alphadef} and build a weighted regularity estimate for $f$ along with it for the RVM system with boundary. We will discuss the role of kinetic weight in Definition \ref{def:alpha} and its remarks. There are several interesting related research lines. Here we only list some of them for readers' convenience: stationary solutions of the RVM (\cite{R}), initial-boundary value problem of Maxwell system in time-dependent domains (\cite{CS}), an inverse boundary value problem of Maxwell's equations (\cite{OPS}), a dielectric boundary problem \cite{BronoldFehske}, and a non-perfect conductor boundary problem (\cite{Dantas,Matus}).

Now we consider the gravitation (and the gravitation constant $g>0$), an ambient polarization electric field $E_{\text{ext}}$, and geomagnetic field $B_{\text{ext}}$ near the exobase. As we are only interested in the dynamics near the exobase, we can assume that $E_{\text{ext}}$ and $B_{\text{ext}}$ take forms of \Be \label{EBext_I}
E_{\text{ext}} = E_e \mathbf e_3 \text{ and } \Bex = B_e \mathbf e_3, 
\Ee
where $E_e, B_e$ are the magnitude of the fields, and $\mathbf e_3$ is a unit vector $(0 \ 0 \ 1)^T$. In the early 1920's Pannekoek and Rosseland independently calculated an electric potential of a Sun-like gaseous star, which consists of fully ionized matter in isothermal equilibrium (temperature$=T$). Recall that, for the two-species model, we have $e_{\pm}$ and $m_{\pm}$ be the charge and mass of negative/positive ions, respectively. Pannekoek and Rosseland conclude that the gravitational constant $g>0$ and the polarization electrical field $E_{\text{ext}}= E_e (0 \ 0 \ 1)^T$ satisfy the following condition at the exobase:  
\Be\label{PR_field}
\frac{E_{e}}{-g} = - \frac{m_+- m_-}{ e_+ + |e_-|}.
\Ee
For electron/proton gaseous star ($m_+>1800 m_-$ and $e_+=1=-e_-$), 
this identity implies that the polarization electrical field is \textit{upward}. Moreover, from $e_+ E_e= \frac{1}{2}(m_+ - m_-) g $, we derive the \textit{Pannekoek-Rosseland condition}:
\Be\label{PR}
 m_+ g >2e_+ E_{e}.
\Ee
This condition crucially implies that the gravitation effect dominates ambient electromagnetic one so that the acceleration of particles would be \textit{attractive to the boundary}. We will explain the importance of the Pannekoek-Rosseland condition qualitatively when defining the kinetic weight in Definition \ref{def:alpha} and its remarks. It might be worth mentioning another important physical domain with boundary in plasma physics which is a fusion reactor such as tokamak. In a lab on the earth fusion can happen above 100 million Celsius (much higher than Sun's temperature) and no boundary materials can effectively withstand direct contact with such heat. To solve this problem, scientists have devised plasma held inside a doughnut-shaped magnetic field: if a confining external magnetic field is large enough, the plasma is localized away from the boundary. In other words, the acceleration of particles due to this external electromagnetic field is \textit{repellent to the boundary}, which is the exact opposite effect of gravitation/polarization electric field satisfying the Pannekoek-Rosseland condition. In some sense, one can reduce the initial-boundary value problem to the Cauchy problem when the confining magnetic field is dominant (\cite{Zhang2,JSW}).

Finally we consider a boundary condition of density distribution of plasma particles on  the incoming phase boundary $\gamma_- := \{(x,v) \in \p \O \times \mathbb R^3: v_3>0 \}.$ In addition let $\gamma_+ := \{(x,v) \in \p \O \times \mathbb R^3: v_3<0 \}$ and $\gamma_0 := \{(x,v) \in \p \O \times \mathbb R^3: v_3=0 \}$ denote the outgoing phase boundary and grazing phase boundary, respectively. In this paper we consider the following three simple physical boundary conditions, which were originally proposed by James Clerk Maxwell \cite{Maxwell}. An inflow boundary condition (inflow BC) is given by a prescribed date $g_\pm : \R_+ \times \gamma_- \to \mathbb R$:
	\Be \label{inflow_I}
	f_\pm(t,x,v)  =  g_\pm(t,x,v), \ \ \text{on} \  \R_+ \times  \gamma_-.
	\Ee  
A diffuse boundary condition (diffuse BC) takes the form of 
	\Be \label{diffuseBC_I}
	f_\pm(t,x,v )   = 
	\frac{1}{2\pi T_w^2}e^{- \frac{|v|^2}{2 T_w}}
	 \int_{u_3 < 0 }  - f_\pm(t,x,u) \hat u_{\pm,3} du, \ \ \text{on} \  \R_+ \times  \gamma_-,
	\Ee
	where $T_w(x)$ is a positive smooth prescribed boundary temperature. As we are interested in a short-time dynamics from now on we assume the isothermal case $T_w(x)= 1$ for the sake of simplicity. We also have a generalized diffuse boundary condition \cite{CKLi}. Finally, a specular reflection boundary condition (specular BC) is given by
	\Be \label{spec_I}
	f_\pm(t,x,v_\parallel, v_3 )  = f_\pm(t,x, v_\parallel, -v_3 )  \ \ \text{on} \  \R_+ \times  \gamma_-.
	\Ee
For the diffuse BC and specular BC, the boundary conditions enjoy a null flux condition:
	$\int_{\mathbb R^3} f_\pm(t,x,v) \hat v_{\pm,3} dv = 0 \text{ for } x \in \p \O$
, which implies a conservation of mass for a strong solution of RVM
	$\int_{\O \times \mathbb R^3} f_\pm(t,x,v) dv dx = \int_{\O \times \mathbb R^3} f_\pm(0,x,v) dv dx \text{ for all } t \ge 0.$
One of the advantages of kinetic theory is that we can devise different boundary conditions from the microscopic interaction law of particles and boundaries. For example in a recent solar wind model, a non-Maxwellian inflow boundary condition is used to explain coronal heating phenomena (\cite{PP}).

 Stability of the RVM system has also been studied extensively. Notably, in \cite{LS1, LS2, LS3}, the authors study the spatially inhomogeneous equilibrium in domains without boundary. A sharp criterion for spectral stability was given in \cite{LS3} and the nonlinear stability is studied in \cite{LS2}. In the case of bounded domains, stability analysis of the system was carried out in \cite{NS1} when the domain is a $2$D disk with perfect conducting boundary which reflects particles specularly. And then later the authors consider the case when the domain is a $3$D solid torus. More recently, the stability analysis was generalized to any axisymmetric domains in \cite{Zhang1}.

\hide

In this paper, we consider a plasma particle dynamics near above the exobase using the two species relativistic Vlasov-Maxwell system subjected to the solar gravitation and given geoelectricmagentic fields. Our geoelectric fields only agree with the Pannekoek-Rosseland condition at the top of the exobase (not for all altitude)

 $(E_{\text{ext}}, B_{\text{ext}})$, which only satisfy so-called Pannekoek-Rosseland condition at the top of the exobase (\ref{PR}).

the Pannekoek-Rosseland electric field

Pannekoek-Rosseland

 The first one is the solar breeze model. It was pro- posed by Chamberlain (1960) who suggested that the protons with a velocity exceeding the critical escape velocity evapo- rate like neutral particles escape out of a planetary atmosphere (Jeans 1923, Brandt  Chamberlain 1960). Thus Chamberlain suggested that the radial expansion of the solar corona results from the thermal evaporation of the hot coronal protons out of the gravitational field of the Sun.

As indicated above, the present work is based on the ki- netic/exospheric model of the ion-exosphere originally devel- oped by Lemaire-Scherer (1970; 1971a) for geomagnetic field lines open to the magnetospheric tail. This model, initially ded- icated to the study of the polar wind, had subsequently been applied to model the solar wind (Lemaire-Scherer 1971b), with the assumption of Maxwellian VDF’s for the protons and electrons at the top of the collision-dominated part of the corona.

In zero order kinetic approximations, or exospheric models, two separate regions are considered: first, the collision-dominated barosphere at low altitude, in which the particles are assumed to be in hydrostatic/hydrodynamic equilibrium, and secondly, an exosphere in which the collision rate between particles is assumed to be negligibly small. These two extreme Knudsen number regimes are separated by a surface which is called the exobase.

To obtaina kineticdescriptionof the solar windphenomenas,everaaluthorshaveapplie exosperic theoreis o the collisionless region of the solar corona

 Among the kinetic approaches, the purely collisionless one is generally called the exospheric approach. Two classes of exo- spheric solar wind models have been developed during the last fourty years.

An average of 95 of the solar wind ions are protons. Helium is the most abundant heavy ion with 3.2 in average in the slow solar wind and 4.2 in the high speed solar wind [Schwenn, 1990]. However, this fraction is variable, espe- cially in the slow speed solar wind, and helium concentra- tion can sometimes be as high as 10 of the total ions concentration. Oxygen, carbon, neon, nitrogen, silicon, magnesium, iron, sulfur, and other heavy minor ions are also detected in much smaller amounts (around 1 all together).

three internal (thermonuclear, radiative (energy is transported
mainly by radiative diffusion), and
convective) zones, the solar surface (photosphere), the
lower (chromosphere) and upper atmosphere (corona),

The solar interior further
consists of a radiative zone, where energy is transported
mainly by radiative diffusion

Markus J. Aschwanden,
Chapter 11 - The Sun,
Editor(s): Tilman Spohn, Doris Breuer, Torrence V. Johnson,
Encyclopedia of the Solar System (Third Edition),
Elsevier,
2014,
Pages 235-259,

=
A kinetic model of the solar wind with Kappa distribution functions in the corona
M. Maksimovic1, V. Pierrard2, and J.F. Lemaire2
Astron. Astrophys. 324, 725–734 (1997)
==

\unhide

%

\subsubsection*{\textbf{\large{Main Theorems}}} 
As a major goal of this paper, we construct a weak solution of RVM in a locally Lipschitz space, in which we can guarantee a \textit{uniqueness!} The major difficulty is that the density distribution $f_{\pm}$ is singular at the grazing set $\gamma_0$ in general. Notably, a solution is discontinuous at the grazing set $\gamma_0$ (\cite{Kim}), and a derivative $\nabla_{x,v} f_{\pm}$ blows up at $\gamma_0$ (\cite{Guo95, GKTT1}). If a trajectory emanating from the grazing set can propagate inside the domain (either if the domain is not convex or the field is repellent to the boundary) then such singularities propagate inside the domain and the regularity of solutions become restrictive (\cite{Kim, Guo96, GKTT2, KimLee}). In particular, following the proof of Guo-Kim-Tonon-Trescases \cite{GKTT1}, we can deduce that a global $H^1(\O)$ bound is not possible for solutions $f_\pm$ of \eqref{VMfrakF}-\eqref{rhoJ} \& \eqref{E12B3bc}, in general. Unfortunately, such low regularity hardly guarantees uniqueness due to the nonlinear term $\mathfrak{F}_\pm \cdot \nabla_v f_{\pm}$. To overcome such obstacle, we adopt a kinetic weight function $\alpha_\pm(t,x,v)$ in the regularity estimate of a locally Lipschitz space, inspired by \cite{Guo96, GKTT1}.
 \begin{definition}[Kinetic weight]\label{def:alpha}Recall the Lorentz force $\mathfrak{F}_{\pm}$ in \eqref{frakF} with $(E_{\text{ext}}, B_{\text{ext}})$ in \eqref{EBext}. We define
 	\Be \label{alphadef}
 	\begin{split}
 		\alpha_\pm (t, x_\parallel, x_3, v) :=  & \sqrt{(x_3)^2+(\hat{v}_{\pm,3})^2 -2 \mathfrak F_{\pm,3}(t,x_\parallel, 0 ,v) \frac{x_3}{\langle v_\pm \rangle}}
 		\\ =  & \sqrt{(x_3)^2+(\hat{v}_{\pm,3})^2 +2\bigg(  m_\pm g  - e_\pm \Big(E_3 +E_e +  \frac{1}{c} (\hat v_\pm  \times B)_3 \Big)_{x_3=0}  \bigg) \frac{x_3}{\langle v _\pm\rangle}},
 	\end{split}
 	\Ee
where we have used that  $( \hat v_{\pm } \times B_{\text{ext}} )_3 = 0$ for \eqref{EBext}.
 \end{definition}
\begin{remark}
	Clearly, $\alpha_\pm$ is well-defined when $-\mathfrak{F}_{\pm,3} (t, x_\parallel, 0, v)$ is positive. In this paper, we assume this condition on the initial data at the boundary:  
	\Be \label{E0B0g}
	m_\pm g -e_\pm \Big(E_{0,3}(x) +E_e +  \frac{1}{c} (\hat v_\pm  \times B_0(x))_3 \Big)_{x_3=0}   >  c_1, \ \ \text{for some } \   c_1>0.
	\Ee
	\end{remark}
	
	\begin{remark}
		The condition \eqref{E0B0g} is not very restrictive under the Pannekoek-Rosseland condition \eqref{PR}. Note that $-\mathfrak{F}_{\pm,3} (t, x_\parallel, 0, v)$ equals 
	\Be\notag
\underbrace{\big(  m_\pm g  - e_\pm  E_e   \big)}  \frac{x_3}{\langle v _\pm\rangle}
  - e_\pm \Big(E_{0,3}  +  ( \frac{\hat v_\pm}{c}  \times B_0)_3 \Big)_{x_3=0}   \frac{x_3}{\langle v _\pm\rangle}.
	\Ee
	If the Pannekoek-Rosseland condition \eqref{PR} holds then the underlined coefficients of the first term, which corresponds to the net force at the equilibrium, has lower bounds:
	\Be\notag
\big(  m_+ g  - e_+ E_e   \big) 	> \frac{m_+g }{2}, \ \ \ \big(  m_- g  - e_-  E_e   \big)>|e_-| E_e.
	\Ee 
From \eqref{PR_field}, we know that both lower bounds are of the same size. If $E_{0,3}|_{x_3=0}$ and $B_{0,1}|_{x_3=0}$, $B_{0,2}|_{x_3=0}$ are smaller than such lower bounds then the condition \eqref{E0B0g} holds. It is the case when the initial state of plasma is either close to the neutral state or vacuum at the boundary. 
\end{remark}

 
 \begin{remark}Since being introduced in \cite{Guo95}, such weight function $\alpha$ and its variants have served important roles in the regularity analysis for various kinetic equations with boundary such as \cite{CAO2, CAO3, CAO1, CK, CKL, CKLi, GKTT1, HV}. Notably in \cite{GKTT1}, an $\alpha$-weighted $C^1$ solution for the Boltzmann equation was constructed in convex domains. In \cite{CKL}, the authors used a different version of kinetic weight to construct the global strong solution to the Vlasov-Poisson-Boltzmann (VPB) system in convex domains with diffuse BC. The result was generalized to the two-species case in \cite{CAO3}, and to the case of the presence of the external field in \cite{CAO2}. A generalized diffuse boundary condition (namely the Cercignani-Lampis boundary condition) for the VPB system is studied in \cite{CKLi}. A survey on the recent development in this direction can be found in \cite{CK}. 

 	\end{remark}


Although the problem has been set already (RVM system \eqref{VMfrakF_I}-\eqref{rhoJ_I} under the perfect conductor boundary condition of electromagnetic field \eqref{E12B3bc_I}
), we list them here redundantly for the sake of the reader's convenience: Let $\O$ the half space \eqref{domain}. We read the RVM system 
\Be \label{VMfrakF}
\begin{split}
	\p_t f_\pm + \hat v_\pm \cdot \nabla_x f_\pm +   \mathfrak F_\pm \cdot \nabla_v f_\pm = 0,  & \ \  \text{ in } \R_+ \times \O \times \mathbb R^3 ,
	\\ f_\pm(0,x,v) =  f_{0,\pm}(x,v) ,  & \ \  \text{ in }     \O \times \mathbb R^3.
\end{split}
\Ee
with the relativistic velocity, Lorentz force, and the external fields
\begin{align}
\hat v_\pm &=  {v}\Big/ {\sqrt{m_\pm^2 + |v|^2 / c^2}},\label{rel_v}\\
\mathfrak F_\pm  &=  e_\pm \left(E+ E_{\text{ext}} +  \frac{\hat v_\pm}{c}  \times (B + B_{\text{ext}}) \right) - m_\pm g \mathbf e_3,\label{frakF}\\
E_{\text{ext}} &= E_e \mathbf e_3 \text{ and } \Bex = B_e \mathbf e_3.  \label{EBext}
\end{align}  
The Maxwell's equations solve
\Be \label{Maxwell}
\begin{split}
	\p_t E & = c  \nabla_x \times B - 4 \pi J, \, \ \  \  \nabla_x \cdot E = 4\pi \rho, \ \ \text{ in }  \R_+ \times  \O,
	\\ \p_t B & = - c  \nabla_x \times E, \, \ \  \  \nabla_x \cdot B = 0, \ \ \text{ in }  \R_+ \times  \O,
\end{split}
\Ee 
\Be \label{BEinitdata}
E(0,x) = E_0(x),  \ \ B(0,x) = B_0(x), \ \ \text{ in } \O.
\Ee
where the electric density and current are defined as
\Be \label{rhoJ}
\rho = \int_{\mathbb R^3} ( e_+f_+ + e_- f_-) \dd v,  \ \ \  J = \int_{\mathbb R^3}  (\hat v_+ e_+ f_+  +  \hat v_- e_- f_-) \dd v.
\Ee
Finally we impose the perfect conductor boundary condition
\Be \label{E12B3bc}
E_1 = E_2 = 0, \,  B_3 = 0 , \text{ on }  \  \R_+ \times  \p \O,
\Ee
and consider the inflow BC, diffuse BC, and specular BC on the incoming boundary $\gamma_-$:
\begin{align}
	f_\pm(t,x,v)  =  g_\pm(t,x,v) \ \ &\text{on} \  \R_+ \times  \gamma_-,
	\label{inflow}\\
	f_\pm(t,x,v )    = 
	\frac{1}{2\pi T_w^2}e^{- \frac{|v|^2}{2 T_w}}
	\int_{u_3 < 0 }  - f_\pm(t,x,u) \hat u_{\pm,3} du \ \ &\text{on} \  \R_+ \times  \gamma_-,
	\label{diffuseBC}\\
	f_\pm(t,x,v_\parallel, v_3 )   = f_\pm(t,x, v_\parallel, -v_3 )  \ \ &\text{on} \  \R_+ \times  \gamma_-
	.\label{spec}
	\end{align}

We define a notation of weak solutions to this initial-boundary value problem.
\begin{definition}[Definition 1.5 of \cite{Guo93}] \label{weaksoldef}
	Let $f_\pm \in L^1_{\text{loc} } ((0,T) \times \O \times \mathbb R^3 ) \cap   L^1_{\text{loc} } ((0,T) \times \gamma_+ ) $, $f_{0, \pm } \in L^1_{\text{loc}}( \O \times \mathbb R^3 ) $, $g \in L^1_{\text{loc} } ((0,T) \times \gamma_- )$. Let $E, B \in  L^1_{\text{loc}}( (0,T) \times  \O ) $, $E_0, B_0  \in L^1_{\text{loc} }(\O ) $. Then $(f_\pm, E, B)$ is a weak solution of \eqref{VMfrakF}-\eqref{rhoJ} under the perfect conductor boundary condition of electromagnetic field \eqref{E12B3bc} and different boundary conditions for $f_{\pm}$ \eqref{inflow}, \eqref{diffuseBC}, or \eqref{spec}, if for any test functions 
	\[
	\begin{split}
		& \phi(t,x,v) \in C_c^\infty([0,T) \times \O \times \mathbb R^3 ), \text{ with } \text{ supp } \phi  \subset \{ [0, T) \times \bar \O \times \mathbb R^3 \} \setminus \{ (\{0\} \times \gamma ) \cup (0,T) \times \gamma_0 \}, \text{ and }
		\\ & \Psi(t,x)  \in C_c^\infty([0,T) \times \bar \O ; \mathbb R^3 ), \    \Phi(t,x) \in C_c^\infty([0,T) \times \O ; \mathbb R^3 ), 
	\end{split}
	\]
	we have
	\Be \label{weakf}
	\begin{split}
		&  \iint_{\O \times \mathbb R^3 } f_{0, \pm} \phi(0) dv dx + \int_0^T \iint_{\O \times \mathbb R^3 } (\p_t \phi + \nabla \phi \cdot \hat v + \mathfrak F_{\pm} \cdot \nabla_v \phi ) f_\pm  dv dx dt
		\\ = & \int_0^T \int_{\gamma_+ } \phi f_\pm \hat v_3 d v dS_x + \underbrace{  \int_0^T \int_{\gamma_- } \phi f_\pm \hat v_3 d v dS_x }_{\eqref{weakf}_{\text{BC} } },
	\end{split}
	\Ee 
		and
	\Be \label{Maxweak1}
	\int_0^T \int_\O E \cdot \p_t \Psi dx dt - \int_\O \Psi(0,x) \cdot E_0 dx  = - \int_0^T \int_\O (\nabla_x \times \Psi) \cdot B dx dt + 4\pi \int_0^T \int_\O \Psi \cdot J dx dt,
	\Ee
	\Be \label{Maxweak2}
	\int_0^T \int_\O B \cdot \p_t \Phi dx dt + \int_\O \Phi(0,x) \cdot B_0 dx = \int_0^T \int_\O (\nabla_x \times \Phi) \cdot E dx dt,
	\Ee
	and
	\Be \label{nablaEBweak}
	\nabla \cdot E = 4 \pi \rho, \  \nabla \cdot B = 0 \text{ in the sense of distributions in } (0,T) \times \O \times \mathbb R^3.
	\Ee
	Here, the boundary term of \eqref{weakf} is determined by different boundary conditions:	
	\[
	\begin{split}
		\eqref{weakf}_{\text{BC} }  = 
		\begin{cases}
			\int_0^T \int_{\gamma_- } \phi g_\pm \hat v_3  \, d v dS_x, \text{ for the inflow BC } \eqref{inflow},
			\\   \int_0^T \int_{\gamma_+ }  \left( - \frac{1}{2\pi T_w^2}   \int_{u_3 > 0 } e^{- \frac{|u|^2}{2 T_w}} \phi (t,x,u )\hat u_3 du  \right) \hat v_3   f_\pm  \, dv dS_x, \text{ for the diffuse BC } \eqref{diffuseBC},
			\\ \int_0^T \int_{\gamma_+} \phi(t,x,v_\parallel, - v_3 ) f_\pm \hat v_3 \  dv dS_x, \text{ for the specular BC } \eqref{spec}.
		\end{cases}
	\end{split}
	\]

\end{definition}

Now we state the main theorems.

\begin{theorem} [inflow BC] \label{main1}
Suppose the initial datum $f_{0,\pm}$ satisfies, for some $\delta>0$,
\Be \label{f0bdd}
\begin{split}
& \| \langle v \rangle^{4 + \delta } f_{0,\pm} \|_{L^\infty(\O \times \mathbb R^3) }   + \| \langle v \rangle^{5 + \delta }   \nabla_{x_\parallel} f_{0,\pm} \|_{L^\infty(\O \times \mathbb R^3) }  
 \\ & +  \| \langle v \rangle^{5 + \delta }   \alpha_\pm \p_{x_3} f_{0,\pm} \|_{L^\infty(\O \times \mathbb R^3) } + \| \langle v \rangle^{5 + \delta }    \nabla_v f_{0,\pm} \|_{L^\infty(\O \times \mathbb R^3) }  < \infty,
 \end{split}
 \Ee
  and the inflow boundary datum $g_\pm$ satisfies
  \Be \label{inflowdata}
    \| \langle v \rangle^{5 + \delta }   \p_t g_\pm \|_{L^\infty( (0 , \infty) \times \gamma_-) } +   \| \langle v \rangle^{5 + \delta }   \nabla_{x_\parallel} g_\pm \|_{L^\infty (0 , \infty) \times \gamma_-) } +  \| \langle v \rangle^{5 + \delta }   \nabla_{v} g_\pm \|_{L^\infty( (0 , \infty) \times \gamma_-) }    < \infty.
\Ee
Moreover, $E_0,B_0, g$ satisfies \eqref{E0B0g}, and the compatibility conditions 
\Be \label{EBintialC}
\begin{split}
\nabla \cdot E_0 =  4 \pi \rho_0 , \, \nabla \cdot B_0 = 0, & \text{ in } \O,
\\ E_{0, \parallel }  = 0, \,  B_{0,3} = 0 , & \text{ on } \p \O,
\end{split}
\Ee
and
\Be \label{E0B0bdd}
\| E_0 \|_{C^2(\O ) } + \| B_0 \|_{C^2(\O) } < \infty.
\Ee
Then there exists a unique solution $f_\pm(t,x,v), E(t,x,v), B(t,x,v) $ for $0 \le t \le T$ with $T \ll 1 $ to RVM for the inflow BC \eqref{inflow} in the sense of Definition \ref{weaksoldef}, such that,
\Be \label{inflowfreg}
\begin{split}
  \sup_{0 \le t \le T} \Big( & \| \langle v \rangle^{4+\delta}   \nabla_{x_\parallel} f_\pm(t) \|_{L^\infty(\O \times \mathbb R^3 )}   + \| \langle v \rangle^{5+\delta}  \alpha_\pm \p_{x_3} f_\pm(t) \|_{L^\infty(\O \times \mathbb R^3 )}   \\
 &  +  \| \langle v \rangle^{5+\delta}   \nabla_{v} f_\pm(t) \|_{L^\infty(\O \times \mathbb R^3 )}  \Big) <  \infty,
  \end{split}
\Ee
and
\Be \label{inflowEBreg}
\sup_{0 \le t \le T} \left( \| \nabla_x E(t) \|_{L^\infty( \O ) }  + \| \nabla_x B(t) \|_{L^\infty( \O ) } \right)  < \infty. 
\Ee
\end{theorem}

\begin{theorem} [diffuse BC] \label{main2}
Suppose $f_{0,\pm}$ satisfies \eqref{f0bdd}, and $E_0, B_0, g$ satisfy \eqref{E0B0g}, \eqref{EBintialC}, and \eqref{E0B0bdd}.  Then there exists a unique solution $f_\pm(t,x,v), E(t,x,v), B(t,x,v) $ for $0 \le t \le T$ with $T \ll 1 $ to RVM for the diffuse BC \eqref{diffuseBC} in the sense of Definition \ref{weaksoldef}, such that both \eqref{inflowfreg} and \eqref{inflowEBreg} hold. 
\hide\Be \label{inflowfreg}
 \sup_{0 \le t \le T} \left(  \| \langle v \rangle^{4+\delta}   \nabla_{x_\parallel} f_\pm(t) \|_{L^\infty(\O \times \mathbb R^3 )}   + \| \langle v \rangle^{5+\delta}  \alpha_\pm \p_{x_3} f_\pm(t) \|_{L^\infty(\O \times \mathbb R^3 )}   +  \| \langle v \rangle^{5+\delta}   \nabla_{v} f_\pm(t) \|_{L^\infty(\O \times \mathbb R^3 )}  \right) <  \infty,
\Ee
and
\Be \label{inflowEBreg}
 \sup_{0 \le t \le T} \left( \| \nabla_x E(t) \|_{L^\infty( \O ) }  + \| \nabla_x B(t) \|_{L^\infty( \O ) } \right)  < \infty. 
\Ee\unhide
\end{theorem}

\begin{theorem} [specular BC] \label{main3}
Suppose $f_{0,\pm}$ satisfies
\Be \label{f0spec}
\begin{split}
 \| \langle v \rangle ^{5 + \delta }    e^{\frac{C}{\sqrt{ \alpha_\pm \langle v \rangle } } }   \nabla_{x} f_{0,\pm} \|_\infty +  \| \langle v \rangle^{5 + \delta }    e^{\frac{C}{\sqrt{ \alpha_\pm \langle v \rangle } } }      \nabla_v f_{0,\pm} \|_\infty  < \infty,
\end{split}
\Ee
for some $C> 0$. $E_0$, $B_0$ satisfy \eqref{E0B0g}, \eqref{EBintialC}, and \eqref{E0B0bdd}. Then there exists a unique solution $f_\pm(t,x,v)$, $E(t,x,v)$, $B(t,x,v) $ for $0 \le t \le T$ with $T \ll 1 $ to the RVM with system \eqref{spec} in the sense of Definition \ref{weaksoldef}, such that
\Be \label{specfbd}
 \sup_{0 \le t \le T} \left(  \| \langle v \rangle^{4+\delta}   \nabla_{x} f_\pm (t) \|_{L^\infty(\O \times \mathbb R^3 )}   + \| \langle v \rangle^{4+\delta}   \nabla_{v} f_\pm (t) \|_{L^\infty(\O \times \mathbb R^3 )}  \right) <  \infty,
\Ee
and \eqref{inflowEBreg} holds.
\end{theorem}

\begin{remark}
	A large class of functional spaces satisfy the condition \eqref{f0bdd}. Indeed any function, whose weak derivatives $\nabla_{x,v} f$ are bounded in $L^\infty(\O \times \mathbb R^3)$, and decays fast enough as $|v| \to \infty$, belongs to the space of \eqref{f0bdd}. Actually $\p_{x_3} f_{0, \pm}$ is allowed to be singular at the grazing set $\gamma_0$.
\end{remark}

\begin{remark}
	As far as the authors know, Theorem \eqref{main1}- \eqref{main3} provide the first unique solvability of the RVM system when the physical boundary has a global effect (cf. \cite{Zhang2,JSW}). The time span $T$ of existence depends on the size of the initial data $f_0$, $E_0$, $B_0$ and their derivatives, and $c_1$ in \eqref{E0B0g}. 
\end{remark}

\begin{remark}We prove the weighted regularity estimate using a Lagrangian approach of \cite{GKTT1}. We take a direct differentiation to the Lagrangian solution along the generalized characteristics. The generalized characteristics depend on the boundary condition. 
	
\end{remark}

\begin{remark}
Here we require that the initial data $f_0$ to vanish exponentially fast towards the grazing set \eqref{specfbd}. This allows us to establish the regularity estimate for specular BC \eqref{specfbd} with the $W^{1,\infty}$ field. We prove this theorem in section \ref{chapspec}.
\end{remark}







\subsubsection*{\textbf{\large{Difficulties and Key Ingredients}}} \label{diffidea}
The problem in this paper is a coupled system of hyperbolic equations and kinetic Vlasov equation with characteristic boundary condition: the problem suffers \textit{a loss of derivative of wave equation} (cf. \cite{CS}) and \textit{the boundary singularity of Vlasov equation} (cf. \cite{GKTT1, Guo95}) at the same time. The key difficulty in the construction of a \textit{unique} solutions of the RVM system with physical boundary conditions is a control of the nonlinear term $\mathfrak{F}_\pm \cdot \nabla_v f_\pm$. We overcome this difficulty by establishing a regularity estimate for both the electromagnetic field $E$ and $B$, and the density distribution $f_\pm$ using the Glassey-Strauss representation. We detail several key difficulties along the road. In this section and the rest of the paper, for sake of simplicity, we will consider the one-species relativistic Vlasov-Maxwell-system since the analysis of the two-species case does not process essential difference from that of the one species case:
\Be \label{VMfrakF1}
\begin{split}
	\p_t f + \hat v \cdot \nabla_x f +   \mathfrak F \cdot \nabla_v f = 0,  & \text{ in } \R_+ \times \O \times \mathbb R^3 ,
	\\ f(0,x,v) =  f_{0}(x,v) ,  & \text{ in } \O \times \mathbb R^3,
\end{split}
\Ee
We also set all the charge and mass of the plasma $f$ equal to one, so here $\hat v = \frac{1}{\sqrt{1 +|v|^2 } } $, and
\Be \label{frakF1}
\mathfrak F  =   \left(E + E_{\text{ext}}+  \frac{\hat v}{c}  \times (B + B_{\text{ext}}) \right) -  g \mathbf e_3,
\Ee 
and $E,B$ satisfies the Maxwell equations \eqref{Maxwell}, with
\Be \label{rhoJ1}
\rho = \int_{\mathbb R^3} f dv, \, J = \int_{\mathbb R^3}  \hat v f dv.
\Ee

\subsubsection*{$\bullet$\textbf{Wave equation and the Neumann BC}} From the Maxwell's equations \eqref{Maxwell}, we have the inhomogenous wave equations for $E$ and $B$:
\begin{align}
\p_t^2 E - \Delta_x E = -4\pi \nabla_x \rho - 4\pi \p_t J,  \  &\text{ in } \  \R_+ \times \O,\label{wave_eq_E}\\ 
\p_t^2 B - \Delta_x B = 4 \pi \nabla_x \times J, \  &\text{ in } \  \R_+ \times \O,\label{wave_eq_B}
\end{align}
with the boundary condition $E_1 = E_2 = 0, B_3 = 0$ on $\p\O$ in \eqref{E12B3bc} and the initial condition
\Be\label{initialC}
\begin{split}
	E|_{t=0} = E_0,\ \ \p_t E|_{t=0} = \p_t E_0:= \nabla_x \times B_0 - 4\pi J|_{t=0}, \text{ in } \O
	,
	\\
	B|_{t=0} = B_0,\ \ \p_t B|_{t=0} = \p_t B_0:=  - \nabla_x \times E_0, \text{ in } \O
	.
\end{split}
\Ee
The boundary conditions of $E_3, B_\parallel$ components are not a priori given, which causes some trouble handling weak solutions based on the Glassey-Strauss representation. Of course, if the fields $E,B \in C^2 (\O) \cap C^1( \bar \O)$, and $\rho \in C^1(\O ) \cap C(\bar \O)$, then from the Maxwell's equations \eqref{Maxwell} and the perfect conductor boundary condition \eqref{E12B3bc}, we deduce the Neumann boundary condition
\Be \label{E3B1B2bc}
\p_3 E_3 =  4\pi \rho, \p_3 B_1 = 4\pi J_2, \  \p_3 B_2 = - 4\pi J_1  \ \text{ on } \R_+ \times \p \O.
\Ee
\hide\Be
\p_{3} E_3 |_{\p \O} = 4 \pi \rho |_{\p \O }  + \p_1 E_1  |_{\p \O } + \p_2 E_2  |_{\p \O } = 4 \pi \rho|_{\p \O},
\Ee
and
\Be
\begin{split}
	\p_3 B_2  |_{\p \O} = &  \p_2 B_3 |_{\p \O} - \p_t E_1 |_{\p \O} - 4\pi J_1 |_{\p \O} = - 4 \pi J_1 |_{\p \O},
	\\ \p_3 B_1  |_{\p \O} = & \p_1 B_3 |_{\p \O} + \p_t E_2 |_{\p \O} + 4\pi J_2 |_{\p \O} = - 4 \pi J_2 |_{\p \O}.
\end{split}
\Ee 
Therefore $E_3$ and $B_\parallel$ satisfy the following Neumann boundary condition
\Be \label{E3B1B2bc}
\p_3 E_3 =  4\pi \rho, \p_3 B_1 = 4\pi J_2, \  \p_3 B_2 = - 4\pi J_1  \ \text{ on } \p \O.
\Ee\unhide
One of the main goals in this paper is to equip a solution space in which we can realize the Neumann BC \eqref{E3B1B2bc} in a suitable sense and hence guarantee a unique solvability. Indeed we can justify the Neumann boundary condition \eqref{E3B1B2bc} in a weak solution formulation testing
against smooth test functions that do not vanish at the boundary $\p \O$ in Lemma \ref{Maxtowave}, and prove the uniqueness of weak solution. 
We then carefully show in Lemma \ref{wavetoMax} that, assuming the continuity equation 
\[
\nabla \cdot J + \p_t \rho = 0,
\]
and some compatibility conditions of the initial datum \eqref{EBintialC}, the weak solution of wave equations with boundary conditions \eqref{wave_eq_E}-\eqref{E3B1B2bc} is indeed a solution to the Maxwell equations. This equivalence allows us to solve the RVM system by looking for solutions to the wave equations with boundary conditions \eqref{wave_eq_E}, \eqref{wave_eq_B}, which is the first step of our analysis.


\subsubsection*{$\bullet$\textbf{Glassey-Strauss Representation in the half space}}  
The wave equations \eqref{wave_eq_E}, \eqref{wave_eq_B} suffer from the ``loss of derivatives'' of $(E,B)$ with respect to the regularity of the source terms $\rho$ and $J$. As Glassey mentions in his book \cite{Glassey}, the key idea of the Glassey-Strauss representation is replacing the derivatives $\p_t, \nabla_x$ by a geometric operator $T$ in \eqref{def:T1} and a kinetic transport operator $S$ in \eqref{def:S1}:
\Be \label{pxptST}
\begin{split}
	\p_t  = \frac{S- \hat{v} \cdot T}{1+ \hat v \cdot \o}, \ \
	\p_i  = \frac{\o_i S}{ 1+ \hat v \cdot \o} + \left( \delta_{ij} - \frac{\o_i \hat{v}_j}{1+ \hat v \cdot \o}\right) T_j,
\end{split}\Ee
while, for $\o= \o(x,y) = \frac{y-x}{|y-x|}$,
\begin{align}
	T_i &:= \p_i - \o_i \p t, \label{def:T1}\\
	S &:= \p_t + \hat v \cdot \nabla_x. \label{def:S1}
\end{align}
Note that
\Be\label{T=y}
T_j f (t- |y-x| , y, v ) = \p_{y_j } [ f(t- |y-x|, y, v ) ],
\Ee
which is a tangential derivative along the surface of a backward light cone \cite{Glassey}. On the other hand, the Vlasov equation \eqref{VMfrakF1} implies that
\Be\label{S=Lf_v}
Sf= - \nabla_v \cdot [ (E + E_{\text{ext}} + \hat v \times ( B + B_{\text{ext}})- g \mathbf e_3)f].
\Ee
Therefore, in \cite{GS,Glassey}, they can take off the derivatives $T_j$, $S$ from $f$ using the integration by parts within the Green's formula of \eqref{wave_eq_E}--\eqref{wave_eq_B} by connecting the source terms to $f$ via \eqref{rhoJ1}. 

For our problem, we derive the Glassey-Strauss representation in the presence of a boundary. For $E_\parallel$ and $B_3$, to solve the wave equation with the Dirichlet boundary condition \eqref{E12B3bc}, we employ the odd extension of the initial data and the forcing term into the lower half space $\mathbb R^3_- : = \{ (x_1, x_2, x_3 ) \in \mathbb R^3: x_3 < 0 \} $. Solving the whole space wave equation with the oddly extended data gives us the solution that satisfies \eqref{E12B3bc}. On the other hand, for $E_3$ and $B_\parallel$, we decompose the solution into two parts: one with the Neumann boundary condition of \eqref{E3B1B2bc} and the zero forcing term and initial data, and the other part satisfying \eqref{wave_eq_E}-\eqref{initialC} with the zero Neumann boundary condition. For the first part, we find out the expression using the fundamental solution of the Helmholtz equation. And for the second part, we use the even extension to get the solution with the zero Neumann boundary condition.

For the expression in the lower half space, we introduce
\Be\label{def:o-}
\bar \o   =  \begin{bmatrix} \o_1 & \o_2  & - \o_3 \end{bmatrix} ^T,
\Ee
and
\Be\label{def:T-}
\begin{split}
	\bar{T}_3 f &=  -\p_{y_3} [f(t-|y-x|, y_\parallel, - y_3, v)] =   \p_{y_3} f - \bar{\omega}_3 \p_t f ,
	\\ \bar T _i f &=  \p_{y_i} [f(t-|y-x|, y_\parallel, - y_3, v)] =   \p_{y_i} f - \bar{\omega}_i \p_t f \, \text{ for } \ i=1,2.
\end{split}
\Ee
Then through direct calculation we obtain an explicit expressions of $E$ and $B$ by solving the wave equations \eqref{wave_eq_E}--\eqref{initialC} under the boundary condition \eqref{E12B3bc} and \eqref{E3B1B2bc} in Proposition \ref{Eiform} and Proposition \ref{Biform} respectively.

\subsubsection*{$\bullet$\textbf{Weighted $W^{1,\infty}$ estimate of $f$ and the regularity of the fields}} An intrinsic feature of the transport equation in domains with boundary is the singular behavior of its derivatives: the solution of a linear transport equation with physical boundaries is known to not have high regularity \cite{GKTT1}. However, to get the unique solvability, one must control $\nabla_v f $ effectively. This in turn requires the control of spatial derivatives of the distribution function and the spatial derivatives of electromagnetic fields. But due to the characteristic boundary, $f$ does not have high enough regularity to achieve the required regularity for $E$ and $B$ directly from the hyperbolic equations. We explain the ideas of the paper and the methods we use to overcome the difficulties in the rest of this section and the next.


 Let's consider a solution of the RVM system with inflow boundary data \eqref{VMfrakF1}-\eqref{rhoJ1}, \eqref{inflow}. The \textit{characteristics (trajectory)} is determined by the Hamilton ODEs, 
\Be \label{HamiltonODE}
\begin{split}
	\frac{d}{ds} X(s;t,x,v)&= \hat{V}(s;t,x,v),\\
	\frac{d}{ds} V(s;t,x,v)&= \mathfrak F (s,X(s;t,x,v),V(s;t,x,v)).
\end{split}
\Ee
We define \textit{the backward exit time} $\tb(t,x,v)$ as   
\Be\label{tb}
\tb (t,x,v) := \sup \{s \geq 0 : X(\tau;t,x,v) \in \O \ \ \text{for all } \tau \in (t-s,t) \}.
\Ee
Furthermore, we define $\xb (t,x,v) := X(t-\tb(t,x,v);t,x,v)$, and $\vb(t,x,v) := V(t-\tb(t,x,v);t,x,v)$. We can solve the Vlasov equation \eqref{VMfrakF1} with the inflow boundary condition \eqref{inflow} as
\[
f(t,x,v) = g( t -\tb(t,x,v) ,  X(t-\tb(t,x,v);t,x,v), V(t-\tb(t,x,v);t,x,v) ) \text{ for } t \ge \tb(t,x,v).
\]
From some direct computations (see \eqref{pxbvb} and \eqref{pxiF}), the derivatives of $f$ have a bound in general as
\[
\nabla_x f(t,x,v) \sim \nabla_x \tb(t,x,v),
\]
which can be further bounded from the direct computation of the characteristics (see \eqref{pxitb}) as
\Be \label{dtbbd}
\nabla_x \tb(t,x,v) \lesssim  \frac{1 + \sup_{0 \le s \le t} \| \nabla_x \mathfrak F(s) \|_\infty }{\hat{v}_{\mathbf{b},3}}.
\Ee
The formation of such singularity motivates us to introduce the following notion. As a first order approximation of $\hat{v}_{\mathbf{b},3}(t,x,v)$, we define the kinetic weight
\Be \label{alphadef}
\begin{split}
	\alpha(t, x_\parallel, x_3, v) :=  & \sqrt{(x_3)^2+(\hat{v}_{3})^2 -2 \mathfrak F_3(t,x_\parallel, 0 ,v) \frac{x_3}{\langle v \rangle}}
	\\ =  & \sqrt{(x_3)^2+(\hat{v}_{3})^2 -2\left( E_3(t,x_\parallel, 0 ) + E_e + (\hat v \times B)_3(t,x_\parallel, 0 )  -  g \right) \frac{x_3}{\langle v \rangle}}.
\end{split}
\Ee
Note that $\alpha = \hat v_3$ on $\p \O$. Crucially $\alpha$ is effectively invariant along the characteristics, thanks to the velocity lemma (Lemma \ref{vlemma}). This allows us to prove an $\alpha$-weighted bound on the derivatives of $f$, more specifically, we prove that for any $0 < \delta < 1$,
\Be
\langle v \rangle^{5 + \delta} \alpha \nabla_x f (t) \in L^\infty( \O \times \mathbb R^3 ), \text{ for } 0< t  < T.
\Ee

On the other hand, due to the generic singularity \eqref{dtbbd}, to close the estimate we need to bound $\nabla_x E$, $\nabla_x B$ by $ \langle v \rangle^{5 + \delta} \alpha \nabla_x f $ in the generalized Glassey-Strauss representation. By taking the derivatives directly to the formulas of $E$ and $B$, in Lemma \ref{EBW1inftylemma} we achieve the bound
\[
|  \nabla_x E | +  | \nabla_x B | \lesssim ``\text{ initial data }" + \sup_{0 \le t \le T}  \| \langle v \rangle^{5 + \delta} \alpha \p_{x_3} f (t) \|_\infty  \int_{ \O \cap \{ |x| < T \}} \int_{ \mathbb R^3} \frac{1}{ \langle v \rangle^{4 + \delta} \alpha (t,x,v) } dv dx.
\]  
Then from the local-to-nonlocal estimate (Lemma \ref{1alphaintv}), we derive 
\[
\int_{ \mathbb R^3} \frac{1}{ \langle v \rangle^{4 + \delta} \alpha (t,x,v) } dv \lesssim \ln(1+ \frac{1}{x_3} )  \in L^1_{\text{loc} } (\O),
\]
and we are able to close the estimate and conclude $E, B \in W^{1,\infty}((0,T) \times \O ).$

In the construction of solution, we study a sequence of solutions $(f^\ell, E^\ell, B^\ell )$ and pass the limit. To achieve a uniform estimate, we use a weight $\alpha^\ell(t,x,v)$, which is the same form of \eqref{alphadef} with exchanging $E,B$ to $E^\ell, B^\ell$. Since $\alpha^\ell$ depends on $E^\ell, B^\ell$ and hence $f^\ell$, when passing the limit of the sequence $\{ \alpha^{\ell-1} \p_{x_3} f^\ell \}_{\ell = 1}^\infty$, we need to verify that
\Be \label{alphp3flim}
\alpha^{\ell-1} \p_{x_3} f^\ell \overset{\ast}{\rightharpoonup} \alpha \p_{x_3} f \text{ in } L^\infty((0,T) \times \O \times \mathbb R^3 ).
\Ee
Obviously this convergence is nontrivial since the norm itself is nonlinear. To obtain this, we observe that since we can bound $\nabla E^\ell , \nabla B^\ell$ pointwisely, they have traces and a strong convergence 
\[
E^\ell |_{\p \O} \to E |_{\p \O} , \ B^\ell|_{\p \O} \to B |_{\p \O}. 
\]
Thus we can prove that a strong convergence $\alpha^{\ell-1}  \to \alpha $ in $L^\infty$. 
On the other hand, using a positive lower bound of $\alpha^\ell$ away from the grazing set, we obtain a upper uniform bound of $\p_{x_3} \alpha^{\ell-1}- \p_{x_3} \alpha$ locally. We then achieve the desired convergence \eqref{alphp3flim} using uniform bound of $f^\ell$. 
We refer to Lemma \ref{fEBregin} for more details.

Among other boundary conditions, we find that the specular boundary condition suffers most. Due to the lack of higher regularity of the fields (e.g. compare to \cite{CAO1} where the field is $C^2$), we can only derive an exponential-in-$\alpha$ singularity of the derivative of trajectory
\begin{equation}\label{lemma_Dxv1}
	\begin{split}
		| \p_{\mathbf e } X_{\mathbf{cl}}(s;t,x,v)| & \le C_1 \langle v \rangle  e^{\frac{C_1}{\sqrt{ \alpha(t,x,v) \langle v \rangle } } }   , 
		\\ | \p_{\mathbf e } V_{\mathbf{cl}}(s;t,x,v)| &  \le  C_1 \langle v \rangle  e^{\frac{C_1}{\sqrt{ \alpha(t,x,v) \langle v \rangle } } }.
	\end{split}
\end{equation}
Clearly, such a strong singularity can be harmful in our analysis based on the Glassey-Strauss representation. We study the specular BC problem with great care. Details are presented in section \ref{chapspec}.

\subsubsection*{$\bullet$\textbf{A Priori $L^\infty$ estimate of $\nabla_v f$ and uniqueness}} A simple Gronwall's inequality implies  
\Be \label{fgdiff}
\| \langle v \rangle^{4+ \delta } (f- g ) (t) \|_\infty \lesssim \| \langle v \rangle^{4+ \delta } (f- g )(0) \|_\infty + \sup_{0 \le t \le T } \|  \langle v \rangle^{4+ \delta } \nabla_v f(t) \|_\infty \int_0^t \|( E_f - E_g + B_f - B_g)(s) \|_\infty.
\Ee 
For constructing a solution and proving its uniqueness, we establish an effective stability estimate of the difference of solutions $f-g$, and $E_f - E_g$, $B_f - B_g$. To control the nonlinear term of the equation of $f-g$, $(E_f - E_g + B_f - B_g) \cdot \nabla_v f $, we establish an estimate of $\nabla_v f $. From the Lagrangian view point along the characteristics \eqref{HamiltonODE}, we have
\[
\begin{split}
& \nabla_v f(t,x,v) 
\\ & \sim \nabla_{x} f_0(X(0;t,x,v) , V(0;t,x,v) )  \cdot \nabla_v X(0;t,x,v) +  \nabla_{v} f_0(X(0;t,x,v) , V(0;t,x,v) )  \cdot \nabla_v V(0;t,x,v).
\end{split}
\]
Clearly effective control of $\nabla_x E$, $\nabla_x B$ is necessary. 
Now we crucially use our estimate of $\alpha \p_{x_3} f $ to obtain bounds for $\nabla_x E$, $\nabla_x B$ in the Glassey-Strauss representation, which in turn gives the a priori $L^\infty$ estimate of $\nabla_v f$. Using this $\nabla_v f$-bound and a pointwise bound from the Glassey-Strauss representation (Lemma \ref{EBlinflemma})
\[
\| E_{f-g}(t) \|_\infty +  \| E_{f-g}(t) \|_\infty \lesssim \sup_{0 \le s \le t } \| \langle v \rangle^{4+ \delta } ( f - g)(s) \|_\infty,
\]
we achieve 
an $L^\infty$ stability as $
\sup_{0 \le s \le t } \|  \langle v \rangle^{4+ \delta } (f- g ) (s) \|_\infty \lesssim e^{Ct } \| \langle v \rangle^{4+ \delta } (f- g )(0) \|_\infty.$

\hide

On the other hand, notice that $E_f - E_g = E_{f-g}$, $B_f - B_g = B_{f-g}$ solves the Maxwell equations 
\Be \label{Maxwellfg}
\begin{split}
\p_t E_{f-g} & = \nabla_x \times B_{f-g} - 4 \pi J_{f-g}, \, \nabla_x \cdot E_{f-g} = 4\pi \rho_{f-g},
\\ \p_t B_{f-g} & = - \nabla_x \times E_{f-g}, \, \nabla_x \cdot B_{f-g} = 0.
\end{split}
\Ee
Then Lemma \ref{Maxtowave} implies that $E_{f-g}$ and $B_{f-g}$ solves the wave equation
\Be
\begin{split}
\p_t^2 E_{f-g} - \Delta_x E_{f-g} = & -4\pi \nabla_x \rho_{f-g} - 4\pi \p_t J_{f-g},  
\\ \p_t^2 B_{f-g} - \Delta_x B_{f-g} =  &4 \pi \nabla_x \times J_{f-g},
\end{split}
\Ee
with
\Be
\begin{split}
& E_{f-g, \parallel } =0,  B_{f-g,3} = 0,  \text{ on } \p \O
\\ & \p_3 E_{f-g, 3} = 4 \pi \rho_{f-g}, \p_3 B_{f-g, 1 } = 4 \pi J_{f-g,2} , \p_3 B_2 = -4 \pi J_{f-g,1}, \text{ on } \p \O,
\end{split} 
\Ee
in the sense of \eqref{waveinner}, \eqref{waveD_weak}. Therefore, from Glassey-Strauss representation taking into account of the boundary, $E_{f-g}$ and $B_{f-g}$ has the form in Proposition \ref{Eiform} and Proposition \ref{Biform} with $f$ changes to $f-g$ everywhere. This allows us to bound (same as in Lemma \ref{EBlinflemma})
\[
\| E_{f-g}(t) \|_\infty +  \| E_{f-g}(t) \|_\infty \lesssim \sup_{0 \le s \le t } \| \langle v \rangle^{4+ \delta } ( f - g)(s) \|_\infty. 
\]
Together with \eqref{fgdiff}, we use Gronwall's inequality to get the $L^\infty$ stability
\[
\sup_{0 \le s \le t } \|  \langle v \rangle^{4+ \delta } (f- g ) (s) \|_\infty \lesssim e^{Ct } \| \langle v \rangle^{4+ \delta } (f- g )(0) \|_\infty.
\] \unhide

\hide

\subsubsection*{$\bullet$\textbf{Specular BC}} Among other boundary conditions, we find that the specular boundary condition suffers most. Due to the lack of higher regularity of the fields (e.g. compare to \cite{CAO1} where the field is $C^2$), we can only derive an exponential-in-$\alpha$ singularity of the derivative of trajectory
\begin{equation}\label{lemma_Dxv1}
	\begin{split}
		| \p_{\mathbf e } X_{\mathbf{cl}}(s;t,x,v)| & \le C_1 \langle v \rangle  e^{\frac{C_1}{\sqrt{ \alpha(t,x,v) \langle v \rangle } } }   , 
		\\ | \p_{\mathbf e } V_{\mathbf{cl}}(s;t,x,v)| &  \le  C_1 \langle v \rangle  e^{\frac{C_1}{\sqrt{ \alpha(t,x,v) \langle v \rangle } } }.
	\end{split}
\end{equation}
We achieve 

The proof of the theorem relies on the estimate of taking the derivatives to $f$ under the generalized characteristics for the specular reflection, i.e. the ``specular cycles",  $X_{\mathbf{cl}}(s;t,x,v) , V_{\mathbf{cl}}(s;t,x,v)  $ as in \eqref{cycles}. One has
\Be \label{pef1}
\begin{split}
	\p_{\mathbf e }  f(t,x,v ) & = \p_{\mathbf e} (f(0, X_{\mathbf{cl}}(0;t,x,v) , V_{\mathbf{cl}}(0;t,x,v) ) 
	\\ & = \nabla_x f_0 \cdot \p_{\mathbf e } X_{\mathbf{cl}}(0;t,x,v) +  \nabla_v f_0 \cdot \p_{\mathbf e } V_{\mathbf{cl}}(0;t,x,v).
\end{split}
\Ee

 Note that here we have a high singularity in terms of $\alpha$ due to the fact of lacking higher regularity of the fields (e.g. compare to \cite{CAO1} where the field is $C^2$). 

The proof of the theorem relies on the estimate of taking the derivatives to $f$ under the generalized characteristics for the specular reflection, i.e. the ``specular cycles",  $X_{\mathbf{cl}}(s;t,x,v) , V_{\mathbf{cl}}(s;t,x,v)  $ as in \eqref{cycles}. One has
\Be \label{pef1}
\begin{split}
\p_{\mathbf e }  f(t,x,v ) & = \p_{\mathbf e} (f(0, X_{\mathbf{cl}}(0;t,x,v) , V_{\mathbf{cl}}(0;t,x,v) ) 
\\ & = \nabla_x f_0 \cdot \p_{\mathbf e } X_{\mathbf{cl}}(0;t,x,v) +  \nabla_v f_0 \cdot \p_{\mathbf e } V_{\mathbf{cl}}(0;t,x,v).
\end{split}
\Ee
Then from the crucial estimate of the derivatives of the generalized characteristics in Lemma \ref{dXVcl}, for $\p_{\mathbf  e } \in \{ \nabla_x , \nabla_v \}$, we have
\begin{equation}\label{lemma_Dxv1}
\begin{split}
| \p_{\mathbf e } X_{\mathbf{cl}}(s;t,x,v)| & \le C_1 \langle v \rangle  e^{\frac{C_1}{\sqrt{ \alpha(t,x,v) \langle v \rangle } } }   , 
\\ | \p_{\mathbf e } V_{\mathbf{cl}}(s;t,x,v)| &  \le  C_1 \langle v \rangle  e^{\frac{C_1}{\sqrt{ \alpha(t,x,v) \langle v \rangle } } }.
\end{split}
\end{equation}
Note that here we have a high singularity in terms of $\alpha$ due to the fact of lacking higher regularity of the fields (e.g. compare to \cite{CAO1} where the field is $C^2$). From \eqref{lemma_Dxv1} and the strong decay of the initial data $f_0$ towards the grazing set in \eqref{f0spec}, we can control $\nabla_x f, \nabla_v f $ in an $L^\infty$ space. Details are presented in section \ref{chapspec}.

\unhide


\bigskip

\subsubsection*{\textbf{{Acknowledgements.}}} This project is supported in part by National Science Foundation under Grant No. 1900923 and 2047681 (NSF-CAREER). CK was supported by Brain Pool program funded by the Ministry of Science and ICT through the National Research Foundation of Korea (2021H1D3A2A01039047), and he thanks Professor Seung Yeal Ha for the kind hospitality of the BP program.

\tableofcontents

\section{Uniqueness of the Maxwell equations}
In this section, we consider the uniqueness of solution to the Maxwell equations in $(0,T) \times \O$ in a presence of free charge:
\begin{align}
 \label{Maxwell1} \p_t E &  = \nabla_x \times B - 4 \pi J,
\\ \label{Maxwell2}  \p_t B & = - \nabla_x \times E,
\\ \label{Maxwell3}  \nabla_x \cdot E&  = 4 \pi \rho,
\\  \label{Maxwell4} \nabla_x \cdot B  &= 0,
\end{align}
with initial condition:
\Be \label{EBat0}
E(0,x) = E_0(x), \ B(0,x) = B_0(x) \text{ in } \O,
\Ee
and the perfect conductor boundary condition:
\Be \label{percond}
E_1  = E_2 = 0 \text{ on } \p \O, \ B_3 = 0  \text{ on } \p \O.
\Ee
It is worth to recall that the boundary conditions for $E_3$ and $B_1, B_2$ are not given originally. 

\begin{definition}\label{sol_M}
For given $E_0, B_0, \rho , j$, we say functions
\Be \label{EBspace}
E(t,x), B(t,x) \in W^{1,\infty}((0,T) \times \O),
\Ee 
is a solution to the equations \eqref{Maxwell1}-\eqref{percond} if \eqref{Maxwell1}-\eqref{Maxwell4} holds almost everywhere in $(0,T) \times \O$, and \eqref{EBat0}, \eqref{percond} holds in the sense of trace.
\end{definition}
\begin{remark}
Traces of $W^{1,\infty}((0,T) \times \O)$ are well-defined in a classical sense since any uniformly continuous function in space and time can be extended up to $[0,T] \times \bar\O$. 
	\end{remark}

The goal of this section is to prove the following uniqueness result:
\begin{theorem}\label{EBunique}
Suppose $E_0(x), B_0(x) \in W^{1,p}(\O) $, and $\nabla_x \rho, \nabla_x J,  \p_t J \in  L^\infty((0,T); L_{\text{loc}}^p(\O))$ for some $p>1$, and 	\Be \label{conteq}
\nabla \cdot J = - \p_t \rho.
\Ee Then a solution $E(t,x), B(t,x) \in W^{1,\infty}((0,T) \times \O)$ to the equations \eqref{Maxwell1}-\eqref{percond} in the sense of Definition \ref{sol_M} is unique.
%
%
\end{theorem}

The key of proof is to realize $E_3, B_1, B_2$ as weak solutions of inhomengenous wave equations with the Neumann boundary condition from a weak solution of Maxwell equations in the sense of Definition \ref{sol_M}. 


\begin{definition}\label{def:weak_wave}
Given any $u_0, u_1 : \O \to \mathbb R $, $G: (0,T) \times \O \to \mathbb R$, and $g : (0,T) \times \p \O \to \mathbb R$, we define a function $u(t) \in W^{1,p}(\O)$ for $t \in (0,T)$, $p > 1$ to be a weak solution of the inhomengenous wave equation with Neumann boundary condition:
\Be \label{waveNeu}
\begin{split}
\p_t^2 u - \Delta_x u &= G,
\\ -\p_{x_3} u(t,x) |_{\p \O } &= g, 
\\ u(0,x) = u_0, \, \p_t u(0,x ) & = u_1,
\end{split}
\Ee
if  for any $\phi \in C_c^\infty ([0,T) \times \bar \O )$, we have
\Be \label{waveinner}
\begin{split}
&\langle u, \phi \rangle_N 
\\ : =  & \int_{\O} ( u_1 (x) \phi(0,x) - u_0 (x) \p_t \phi(0,x)  )  dx + \int_0^T \int_\O  u(t,x) \left( \p_t^2 \phi(t,x)  - \Delta_x  \phi(t,x) \right)  dx dt  
\\ & - \int_0^T \int_{\p \O} u(t,x_\parallel, 0 ) \p_{x_3} \phi(t,x_\parallel, 0 ) dx_\parallel dt   + \int_0^T \int_{\p \O } g(t,x_\parallel ) \phi(t, x_\parallel, 0 )   dx_\parallel dt  - \int_0^T \int_\O  G \phi  \,  dx dt
\\  = &  0 ,
\end{split}
\Ee
and each terms in \eqref{waveinner} are all bounded. Note that since $u(t) \in W^{1,p}(\O)$, it has a trace $u(t) |_{\p \O} \in L^p (\p \O)$. Also, note that $\text{supp} ( \phi  ) \subset [0,T) \times \bar \O $ is compact, but $ \phi |_{t = 0 }  \neq 0$, and $\phi |_{\p \O } \neq 0 $ in general.

We also define a weak solution of the Dirichlet boundary problem:
\Be\label{waveD}
\begin{split}
\p_t^2 u - \Delta_x u &= G,
\\ u(t,x) |_{\p \O } &= g, 
\\ u(0,x) = u_0, \, \p_t u(0,x ) & = u_1,
	\end{split}
\Ee
if for any $\phi \in C_c^\infty((0,T) \times \bar \O) $, with $\phi |_{\p \O} = 0$, we have
\Be\label{waveD_weak}
\begin{split}
\langle u, \phi \rangle _{D} := &  \int_{\O} ( u_1 (x) \phi(0,x) - u_0 (x) \p_t \phi(0,x)  )  dx + \int_0^T \int_\O  u(t,x) \left( \p_t^2 \phi(t,x)  - \Delta_x  \phi(t,x) \right)  dx dt  
\\ & - \int_0^T \int_{\p \O } g(t,x_\parallel ) \p_{x_3} \phi(t, x_\parallel, 0 )   dx_\parallel dt  - \int_0^T \int_\O  G \phi  \,  dx dt
\\ = & 0,
\end{split}
\Ee
and each terms in \eqref{waveD_weak} are all bounded.
\end{definition}

\begin{lemma} \label{Maxtowave}
Suppose $E(t,x), B(t,x) \in W^{1,\infty}((0,T) \times \O)$ is a solution to the equations \eqref{Maxwell1}-\eqref{percond} in the sense of Definition \ref{sol_M}, and $E_0(x), B_0(x) \in W^{1,p}(\O) $, $\nabla_x \rho, \nabla_x J,  \p_t J \in  L^\infty((0,T); L_{\text{loc}}^p(\O))$ for some $p>1$. Then $E_1, E_2, B_3$ solve the wave equation with the Dirichlet boundary condition \eqref{waveD} in the sense of \eqref{waveD_weak} with \begin{align}
u_0 = E_{0,i}, \  u_1 = \p_t E_{0,i} : = (\nabla_x \times B)_i - 4 \pi J_{0,i} , \ G = -4\pi \p_{x_i} \rho - 4 \pi \p_t J_i, \ g = 0 , \ \ \text{for} \  E_i,  i =1,2, \label{E12sol} \\
 u_0 = B_{0,3}, \ u_1 = \p_t B_{0,3} := - (\nabla_x \times E_0)_3, \   G =  4 \pi (\nabla_x \times J )_3, \  g = 0, \ \ \text{for} \ B_3,  \label{B3sol}
 \end{align} 
respectively.

Moreover, $E_3, B_1, B_2$ solve the wave equation with the Neumann boundary condition \eqref{waveNeu}  in the sense of \eqref{waveinner} \text{ with }
\begin{align}
u_0 = E_{0,3}, \  u_1 = \p_t E_{0,3} : = - \p_2 B_{0,1} + \p_1 B_{0,2} - 4 \pi J_{0,3} , \ G = -4\pi \p_{x_3} \rho - 4 \pi \p_t J_3, \ g = - 4\pi \rho, \ \ \text{for} \  E_3, \label{E3sol} \\
 u_0 = B_{0,i}, \ u_1 = \p_t B_{0,i} := - (\nabla_x \times E_0)_i, \   G =  4 \pi (\nabla_x \times J )_i, \  g = (-1)^{i+1} 4 \pi J_{\underline i}, \ \ \text{for} \ B_i, \ i=1,2, \label{B12sol}
 \end{align} 
 respectively.
 \end{lemma}
\begin{proof}\textit{Step 1. }We first consider $E_3, B_1, B_2$. Let $\phi(t,x) \in C^\infty_c( [0,T) \times \bar \O ) $. Since $E(t,x), B(t,x) \in W^{1,\infty}((0,T) \times \O)$, we can multiply $\p_{x_3} \phi $ to equation \eqref{Maxwell3}  and integrate over $(0,T) \times \O$ to get
\Be
\int_0^T \int_\O ( \nabla_x \cdot E ) \p_3 \phi dx dt = \int_0^T \int_\O 4 \pi \rho \p_3 \phi dx dt,
\Ee
then from integration by parts, we get
\Be \label{E3phi1}
\begin{split}
&\int_0^T \int_\O -  ( E_1 \p_1 \p_3 \phi + E_2 \p_2 \p_3 \phi + E_3 \p_3^2 \phi )  dx dt - \int_0^T \int_{\p \O} E_3 \p_3 \phi dx_\parallel dt  
\\ = &  \int_0^T \int_\O - 4 \pi \p_3 \rho  \phi dx dt -  \int_0^T \int_{\p \O } 4 \pi \rho \phi dx_\parallel dt.
\end{split}
\Ee

Next, multiplying $\p_1 \phi  $ to $\p_t B_2 = - (\p_3 E_1 - \p_1 E_3) $ and integrate over $(0,T) \times \O$ we get
\Be \label{ptB2p1phi}
\int_0^T \int_\O \p_t B_2 \p_1 \phi  dx dt = \int_0^T \int_\O   - (\p_3 E_1 - \p_1 E_3)  \p_1 \phi   dx dt.
\Ee
From integration by parts, the LHS of \eqref{ptB2p1phi} equals to
\[
\begin{split}
& \int_0^T \int_\O - B_2 \p_1 \p_t \phi  dx - \int_\O B_2(0,x) \p_1 \phi(0,x) dx
\\ = &  \int_0^T \int_\O \p_1 B_2 \p_t \phi dt dx + \int_\O \p_1 B_2(0,x)  \phi(0,x) dx.
\end{split}
\]
And the RHS of \eqref{ptB2p1phi} equals to 
\[
\int_0^T \int_\O (E_1 \p_3 \p_1 \phi - E_3 \p_1^2 \phi )dx dt,
\]
where we've used that $E_1 |_{\p \O} = 0 $. Thus we get
\Be \label{E3phi2}
\begin{split}
  \int_0^T \int_\O \p_1 B_2 \p_t \phi dt dx + \int_\O \p_1 B_2(0,x)  \phi(0,x) dx = \int_0^T \int_\O (E_1 \p_3 \p_1 \phi - E_3 \p_1^2 \phi )dx dt.
 \end{split}
\Ee
Similarly, multiplying $- \p_2 \phi  $ to $\p_t B_1 = - (\p_2 E_3 - \p_3 E_2) $ and integrate over $(0,T) \times \O$, using integration by parts and that $E_2 |_{\p \O} = 0 $, we get
\Be \label{E3phi3}
\begin{split}
 - \int_0^T \int_\O \p_2 B_1 \p_t \phi dt dx - \int_\O \p_2 B_1(0,x)  \phi(0) dx = \int_0^T \int_\O  (E_2 \p_3 \p_2 \phi - E_3 \p_2^2 \phi )dx dt.
 \end{split}
\Ee

Finally, multiplying $ - \p_t \phi$ to $\p_t E_3 = \p_1 B_2 - \p_2 B_1 - 4\pi J_3 $ and integrate over $(0,T) \times \O$ we get
\Be
 - \int_0^T \int_\O \p_t E_3 \p_t \phi dx dt = - \int_0^T \int_\O  (  \p_1 B_2 - \p_2 B_1 - 4\pi J_3 ) \p_t \phi dx dt.
\Ee
From integration by parts, this gives
\Be \label{E3phi4}
\begin{split}
&  \int_0^T \int_\O  E_3 \p_t^2 \phi dx dt  + \int_\O E_3(0) \p_t \phi(0) dx 
\\= & - \int_0^T \int_\O  (  \p_1 B_2 - \p_2 B_1 ) \p_t \phi dx dt  - \int_0^T \int_\O  4 \pi ( \p_t J_3 \phi ) dx dt - \int_\O  4 \pi J_3(0,x) \phi(0,x) dx.
\end{split}
\Ee

Adding up \eqref{E3phi1}, \eqref{E3phi2}, \eqref{E3phi3}, and \eqref{E3phi4}, we get
\Be
\begin{split}
& \int_0^T \int_\O ( \p_t^2 \phi -  \Delta_x \phi )E_3  dx dt + \int_\O \left( E_3(0,x)  \p_t \phi(0) + (\p_2 B_1(0,x) - \p_1 B_2(0,x)  + 4\pi J_3(0,x) ) \phi(0,x)  \right) dx 
\\ = &   \int_0^T \int_\O  ( - 4 \pi \p_3 \rho - 4\pi \p_t J_3 )   \phi dx dt + \int_0^T \int_{\p \O} E_3 \p_3 \phi dx_\parallel dt -  \int_0^T \int_{\p \O } 4 \pi \rho \phi dx_\parallel dt.  \end{split}
\Ee
This proves \eqref{E3sol}.

\textit{Step 2. }Next, we do the same procedure for $B_1, B_2$. Multiply $\p_t B_1 = - \p_2 E_3 + \p_3 E_2 $ by $ - \p_t \phi $ and integrate over $(0,T) \times \O$, we get
\[
- \int_0^T \int_\O \p_t B_1 \p_t \phi dx dt = -  \int_0^T \int_\O (  - \p_2 E_3 + \p_3 E_2 ) \p_t \phi dx dt. 
\]
From integration by parts and that $E_2 |_{\p \O } = 0$, this gives
\Be \label{B1phi1}
\int_0^T \int_\O B_1 \p_t^2 \phi dx dt + \int_\O B_1(0,x) \p_t \phi(0,x) dx =  - \int_0^T \int_\O ( E_3 \p_2 \p_t \phi - E_2 \p_3 \p_t \phi ) dx dt.
\Ee
Multiply $\p_1 \phi $ to \eqref{Maxwell4} and integrate over $(0,T) \times \mathbb R^3$, we get
\[
\int_0^T \int_\O (\nabla_x \cdot B) \p_1 \phi  dx dt = 0.
\]
From integration by parts and that $B_3 |_{\p \O } = 0$, this gives
\Be
\int_0^T \int_\O (   - B_1 \p_1^2 \phi - B_2 \p_2\p_1 \phi - B_3 \p_3 \p_1 \phi ) dx dt = 0. 
\Ee
Multiply $- \p_2 \phi $ to $ \p_t E_3 = \p_1 B_2 - \p_2 B_1 - 4 \pi J_3$ and integrate over $(0,T) \times \O$, we get
\[
- \int_0^T \int_\O \p_t E_3 \p_2 \phi dx dt =  - \int_0^T \int_\O  ( \p_1 B_2 - \p_2 B_1 - 4 \pi J_3 ) \p_2 \phi dx dt.
\]
From integration by parts, this gives
\Be
  \int_0^T \int_\O  E_3 \p_t \p_2 \phi dx dt - \int_\O \p_2 E_3(0,x)  \phi(0,x) dx = \int_0^T \int_\O (  B_2 \p_1 \p_2 \phi - B_1 \p_2^2 \phi  - 4\pi \p_2 J_3 \phi ) dx dt.
\Ee
Multiply $  \p_3 \phi $ to $\p_t E_2 = - (\p_1 B_3 - \p_3 B_1 ) - 4 \pi J_2$ and integrate over  $(0,T) \times \O$, we get
\[
 \int_0^T \int_\O \p_t E_2 \p_3 \phi dx dt =  \int_0^T \int_\O  ( - \p_1 B_3 + \p_3 B_1 - 4 \pi J_2 ) \p_3 \phi dx dt.
\]
From integration by parts and that $ E_2 |_{\p \O } = 0$, we have
\Be \label{B1phi4}
\begin{split}
 &-  \int_0^T \int_\O  E_2 \p_t \p_3 \phi dx dt + \int_\O \p_3 E_2(0,x) \phi (0,x) dx
 \\ &  = \int_0^T \int_\O ( B_3 \p_1 \p_3 \phi - B_1 \p_3^2 \phi + 4 \pi \p_3 J_2 \phi ) dx dt - \int_0^T \int_{\p \O} B_1 \p_3 \phi dx_\parallel dt + \int_0^T \int_{\p \O } 4 \pi J_2 \phi dx_\parallel dt.
 \end{split}
\Ee

Adding up \eqref{B1phi1}--\eqref{B1phi4}, we get
\Be
\begin{split}
& \int_0^T \int_\O  ( \p_t^2 \phi - \Delta_x  \phi  ) B_1  + \int_\O \left(B_1(0,x) \p_t \phi(0,x)  +  (\p_2 E_3(0,x)  - \p_3 E_2(0,x)  )\phi(0,x) \right) dx 
\\ = & \int_0^T \int_\O ( 4 \pi \p_2 J_3 -  4 \pi \p_3 J_2 ) dx dt  + \int_0^T \int_{\p \O} B_1 \p_3 \phi dx_\parallel dt- \int_0^T \int_{\p \O } 4 \pi J_2 \phi dx_\parallel dt.
\end{split}
\Ee
Using the same argument we also derive
\[
\begin{split}
& \int_0^T \int_\O  ( \p_t^2 \phi - \Delta_x  \phi  ) B_2  + \int_\O \left(B_2(0,x) \p_t \phi(0,x)  + (- \p_1 E_3(0,x)  - \p_3 E_1(0,x)  )\phi(0,x) \right) dx 
\\ = & \int_0^T \int_\O ( - 4 \pi \p_1 J_3 +  4 \pi \p_3 J_1 ) dx dt + \int_0^T \int_{\p \O} B_2 \p_3 \phi dx_\parallel dt +  \int_0^T \int_{\p \O } 4 \pi J_1 \phi dx_\parallel dt.
\end{split}
\]
This proves \eqref{B12sol}.
%

\textit{Step 3. }Next, we consider $E_1, E_2, B_3$. Let $\phi(t,x) \in C_c^\infty([0,T) \times \bar \O ) $ such that $\phi |_{\p \O } = 0$. Note that this implies $\p_1 \phi|_{\p \O} = \p_2 \phi|_{\p \O} = 0 $. Multiply $\p_{x_1} \phi $ to equation \eqref{Maxwell3}  and integrate over $(0,T) \times \O$ to get
\[
\int_0^T \int_\O ( \nabla_x \cdot E ) \p_1 \phi dx dt = \int_0^T \int_\O 4 \pi \rho \p_1 \phi dx dt,
\]
then from integration by parts and that $\p_2 \phi |_{\p \O} = 0$, we get
\Be \label{E1phi1}
\int_0^T \int_\O (-E_1 \p_1^2 \phi - E_2 \p_2 \p_1 \phi - E_3 \p_3 \p_1 \phi ) dx dt = -  \int_0^T \int_\O 4 \pi \p_1 \rho \phi dx dt.
\Ee
Next, multiply $\p_2 \phi $ to $\p_t B_3 = - (\p_1 E_2 - \p_2 E_1 ) $ and integrate over $(0,T) \times \O$ we get
\[
\int_0^T \int_\O \p_t B_3 \p_2 \phi   dx dt  = \int_0^T \int_\O   - (\p_1 E_2 - \p_2 E_1) (\p_2 \phi ) dx dt .
\]
From integration by parts, 
\Be \label{E1phi2}
\begin{split}
\int_0^T \int_\O \p_2 B_3 \p_t \phi dx dt  + \int_\O  \p_2 B_3(0,x)  \phi(0,x)  dx  = \int_0^T \int_\O (  E_2 \p_1 \p_2 \phi -  E_1 \p_2^2 \phi )  dx dt.
\end{split}
\Ee
Multiply $- \p_3 \phi $ to $\p_t B_2 =  \p_1 E_3 - \p_3 E_1  $ and integrate over $(0,T) \times \O$ we get
\[
- \int_0^T \int_\O \p_t B_2 \p_3 \phi   dx dt  = - \int_0^T \int_\O   (\p_1 E_3 - \p_3 E_1) (\p_3 \phi ) dx dt .
\]
From integration by parts and that $E_1 |_{\p \O} = \phi |_{\p \O} = 0$, we have
\Be \label{E1phi3}
- \int_0^T \int_\O \p_3 B_2 \p_t \phi dx dt - \int_\O \p_3 B_2(0,x) \phi(0,x) dx dt = \int_0^T \int_\O ( E_3 \p_1 \p_3 \phi - E_1 \p_3^2 \phi ) dx dt.
\Ee
Then multiplying $ - \p_t \phi$ to $\p_t E_1 = \p_2 B_3 - \p_3 B_2 - 4\pi J_1 $ and integrate over $(0,T) \times \O$, we get
\[
 - \int_0^T \int_\O \p_t E_1 \p_t \phi dx dt = - \int_0^T \int_\O  (  \p_2 B_3 - \p_3 B_2 - 4\pi J_1 ) \p_t \phi dx dt.
\]
From integration by parts, this gives
\Be \label{E1phi4}
\begin{split}
&  \int_0^T \int_\O  E_1 \p_t^2 \phi dx dt  + \int_\O E_1(0) \p_t \phi(0) dx 
\\= & - \int_0^T \int_\O  (  \p_2 B_3 - \p_3 B_2 ) \p_t \phi dx dt  - \int_0^T \int_\O  4 \pi ( \p_t J_1 \phi ) dx dt - \int_\O  4 \pi J_1 (0,x) \phi(0,x) dx.
\end{split}
\Ee

Adding up \eqref{E1phi1}-\eqref{E1phi4}, we get
\[
\begin{split}
& \int_0^T \int_\O ( \p_t^2 \phi -  \Delta_x \phi )E_1  dx dt + \int_\O \left( E_1(0,x)  \p_t \phi(0) + (\p_2 B_3(0,x) - \p_3 B_2(0,x)  + 4\pi J_1(0,x) ) \phi(0,x)  \right) dx 
\\ = &   \int_0^T \int_\O  ( - 4 \pi \p_1 \rho - 4\pi \p_t J_1 )   \phi dx dt.  \end{split}
\]
Using the same argument, we derive
\[
\begin{split}
& \int_0^T \int_\O ( \p_t^2 \phi -  \Delta_x \phi )E_2  dx dt + \int_\O \left( E_2(0,x)  \p_t \phi(0) + (- \p_1 B_3(0,x) + \p_3 B_1(0,x)  + 4\pi J_2(0,x) ) \phi(0,x)  \right) dx 
\\ = &   \int_0^T \int_\O  ( - 4 \pi \p_2 \rho - 4\pi \p_t J_2 )   \phi dx dt.  
\end{split}
\]
This proves \eqref{E12sol}. 

\textit{Step 4. }Next, multiply $\p_t B_3 = - \p_1 E_2 + \p_2 E_1 $ by $ - \p_t \phi $ and integrate over $(0,T) \times \O$, we get
\[
- \int_0^T \int_\O \p_t B_3 \p_t \phi dx dt = -  \int_0^T \int_\O (  - \p_1 E_2 + \p_2 E_1 ) \p_t \phi dx dt. 
\]
From integration by parts,
\Be \label{B3phi1}
 \int_0^T \int_\O B_3 \p_t^2 \phi dx dt + \int_\O B_3(0,x) \p_t \phi(0,x) dx  =  \int_0^T \int_\O ( -  E_2 \p_1 \p_t \phi + E_1 \p_2 \p_t \phi ) dx dt .
\Ee
Multiply $\p_3 \phi $ to \eqref{Maxwell4} and integrate over $(0,T) \times \mathbb R^3$, we get
\[
\int_0^T \int_\O (\nabla_x \cdot B) \p_3 \phi  dx dt = 0.
\]
From integration by parts and that $B_3 |_{\p \O } = 0$, this gives
\Be \label{B3phi2}
\int_0^T \int_\O (   - B_1 \p_1\p_3 \phi - B_2 \p_2\p_3 \phi - B_3 \p_3^2 \phi ) dx dt = 0. 
\Ee
Multiply $- \p_1 \phi $ to $ \p_t E_2 = - \p_1 B_3 + \p_3 B_1 - 4 \pi J_2$ and integrate over $(0,T) \times \O$, we get
\[
- \int_0^T \int_\O \p_t E_2 \p_1 \phi dx dt =  - \int_0^T \int_\O  ( -\p_1 B_3 + \p_3 B_1 - 4 \pi J_2 ) \p_1 \phi dx dt.
\]
From integration by parts and that $\p_1 \phi |_{\p \O } = 0$, this gives
\Be \label{B3phi3}
  \int_0^T \int_\O  E_2 \p_t \p_1 \phi dx dt - \int_\O \p_1 E_2(0,x)  \phi(0,x) dx = \int_0^T \int_\O ( - B_3 \p_1^2 \phi + B_1 \p_3 \p_1 \phi  - 4\pi \p_1 J_2 \phi ) dx dt.
\Ee
Multiply $  \p_2 \phi $ to $\p_t E_1 =  \p_2 B_3 - \p_3 B_2 - 4 \pi J_1$ and integrate over  $(0,T) \times \O$, we get
\[
 \int_0^T \int_\O \p_t E_1 \p_2 \phi dx dt =  \int_0^T \int_\O  (  \p_2 B_3 - \p_3 B_2 - 4 \pi J_1 ) \p_2 \phi dx dt.
\]
From integration by parts and that $\p_2 \phi |_{\p \O } = 0$, we have
\Be \label{B3phi4}
\begin{split}
-  \int_0^T \int_\O  E_1 \p_t \p_2 \phi dx dt + \int_\O \p_2 E_1(0,x) \phi (0,x) dx
 = \int_0^T \int_\O ( - B_3 \p_2^2 \phi + B_2 \p_3 \p_2 \phi + 4 \pi \p_2 J_1 \phi ) dx dt.
 \end{split}
\Ee

Adding up \eqref{B3phi1}-\eqref{B3phi4}, we get
\[
\begin{split}
& \int_0^T \int_\O  ( \p_t^2 \phi - \Delta_x  \phi  ) B_3  + \int_\O \left(B_3(0,x) \p_t \phi(0,x)  +  (\p_1 E_2(0,x)  - \p_2 E_1(0,x)  )\phi(0,x) \right) dx 
\\ = & \int_0^T \int_\O ( 4 \pi \p_1 J_2 -  4 \pi \p_2 J_1 ) dx dt .
\end{split}
\]
This proves \eqref{B3sol}.\end{proof}



Next, we prove the uniqueness of wave equation with Neumann BC \eqref{waveNeu} and Dirichlet BC \eqref{waveD}.
\begin{lemma} \label{wavesol}
%
Suppose $u(t,x), \tilde u(t,x)$ are weak solutions with the Neumann BC \eqref{waveNeu} with the same $u_0$, $u_1$, $G$, and $g$ in the sense of weak formulation \eqref{waveinner}. Then $u(t,x) = \tilde u(t,x)$.
\end{lemma}
\begin{proof}
It suffices to show that if $u$ is the solution of \eqref{waveNeu} with
\Be \label{waveNeu0}
\begin{split}
u_0 = u_1 = G = g = 0,
\end{split}
\Ee
in the sense of \eqref{waveinner}, then for any $\psi \in C_c^\infty((0,T) \times \O )$,
\Be \label{intupsi}
\int_0^T \int_\O u(t,x) \psi(t,x)  dx dt = 0.
\Ee

Let $\tilde \psi (t,x) = \psi(T-t, x )$, then $\tilde \psi \in C_c^\infty((0,T) \times \O )$. We consider the function $\tilde v(t,x) : (0,T) \times \mathbb R^3 \to \mathbb R$ given by
\Be \label{vtest}
\begin{split}
\tilde v(t,x) :=  & \frac{1}{4 \pi } \int_{B(x;t)  \cap \{ y_3 > 0 \} }\frac{ \tilde \psi (t - |y-x|, y_\parallel, y_3 ) }{|y-x| } dy
\\ &  + \frac{1}{4 \pi } \int_{B(x;t)  \cap \{ y_3 < 0 \} }\frac{ \tilde  \psi ( t - |y-x|, y_\parallel, - y_3 ) }{|y-x| } dy.
\end{split}
\Ee
Then the function $\tilde v(t,x)$ is a weak solution of the wave equation in $(0,T) \times \mathbb R^3$
\Be \label{tildevwave}
\begin{split}
(\p_t^2 - \Delta_x) \tilde v(t,x)   = & \mathbf 1_{x_3 > 0 }  \tilde \psi(t,x) +   \mathbf 1_{x_3 < 0 }  \tilde \psi(t, \bar x) 
\\ \tilde v(0,x) = & 0, \ \p_t \tilde v(0,x) = 0,
\end{split}
\Ee
where $x= (x_1, x_2, -x_3)$. And since $\psi \in C_c^\infty((0,T) \times \O )$, the function $ \mathbf 1_{x_3 > 0 }  \tilde \psi(t,x) +   \mathbf 1_{x_3 < 0 }  \tilde \psi(t, \bar x)  $ is smooth in $(0,T) \times \mathbb R^3$. Thus $\tilde v $ is smooth. Moreover, for some small $\delta > 0$, 
\Be \label{vtsupp}
\tilde v(s,x) = 0 \text{ for }   s \in [ 0,\delta ),
\Ee
and a direct computation yields
\Be \label{p3tildev}
\p_{x_3} \tilde v = 0  \text{ on } (0,T) \times \p \O.
\Ee

Now, let $v(t,x) : [0,T) \times \bar \O \to \mathbb R^3$ be given by 
\Be \label{defv}
v(t,x) = \tilde v(T-t,x ).
\Ee
Then $v(t,x)$ is smooth, and by \eqref{vtsupp}, $v(t,x) \in C_c^\infty([0,T) \times \bar \O )  $. Moreover, from \eqref{tildevwave}, \eqref{p3tildev},
\Be \label{wavevneu}
\begin{split}
(\p_t^2 - \Delta_x )v(t,x) = & \tilde \psi(T-t, x ) \text{ in } (0,T) \times \O,
\\  \p_{x_3} v(t,x) = & 0 \text{ on } (0,T) \times \p \O.
\end{split}
\Ee
Now, since $u$ is the solution of \eqref{waveNeu} with $u_0 = u_1 = G = g = 0$, and $v(t,x) \in C_c^\infty([0,T) \times \bar \O )  $, $\p_{x_3} v|_{\p \O } = 0$, from \eqref{waveinner} we have
\Be \label{inner0}
\begin{split}
0 =  &  \int_{\O} ( u_1 (x) v(0,x) - u_0 (x) \p_t v(0,x)  )  dx + \int_0^T \int_\O   u(t,x) \left( \p_t^2 v(t,x)  - \Delta_x  v(t,x) \right) dx dt  
\\ & - \int_0^T \int_{\p \O } u(t,x_\parallel ) \p_{x_3} v(t, x_\parallel, 0 )   dx_\parallel dt  + \int_0^T \int_{\p \O } g(t,x_\parallel ) v(t, x_\parallel, 0 )   dx_\parallel dt  - \int_0^T \int_\O  G v  \,  dx dt
\\ = &   \int_0^T \int_\O u (t,x)  \tilde \psi (T-t, x ) dx dt
\\ = & \int_0^T \int_\O u(t,x) \psi (t,x) dx dt.
\end{split}
\Ee
Thus, we proved \eqref{intupsi} and this conclude the lemma.
\end{proof}

We also prove a similar version of the lemma that will be used later.
\begin{lemma} \label{wavesol2}
Let $u :(0,T) \times \O \to \mathbb R$ be a function such that for any $\phi \in C_c^\infty([0,T) \times \bar \O )$, with $\p_{x_3} \phi |_{\p \O} = 0$,
\Be \label{waveuN0}
\int_0^T \int_\O u ( \p_t^2- \Delta_x ) \phi dx dt  = 0,
\Ee
then $u =0$.
\end{lemma}
\begin{proof}
Take any $\psi \in C_c^\infty ((0,T) \times \O ) $. Let $\tilde \psi (t,x) = \psi(T-t, x )$, and define $\tilde v (t,x) $ in the same way as \eqref{vtest}. Then define $v(t,x)$ as in \eqref{defv}. Then, as showed in \eqref{vtest}-\eqref{wavevneu}, $v(t,x) \in C_c^\infty([0,T) \times \bar \O ) $, 
\[
\begin{split}
(\p_t^2 - \Delta_x )v(t,x) =  \psi(t, x ) \text{ in } (0,T) \times \O, \text{ and }  \p_{x_3} v  |_{\p \O } = 0 .
\end{split}
\]
Therefore, from \eqref{waveuN0}
\[
 \int_0^T \int_\O u \psi dx dt =  \int_0^T \int_\O u ( \p_t^2- \Delta_x ) v dx dt = 0.
\]
Thus $u = 0$.
\end{proof}

\begin{lemma} \label{wavesolD}
Suppose $u(t,x), \tilde u(t,x)$ are weak solutions with the Dirichlet BC \eqref{waveD} with the same $u_0$, $u_1$, $G$, and $g$ in the sense of weak formulation \eqref{waveD_weak}. Then $u(t,x) = \tilde u(t,x)$.
\end{lemma}
\begin{proof}
It suffices to show that if $u$ is the solution of \eqref{waveD} with
\Be \label{waveD0}
\begin{split}
u_0 = u_1 = G = g = 0,
\end{split}
\Ee
in the sense of \eqref{waveD_weak}, then for any $\psi \in C_c^\infty((0,T) \times \O )$,
\Be \label{intupsiD}
\int_0^T \int_\O u(t,x) \psi(t,x)  dx dt = 0.
\Ee
Now, let $\tilde \psi(t,x) = \psi (T-t,x)$, and define the function $\tilde w(t,x) : (0,T) \times \mathbb R^3 \to \mathbb R^3$ as
\Be
\tilde w(t,x) : =  \frac{1}{4 \pi } \int_{B(x;t)  \cap \{ y_3 > 0 \} }\frac{ \tilde \psi (t - |y-x|, y_\parallel, y_3 ) }{|y-x| } dy - \frac{1}{4 \pi } \int_{B(x;t)  \cap \{ y_3 < 0 \} }\frac{ \tilde  \psi ( t - |y-x|, y_\parallel, - y_3 ) }{|y-x| } dy,
\Ee
and let $w(t,x) : [0,T) \times \bar \O \to \mathbb R^3$ be given by 
\Be
w(t,x) = \tilde w(T-t,x ).
\Ee
Then from direct computation and using the same argument as \eqref{vtest}-\eqref{wavevneu}, we get $w(t,x) \in C_c^\infty([0,T) \times \bar \O )  $, and
\Be \label{wavewneu}
\begin{split}
(\p_t^2 - \Delta_x )w(t,x) = & \tilde \psi(T-t, x ) = \psi(t,x) \text{ in } (0,T) \times \O,
\\   w(t,x) = & 0 \text{ on } (0,T) \times \p \O.
\end{split}
\Ee
Therefore from \eqref{waveD_weak} and \eqref{waveD0}, we have
\[
 0 =  \int_0^T \int_\O u  (\p_t^2 - \Delta_x )w dx dt  =  \int_0^T \int_\O u \psi dx dt.
\]
This proves \eqref{intupsiD}.\end{proof}

Next, we show that the Lipschitz solutions of the wave equations solves the Maxwell equations if the continuity equation {conteq} and some initial compatibility condition are satisfied. 
\begin{lemma} \label{wavetoMax}
	Suppose $E(t,x), B(t,x) \in W^{1, \infty}((0,T) \times \O  ) $, $\nabla_x \rho, \p_t J, \nabla_x J \in L^\infty((0,T); L_{\text{loc}}^{p}( \O )) $, with 
	\Be 
	\nabla \cdot J = - \p_t \rho.
	\Ee
	Assume 
	\Be \label{weakEass}
	\begin{split}
		& E_1, E_2 \text{ solves } \eqref{waveD} \text{ with } \eqref{E12sol}, \text{ and } E_3 \text{ solves } \eqref{waveNeu} 
		\text{ with } \eqref{E3sol},
		\\ & B_3 \text{ solves } \eqref{waveD} \text{ with } \eqref{B3sol}, \text{ and } B_1, B_2 \text{ solves } \eqref{waveNeu} 
		\text{ with } \eqref{B12sol}.
	\end{split}
	\Ee
	Further we assume compatibility conditions
	\begin{align}
		\nabla \cdot E_0 = 4 \pi \rho_0, \ \nabla_x \cdot B_0 =0, \text{ in } \O, \label{GaussE0} \\
		E_{0,1} = E_{0,2} = B_{0,3} = 0 \text{ on } \p \O. \label{Dirt0}
	\end{align}
	Then we have
	\Be 
	\begin{split}
		\p_t E & = \nabla_x \times B - 4 \pi J, \, \nabla_x \cdot E = 4\pi \rho,
		\\ \p_t B & = - \nabla_x \times E, \, \nabla_x \cdot B = 0.
	\end{split}
	\Ee
\end{lemma}
\begin{proof}
	Let's first prove $ \nabla \cdot E =4 \pi \rho$.
	
	In the view of Lemma \ref{wavesolD}, it suffices to show that for any $\phi(t,x) \in C_c^\infty( [0,T) \times \bar \O ) $ with $\phi |_{\p \O } = 0 $, we have
	\Be \label{GaussEweak}
	\int_0^T \int_\O ( \nabla \cdot E - 4 \pi \rho ) (\p_t^2 - \Delta_x)  \phi dx dt = 0.
	\Ee
	Now by direct computation with integration by parts, we have
	\[
	\begin{split}
		& \int_0^T \int_\O ( \nabla \cdot E - 4 \pi \rho ) (\p_t^2 - \Delta_x)  \phi dx dt
		\\ = & \int_0^T \int_\O  -  \sum_{i=1}^3 E_i (\p_t^2 - \Delta_x) (\p_{x_i} \phi ) dx dt - \int_0^T \int_{\p \O } E_3 (\p_t^2 - \Delta_x) \phi \ dx_\parallel dt  - \int_0^T \int_{\O } 4 \pi \rho (\p_t^2 - \Delta_x)  \phi dx dt
		\\ = &  \int_\O   (   -  \p_t E_{0} \cdot \nabla  \phi(0,x) +   E_{0} \cdot \nabla \p_t   \phi(0,x) ) dx  + 4 \pi \int_0^T \int_\O  \sum_{i=1}^3 ( \p_{x_i} \rho + \p_t J_i )    (\p_{x_i} \phi )  dx dt 
		\\ & - \int_0^T \int_{\p \O }  (E_3 \p_{x_3}^2 \phi  - 4 \pi \rho  \p_{x_3} \phi  ) \ d x_\parallel dt  - \int_0^T \int_{\p \O } E_3 (- \p_{x_3}^2 ) \phi \ dx_\parallel dt  - \int_0^T \int_{\O } 4 \pi \rho (\p_t^2 - \Delta_x)  \phi dx dt
		\\ = &    \int_\O   (    -  \p_t E_{0} \cdot \nabla  \phi(0,x) +   E_{0} \cdot \nabla \p_t   \phi(0,x)  ) dx  - 4 \pi \int_0^T \int_\O \rho \Delta_x \phi dx dt - 4\pi \int_0^T \int_\O \sum_{i=1}^3 J_i \p_t \p_{x_i } \phi dx dt
		\\ & - 4\pi \int_\O \sum_{i=1}^3  J_{0,i} \p_{x_i} \phi(0,x)  dx  - \int_0^T \int_{\O } 4 \pi \rho (\p_t^2 - \Delta_x)  \phi dx dt
		\\ = &   \int_\O     (    -  \p_t E_{0} \cdot \nabla  \phi(0,x) +   E_{0} \cdot \nabla \p_t   \phi(0,x)  ) dx + 4\pi \int_0^T \int_\O \nabla \cdot J \p_t \phi dx dt 
		\\ & + 4\pi \int_\O   J_0 \cdot \nabla  \phi(0,x) dx  - 4 \pi \int_0^T \int_\O \rho  \p_t^2 \phi dx dt
	\end{split} 
	\]
	where in the second equality we've used \eqref{weakEass}.
	
	Now, using $\nabla \cdot J = - \p_t \rho$, \eqref{GaussE0}, and integration by parts we obtain
	\[
	\begin{split}
		&  \int_0^T \int_\O ( \nabla \cdot E - 4 \pi \rho ) (\p_t^2 - \Delta_x)  \phi dx dt 
		\\ = &  \int_\O     (    -  \p_t E_{0} \cdot \nabla  \phi(0,x) +   E_{0} \cdot \nabla \p_t   \phi(0,x)  ) dx
		\\ &  - 4\pi \int_0^T \int_\O \p_t \rho \p_t \phi dx dt + 4\pi \int_\O J_0 \cdot \nabla  \phi(0,x) dx  - 4 \pi \int_0^T \int_\O \rho  \p_t^2 \phi dx dt
		\\ = &  \int_\O     (    - \p_t E_{0} + 4 \pi J_0 )  \cdot \nabla  \phi(0,x)  dx  - \int_\O  ( \nabla \cdot  E_{0} - 4 \pi \rho_0 )  \p_t   \phi(0,x) ) dx   
		\\ = & \int_\O - (\nabla_x \times B_0)   \cdot \nabla  \phi(0,x) dx =  0.
	\end{split} 
	\]
	This proves \eqref{GaussEweak}.
	
	Next, let's show that $ \p_t E_1 =  (\nabla_x \times B)_1 - 4 \pi J_1 $. It suffices to prove for any $\phi(t,x) \in C_c^\infty( [0,T) \times \bar \O ) $ with $ \phi |_{\p \O } = 0 $, we have
	\Be \label{Ampere1}
	\int_0^T \int_\O (\p_t E_1 - ( \nabla_x \times B)_1 + 4 \pi J_1 )(\p_t^2 - \Delta_x) \phi dx dt = 0.
	\Ee
	Using \eqref{weakEass} and integration by parts, we compute
	\[
	\begin{split}
		& \int_0^T \int_\O (\p_t E_1 - ( \nabla_x \times B)_1 + 4 \pi J_1 )(\p_t^2 - \Delta_x) \phi dx dt
		\\ = & - \int_0^T \int_\O E_1 ( \p_t^2 - \Delta_x ) (\p_t \phi  )dx dt - \int_\O E_{0,1} (\p_t^2 - \Delta_x) \phi(0,x)  dx  +  \int_0^T \int_\O 4\pi J_1 (\p_t^2 - \Delta_x) \phi dx dt 
		\\ &  + \int_0^T \int_\O   B_3  ( \p_t^2 - \Delta_x ) (\p_2 \phi )  dx dt  - \int_0^T \int_\O B_2 ( \p_t^2 - \Delta_x ) (\p_3 \phi ) dx dt  - \int_0^T \int_{\p \O } B_2(\p_t^2 - \Delta_x ) \phi d x_\parallel dt 
		\\ = & \int_\O \left( E_{0,1} \p_t^2 \phi(0,x) - \p_t E_{0,1} \p_t \phi(0,x) \right)  dx  + \int_0^T \int_\O 4 \pi ( \p_1 \rho  + \p_t J_1 ) (\p_t \phi ) dx dt - \int_\O E_{0,1} \p_t^2 \phi (0,x) dx + \int_\O  \Delta_x E_{0,1}  \phi (0,x) dx 
		\\ & + \int_0^T \int_\O 4\pi J_1 (\p_t^2 - \Delta_x) \phi dx dt - \int_\O B_{0,3}   (\p_t \p_2 \phi(0,x)  )dx  + \int_\O \p_t B_{0,3 } \p_2 \phi(0,x) dx + \int_0^T \int_\O 4 \pi (\nabla_x \times J )_3 (\p_2 \phi )  dx dt
		\\ & + \int_\O B_{0,2} ( \p_t \p_3 \phi(0,x) ) dx - \int_\O \p_t B_{0,2} \p_3 \phi(0,x) dx  - \int_0^T \int_\O  4 \pi (\nabla_x \times J)_2 (\p_3 \phi ) dx dt  - \int_0^T \int_{\p \O } B_2 \p_3^2 \phi dx_\parallel dt
		\\ & - \int_0^T \int_{\p \O} 4 \pi J_1 \p_3 \phi d x_\parallel dt   + \int_0^T \int_{\p \O } B_2 \p_3^2 \phi dx_\parallel dt
		\\ = & - \int_\O \p_t E_{0,1} \p_t \phi(0,x) dx  + \int_\O 4 \pi \rho_0 \p_1 \phi(0,x) dx + \int_0^T \int_\O 4 \pi \p_t \rho \p_1 \phi dx dt - \int_\O  4 \pi J_{0,1} \p_t \phi(0,x) dx + \int_\O  E_{0,1} \Delta_x \phi(0,x) dx
		\\ & - \int_0^T \int_\O 4 \pi J_1 \Delta_x \phi dx dt + \int_\O ( B_{0,2}  \p_t \p_3 \phi - B_{0,3}  \p_t \p_2 \phi (0,x) ) dx  - \int_\O ( - \p_t B_{0,3} \p_2 \phi(0,x)  + \p_t B_{0,2} \p_3 \phi(0,x) )dx 
		\\ & + \int_0^T \int_\O 4 \pi (-   J_2 \p_1 \p_2 \phi  +  J_1 \p_2^2 \phi  - J_3 \p_1 \p_3 \phi + J_1 \p_3^2 \phi ) dx dt + \int_0^T \int_{\p \O } 4 \pi J_1 \p_3 \phi dx_\parallel dt- \int_0^T \int_{\p \O} 4 \pi J_1 \p_3 \phi d x_\parallel dt.
	\end{split}
	\]
	Then from \eqref{conteq} and integration by parts,
	\[
	\begin{split}
		& \int_0^T \int_\O 4 \pi \p_t \rho \p_1 \phi dx dt  - \int_0^T \int_\O 4 \pi J_1 \Delta_x \phi dx dt+ \int_0^T \int_\O 4 \pi (-   J_2 \p_1 \p_2 \phi  +  J_1 \p_2^2 \phi  - J_3 \p_1 \p_3 \phi + J_1 \p_3^2 \phi ) dx dt
		\\ = & -  \int_0^T \int_\O 4 \pi \nabla \cdot J \p_1 \phi dx dt  - \int_0^T \int_\O 4 \pi J_1 \Delta_x \phi dx dt+ \int_0^T \int_\O 4 \pi (-   J_2 \p_1 \p_2 \phi  +  J_1 \p_2^2 \phi  - J_3 \p_1 \p_3 \phi + J_1 \p_3^2 \phi ) dx dt
		\\ = & \int_0^T \int_\O 4 \pi ( J_1 \p_1^2 \phi + J_2 \p_2 \p_1 \phi + J_3 \p_3 \p_1 \phi ) dx dt  - \int_0^T \int_\O 4 \pi J_1 \Delta_x \phi dx dt
		\\ & + \int_0^T \int_\O 4 \pi (-   J_2 \p_1 \p_2 \phi  +  J_1 \p_2^2 \phi  - J_3 \p_1 \p_3 \phi + J_1 \p_3^2 \phi ) dx dt
		\\ = & 0.
	\end{split}
	\]
	Thus,
	\Be
	\begin{split}
		& \int_0^T \int_\O (\p_t E_1 - ( \nabla_x \times B)_1 + 4 \pi J_1 )(\p_t^2 - \Delta_x) \phi dx dt
		\\ = &  \int_\O  ((-\p_t E_{0,1} - 4\pi  J_{0,1} ) \p_t \phi(0,x)   +  B_{0,2} \p_3 \p_t \phi(0,x)  -  B_{0,3}   \p_2 \p_t \phi(0,x) ) dx  
		\\ & + \int_\O ( -4\pi \p_1 \rho_0 \phi(0,x)   +  E_{0,1}  \Delta_x \phi(0,x) +  \p_t B_{0,3} \p_2 \phi(0,x) -  \p_t B_{0,2}   \p_3 \phi(0,x) ) dx.
	\end{split}
	\Ee
	From \eqref{weakEass}, we have $ -\p_t E_{0,1} - 4\pi  J_{0,1}  - \p_3 B_{0,2} + \p_2 B_{0,3}  = 0 $, and $\p_t B_0 = - \nabla_x \times E_0$. Together with \eqref{GaussE0}, \eqref{Dirt0}, we use integration by parts to get
	\[
	\begin{split}
		&  \int_\O  ((-\p_t E_{0,1} - 4\pi  J_{0,1} ) \p_t \phi(0,x)   +  B_{0,2} \p_3 \p_t \phi(0,x)  -  B_{0,3}   \p_2 \p_t \phi(0,x) ) dx  
		\\ & + \int_\O ( -4\pi \p_1 \rho_0 \phi(0,x)   +  E_{0,1}  \Delta_x \phi(0,x) +  \p_t B_{0,3} \p_2 \phi(0,x) -  \p_t B_{0,2}   \p_3 \phi(0,x) ) dx
		\\ = & 0
	\end{split}
	\]
	Thus, we conclude \eqref{Ampere1}. And from the same argument we can show that $ \p_t E_2 =  (\nabla_x \times B)_2 + 4 \pi J_2 $. 
	
	Next, let's prove $\p_t E_3  = ( \nabla_x \times B)_3 - 4 \pi J_3$. In the view of Lemma \ref{wavesol2}, it suffices to show that for any $\psi \in C_c^\infty( [0,T) \times \bar \O ) $ with $\p_{x_3} \psi |_{\p \O } = 0 $, we have
	\Be \label{Ampere3}
	\int_0^T \int_\O   (\p_t E_3 - ( \nabla_x \times B)_3 + 4 \pi J_3 )(\p_t^2 - \Delta_x) \psi dx dt = 0.
	\Ee
	
	Using \eqref{weakEass} and integration by parts, we compute
	\[
	\begin{split}
		& \int_0^T \int_\O (\p_t E_3 - ( \nabla_x \times B)_3 + 4 \pi J_3 )(\p_t^2 - \Delta_x) \psi dx dt
		\\ = & - \int_0^T \int_\O E_3 ( \p_t^2 - \Delta_x ) (\p_t \psi  )dx dt - \int_\O E_{0,3} (\p_t^2 - \Delta_x) \psi(0,x)  dx  +  \int_0^T \int_\O 4\pi J_3 (\p_t^2 - \Delta_x) \psi dx dt 
		\\ &  + \int_0^T \int_\O   B_2  ( \p_t^2 - \Delta_x ) (\p_1 \psi )  dx dt  - \int_0^T \int_\O B_1 ( \p_t^2 - \Delta_x ) (\p_2 \psi ) dx dt 
		\\ = & \int_\O \left( E_{0,3} \p_t^2 \psi(0,x) - \p_t E_{0,3} \p_t \psi(0,x) \right)  dx  + \int_0^T \int_\O 4 \pi ( \p_3 \rho  + \p_t J_3 ) (\p_t \psi ) dx dt   + \int_0^T \int_{\p \O } 4 \pi \rho \p_t \psi  dx_\parallel dt
		\\ & - \int_\O E_{0,3} \p_t^2 \psi (0,x) dx + \int_\O  \Delta_x E_{0,3}  \psi (0,x) dx  + \int_{\p \O } 4 \pi \rho_0 \psi(0,x_\parallel ) dx_\parallel + \int_0^T \int_\O 4\pi J_3 (\p_t^2 - \Delta_x) \psi dx dt
		\\ &  - \int_\O B_{0,2}   (\p_t \p_1 \psi(0,x)  )dx  + \int_\O \p_t B_{0,2 } \p_1 \psi(0,x) dx + \int_0^T \int_\O 4 \pi (\nabla_x \times J )_2 (\p_1 \psi )  dx dt - \int_0^T \int_{\p \O } (-4 \pi J_1  \p_1 \psi ) dx_\parallel dt
		\\ & + \int_\O B_{0,1} ( \p_t \p_2 \psi(0,x) ) dx - \int_\O \p_t B_{0,1} \p_2 \psi(0,x) dx  - \int_0^T \int_\O  4 \pi (\nabla_x \times J)_1 (\p_2 \psi ) dx dt  + \int_0^T \int_{\p \O } ( 4 \pi J_2  \p_2 \psi )  dx_\parallel dt
		\\ = & -  \int_\O   \p_t E_{0,3} \p_t \psi(0,x) dx + \int_0^T \int_\O 4 \pi \p_t \rho  \p_3 \psi dx dt + \int_\O 4 \pi \rho_0 \p_3 \psi(0,x) dx - \int_\O J_{0,3} \p_t \psi (0,x) dx
		\\ &  + \int_\O   E_{0,3} \Delta_x \psi (0,x) dx + \int_{\p \O } 4 \pi \rho_0 \psi(0,x_\parallel ) dx_\parallel - \int_0^T \int_\O 4 \pi J_3 \Delta_x \psi dx dt + \int_\O ( \p_1 B_{0,2} - \p_2 B_{0,1} ) \p_t \psi(0,x) dx  
		\\ & + \int_\O (- \p_t B_{0,1} \p_2 \psi(0,x) +  \p_t B_{0,2}  \p_1 \psi(0,x) ) dx  + \int_0^T \int_\O 4\pi (J_3 \p_2^2 \psi - J_2 \p_3 \p_2 \psi - J_1 \p_3 \p_1 \psi + J_3 \p_1^2 \psi) dx dt.
	\end{split}
	\]
	From \eqref{conteq} and integration by parts,
	\[
	\begin{split}
		& \int_0^T \int_\O 4 \pi \p_t \rho \p_3 \psi dx dt  - \int_0^T \int_\O 4 \pi J_3 \Delta_x \psi dx dt+ \int_0^T \int_\O 4 \pi  (J_3 \p_2^2 \psi - J_2 \p_3 \p_2 \psi - J_1 \p_3 \p_1 \psi + J_3 \p_1^2 \psi) dx dt
		\\ = & -  \int_0^T \int_\O 4 \pi \nabla \cdot J \p_3 \psi dx dt  - \int_0^T \int_\O 4 \pi J_3 \Delta_x \psi dx dt+ \int_0^T \int_\O 4 \pi (J_3 \p_2^2 \psi - J_2 \p_3 \p_2 \psi - J_1 \p_3 \p_1 \psi + J_3 \p_1^2 \psi)  dx dt
		\\ = & \int_0^T \int_\O 4 \pi ( J_1 \p_1 \p_3 \psi + J_2 \p_2 \p_3 \psi + J_3 \p_3^2 \psi ) dx dt   - \int_0^T \int_\O 4 \pi J_3 \Delta_x \psi dx dt
		\\ & + \int_0^T \int_\O 4 \pi  (J_3 \p_2^2 \psi - J_2 \p_3 \p_2 \psi - J_1 \p_3 \p_1 \psi + J_3 \p_1^2 \psi) dx dt
		\\ = & 0.
	\end{split}
	\]
	Thus,
	\Be
	\begin{split}
		& \int_0^T \int_\O (\p_t E_3 - ( \nabla_x \times B)_3 + 4 \pi J_3 )(\p_t^2 - \Delta_x) \psi dx dt
		\\ = &  \int_\O (-\p_t E_{0,3} - 4\pi  J_{0,3}  + \p_1 B_{0,2} - \p_2 B_{0,1}  ) \p_t \psi(0,x) dx  
		\\ & + \int_\O ( -4\pi \p_3 \rho_0 \psi(0,x)   +  E_{0,3} \Delta_x \psi(0,x)  -  \p_t B_{0,1} \p_2 \psi(0,x) +  \p_t B_{0,2} \p_1 \psi(0,x) ) dx.
	\end{split}
	\Ee
	From \eqref{weakEass}, we have $ -\p_t E_{0,3} - 4\pi  J_{0,3}  + \p_1 B_{0,2} - \p_2 B_{0,1}  = 0 $, and $\p_t B_0 = - \nabla_x \times E_0$. Together with \eqref{GaussE0}, \eqref{Dirt0}, we use integration by parts to get
	\[
	\begin{split}
		& \int_\O (-\p_t E_{0,3} - 4\pi  J_{0,3}  + \p_1 B_{0,2} - \p_2 B_{0,1}  ) \p_t \psi(0,x) dx  
		\\ & + \int_\O ( -4\pi \p_3 \rho_0 \psi(0,x)   +  E_{0,3} \Delta_x \psi(0,x)  -  \p_t B_{0,1} \p_2 \psi(0,x) +  \p_t B_{0,2} \p_1 \psi(0,x) ) dx
		\\ & = 0.
	\end{split}
	\]
	This concludes \eqref{Ampere3}.
	
	Next, we prove $ \p_t B_1 = -  (\nabla_x \times E)_1$. From Lemma \ref{wavesol2}, it suffices to prove that for any $\psi(t,x) \in C_c^\infty( [0,T) \times \bar \O ) $ with $ \p_{x_3} \psi |_{\p \O } = 0 $, we have
	\Be \label{Faraday1}
	\int_0^T \int_\O (\p_t B_1 + (\nabla_x \times E)_1  ) (\p_t^2 - \Delta_x ) \psi dx dt = 0.
	\Ee
	
	Using \eqref{weakEass} and integration by parts, we compute
	\[
	\begin{split}
		& \int_0^T \int_\O (\p_t B_1 + (\nabla_x \times E)_1  ) (\p_t^2 - \Delta_x ) \psi dx dt
		\\ = & - \int_0^T \int_\O B_1 ( \p_t^2 - \Delta_x ) (\p_t \psi  )dx dt - \int_\O B_{0,1} (\p_t^2 - \Delta_x) \psi(0,x)  dx   
		\\ &  - \int_0^T \int_\O   E_3  ( \p_t^2 - \Delta_x ) (\p_2 \psi )  dx dt  + \int_0^T \int_\O E_2 ( \p_t^2 - \Delta_x ) (\p_3 \psi ) dx dt  
		\\ = & \int_{\O}  ( B_{0,1} \p_t^2 \psi(0,x) - \p_t B_{0,1} \p_t \psi(0,x) )  dx  + \int_0^T \int_{\p \O }  4 \pi J_2 \p_t \psi dx_\parallel dt - \int_0^T \int_\O 4 \pi (\nabla_x \times J )_1  \p_t \psi dx dt 
		\\ & - \int_\O B_{0,1} (\p_t^2 - \Delta_x) \psi(0,x)  dx  
		\\ & + \int_\O (E_{0,3} \p_t \p_2 \psi(0,x) - \p_t E_{0,3} \p_2 \psi(0,x) ) dx + \int_0^T \int_{\p \O } 4 \pi \rho \p_2 \psi dx_\parallel dt + \int_0^T \int_\O 4 \pi ( \p_3 \rho + \p_t J_3 ) \p_2 \psi  dx dt
		\\ & + \int_\O (-E_{0,2}  \p_t \p_3 \psi(0,x) + \p_t E_{0,2}  \p_3 \psi (0,x) ) dx  - \int_0^T \int_\O 4 \pi ( \p_2 \rho + \p_t J_2 ) \p_3 \psi dx dt
		\\ = & - \int_\O \p_t B_{0,1} \p_t \psi (0,x) dx   + \int_0^T \int_\O 4\pi (J_3 \p_2 \p_t \psi - J_2 \p_3 \p_t \psi ) dx dt + \int_\O  B_{0,1} \Delta_x \psi(0,x) dx + \int_{ \p \O } 4\pi J_{0,2} \psi (0,x_\parallel ) dx_\parallel
		\\ & + \int_\O (- \p_2 E_{0,3}  + \p_3 E_{0,2} ) \p_t \psi(0,x) dx  + \int_\O (- \p_t E_{0,3 } \p_2 \psi(0,x) +  \p_t E_{0,2 } \p_3 \psi (0,x) ) dx
		\\ & +  \int_{\O} 4\pi ( - J_{0,3}  \p_2 \psi(0,x) + J_{0,2} \p_3 \psi(0,x) )  dx  + \int_0^T \int_\O 4\pi (-  J_3 \p_t \p_2 \psi + J_2 \p_t \p_3 \psi ) dx dt
		\\ = & \int_\O (-\p_t B_{0,1}  - \p_2 E_{0,3} + \p_3 E_{0,2} ) \p_t \psi(0,x) dx 
		\\ & + \int_\O (- \p_t E_{0,3} \p_2 \psi(0,x) + \p_t E_{0,2}  \p_3 \psi(0,x) +  B_{0,1} \Delta_x \psi(0,x) - 4\pi  J_{0,3} \p_2 \psi(0,x)  +  4\pi  J_{0,2}  \p_3 \psi(0,x) )  dx.
	\end{split}
	\]
	From \eqref{weakEass}, we have $ -\p_t B_{0,1}  - \p_2 E_{0,3} + \p_3 E_{0,2}  = 0$, and $\p_t E_0= \nabla_x \times B_0 - 4\pi J_0$.  Together with \eqref{GaussE0}, we use integration by parts to get
	\[
	\begin{split}
		& \int_\O (-\p_t B_{0,1}  - \p_2 E_{0,3} + \p_3 E_{0,2} ) \p_t \psi(0,x) dx 
		\\ & + \int_\O (- \p_t E_{0,3} \p_2 \psi(0,x) + \p_t E_{0,2}  \p_3 \psi(0,x) +  B_{0,1} \Delta_x \psi(0,x) - 4\pi  J_{0,3} \p_2 \psi(0,x)  +  4\pi  J_{0,2}  \p_3 \psi(0,x) )  dx
		\\  &= 0.
	\end{split}
	\]
	This proves \eqref{Faraday1}. Using the same argument we also prove $ \p_t B_2 = -  (\nabla_x \times E)_2$.
	
	Next, we prove $ \p_t B_3 = -  (\nabla_x \times E)_3$. From Lemma \ref{wavesolD}, it suffices to prove that for any $\phi \in C_c^\infty( [0,T) \times \bar \O ) $ with $ \phi |_{\p \O } = 0 $,
	\Be \label{Faraday3}
	\int_0^T \int_\O (\p_t B_3 +(\nabla_x \times E)_3 ) (\p_t^2 - \Delta_x) \phi dx dt = 0.
	\Ee
	
	Using \eqref{weakEass} and integration by parts, we compute
	\[
	\begin{split}
		& \int_0^T \int_\O (\p_t B_3 +(\nabla_x \times E)_3 ) (\p_t^2 - \Delta_x) \phi dx dt
		\\ = & - \int_0^T \int_\O B_3 ( \p_t^2 - \Delta_x ) (\p_t \phi  )dx dt - \int_\O B_{0,3} (\p_t^2 - \Delta_x) \phi(0,x)  dx   
		\\ &  - \int_0^T \int_\O   E_2  ( \p_t^2 - \Delta_x ) (\p_1 \phi )  dx dt  + \int_0^T \int_\O E_1 ( \p_t^2 - \Delta_x ) (\p_2 \phi ) dx dt  
		\\ = & \int_{\O}  ( B_{0,3} \p_t^2 \phi(0,x) - \p_t B_{0,3} \p_t \phi(0,x) )  dx   - \int_0^T \int_\O 4 \pi (\nabla_x \times J )_3  \p_t \phi dx dt - \int_\O B_{0,3} (\p_t^2 - \Delta_x) \phi(0,x)  dx  
		\\ & + \int_\O (E_{0,2} \p_t \p_1 \phi(0,x) - \p_t E_{0,2} \p_1 \phi(0,x) ) dx + \int_0^T \int_\O 4 \pi ( \p_2 \rho + \p_t J_2 ) \p_1 \phi  dx dt
		\\ & + \int_\O (-E_{0,1}  \p_t \p_2 \phi(0,x) + \p_t E_{0,1}  \p_2 \phi (0,x) ) dx  - \int_0^T \int_\O 4 \pi ( \p_1 \rho + \p_t J_1 ) \p_2 \phi dx dt
		\\ = & -  \int_\O \p_t B_{0,3} \p_t \phi (0,x) dx  + \int_0^T \int_\O 4\pi (J_2 \p_1 \p_t \phi - J_1 \p_2 \p_t \phi ) dx dt + \int_\O  B_{0,3} \Delta_x \phi(0,x) dx
		\\ & + \int_\O (- \p_1 E_{0,2} + \p_2 E_{0,1} ) \p_t \phi(0,x) dx  + \int_\O (- \p_t E_{0,2 } \p_1 \phi(0,x) +  \p_t E_{0,1} \p_2 \phi(0,x) ) dx
		\\ & + \int_0^T \int_\O 4 \pi (- J_2  \p_t \p_1 \phi + J_1 \p_t \p_2 \phi ) dx dt + \int_\O  4\pi (- J_{0,2} \p_1 \phi(0,x) + J_{0,1} \p_2 \phi(0,x) ) dx
		\\ = & \int_\O ( - \p_t B_{0,3} - \p_1 E_{0,2} + \p_2 E_{0,1} ) \p_t \phi(0,x) dx  
		\\ & + \int_\O (  B_{0,3}  \Delta_x \phi(0,x) -  \p_t E_{0,2} \p_1 \phi(0,x) +  \p_t E_{0,1} \p_2 \phi(0,x)  - 4\pi  J_{0,2}  \p_1\phi(0,x)  + 4\pi  J_{0,1} \p_2 \phi(0,x)  )dx
	\end{split}
	\]
	From \eqref{weakEass}, we have $ -\p_t B_{0,3}  - \p_1 E_{0,2} + \p_2 E_{0,1}  = 0$, and $\p_t E_0= \nabla_x \times B_0 - 4\pi J_0$.  Together with \eqref{GaussE0}, we use integration by parts to get
	\[
	\begin{split}
		& \int_\O ( - \p_t B_{0,3} - \p_1 E_{0,2} + \p_2 E_{0,1} ) \p_t \phi(0,x) dx  
		\\ & + \int_\O (  B_{0,3}  \Delta_x \phi(0,x) -  \p_t E_{0,2} \p_1 \phi(0,x) +  \p_t E_{0,1} \p_2 \phi(0,x)  - 4\pi  J_{0,2}  \p_1\phi(0,x)  + 4\pi  J_{0,1} \p_2 \phi(0,x)  )dx
		\\ & = 0.
	\end{split}
	\]
	This proves \eqref{Faraday3}. 
	
	Lastly, we prove $\nabla_x \cdot B = 0$. From Lemma \ref{wavesol2}, it suffices to prove that for any $\psi(t,x) \in C_c^\infty( [0,T) \times \bar \O ) $ with $ \p_{x_3} \psi |_{\p \O } = 0 $, we have
	\Be \label{GaussM1}
	\int_0^T \int_\O (\nabla_x \cdot B ) (\p_t^2 - \Delta_x )  \psi dx dt = 0.
	\Ee
	
	Using \eqref{weakEass}, integration by parts, and \eqref{Dirt0}, we compute
	\[
	\begin{split}
		& \int_0^T \int_\O (\nabla_x \cdot B ) (\p_t^2 - \Delta_x )  \psi dx dt
		\\ = &  - \int_0^T \int_\O \sum_{i=1}^3 B_i  (\p_t^2 - \Delta_x )  \p_{x_i}  \psi dx dt
		\\ = & \int_\O \sum_{i=1}^3 ( B_{0,i} \p_{x_i} \p_t \psi(0,x) -\p_t B_{0,i} \p_{x_i}  \psi(0,x) ) dx  - \int_0^T \int_\O \sum_{i=1}^3 4 \pi (\nabla_x \times J)_i \p_{x_i} \psi dx dt 
		\\ & + \int_0^T \int_{\p \O } 4\pi J_2 \p_{x_1} \psi dx_\parallel dt + \int_0^T \int_{\p \O } - 4\pi J_1 \p_{x_2} \psi dx_\parallel dt 
		\\ = & \int_\O (- \nabla_x \cdot B_0 ) \p_t \psi(0,x) dx  -  \int_\O \left( \sum_{i=1}^3 \p_t B_{0,i} \p_{x_i}  \psi(0,x) \right) dx.
	\end{split}
	\]
	From \eqref{GaussE0} we have $\nabla_x \cdot B_0 = 0$. And from \eqref{weakEass}, $\p_t B_0 = - \nabla_x \times E_0$, so from integration by parts and \eqref{Dirt0}, we have
	\[
	\int_\O \left( \sum_{i=1}^3 \p_t B_{0,i} \p_{x_i}  \psi(0,x) \right) dx = 0.
	\]
	This proves \eqref{GaussM1}.\end{proof}

Now we are ready to prove Theorem \ref{EBunique}.

\begin{proof}[\textbf{Proof of Theorem \ref{EBunique}}]
	%
	The proof is a direct consequence of previous lemmas. Let $(E,B) \in W^{1,\infty}((0,T) \times \O )$, and $(\tilde E, \tilde B) \in W^{1,\infty}((0,T) \times \O )$ be two solutions of \eqref{Maxwell1}-\eqref{percond}. 

	We consider $E_1- \tilde E_1$,  $E_2 - \tilde E_2$, and $B_3 - \tilde B_3$. From Lemma \ref{Maxtowave}, both $E_i$ and $\tilde E_i$ satisfy \eqref{E12sol}  for $i=1,2$, and both $B_3$ and $\tilde B_3$ satisfy \eqref{B3sol}. Therefore, from Lemma \ref{wavesolD}, we have
	\Be
	E_1 = \tilde E_1, \ E_2 = \tilde E_2, \ B_3 = \tilde B_3. 
	\Ee
	
	And for $E_3 - \tilde E_3$, $B_1 - \tilde B_1$, and $B_2 - \tilde B_2$. From Lemma \ref{Maxtowave}, we have both $E_3$ and $\tilde E_3$ satisfy \eqref{E3sol}, and both $B_i$ and $\tilde B_i$ satisfy \eqref{B12sol} for $i = 1,2$. Therefore, from Lemma \ref{wavesol}, we deduce that
	\Be
	E_3 = \tilde E_3, \ B_1 = \tilde B_1, \ B_2 = \tilde B_2.
	\Ee
	Thus, we get $E = \tilde E$, $B = \tilde B$, and this concludes the uniqueness of the solution. Now this solution should solve the Maxwell equation by Lemma \ref{wavetoMax}.
\end{proof}


\section{Glassey-Strauss Representation of $E$ and $B$}

In this section, we give a representation of the field $E$ and $B$ by solving the wave equations \eqref{wave_eq_E}, \eqref{wave_eq_B}, under the boundary condition \eqref{E12B3bc} and \eqref{E3B1B2bc}.

We first consider the electronic field $E$. The tangential component $E_\parallel = (E_1, E_2)$ satisfies
\Be \label{waveEparallel}
\begin{split}
\p_t^2 E_\parallel - \Delta_x E_\parallel = G_\parallel:=  -4 \pi \nabla_\parallel \rho - 4\pi \p_t J_\parallel   ,
\\
E_\parallel |_{t=0}=     E_{0\parallel}, \  \p_t E_\parallel |_{t=0} = \p_t E_{0\parallel}  ,
\end{split}
\Ee
and
\Be\label{Dirichlet}
E_\parallel =    0 \ \   \text{ on }   \  \p \O.
\Ee

Define
\Be\label{def:bar}
\begin{split}
\bar{x}= (x_\parallel, -x_3)  \ \ \ \text{for} \ \ x=  (x_\parallel,  x_3)=  (x_1, x_2,  x_3).
\end{split}
\Ee
To solve the Dirichlet boundary condition, we employ the odd extension of the data: for $i=1,2$, and $x \in \R^3$,
\Be
\begin{split}\label{odd_ext}
G_i(t,x_\parallel, x_3) = &  \mathbf 1_{x_3 > 0 }  G_i(t,x ) -  \mathbf 1_{x_3 <  0 }  G_i(t,\bar x),
\\  E_{0i} ( x_\parallel, x_3) = &  \mathbf 1_{x_3 > 0 }   E_{0i} (  x) -  \mathbf 1_{x_3 <  0 }   E_{0i}( \bar x),
\\  \p_t  E_{0i} ( x_\parallel, x_3) = &  \mathbf 1_{x_3 > 0 } \p_t  E_{0i} (  x) -  \mathbf 1_{x_3 <  0 } \p_t  E_{0i}( \bar x).
\end{split}
\Ee
Then the weak solution of $E_\parallel(t,x)$ to \eqref{waveEparallel} with data \eqref{odd_ext} in the whole space $\R^3$ takes a form of, for $i=1,2,$
\begin{align}
 &E_i(t,x ) =    \frac{1}{4 \pi t^2} \int_{\p B(x;t) \cap \{ y_3 > 0 \} }  \left( t \p_t  E_{0i}( y  ) +  E_{0i}(y ) + \nabla  E_{0i} (y ) \cdot (y-x) \right) dS_y \notag
\\  &  \ \ \ \ \  + \frac{1}{4 \pi t^2} \int_{\p B(x;t) \cap \{ y_3 < 0 \} }  \big( - t \p_t  E_{0i}( \bar y  ) -  E_{0i} ( \bar y)
 - \nabla  E_{0i} (\bar y) \cdot (\bar y - \bar x) 
\big)
 dS_y \notag
\\ & \ \ \ \  \ + \frac{1}{4 \pi } \int_{B(x;t)  \cap \{ y_3 > 0 \} }\frac{ G_i( t - |y-x|, y  ) }{|y-x| } dy \label{Eexpanmajor1}\\
&\ \  \ \ \  + \frac{1}{4 \pi } \int_{B(x;t)  \cap \{ y_3 < 0 \} }\frac{ -  G_i( t - |y-x|,  \bar y  ) }{|y-x| } dy, \label{Eexpanmajor2}
\end{align}
where $B(x,t)= \{ y \in \R^3: |x-y| <t\}$ and $\p B(x,t)= \{ y \in \R^3: |x-y| =t\}$. The above form is then a (weak) solution of \eqref{waveEparallel} and \eqref{Dirichlet}. Next, we express \eqref{Eexpanmajor1} and \eqref{Eexpanmajor2} using the Glassey-Strauss representation \cite{GS} in $\O$. Define
  \Be \label{pxptST}
\begin{split}
 \p_t  = \frac{S- \hat{v} \cdot T}{1+ \hat v \cdot \o}, \ \
\p_i  = \frac{\o_i S}{ 1+ \hat v \cdot \o} + \left( \delta_{ij} - \frac{\o_i \hat{v}_j}{1+ \hat v \cdot \o}\right) T_j,
\end{split}\Ee
 while, for $\o= \o(x,y) = \frac{y-x}{|y-x|}$,
 \begin{align}
T_i &:= \p_i - \o_i \p t, \label{def:T}\\
S &:= \p_t + \hat v \cdot \nabla_x. \label{def:S}
\end{align}
Note that
\Be\label{T=y}
T_j f (t- |y-x| , y, v ) = \p_{y_j } [ f(t- |y-x|, y, v ) ],
\Ee
and the Vlasov equation \eqref{VMfrakF1} implies that
 \Be\label{S=Lf_v}
 Sf= - \nabla_v \cdot [ (E + E_{\text{ext}} + \hat v \times ( B + B_{\text{ext}})- g \mathbf e_3)f].
 \Ee



 From \eqref{rhoJ1} and \eqref{pxptST},
\begin{align}
&\eqref{Eexpanmajor1}=- \int_{B(x;t) \cap \{y_3 > 0\}}  \frac{(   \p_i \mathbf{\rho} +\p_t {J}_i )(t-|y-x|,y )}{|y-x|} \dd y
\notag\\ &\ \  =- \int_{B(x;t) \cap \{y_3 > 0\}}   \int_{\R^3} (\p_i f + \hat{v}_i \p_t f) (t-|y-x|,y ,v)  \dd v \frac{ \dd y}{|y-x|}\notag
\\ &\ \  =- \int_{B(x;t) \cap \{y_3 > 0\}}  \int_{\R^3} \frac{\o_i + \hat{v}_i}{1+ \hat{v} \cdot \o}(Sf) (t-|y-x|,y ,v) \dd v \frac{ \dd y}{|y-x|} \notag
\\ & \ \ \  \ \  - \int_{
\substack{
B(x;t) \\  \cap \{y_3 > 0\}
}
}   \int_{\R^3}  \left( \delta_{ij} - \frac{(\o_i + \hat{v}_i)\hat{v}_j}{1+\hat{v} \cdot \o} \right) T_j f (t-|y-x|,y ,v) \dd v \frac{ \dd y}{|y-x|}.  \notag
\end{align}
Here, we followed the Einstein convention (when an index variable appears twice, it implies summation of that term over all the values of the index) and will do throughout this section.


Then replace $T_j f$ with \eqref{T=y} and apply the integration by parts to get the last term equals
\Be \label{upperT1}
\begin{split}
 &-    \int_{ \p B(x;t) \cap \{y_3>0 \}}    \o_j  \left( \delta_{ij} - \frac{(\o_i + \hat{v}_i)\hat{v}_j}{1+\hat{v} \cdot \o} \right) f(0, y,v)  \dd v \frac{\dd S_y}{|y-x|}
 \\  &+  \int_{B(x;t) \cap \{y_3= 0\}}   \int_{\R^3}  \left( \delta_{i3} - \frac{(\o_i + \hat{v}_i)\hat{v}_3}{1+\hat{v} \cdot \o} \right)  f (t-|y-x|,y_\parallel, 0 ,v)  \dd v \frac{ \dd y_\parallel}{|y-x|} \\
  &+  \int_{\substack{
B(x;t) \\  \cap \{y_3 > 0\}
}}   \int_{\R^3}
\frac{ (|\hat{v}|^2-1  )(\hat{v}_i +  \o_i ) }{  (1+  \hat v \cdot \o )^2}
  f (t-|y-x|,y ,v)  \dd v  \frac{\dd y}{|y-x|^2}.
\end{split}
\Ee
where we have used that, from \cite{GS,Glassey},
\Be\notag
 \frac{\p}{\p y_j}\left[\frac{1}{|y-x|}\left(\delta_{ij} - \frac{(\o_i + \hat{v}_i)\hat{v}_j}{1+\hat{v} \cdot \o} \right)\right]= \frac{ (|\hat{v}|^2-1  )(\hat{v}_i +  \o_i ) }{|y-x|^2 (1+  \hat v \cdot \o )^2}.
\Ee

In order to express \eqref{Eexpanmajor2} in the lower half space we modify the idea of Glassey-Strauss slightly. Define
 \Be\label{def:o-}
 \bar \o   =  \begin{bmatrix} \o_1 & \o_2  & - \o_3 \end{bmatrix} ^T.
 \Ee
We use the same $S$ of \eqref{def:S} but
\Be\label{def:T-}
\begin{split}
 \bar{T}_3 f &=  -\p_{y_3} [f(t-|y-x|, y_\parallel, - y_3, v)] =   \p_{y_3} f - \bar{\omega}_3 \p_t f ,
\\ \bar T _i f &=  \p_{y_i} [f(t-|y-x|, y_\parallel, - y_3, v)] =   \p_{y_i} f - \bar{\omega}_i \p_t f \, \text{ for } \ i=1,2.
\end{split}
\Ee
Then we get
\begin{align}
\p_t &=  \frac{S-   \hat{v} \cdot \bar{T} }{1+ \hat{v}  \cdot \bar{\omega} }, \label{ptST-}\\
 \p_{y_i} &=    \bar{T}_i  +   \bar{\o}_i  \frac{S-   \hat{v} \cdot \bar{T} }{1+ \hat{v}  \cdot \bar{\omega} } =  \frac{ \bar{\o}_i S }{1+ \hat{v}  \cdot \bar{\omega}   }
+ \bar{T}_i   -  \bar{\o}_i \frac{     \hat{v}   \cdot \bar{T} }{1+ \hat{v}  \cdot \bar\omega }.
 \label{pxST-}
\end{align}
Therefore, we derive
\hide\Be
\begin{split}
 \p_i + \hat{v}_i \p_t
=& \frac{\o_i + \hat{v}_i }{1+ \hat{v}_\parallel \cdot \omega_\parallel - \hat{v}_3 \o_3}S
 + T^-_i  - (\o_i   + \hat{v}_i  )\frac{    \hat{v}_\parallel  \cdot T^-_\parallel - \hat{v}_3 T_3^-}{1+ \hat{v}_\parallel \cdot \omega_\parallel - \hat{v}_3 \o_3},
 \\
 \p_3 + \hat{v}_3 \p_t = & \frac{-\o_3 + \hat{v}_3 }{1+ \hat{v}_\parallel \cdot \omega_\parallel - \hat{v}_3 \o_3}S
- T^-_3  - (-\o_3   + \hat{v}_3  )\frac{    \hat{v}_\parallel  \cdot T^-_\parallel - \hat{v}_3 T_3^-}{1+ \hat{v}_\parallel \cdot \omega_\parallel - \hat{v}_3 \o_3}.
\end{split}
\Ee
Or
\unhide
\Be\label{ST_lower}
\begin{split}
  \p_i + \hat{v}_i \p_t 
=  \frac{ \bar{\o}_i+ \hat{v}_i }{1+
  \hat v \cdot \bar\o
  }S
  +  \left(  \delta_{ij}  -  \frac{     \bar{\o}_i \hat{v}_j   + \hat{v}_i \hat{v}_j     }{1+ \hat v \cdot \bar\o }\right)  \bar{T}_j .
\end{split}\Ee

Now we consider \eqref{Eexpanmajor2}. From \eqref{ST_lower},
\begin{align}
&\eqref{Eexpanmajor2}=  \int_{B(x;t) \cap \{y_3 <0\}}  \int_{\R^3} (\p_i f + \hat{v}_i \p_t f) (t-|y-x|,\bar y ,v) \dd v \frac{ \dd y}{|y-x|}\notag
\\& \ \ =  \int_{B(x;t) \cap \{y_3 <0\}}  \int_{\R^3} \frac{ \bar{\o}_i + \hat{v}_i }{1+ \hat v \cdot \bar{\o} }(Sf) (t-|y-x|, \bar y ,v) \dd v \frac{ \dd y}{|y-x|} \notag
\\& \ \ +  \int_{\substack{
B(x;t) \\  \cap \{y_3 < 0\}
} }  \int_{\R^3}   \Big(  \delta_{ij} - \frac{ \bar{\o}_i \hat{v}_j+ \hat{v}_i \hat{v}_j}{1+ \hat v \cdot \bar{\o}}   \Big)  \bar{T}_j f (t-|y-x|, \bar y ,v)  \dd v \frac{ \dd y}{|y-x|}. \notag
\end{align}
%
 Applying \eqref{def:T-} and the integration by parts, we derive that the last term equals
\Be \label{GSlowerhalf2}
\begin{split}
 &  \int_{ \p B(x;t) \cap \{  y_3 < 0 \}} \int_{\R^3}   \bar{\o}_j    \Big(  \delta_{ij} - \frac{   \bar{\o} _i \hat{v}_j + \hat{v}_i \hat{v}_j}{1+ \hat v \cdot \bar{\o} }   \Big)  f(0, \bar y ,v)   \dd v \frac{\dd S_y}{|y-x|}
  \\ &+  \int_{B(x;t) \cap \{y_3 =0\}}  \int_{\R^3}  \iota_3   \Big( \delta_{i3} - \frac{\bar{\o}_i  \hat{v}_3 + \hat{v}_i  \hat{v}_3}{1+ \hat v \cdot \bar{\o} }   \Big)  f(t-|y-x|,y_\parallel, 0 ,v) \dd v  \frac{ \dd y}{|y-x|}\\
  &-  \int_{B(x;t) \cap \{y_3 <0\}}  \int_{\R^3}
   \frac{ (|\hat{v}|^2-1  )(\hat{v}_i +   \bar{\o}_i ) }{ (1+ \hat v \cdot \bar{\o} )^2}
   f(t-|y-x|, \bar y ,v) \dd v\frac{ \dd y}{|y-x|^2},
\end{split}
\Ee
where we have utilized the notation
\Be\label{iota}
\iota_i = +1 \ \ \text{ for } \ i=1,2,  \ \
\iota_3=-1,
\Ee
and the direct computation
\Be \label{DyT}
\iota_j
\frac{\p}{\p y_j}\left[
 \frac{1}{|y-x|}
  \Big(
  \delta_{ij} - \frac{ \iota_i \o_i \hat{v}_j + \hat{v}_i \hat{v}_j}{1+ \hat v \cdot \bar \o}
  \Big)\right]
 =   \frac{ (|\hat{v}|^2-1  )(\hat{v}_i +   \bar{\o}_i ) }{|y-x|^2 (1+ \hat v \cdot \bar{\o} )^2}.
\Ee
\hide

Now we compute $\sum_j \iota_j \frac{\p}{\p y_j}\left[ \frac{1}{|y-x|}  \Big( \delta_{ij} - \frac{ \iota_i \o_i \hat{v}_j + \hat{v}_i \hat{v}_j}{1+ \hat v \cdot \o^-}   \Big)\right] $. Note that
\[
\p_{y_j} |y-x|= \frac{(y-x)_j}{|y-x|}, \ \ \
\p_{y_j} \o_i = \frac{1}{|y-x|} \left(
\delta_{ij} - \frac{(y-x)_i (y-x)_j}{|y-x|^2}
\right)
 \]
We have

where we have computed
\Be
\begin{split}
 \sum_{j} \bigg\{& - \iota_j |y-x|^{-2}  \o_j  \Big(
  \delta_{ij} - \frac{ \iota_i  \o_i \hat{v}_j + \hat{v}_i \hat{v}_j}{1+ \hat v \cdot \o^-}
  \Big)\\
 &+ \iota_j  |y-x|^{-1} (-1)^2 [1+ \hat v \cdot \o^- ]^{-2}
  \left(\sum_k
 \hat{v}_k \iota_k \p_{y_j} \o_k
 \right)
 \left( \iota_i \o_i \hat{v}_j + \hat{v}_i \hat{v}_j \right)\\
 &+ \iota_j |y-x|^{-1} (-1) [1+ \hat v \cdot \o^-]^{-1}
 \iota_i \p_{y_j} \o_i \hat{v}_j \bigg\}\\
 \end{split}
\Ee
\Be
\begin{split}
 = \frac{1}{|y-x|^2 [1+ \hat v \cdot \o^-]^2}\sum_{j} \bigg\{&
 -[1+ \hat v \cdot \o^-]^2 \delta_{ij} \iota_j \o_j
 + [1+ \hat v \cdot \o^-] \{ \iota_i \o_i \iota_j\hat{v}_j + \hat{v}_i \iota_j\hat{v}_j\}   \o_j \\
 &+ \{ \iota_i \o_i  \iota_j\hat{v}_j + \hat{v}_i \iota_j \hat{v}_j\}
 \sum_k \iota_k \hat{v}_k (\delta_{kj} - \o_k \o_j) \\
 &- [1+ \hat v \cdot \o^-] \iota_j \hat{v}_j ( \iota_i\delta_{ij} - \iota_i \o_i \o_j )
 \bigg\}\\
\end{split}
\Ee
Now the summation equals
\Be
\begin{split}
& (\hat v \cdot \o^- ) ^2\hat{v}_i  - (\hat v \cdot \o^-)  ^2\hat{v}_i  + |\hat{v}|^2 \hat{v}_i\\
&-  (\hat v \cdot \o^-)^2 \iota_i \o_i + (\hat v \cdot \o^-)^2 \iota_i \o_i
+ (\hat v \cdot \o^-) \hat{v}_i
+|\hat{v}|^2 \iota_i \o_i -(\hat v \cdot \o^-)^2 \iota_i \o_i
- (\hat v \cdot \o^-) \hat{v}_i  + (\hat v \cdot \o^-)^2 \iota_i \o_i \\
&- 2(\hat v \cdot \o^-) \iota_i \o_i +(\hat v \cdot \o^-) \iota_i \o_i
- (\iota_i)^2 \hat{v}_i  + (\hat v \cdot \o^-) \iota_i \o_i \\
&- \iota_i \o_i \\
=&  |\hat{v}|^2 \hat{v}_i  + |\hat{v}|^2 \iota_i \o_i- (\iota_i)^2 \hat{v}_i - \iota_i \o_i .
\end{split}
\Ee

\unhide

Next, we consider the normal components of the Electronic field $E_3$. From \eqref{wave_eq_E}, \eqref{initialC}, and \eqref{E3B1B2bc}, we have
\Be \label{waveE3}
\begin{split}
\p_t^2 E_3 - \Delta_x E_3 = G_3:=   -4 \pi \p_3 \rho - 4\pi \p_t J_3  ,
\\ E_3 |_{t=0}=     E_{03}, \  \p_t E_3 |_{t=0} = \p_t E_{03}    ,
\end{split}
\Ee
and
\Be
  \p_3 E_3 =    4 \pi \rho   \ \   \text{ on } \ \  \p \O. \label{Neumann}
\Ee
It is convenient to decompose the solution into two parts: one with the Neumann boundary condition of \eqref{waveE3} and the zero forcing term and initial data
\Be \label{waveE3b}
\begin{split}
\p_t^2 w  - \Delta_x w  =  0     \ \  &\text{ in } \ \  \O,
\\ w  |_{t=0} =  0, \ \p_t w |_{t=0}  = 0   \ \  &\text{ in } \ \  \O,
\\ \p_3 w= 4 \pi \rho   \ \  &\text{ on } \ \  \p \O,
\end{split}\Ee
and the other part $\tilde E_3$ with the initial data of \eqref{waveE3} and the zero Neumann boundary condition. We achieve it by the even extension trick. Recall $\bar x$ in \eqref{def:bar}. For $x \in \R^3$, define \Be \label{wavetildeE32}
\begin{split}
G_3(t,x) = & \mathbf{1}_{x_3 > 0 } G_3(t,x ) +   \mathbf{1}_{x_3 < 0 } G_3(t,\bar x )  ,
\\  E_{03} (x)  = &  \mathbf 1_{x_3 > 0 }    E_{03  } (x ) +  \mathbf 1_{x_3 <  0 }    E_{03 } ( \bar x ),
\\  \p_t    E_{03} (x)   = &  \mathbf 1_{x_3 > 0 } \p_t  E_{03} ( x ) +  \mathbf 1_{x_3 <  0 } \p_t  E_{03}( \bar x ).
\end{split}
\Ee
The weak solution $\tilde E_3$ to \eqref{waveE3} with the data \eqref{wavetildeE32} in the whole space $\R^3$ take a form of
\begin{align}
\tilde E_3(t,x ) &=    \frac{1}{4 \pi t^2}
\int_{
\substack{\p B(x;t)\\ \cap \{ y_3 > 0 \}} }  \left( t \p_t  E_{03}( y ) +  E_{03}(y ) + \nabla  E_{03} (y ) \cdot (y-x) \right) dS_y\notag
\\  &   + \frac{1}{4 \pi t^2} \int_{ \substack{\p B(x;t) \\ \cap \{ y_3 < 0 \} }}  \big(  t \p_t  E_{03}( \bar y  ) +  E_{03} (\bar y)
  + \nabla  E_{03} ( \bar y ) \cdot (\bar y - \bar x  )   \big) dS_y\notag
\\ &  + \frac{1}{4 \pi } \int_{ \substack{ B(x;t) \\ \cap \{ y_3 > 0 \}} }\frac{ G_3( t - |y-x|, y  ) }{|y-x| } dy \label{tildeE3formula1}\\
&   + \frac{1}{4 \pi } \int_{\substack{ B(x;t)  \\  \cap \{ y_3 < 0 \} }}\frac{  G_3( t - |y-x|, \bar y ) }{|y-x| } dy. \label{tildeE3formula2}
\end{align}
Following the same argument to expand \eqref{Eexpanmajor1} and \eqref{Eexpanmajor2}, we derive that
\Be \label{E3expression1}
\begin{split}
&\eqref{tildeE3formula1}\\
& = - \int_{B(x;t) \cap \{y_3 > 0\}}  \int_{\R^3} \frac{\o_3 + \hat{v}_3}{1+ \hat{v} \cdot \o}(Sf) (t-|y-x|,y ,v) \dd v \frac{ \dd y}{|y-x|}
 \\  & + \int_{B(x;t) \cap \{y_3 > 0\}}   \int_{\R^3}\frac{ (|\hat{v}|^2-1  )(\hat{v}_3 +  \o_3 ) }{ (1+  \hat v \cdot \o )^2} f (t-|y-x|,y ,v)  \dd v  \frac{ \dd y}{|y-x|^2}
 \\ & -  \int_{ \p B(x;t) \cap \{  y_3>0 \}}   \o_j  \left( \delta_{3j} - \frac{(\o_3 + \hat{v}_3)\hat{v}_j}{1+\hat{v} \cdot \o} \right) f(0, y,v)  \dd v \frac{\dd S_y}{|y-x|}
 \\ & + \int_{B(x;t) \cap \{y_3= 0\}}   \int_{\R^3}  \left(  1 - \frac{(\o_3 + \hat{v}_3)\hat{v}_3}{1+\hat{v} \cdot \o} \right)  f (t-|y-x|,y_\parallel, 0 ,v)  \dd v \frac{ \dd y_\parallel}{|y-x|},
\end{split}
\Ee
\Be
\begin{split}\label{E3expression2}
&  \eqref{tildeE3formula2}\\
&=   - \int_{B(x;t) \cap \{y_3 <0\}}  \int_{\R^3} \frac{ \bar{\o}_3 + \hat{v}_3 }{1+ \hat v \cdot \bar{\o} }(Sf) (t-|y-x|, \bar y ,v) \dd v \frac{ \dd y}{|y-x|}
\\
&+ \int_{B(x;t) \cap \{y_3 <0\}}  \int_{\R^3}
   \frac{ (|\hat{v}|^2-1  )(\hat{v}_3 +   \bar{\o}_3 ) }{ (1+ \hat v \cdot \bar{\o} )^2}
   f(t-|y-x|,\bar y  ,v) \dd v\frac{ \dd y}{|y-x|^2} \\
&- \int_{\p B(x;t) \cap \{  y_3 < 0 \}} \int_{\R^3}   \bar{\o}_j    \Big(  \delta_{3j} - \frac{   \bar{\o} _3 \hat{v}_j + \hat{v}_3 \hat{v}_j}{1+ \hat v \cdot \bar{\o} }   \Big)  f(0,\bar y ,v)   \dd v \frac{\dd S_y}{|y-x|}
  \\ &+  \int_{B(x;t) \cap \{y_3 =0\}}  \int_{\R^3}     \Big( 1- \frac{ (\bar{\o}_3   + \hat{v}_3 ) \hat{v}_3}{1+ \hat v \cdot \bar{\o} }   \Big)  f(t-|y-x|,y_\parallel, 0 ,v) \dd v  \frac{ \dd y}{|y-x|}.
\end{split}
\Ee
Note that the weak derivative $\p_3$ to the form of $\tilde{E}_3$ solves the linear wave equation \eqref{waveE3} with oddly extended forcing term and the initial data in the sense of distributions. Thus it satisfies
\Be\label{p3tildeE3=0}
\p_3  \tilde E_3=0 \ \   \text{ on } \ \  \p \O.
\Ee

Now we consider \eqref{waveE3b}. We assume $\rho(t,x)$ for all $t\leq 0$, which implies $w(t,x)=0$ for all $t\leq 0$. Define the Laplace transformation:
\Be
W(p, x) = \int^\infty_{- \infty} e^{-p t} w (t,x ) dt, \ \  \  \
R(p, x) = \int^\infty_{- \infty} e^{-p t} \rho (t,x ) dt
.\label{LaplaceT}
\Ee
Then $W$ solves the Helmholtz equation with the same Neumann boundary condition:
\Be \label{Helmholtz}
\begin{split}
p^2  W - \Delta_x   W  =  0 & \ \  \text{ in }   \  \O,
\\ \p_n  W =  4 \pi R    & \ \  \text{ on }   \  \p \O.
\end{split}
\Ee
The solution for $( p^2 - \Delta _x  ) \Phi(x) =   \delta(x) $ in $\mathbb R^3$ is known as $  \frac{1}{4\pi } \frac{e^{\pm p |x| }}{ |x| }$. We choose \Be \label{fundaHelmPhi}
 \Phi(x) =  \frac{1}{4\pi } \frac{e^{- p |x| }}{ |x| }.
\Ee
We have the following identities:
\begin{lemma} \label{Helmu}{\it
Suppose $u \in C^2(\bar \O )$ is an arbitrary function. For a fixed $x \in \O$ and $\Phi$ in \eqref{fundaHelmPhi}, we have
\Be \label{uPhiintrep}
\begin{split}
u(x) &=   \int_\O \Phi(y-x) (   p^2- \Delta_x ) u (y) dy \\
&+ \int_{\p \O } \left[  \Phi(y-x) \p_n u(y) - u(y) \p_n \Phi(y-x) \right] dS_y.\end{split}
\Ee}
\end{lemma}
\begin{proof}The proof is rather standard. Fix $x \in \O$. Let $0 < \e \ll 1 $, and $B(x,\e)$ be a ball centered at $x$ with radius $\e$ such that $B(x, \e ) \subset \O$. Let $V_\e = \O - B(x,\e)$. Then, by the integration by parts,
\Be \notag
\begin{split}
& - \int_{V_\e} \Phi(y-x ) ( \Delta_y - p^2 ) u(y) dy
\\
&=    \int_{ \p \O }  u(y)   \p_n \Phi(y-x ) dS_y +  \int_{\p B(x,\e) }  u(y)   \p_n \Phi(y-x ) dS_y
\\ & \ \ -  \int_{\p \O}    \Phi(y-x )   \p_n  u(y)dS_y - \int_{\p B(x, \e) }    \Phi(y-x )   \p_n  u(y)dS_y.
\end{split}
\Ee
From \eqref{fundaHelmPhi}, $
 \int_{\p B(x, \e) }    \Phi(y-x )   \p_n  u(y)dS_y \lesssim 4\pi \e^2 \frac{e^{|p| \e }}{4 \pi \e }  \to 0,$ as $\e \to 0.$
And by direct computation,
\Be\notag
\begin{split}
 &\int_{\p B(x,\e) }  u(y)   \p_n \Phi(y-x ) dS_y =     \int_{\p B(x,\e) }  u(y)  \frac{-(y-x)}{|y-x| } \cdot  \nabla \Phi(y-x ) dS_y
 \\ = & \frac{1}{4 \pi }  \int_{\p B(x,\e) }  u(y)  \frac{-(y-x)}{|y-x| } \cdot  ( - i p |y-x | - 1 ) \frac{ e^{- i p |y-x | } (y-x) }{|y-x |^3 } dS_y
 \\ = &  \frac{1}{4 \pi }  \int_{\p B(x,\e) } \left( - (- p |y-x| - 1 ) \frac{ e^{- p|y-x |  }}{|y-x |^2 } \right) u(y)   dS_y
  \\ = &   \left( ( 1 - (- p) \e ) e^{- p \e} \right)  \left(  \frac{1}{4 \pi \e^2  }   \int_{\p B(x,\e) } u(y)   dS_y \right)
  \to   u(x), \text{ as } \e \to 0.
\end{split}
\Ee
Combining all together, and letting $\e \to 0$ we get \eqref{uPhiintrep}.\end{proof}

Next for $x \in \O$, let $\phi^x(y)$ be the function such that
\Be \label{phixyforO}
\begin{split}
(\Delta_y  - p^2 ) \phi_N^x(y) =  0 \ \  &\text{ in }  \ \ \O,
\\ \p_n \phi_N^x(y) =   \p_n \Phi(y-x ) \ \  &\text{ on } \ \  \p \O.
\end{split}
\Ee
The integration by parts implies
\Be \label{phixyintbyparts}
\begin{split}
0 = & \int_{\O} ( \Delta_y - p^2 ) \phi_N^x(y) u(y) dy
\\ = & \int_{\O } ( \Delta_y - p^2 )u(y)  \phi_N^x(y) dy + \int_{\p \O } [ \p_n \phi_N^x(y) u(y)  - \phi_N^x(y) \p_n u(y) ] dS_y
\\ = & \int_{\O } ( \Delta_y - p^2 )u(y)  \phi_N^x(y) dy + \int_{\p \O } [ \p_n  \Phi(y-x ) u(y)  - \phi_N^x(y) \p_n u(y) ] dS_y.
\end{split}
\Ee
By adding \eqref{phixyintbyparts} to \eqref{uPhiintrep}, we derive that
\Be \label{urepPhi2}
\begin{split}
u(x) = &  - \int_{\O} \left( \Phi(y-x) -\phi_N^x(y)     \right) ( \Delta_y - p^2 )u(y) dy\\
& + \int_{\p \O } ( \Phi(y-x) -\phi_N^x(y)  ) \p_n u(y) dS_y.
\end{split}
\Ee
For the half space $\O = \mathbb R_+^3$, we have, with $\bar x$ in \eqref{def:bar},
\Be \label{phiNxyformhalf}
\phi_N^x(y) = - \Phi(y- \bar x ) .
\Ee
\hide
Since $x,y \in  \mathbb R_+^3$, so $y - \tilde x \neq 0$, and $(\Delta_y  - p^2 ) ( - \Phi(y- \tilde x )) = 0$ for all $y \in  \mathbb R_+^3$. Also, by direct computation
\[
\p_{y_3}   \Phi(y-  x ) =  \frac{1}{4 \pi } \left(  \frac{(\pm p |y-x|  -1 ) e^{\pm p |y-x | } (y_3 - x_3 )  }{ |y-x |^3 }  \right).
\]
For $y \in \p \O$, $y_3 = 0$ and $|y - x | = |y - \tilde x | $, thus
\Be
\begin{split}
\p_{y_3} \phi_N^x(y) |_{y_3 = 0 }  = - \p_{y_3} \Phi(y- \tilde x )   |_{y_3 = 0 }  = & - \frac{1}{4 \pi } \left(  \frac{(\pm p |y- \tilde x|  -1 ) e^{\pm p |y-\tilde x | } (y_3 + x_3 )  }{ |y-\tilde x |^3 }  \right)  |_{y_3 = 0 }
\\ = &  - \frac{1}{4 \pi } \left(  \frac{(\pm p |y-  x|  -1 ) e^{\pm p |y- x | }  x_3   }{ |y- x |^3 }  \right)
\\ = &  \p_{y_3}   \Phi(y-  x )  |_{y_3 = 0 },
\end{split}
\Ee
and we proved \eqref{phiNxyformhalf}. \unhide

Finally, we derive that, from \eqref{urepPhi2} and \eqref{phiNxyformhalf}:
\begin{lemma} \label{Helmu2} {\it For $\O = \R^3_+$, and $\Phi$ in \eqref{fundaHelmPhi},
\Be
\begin{split}\label{formula:u}
u(x) =& - \int_{\O} \left(  \Phi(y-x) + \Phi(y-\bar x)  \right) ( \Delta_y - p^2 )u(y) dy\\
& + \int_{\p \O } ( \Phi(y-x) + \Phi(y-\bar x)  ) \p_n u(y) dS_y.
\end{split}
\Ee}
\end{lemma}

By applying \eqref{phiNxyformhalf} to \eqref{Helmholtz}, we derive that
\Be
\begin{split}\label{form_W}
W(p,x) &=   \int_{\p \O } ( \Phi(y-x) + \Phi(y-\bar x)  ) 4 \pi R (y) d S_y
\\  &=  2 \int_{\mathbb R^2 }\frac{e^{-p ( (y_1 - x_1)^2 +(y_2 - x_2 )^2 + x_3^2  )^{1/2}}}{ (y_1 - x_1)^2 +((y_2 - x_2 )^2 + x_3^2  )^{1/2}} R(y_1, y_2 ) d y_1 dy_2 .
\end{split}
\Ee
Using the inverse Laplace transform, we derive that
%
%
\Be \label{wexpression1}
\begin{split}
&w(t,x) =   \frac{1}{2\pi } \int_{-\infty}^\infty e^{(p_1 + i p_2 )t} W( p_1 + i p_2 ,x ) dp_2\\
& \ \ =    \frac{1}{ \pi   } \int_{-\infty}^\infty d p_2  \  e^{(p_1 + i p_2 )t}    \int_{\mathbb R^2 }d y_1 dy_2\frac{e^{ -( p_1 + i p_2) ( (y_1 - x_1)^2 +(y_2 - x_2 )^2 + x_3^2  )^{1/2}}}{ (y_1 - x_1)^2 +((y_2 - x_2 )^2 + x_3^2  )^{1/2}} \\
& \ \ \ \ \ \       \times    \int_{-\infty}^{\infty}  ds  \   e^{-( p_1 + i p_2) s } (- \rho(s,y_1,y_2 ))     .
\end{split}
\Ee
Finally, we derive that, using the identity $\int_{-\infty}^\infty e^{i p_2 t } d p_2 = 2 \pi \delta(t) $,
\Be \label{wexpression2}
\begin{split}
 &w(t,x)   = \frac{- 1}{ \pi   } \int_{\mathbb R^2 } \int_\R \int_\R \frac{e^{( p_1 + i p_2 )(t -s -  \sqrt{| y_\parallel - x_\parallel |^2 + x_3^2 } )  }}{ \sqrt{| y_\parallel - x_\parallel |^2 + x_3^2 } }  \rho( s, y_\parallel ) dp_2  ds dy_\parallel
\\  &= -  2  \int_{\mathbb R^2 } \int_\R  \frac{e^{p_1  ( t-s -  \sqrt{| y_\parallel - x_\parallel |^2 + x_3^2 } ) }  \delta(t-s -  \sqrt{| y_\parallel - x_\parallel |^2 + x_3^2 } ) }{ \sqrt{| y_\parallel - x_\parallel |^2 + x_3^2 } } \rho ( s, y_\parallel ) ds dy_\parallel
\\ &=  - 2  \int_{
\sqrt{| y_\parallel - x_\parallel |^2 + x_3^2 }<t
}  \frac{\rho (t -  \sqrt{| y_\parallel - x_\parallel |^2 + x_3^2 },y_\parallel ) }{ \sqrt{| y_\parallel - x_\parallel |^2 + x_3^2 }}   dy_\parallel.
\end{split}
\Ee

Collecting the terms, we conclude the following formula:
\begin{proposition} \label{Eiform}
\begin{align}
\label{Eesttat0pos}  &E_i(t,x)  
\\ &   =   \frac{1}{4 \pi  t^2} \int_{\p B(x;t) \cap \{ y_3 > 0 \} }  \left( t \p_t  E_{0,i}( y_\parallel, y_3 ) +  E_{0,i} (y_\parallel, y_3) + \nabla  E_{0,i} (y_\parallel, y_3) \cdot (y-x) \right) dS_y
\\  \label{Eesttat0neg}  & + \frac{1}{ 4 \pi  t^2} \int_{\p B(x;t) \cap \{ y_3 < 0 \} } \iota_i \big( - t \p_t  E_i(0, y_\parallel,  - y_3 ) -  E_i (0,y_\parallel, -y_3) 
\\ \notag & \quad \quad \quad \quad \quad \quad \quad  \quad \quad \quad \quad  - \nabla_\parallel  E_i (0,y_\parallel, - y_3) \cdot (y_\parallel-x_\parallel)  + \p_3  E_i (0,y_\parallel, - y_3) \cdot (y_3-x_3) \big) dS_y
 \\  \label{Eestbulkpos}  &  +   \int_{B(x;t) \cap \{y_3 > 0\}}  \int_{\R^3}\frac{ (|\hat{v}|^2-1  )(\hat{v}_i +  \o_i ) }{|y-x|^2 (1+  \hat v \cdot \o )^2} f(t-|y-x|,y ,v)\dd v \dd y
 \\ \label{Eestbulkneg} &- \int_{B(x;t) \cap \{y_3 <0\}}   \int_{\R^3} \iota_i \frac{ (|\hat{v}|^2-1 )(\hat{v}_i + \iota_i \o_i ) }{|y-x|^2 (1+ \hat v \cdot \o^-)^2}f(t-|y-x|,y_\parallel, -y_3 ,v) \dd v \dd y
 \\  \label{EestSpos} &- \int_{B(x;t) \cap \{y_3 > 0\}}  \int_{\R^3} \frac{\o_i + \hat{v}_i}{1+  \hat v \cdot \o }(Sf) (t-|y-x|,y ,v) \dd v \frac{ \dd y}{|y-x|}
\\  \label{EestSneg} &+\int_{B(x;t) \cap \{y_3 <0\}}  \int_{\R^3} \iota_i \frac{ \iota_i \o_i + \hat{v}_i }{1+ \hat v \cdot \o^-  }(Sf)(t-|y-x|,y_\parallel, -y_3 ,v) \dd v \frac{ \dd y}{|y-x|}
\\ \label{Eestbdrypos} &  + \int_{B(x;t) \cap \{y_3= 0\}} \int_{\R^3} \left( \delta_{i 3 }  -  \frac{(\o_i + \hat{v}_i)\hat{v}_3}{1+ \hat v \cdot \o }  \right) f (t-|y-x|,y_\parallel, 0 ,v) \dd v\frac{ \dd y_\parallel}{|y-x|}
\\  \label{Eestbdryneg} & -  \int_{B(x;t) \cap \{y_3 =0\}} \int_{\R^3} \iota_i \left( \delta_{i 3 } -   \frac{ \iota_i  \o_i \hat{v}_3 + \hat{v}_i  \hat{v}_3}{1+ \hat v \cdot \o^-} \right) f(t-|y-x|,y_\parallel, 0 ,v) \dd v \frac{ \dd y_\parallel }{|y-x|}
\\ \label{Eestinitialpos} &-  \int_{\p B(x;t) \cap \{ y_3>0 \} }  \int_{\mathbb R^3}  \sum_j \o_j  \left(\delta_{ij} - \frac{(\o_i + \hat{v}_i)\hat{v}_j}{1+ \hat v \cdot \o }\right) f(0,y,v) \dd v \frac{\dd S_y}{|y-x|}
 \\ \label{Eestinitialneg}  &  + \int_{\p B(x;t) \cap \{ y_3<0 \}} \iota_i \int_{\R^3} \sum_j   \iota_j    \omega_j   \Big( \delta_{ij} - \frac{ \iota_i  \o_i \hat{v}_j + \hat{v}_i \hat{v}_j}{1+ \hat v \cdot \o^-}   \Big)  f(0,y_\parallel, -y_3 ,v)  \dd v \frac{\dd S_y}{|y-x|}
 \\ \label{Eest3bdrycontri}  & - \delta_{i3}   \int_{ B(x;t) \cap \{y_3 =0\}}  \int_{\mathbb R^3 }  \frac{ 2 f (t - |y-x |, y_\parallel, 0, v )}{|y-x | }  dv d  S_y .
\end{align}
\end{proposition}

Next, we solve for $B$. For $B_1, B_2$ we have, for $i =1,2,$
\Be \label{wavetildeB}
 \begin{split}
\p_t^2  B_i - \Delta_x  B_i = & 4 \pi ( \nabla_x \times J)_i:= H_i  \ \  \text{ in }  \ \ \O,
\\ \p_{x_3} B_1 = &  4 \pi J_2  , \ \p_{x_3} B_2 =  4 \pi J_1 \ \  \text{ on } \ \  \p \O,
\\  B_i(0,x) = & B_{0 i}, \p_t  B_i(0,x) = \p_t B_{0 i} \ \  \text{ in }  \ \ \O.
\end{split} \Ee

To solve \eqref{wavetildeB} we write $B_i = \tilde B_i + B_{bi} $ with $\tilde B_i$ satisfies the wave equation in $(0, \infty) \times \mathbb R^3 $ with even extension in $x_3$:
\Be \label{wavetildeBi}
\begin{split}
\p_t^2 \tilde B_i - \Delta_x \tilde B_i = &  \mathbf 1_{x_3 > 0 }  H_i(t,x) +  \mathbf 1_{x_3 <  0 }  H_i(t,\bar x),
\\  \tilde B_i(0,x)  = &  \mathbf 1_{x_3 > 0 }   B_{0i} (x) +  \mathbf 1_{x_3 <  0 }   B_{0i} (\bar x),
\\  \p_t  \tilde B_i (0 , x) = &  \mathbf 1_{x_3 > 0 } \p_t  B_{0i} (x) +  \mathbf 1_{x_3 <  0 } \p_t  B_{0i}(\bar x).
\end{split}
\Ee
And $B_{bi}$ satisfies
\Be \label{Bibeq}
\begin{split}
\p_t^2 B_{bi} - \Delta_x B_{bi} = 0 & \text{ in } \O,
\\ B_{bi} (0,x) = 0, \p_t B_{bi} = 0 & \text{ in } \O,
\\ \p_{x_3} B_{b1} = 4 \pi J_2, \ \p_{x_3} B_{b2} = - 4 \pi J_1 & \text{ on } \O.
\end{split}
\Ee
Then from \eqref{wavetildeBi},
\[ \label{tildeE3formula}
\begin{split}
& \tilde B_i (t,x )
\\ =  &  \frac{1}{4 \pi t^2} \int_{\p B(x;t) \cap \{ y_3 > 0 \} }  \left( t \p_t  B_{0i} (y ) +   B_{0i}(y) + \nabla   B_{0i} (y) \cdot (y-x) \right) dS_y
\\ \end{split}
\]
\[ 
\begin{split}
 & + \frac{1}{4 \pi t^2} \int_{\p B(x;t) \cap \{ y_3 < 0 \} }  \big(  t \p_t   B_{0i}(\bar y ) +   B_{0i}(\bar y)   + \nabla   B_{0i} (\bar y) \cdot (\bar y - \bar x)   \big) dS_y
\\ & + \frac{1}{4 \pi } \int_{B(x;t)  \cap \{ y_3 > 0 \} }\frac{ H_i ( t - |y-x|, y ) }{|y-x| } dy + \frac{1}{4 \pi } \int_{B(x;t)  \cap \{ y_3 < 0 \} }\frac{  H_i( t - |y-x|, \bar y ) }{|y-x| } dy.
\end{split}
\]
Applying \eqref{wexpression2} to \eqref{Bibeq},
\Be
\begin{split}
B_{bi}(t,x) = (-1)^i 2  \int_{ B(x;t) \cap \{y_3 =0\}}  \int_{\mathbb R^3 }  \frac{ \hat v_{\underline{i} }  f (t - |y-x |, y_\parallel, 0, v )}{|y-x | }  dv d  S_y,
\end{split}
\Ee
where we have used the notation
\Be\label{def_under_i}
	\underline i = \begin{cases} 2, \text{ if } i=1, \\ 1, \text{ if } i=2. \end{cases}
	\Ee	
Thus,
\Be \label{Bi12rep}
\begin{split}
B_i(t,x) 
= &  \frac{1}{4 \pi t^2} \int_{\p B(x;t) \cap \{ y_3 > 0 \} }  \left( t \p_t  B_{0i} ( y ) +   B_{0i}(y) + \nabla   B_{0i} (y) \cdot (y-x) \right) dS_y
\\  & + \frac{1}{4 \pi t^2} \int_{\p B(x;t) \cap \{ y_3 < 0 \} }  \big(  t \p_t   B_{0i}(\bar y ) +   B_{0i}(\bar y)  + \nabla   B_{0i} (\bar y) \cdot (\bar y - \bar x)  \big) dS_y
\\ & + \frac{1}{4 \pi } \int_{B(x;t)  \cap \{ y_3 > 0 \} }\frac{ H_i ( t - |y-x|, y ) }{|y-x| } dy + \frac{1}{4 \pi } \int_{B(x;t)  \cap \{ y_3 < 0 \} }\frac{  H_i( t - |y-x|, \bar y ) }{|y-x| } dy
\\ & + (-1)^i 2  \int_{ B(x;t) \cap \{y_3 =0\}}  \int_{\mathbb R^3 }  \frac{ \hat v_{\underline i }  f (t - |y-x |, y_\parallel, 0, v )}{|y-x | }  dv d  S_y.
\end{split}
\Ee

On the other hand, $B_3(t,x)$ satisfies
\[
\begin{split}
\p_t^2 B_3 - \Delta_x B_3 = & 4 \pi (\nabla_x \times J )_3 := H_3 \text{ in }  \O,
\\ B_3(0,x) = & B_{03}, \p_t B_3(0,x) = \p_t B_{03} \text{ in } \O,
\\ B_3 = & 0 \text{ on }  \p \O.
\end{split}
\]
Using the odd extension in $x_3$:
\[
\begin{split}
H_3(t,x)  = &  \mathbf 1_{x_3 > 0 }  H_3(t,x) -  \mathbf 1_{x_3 <  0 }  H_3(t,\bar x),
\\   B_{03} (x)  = &  \mathbf 1_{x_3 > 0 }   B_{03} (x) -  \mathbf 1_{x_3 <  0 }   B_{03} (\bar x),
\\  \p_t   B_{03} (0 , x) = &  \mathbf 1_{x_3 > 0 } \p_t  B_{03} (x) -  \mathbf 1_{x_3 <  0 } \p_t  B_{03}(\bar x),
\end{split}
\]
we get the expression for $B_3$:
\Be \label{B3rep}
\begin{split}
B_3(t,x) 
= &  \frac{1}{4 \pi t^2} \int_{\p B(x;t) \cap \{ y_3 > 0 \} }  \left( t \p_t  B_{03} (y ) +   B_{30}(y) + \nabla   B_{03}(y) \cdot (y-x) \right) dS_y
\\  & - \frac{1}{4 \pi t^2} \int_{\p B(x;t) \cap \{ y_3 < 0 \} }  \big(  t \p_t   B_{03}( \bar y ) +   B_{03} (\bar y)  + \nabla   B_{03} (\bar y) \cdot (\bar y - \bar x )  \big) dS_y
\\ & + \frac{1}{4 \pi } \int_{B(x;t)  \cap \{ y_3 > 0 \} }\frac{ H_3 ( t - |y-x|, y ) }{|y-x| } dy \\&- \frac{1}{4 \pi } \int_{B(x;t)  \cap \{ y_3 < 0 \} }\frac{  H_3( t - |y-x|, \bar y ) }{|y-x| } dy.
\end{split}
\Ee
Combining \eqref{Bi12rep} and \eqref{B3rep}, we get for $i=1,2,3$,
\begin{align}
  B_i(t,x) =
 \notag &  \frac{1}{4 \pi t^2} \int_{\p B(x;t) \cap \{ y_3 > 0 \} }  \left( t \p_t  B_{0i} ( y ) +   B_{0i}(y) + \nabla   B_{0i} (y) \cdot (y-x) \right) dS_y
\\   \notag & + \frac{\iota_i}{4 \pi t^2} \int_{\p B(x;t) \cap \{ y_3 < 0 \} }  \big(  t \p_t   B_{0i}( \bar y ) +   B_{0i}(\bar y)   + \nabla  B_{0i} (\bar y) \cdot (\bar y - \bar x )   \big) dS_y
\\ \label{Bexpanmajor1} & + \frac{1}{4 \pi } \int_{B(x;t)  \cap \{ y_3 > 0 \} }\frac{ H_i ( t - |y-x|, y ) }{|y-x| } dy
\\  \label{Bexpanmajor2} & + \frac{\iota_i}{4 \pi } \int_{B(x;t)  \cap \{ y_3 < 0 \} }\frac{  H_i( t - |y-x|, \bar y ) }{|y-x| } dy
\\ \notag & + (-1)^i  2( 1- \delta_{i3} )  \int_{ B(x;t) \cap \{y_3 =0\}}  \int_{\mathbb R^3 }  \frac{ \hat v_{\underline i }  f (t - |y-x |, y_\parallel, 0, v )}{|y-x | }  dv d  S_y.
\end{align}


Using \eqref{pxptST}, we have
\begin{align}
\notag \eqref{Bexpanmajor1}  =  &    \int_{B(x;t)   \cap \{ y_3 > 0 \} } \int_{\mathbb R^3} \frac{(  \nabla_x f \times \hat v )_i ( t - |y-x|, y, v ) }{|y-x| } dy
\\ \notag 
= &   \int_{B(x;t)  \cap \{ y_3 > 0 \} } \int_{\mathbb R^3}   \frac{ (\o \times \hat v)_i  }{ 1+ \hat{v} \cdot \o}  S  f ( t - |y-x|, y ,v  ) dv \frac{ dy} {|y-x|}
\\ \label{BupperT} & + \int_{B(x;t)  \cap \{ y_3 > 0 \} } \int_{\mathbb R^3}   \left(  (T \times \hat v)_i  - \frac{ (\o \times \hat v )_i \hat v \cdot T }{1+ \hat{v} \cdot  \o}\right)    f ( t - |y-x|, y, v ) dv \frac{ dy} {|y-x|}.
\end{align}
%
%
And for \eqref{BupperT}, we replace $T_j f$ with \eqref{T=y} and apply the integration by parts to get \eqref{BupperT} equals
\Be \label{BupperT1}
\begin{split}
 & \int_{  \p B(x; t)  \cap \{ y_3 > 0 \} } \int_{\mathbb R^3}   ( \o \times \hat v)_i  \left( 1   - \frac{  \hat v \cdot \o  }{1+ \hat{v} \cdot  \o}\right)     f ( 0, y, v ) dv \frac{ dS_y} {t}
\\ & +   \int_{ B(x;t)  \cap \{ y_3 = 0 \} } \int_{\mathbb R^3}      \left( - (e_3 \times \hat v)_i   + \frac{ ( \o \times \hat v  )_i }{1+ \hat{v} \cdot  \o}  ( \hat v \cdot e_3 )  \right)     f ( t- |y-x|, y_\parallel, 0, v ) dv \frac{ dy_\parallel} {|y-x|}
\\ & + \int_{B(x;t)  \cap \{ y_3 > 0 \} }  \int_{\mathbb R^3}    \frac{   ( \o \times \hat v  )_i \left( 1   -  |\hat v |^2  \right) }{( 1 + \hat v \cdot \o )^2|y-x |^2}  f   ( t - |y-x|, y, v )  dv dy .
\end{split}
\Ee
where we have used that, from \cite{GS,Glassey},
\Be \notag
\begin{split}
&    \p_{y_j }  \left( \frac{ ( \o \times \hat v  ) \hat v_j }{ ( 1 + \hat v \cdot \o )|y-x | } \right)
 \\ = &  \frac{   ( \o \times \hat v  ) \left( - ( \o \cdot \hat v )( 1 + \o \cdot \hat v ) -  ( \o \cdot \hat v ) ( 1 + \o \cdot \hat v )  -  |\hat v |^2  +  ( \o \cdot \hat v )^2 \right) }{( 1 + \hat v \cdot \o )^2|y-x |^2}
 \\ = &    \frac{   ( \o \times \hat v  ) \left( - 2 ( \o \cdot \hat v )   -  |\hat v |^2  -   ( \o \cdot \hat v )^2 \right) }{( 1 + \hat v \cdot \o )^2|y-x |^2},
 \end{split}
\Ee
and
\Be \notag
\begin{split}
& - \nabla_{y } (   \frac{1}{ |y-x|  }  ) \times \hat v +    \p_{y_j }  \left( \frac{ ( \o \times \hat v  ) \hat v_j }{ ( 1 + \hat v \cdot \o )|y-x | } \right)
\\ & =  \frac{   ( \o \times \hat v  ) \left( ( 1 + \hat v \cdot \o )^2 - 2 ( \o \cdot \hat v )   -  |\hat v |^2  -   ( \o \cdot \hat v )^2 \right) }{( 1 + \hat v \cdot \o )^2|y-x |^2}
 =   \frac{   ( \o \times \hat v  ) \left( 1   -  |\hat v |^2  \right) }{( 1 + \hat v \cdot \o )^2|y-x |^2}.
\end{split}
\Ee
Now we consider \eqref{Bexpanmajor2}. From \eqref{ST_lower},
\begin{align}
\notag \eqref{Bexpanmajor2}
& =  \iota_i \int_{B(x;t)   \cap \{ y_3 < 0 \} } \int_{\mathbb R^3} \frac{(  \nabla_x f \times \hat v )_i ( t - |y-x|, \bar y, v ) }{|y-x| } dy
\\ \notag 
 = &   \iota_i  \int_{B(x;t)  \cap \{ y_3 < 0 \} } \int_{\mathbb R^3}   \frac{ (\bar \o \times \hat v  )_i }{ 1+ \hat{v} \cdot \bar \o }  S  f ( t - |y-x|, \bar y ,v  ) dv \frac{ dy} {|y-x|}
\\ \label{BlowerT} &  + \iota_i \int_{B(x;t)  \cap \{ y_3 < 0 \} } \int_{\mathbb R^3}   \left(  ( \bar T \times \hat v )_i -  \frac{ ( \hat v \cdot  \bar T )  ( \bar \o \times \hat v )_i }{1+ \hat{v} \cdot  \bar \o}\right)    f ( t - |y-x|, \bar y, v ) dv \frac{ dy} {|y-x|}.
\end{align}
And for \eqref{BlowerT}, applying \eqref{def:T-} and the integration by parts, we derive that \eqref{BlowerT} equals
\Be \label{HSlowerhalf2}
\begin{split}
 &   \iota_i  \int_{ \p B(x; t) \cap \{ y_3 < 0 \} } \int_{\mathbb R^3}    (\bar \o \times \hat v)_i  \left( 1   - \frac{  \hat v \cdot \bar \o  }{1+ \hat{v} \cdot  \bar \o }\right)     f ( 0, \bar y, v ) dv \frac{ dS_y} {t}
\\  + &  \iota_i \int_{ B(x;t)  \cap \{ y_3 = 0 \} } \int_{\mathbb R^3}      \left( -(e_3 \times \hat v )_i  + \frac{  ( \bar \o \times \hat v )_i  }{1+ \hat{v} \cdot  \bar \o}  ( \hat v \cdot e_3 )  \right)   \frac{  f ( t- |y-x|, y_\parallel, 0, v ) }{|y-x|} dv d y_\parallel
\\ + & \iota_i \int_{B(x;t)  \cap \{ y_3 < 0 \} }  \int_{\mathbb R^3}  \left(    \frac{   ( \bar \o \times \hat v  )_i \left( 1   -  |\hat v |^2  \right) }{( 1 + \hat v \cdot \bar \o )^2|y-x |^2}  \right) f   ( t - |y-x|, \bar y, v )  dv dy,
\end{split}
\Ee
where we have used the direct computation
\Be \notag
\begin{split}
  \iota_j  \p_{y_j }  \left( \frac{ ( \bar \o \times \hat v  ) \hat v_j }{ ( 1 + \hat v \cdot \bar \o )|y-x | } \right) = &  \frac{   ( \bar \o \times \hat v  ) \left( - ( \hat v \cdot \bar \o )( 1 + \hat v \cdot \bar \o ) -   \hat v \cdot \bar \o - | \hat v |^2 \right) }{( 1 + \hat v \cdot \bar \o )^2|y-x |^2}
 \\ = &    \frac{   ( \bar \o \times \hat v  ) \left( - 2 ( \hat v \cdot \bar \o )   -  |\hat v |^2  -   ( \hat v \cdot \bar \o )^2 \right) }{( 1 +\hat v \cdot \bar \o )^2|y-x |^2},
 \end{split}
\Ee
and
\Be  \notag
\begin{split}
&-   \overline{ \nabla_y (|y-x|^{-1} ) }  \times \hat v  +   \iota_j  \p_{y_j }  \left( \frac{ ( \bar \o \times \hat v  ) \hat v_j }{ ( 1 + \hat v \cdot \bar \o )|y-x | } \right)
\\ & =  \frac{   ( \bar \o \times \hat v  ) \left( ( 1 + \hat v \cdot \bar \o )^2 - 2 ( \hat v \cdot \bar \o )   -  |\hat v |^2  -   ( \hat v \cdot \bar \o )^2 \right) }{( 1 + \hat v \cdot \bar \o )^2|y-x |^2}
 =   \frac{   ( \bar \o \times \hat v  ) \left( 1   -  |\hat v |^2  \right) }{( 1 + \hat v \cdot \bar \o )^2|y-x |^2}.
\end{split}
\Ee

Collecting the terms, we conclude the following formula:
\begin{proposition} \label{Biform}
\begin{align}
\label{Besttat0pos}& B_i(t,x ) 
\\ =  &  \frac{1}{4 \pi t^2} \int_{\p B(x;t) \cap \{ y_3 > 0 \} }  \left( t \p_t  B_{0,i}( y_\parallel, y_3 ) +  B_{0,i} (y_\parallel, y_3) + \nabla  B_{0,i} (y_\parallel, y_3) \cdot (y-x) \right) dS_y
\\ \notag & + \frac{\iota_i}{4 \pi t^2} \int_{\p B(x;t) \cap \{ y_3 < 0 \} }  \big(  t \p_t  B_{0,i}( y_\parallel,  - y_3 ) +  B_{0,i} (y_\parallel, -y_3) 
\\ \label{Besttat0neg}   & \quad \quad \quad \quad \quad \quad \quad  \quad \quad \quad \quad  + \nabla_\parallel   B_{0,i} (0,y_\parallel, - y_3) \cdot (y_\parallel-x_\parallel)  - \p_3   B_{0,i} (0,y_\parallel, - y_3) \cdot (y_3-x_3) \big) dS_y
\\ \label{Bestbulkpos}  & + \int_{B(x;t)  \cap \{ y_3 > 0 \} }  \int_{\mathbb R^3}   \frac{   ( \o \times \hat v  )_i \left( 1   -  |\hat v |^2  \right) }{( 1 + \hat v \cdot \o )^2|y-x |^2} f   ( t - |y-x|, y_\parallel, y_3, v )  dv dy
\\ \label{Bestbulkneg}  & + \int_{B(x;t)  \cap \{ y_3 < 0 \} }  \int_{\mathbb R^3}   \iota_i   \frac{   ( \o^- \times \hat v  )_i \left( 1   -  |\hat v |^2  \right) }{( 1 + \hat v \cdot \o^- )^2|y-x |^2} f   ( t - |y-x|, y_\parallel, - y_3, v )  dv dy
\\ \label{BestSpos} & +   \int_{B(x;t)  \cap \{ y_3 > 0 \} } \int_{\mathbb R^3}   \frac{(\o \times \hat v)_i  }{ 1+ \hat{v} \cdot \o}  S  f ( t - |y-x|, y_\parallel, y_3 ,v  ) dv \frac{ dy} {|y-x|}
\\ \label{BestSneg} & + \int_{B(x;t)  \cap \{ y_3 < 0 \} } \int_{\mathbb R^3}  \iota_i  \frac{ (\o^- \times \hat v)_i  }{ 1+ \hat{v} \cdot \o^-}  S  f ( t - |y-x|, y_\parallel,  - y_3 ,v  ) dv \frac{ dy} {|y-x|}
\\  \label{Bestbdrypos} & +   \int_{ B(x;t)  \cap \{ y_3 = 0 \} } \int_{\mathbb R^3}      \left( -(e_3 \times \hat v)_i   + \frac{ ( \o \times \hat v)_i   \hat v_3 }{1+ \hat{v} \cdot  \o}   \right)     f ( t- |y-x|, y_\parallel, 0, v ) dv \frac{ dy_\parallel} {|y-x|}
\\ \label{Bestbdryneg} & + \int_{ B(x;t)  \cap \{ y_3 = 0 \} } \int_{\mathbb R^3}  \iota_i    \left( - (e_3 \times \hat v )_i  + \frac{ ( \o^- \times \hat v)_i  \hat v_3   }{1+ \hat{v} \cdot  \o^-}    \right)     f ( t- |y-x|, y_\parallel, 0, v ) dv \frac{ dy_\parallel} {|y-x|}
\\ \label{Bestinitialpos} & + \int_{\p B(x;t) \cap \{ y_3>0 \} } \int_{\mathbb R^3}   \left(  \frac{  (  \o \times \hat v)_i   }{1+ \hat{v} \cdot  \o}\right)     f ( 0, y_\parallel, y_3, v ) dv \frac{ dS_y} {t}
\\ \label{Bestinitialneg} & + \int_{\p B(x;t) \cap \{ y_3 < 0 \} } \int_{\mathbb R^3}  \iota_i \left(  \frac{    ( \o^- \times \hat v )_i   }{1+ \hat{v} \cdot  \o^- }\right)     f ( 0, y_\parallel, - y_3, v ) dv \frac{ dS_y} {t}
\\ \label{Bestbdrycontri} & +(-1)^i 2 ( 1 - \delta_{i3} )   \int_{ B(x;t) \cap \{y_3 =0\}}  \int_{\mathbb R^3 }  \frac{ \hat v_{\underline{i} }  f (t - |y-x |, y_\parallel, 0, v )}{|y-x | }  dv d  S_y.
\end{align}
\end{proposition}

\section{Regularity estimate of the field}
With the formula for $E$ as in \eqref{Eesttat0pos}--\eqref{Eest3bdrycontri}, and $B$ as in \eqref{Besttat0pos}--\eqref{Bestbdrycontri}, we have the estimate of the fields.
\begin{lemma} \label{EBlinflemma}
There exists a $0 < T \ll 1 $ such that for any $t \in [0, T]$, we have
\Be \label{EBlinftybdd}
\| E(t) \|_\infty + \| B(t) \|_\infty \lesssim  \| E_0 \|_\infty + \| B_0 \|_\infty + t  \left( \sup_{ 0 \le t \le T} \| \langle v \rangle^{4+\delta} f(t) \|_\infty + \| E_0 \|_{C^1} + \|B_0 \|_{C^1 } \right).
\Ee
\end{lemma}
\begin{proof}
Using the expression for $E_i(t,x)$ in \eqref{Eesttat0pos}-\eqref{Eest3bdrycontri}, we have
\[
\begin{split}
| \eqref{Eesttat0pos}  | \le &    \frac{1}{ t^2} \int_{\p B(x;t) \cap \{ y_3 > 0 \} }  t |  \p_t  E_{0,i}( y ) + |  E_{0,i} (y) | + | \nabla  E_{0,i} (y )  | | (y-x) | dS_y
\\ \lesssim &\| E_0 \|_\infty +  t \left( \| \p_t E_0 \|_\infty +  + \| \nabla_x E_0 \|_\infty \right) \lesssim \| E_0 \|_\infty + t \| E_0 \|_{C^1}.
\end{split}
\]
And the same estimate can be made for \eqref{Eesttat0neg}. Thus
\Be \label{Elinftyest1}
| \eqref{Eesttat0pos}  |  + | \eqref{Eesttat0neg}  |  \lesssim \| E_0 \|_{C^1}.
\Ee
Next, from \cite{GS2} we have
\Be \label{vdecaybasic}
\frac{1}{1 + \hat v \cdot \o } \le \frac{ 2 (1+|v|^2 )}{ 1 + | v \times \o |^2 }  \le 2( 1 + |v|^2 ), \text{ and }  | \o + \hat v |^2 \le 2 ( 1 + \hat v \cdot \o ),
\Ee
and since $ 1 - |\hat{v}|^2 = \frac{1}{ 1 + |v|^2 }$,
\Be \label{vkerneldecay}
 |  \frac{ (|\hat{v}|^2-1 )(\hat{v}_i +  \o_i ) }{ (1+  \hat v \cdot \o )^2} | \le \frac{ \sqrt 2 }{ 1+ |v|^2 } \frac{1}{ (1 + \hat v \cdot \o)^{3/2} }  \le 4 \sqrt{ 1 + |v|^2 }.
\Ee
Thus,
\[
\begin{split}
| \eqref{Eestbulkpos}  | \le &     \int_{B(x;t) \cap \{y_3 > 0\}}  \int_{\R^3} \left| \frac{ (|\hat{v}|^2-1  )(\hat{v}_i +  \o_i ) }{|y-x|^2 (1+  \hat v \cdot \o )^2} \right| |  f(t-|y-x|,y ,v)| \dd v \dd y
\\ \le &   \int_{B(x;t) \cap \{y_3 > 0\}}  \int_{\R^3} \frac{4 \sqrt{ 1 + |v|^2 }  }{|y-x|^2} |  f(t-|y-x|,y ,v)| \dd v \dd y 
\\ \lesssim &  \sup_{ 0 \le t \le T} \| \langle v \rangle^{4+\delta} f(t) \|_\infty   \int_{B(x;t) \cap \{y_3 > 0\}} \frac{1}{|y-x |^2 }   \int_{\R^3} \frac{1}{ 1 + |v|^{3 + \delta } } dv dy
\\ \lesssim  & t  \sup_{ 0 \le t \le T} \| \langle v \rangle^{4+\delta} f(t) \|_\infty.
\end{split}
\]
We have the same estimate for \eqref{Eestbulkneg}, and thus
\Be \label{Elinftyest2}
| \eqref{Eestbulkpos}  | + | \eqref{Eestbulkneg}  | \lesssim  t \sup_{ 0 \le t \le T} \| \langle v \rangle^{4+\delta} f(t) \|_\infty.
\Ee

Next, from the equation \eqref{VMfrakF1} and the definition of $Sf$, we have
\[
Sf = - (E + E_{\text{ext}} + \hat v \times ( B + B_{\text{ext}}) - g\mathbf e_3 ) \cdot \nabla_v f.
\]
From integration by parts in $v$ and the fact that $ \nabla_v \cdot ( \hat v \times ( B + B_{\text{ext}})) =0$,
\Be \label{EestSposibp}
\begin{split}
& \eqref{EestSpos} 
\\ = &  \int_{B(x;t) \cap \{y_3 > 0\}}  \int_{\R^3} \frac{\o_i + \hat{v}_i}{1+  \hat v \cdot \o }((E + E_{\text{ext}} + \hat v \times ( B + B_{\text{ext}}) - g\mathbf e_3 ) \cdot \nabla_v f) (t-|y-x|,y ,v) \dd v \frac{ \dd y}{|y-x|}
\\ = &   \int_{B(x;t) \cap \{y_3 > 0\}}  \int_{\R^3} \nabla_v \left( \frac{\o_i + \hat{v}_i}{1+  \hat v \cdot \o } \right) \cdot (E + E_{\text{ext}} + \hat v \times ( B + B_{\text{ext}}) - g\mathbf e_3 )  f (t-|y-x|,y ,v) \dd v \frac{ \dd y}{|y-x|} 
\\ = &   \int_{B(x;t) \cap \{y_3 > 0\}}  \int_{\R^3}  \mathcal S^E_i ( v ,\o ) \cdot (E + \hat v \times ( B + B_{\text{ext}}) - g\mathbf e_3 )  
 f (t-|y-x|,y ,v) \dd v \frac{ \dd y}{|y-x|},
\end{split}
\Ee
where
\Be \label{calSi}
\mathcal S^E_i(\o,  v ) = \nabla_v \left( \frac{\o_i + \hat{v}_i}{1+  \hat v \cdot \o } \right) =  \frac{ (e_i - \hat v_i \hat v ) ( 1 + \hat v \cdot \o ) - ( \o_i + \hat v_i ) ( \o - (\o \cdot \hat v ) \hat v ) }{\langle v \rangle ( 1 + \hat v \cdot \o )^2 }.
\Ee
By writing  
\Be \label{ominusdotexp}
\o - (\o \cdot \hat v ) \hat v = \o(1 + \hat v \cdot \o ) - ( \hat v \cdot \o )( \o + \hat v ),
\Ee 
we have from \eqref{vdecaybasic},
\Be \label{calSiest}
\begin{split}
|  \mathcal S^E_i(\o, \hat v ) | \le &  | \frac{ (e_i - \hat v_i \hat v ) }{ \langle v \rangle ( 1 + \hat v \cdot \o )}  | + |  \frac{ \o ( \o_i + \hat v_i )  }{ \langle v \rangle ( 1 + \hat v \cdot \o )}  | + |  \frac{ ( \o_i + \hat v_i )  ( \hat v \cdot \o )( \o + \hat v )  }{ \langle v \rangle ( 1 + \hat v \cdot \o )^2}  |
\\ \le & 2 \sqrt{ 1 + |v|^2 } + 2 \sqrt{ 1 + |v|^2 }  + + 8 \sqrt{ 1 + |v|^2 }
\\ = & 12  \sqrt{ 1 + |v|^2 }.
\end{split}
\Ee
Thus,
\[
\begin{split}
& |\eqref{EestSpos}|
\\  \le &   \sup_{ 0 \le t \le T} \| \langle v \rangle^{4+\delta} f(t) \|_\infty   
\\ & \times \int_{B(x;t) \cap \{y_3 > 0\}}  \int_{\R^3}  | E + E_{\text{ext}} + \hat v \times ( B + B_{\text{ext}}) - g \mathbf e_3 |( t-|y-x | , y ) \frac{1}{ 1 + |v|^{3 + \delta }} dv \frac{dy}{|y-x | }  
\\ \lesssim &   \sup_{ 0 \le t \le T} \| \langle v \rangle^{4+\delta} f(t) \|_\infty    \sup_{ 0 \le t \le T} \left(  \| E(t)   \|_\infty +  \| B(t)   \|_\infty + g + E_e + |B_e|  \right)   \int_{B(x;t) \cap \{y_3 > 0\}}  \frac{1}{|y-x | } dy
\\ \lesssim & t^2  \sup_{ 0 \le t \le T} \| \langle v \rangle^{4+\delta} f(t) \|_\infty    \sup_{ 0 \le t \le T} \left(  \| E(t)   \|_\infty +  \| B(t)   \|_\infty  + g + E_e + |B_e| \right).
\end{split}
\]
Applying the same estimate to \eqref{EestSneg} we get
\Be \label{Elinftyest3}
|\eqref{EestSpos}| + |\eqref{EestSneg}| \lesssim t^2  \sup_{ 0 \le t \le T} \| \langle v \rangle^{4+\delta} f(t) \|_\infty    \sup_{ 0 \le t \le T} \left(  \| E(t)   \|_\infty +  \| B(t)   \|_\infty + g + E_e+ |B_e|  \right).
\Ee
Next, we have from \eqref{vdecaybasic},
\Be \label{vdecaybasic2}
| \frac{(\o_i + \hat{v}_i)}{1+ \hat v \cdot \o }  |  \le 2 \sqrt{ 1 + |v|^2}.
\Ee
So
\[
\begin{split}
| \eqref{Eestbdrypos} |  \le &   \int_{ \sqrt{t^2 - |z_\parallel |^2 } > x_3}  
\int_{\R^3} |  \frac{(\o_i + \hat{v}_i)\hat{v}_3}{1+ \hat v \cdot \o } | |  f (t- ( |z_\parallel | + x_3^2 )^{1/2} ,z_\parallel + x_\parallel, 0 ,v)  | \dd v\frac{ \dd z_\parallel}{(|z_\parallel |^2 + x_3^2)^{1/2}}
\\ \lesssim &   \sup_{ 0 \le t \le T} \| \langle v \rangle^{4+\delta} f(t) \|_\infty    \int_{ \sqrt{ |z_\parallel |^2 + x_3^2 } < t } \frac{1}{(|z_\parallel |^2 + x_3^2)^{1/2}}  \int_{\R^3} \frac{1}{1 + |v|^{3 + \delta } } dv  \dd z_\parallel
\\ \lesssim & t  \sup_{ 0 \le t \le T} \| \langle v \rangle^{4+\delta} f(t) \|_\infty.
\end{split} 
\]
And we have the same estimate for \eqref{Eestbdryneg}, thus
\Be \label{Elinftyest4}
| \eqref{Eestbdrypos} | + | \eqref{Eestbdryneg} |  \lesssim t  \sup_{ 0 \le t \le T} \| \langle v \rangle^{4+\delta} f(t) \|_\infty.
\Ee
Again from \eqref{vdecaybasic2}, 
\[
\begin{split}
| \eqref{Eestinitialpos} | \le &     \int_{|x-y| = t , \ y_3>0}  |  \sum_j \o_j  \left(\delta_{ij} - \frac{(\o_i + \hat{v}_i)\hat{v}_j}{1+ \hat v \cdot \o }\right) f(0, y,v) |  \dd v \frac{\dd S_y}{|y-x|}
\\ \lesssim & \frac{1}{t}   \| \langle v \rangle^{4+\delta} f(0) \|_\infty  \int_{|x-y| = t , \ y_3>0}  \int_{\mathbb R^3 }  \frac{1}{1 + |v|^{3 + \delta } } dv dS_y 
\\ \lesssim & t  \| \langle v \rangle^{4+\delta} f(0) \|_\infty.
\end{split}
\]
Using the same estimate for \eqref{Eestinitialneg}, we get
\Be \label{Elinftyest5}
| \eqref{Eestinitialpos} | + | \eqref{Eestinitialneg} |  \lesssim t \| \langle v \rangle^{4+\delta} f(0) \|_\infty.
\Ee
Next, we have
\Be \label{Elinftyest6}
\begin{split}
| \eqref{Eest3bdrycontri} |   \le &     \int_{ B(x;t) \cap \{y_3 =0\}}  \int_{\mathbb R^3 }  \frac{ 8 \pi |  f (t - |y-x |, y_\parallel, 0, v ) | }{|y-x | }  dv d  S_y 
\\ \lesssim &  \sup_{ 0 \le t \le T} \| \langle v \rangle^{4+\delta} f(t) \|_\infty  \int_{ \sqrt{ |z_\parallel |^2 + x_3^2 } < t } \frac{1}{(|z_\parallel |^2 + x_3^2)^{1/2}}  \int_{\mathbb R^3 } \frac{1}{1 + |v|^{4 + \delta } } dv  \dd z_\parallel
\\ \lesssim &t  \sup_{ 0 \le t \le T} \| \langle v \rangle^{4+\delta} f(t) \|_\infty.
\end{split}
\Ee

Next, we go to estimate for $B(t,x)$ in \eqref{Besttat0pos}--\eqref{Bestbdrycontri}. Similar to \eqref{Elinftyest1}, we have
\Be \label{Blinftyest1}
| \eqref{Besttat0pos} | + | \eqref{Besttat0neg} | \lesssim \| B_0 \|_{C^1 }.
\Ee
From \eqref{vdecaybasic} we have
\Be \label{vdecayBbulk}
|  \frac{   ( \o \times \hat v  ) \left( 1   -  |\hat v |^2  \right) }{( 1 + \hat v \cdot \o )^2} |  \le 2 |  \frac{ \o \times  v }{ (1 + |v|^2 )^{3/2}  ( 1 + \hat v \cdot \o )^2 }  | \le  8  \frac{  | \o \times  v |  \sqrt{ 1 + |v|^2 }  }{ (1 + | \o \times v |^2 )^2  } \le 8  \sqrt{ 1 + |v|^2 }.
\Ee
Thus 
\[
\begin{split}
| \eqref{Bestbulkpos} | + | \eqref{Bestbulkneg} |   \lesssim &  \sup_{ 0 \le t \le T} \| \langle v \rangle^{4+\delta} f(t) \|_\infty   \int_{B(x;t)} \frac{1}{|y-x |^2 }   \int_{\R^3} \frac{1}{ 1 + |v|^{3 + \delta } } dv dy
\\ \lesssim  & t \sup_{ 0 \le t \le T} \| \langle v \rangle^{4+\delta} f(t) \|_\infty.
\end{split}
\]
Next, using  the equation \eqref{VMfrakF1} and the definition of $Sf$, from integration by parts in $v$,
\Be \label{BestSposibp}
\begin{split}
& \eqref{BestSpos} 
\\= &  \int_{B(x;t) \cap \{y_3 > 0\}}  \int_{\R^3} \frac{(\o \times \hat v )_i}{1+  \hat v \cdot \o }((E + E_{\text{ext}} + \hat v \times ( B + B_{\text{ext}}) - g\mathbf e_3 ) \cdot \nabla_v f) (t-|y-x|,y ,v) \dd v \frac{ \dd y}{|y-x|}
\\ = &   \int_{B(x;t) \cap \{y_3 > 0\}}  \int_{\R^3} \nabla_v \left( \frac{(\o \times \hat v )_i}{1+  \hat v \cdot \o } \right) \cdot (E + E_{\text{ext}} + \hat v \times ( B + B_{\text{ext}}) - g\mathbf e_3 )  f (t-|y-x|,y ,v) \dd v \frac{ \dd y}{|y-x|} 
\\ = &   \int_{B(x;t) \cap \{y_3 > 0\}}  \int_{\R^3}  \mathcal S^B_i ( v ,\o ) \cdot (E + \hat v \times ( B + B_{\text{ext}}) - g\mathbf e_3 )  
 f (t-|y-x|,y ,v) \dd v \frac{ \dd y}{|y-x|},
\end{split}
\Ee
where by direct calculation we have
\Be \label{calBSirep}
\begin{split}
 \mathcal S^B_i(v,\o ) =  \nabla_v \left( \frac{( \o \times  \hat v )_i }{ 1 + \hat v \cdot \o } \right) = &   \nabla_v \left( \frac{( \o \times   v )_i }{ \sqrt{ 1+ |v|^2 } +  v \cdot \o } \right)  
 \\ = & \frac{ \nabla_v [ (\o \times v)_i ] }{\sqrt{1+|v|^2 }+  v \cdot \o  }  + \frac{ (\o \times v)_i ( \hat v  + \o ) }{ ( \sqrt{1+|v|^2 } + v \cdot \o )^2 }
 \\ = &  \frac{ \nabla_v [ (\o \times v)_i ] }{\sqrt{1+|v|^2 } (1 + \hat  v \cdot \o  ) }  + \frac{ (\o \times v)_i ( \hat v  + \o ) }{ ( \sqrt{1+|v|^2 } (1  + \hat v \cdot \o ) )^2 }.
\end{split}
\Ee
Therefore from \eqref{vdecaybasic}, we have
\Be \label{calBSiest}
\begin{split}
|   \mathcal S^B_i(v,\o ) | \le  \frac{ 2}{\sqrt{1+|v|^2 } (1 + \hat  v \cdot \o  ) }  + 8  \frac{ | \o \times v|  \sqrt{1+|v|^2 }  }{  (1  + | \o \times v |^2 ) ^2 } \le 12 \sqrt{1 + |v|^2 }.
\end{split}
\Ee
So we can use the same argument as in \eqref{EestSposibp}--\eqref{Elinftyest3} with \eqref{BestSposibp} and \eqref{calBSiest} to get
\[
\begin{split}
& |\eqref{BestSpos}| + |\eqref{BestSneg}| 
\\ \lesssim &   \sup_{ 0 \le t \le T} \| \langle v \rangle^{4+\delta} f(t) \|_\infty    \sup_{ 0 \le t \le T} \left(  \| E(t)   \|_\infty +  \| B(t)   \|_\infty + g + E_e + |B_e|  \right)   \int_{B(x;t)}  \frac{1}{|y-x | } dy
\\ \lesssim & t^2  \sup_{ 0 \le t \le T} \| \langle v \rangle^{4+\delta} f(t) \|_\infty    \sup_{ 0 \le t \le T} \left(  \| E(t)   \|_\infty +  \| B(t)   \|_\infty + g  + E_e + |B_e|  \right).
\end{split}
\]
Next, again from \eqref{vdecaybasic},
\Be \label{vdecaybasic3}
| \frac{ ( \o \times \hat v)_i   }{1+ \hat{v} \cdot  \o} | \le 2 \sqrt{1+ |v|^2 }.
\Ee
So similar to \eqref{Elinftyest4}, \eqref{Elinftyest5} we get
\Be 
| \eqref{Bestbdrypos} | + | \eqref{Bestbdryneg} |  \lesssim t \sup_{ 0 \le t \le T} \| \langle v \rangle^{4+\delta} f(t) \|_\infty,
\Ee
and
\Be 
| \eqref{Bestinitialpos} | + | \eqref{Bestinitialneg} |  \lesssim t \| \langle v \rangle^{4+\delta} f(0) \|_\infty.
\Ee
And finally same as \eqref{Elinftyest6}, we have
\Be \label{Blinftyest6}
| \eqref{Bestbdrycontri} |   \lesssim t  \sup_{ 0 \le t \le T} \| \langle v \rangle^{4+\delta} f(t) \|_\infty.
\Ee

Combining \eqref{Elinftyest1}, \eqref{Elinftyest2}, \eqref{Elinftyest3}, \eqref{Elinftyest4}, \eqref{Elinftyest5}, \eqref{Elinftyest6}, and \eqref{Blinftyest1}--\eqref{Blinftyest6}, we get
\Be \label{EBinftybd}
 \begin{split}
& \sup_{0 \le t \le T} \| E(t) \|_\infty + \sup_{0 \le t \le T} \| B(t) \|_\infty
\\  \lesssim & \| E_0 \|_\infty + \|B_0 \|_\infty + t(  \| E_0 \|_{C^1 } +  \| B_0 \|_{C^1 }   )
\\ & +  t \sup_{ 0 \le t \le T} \| \langle v \rangle^{4+\delta} f(t) \|_\infty \left( 1  +  t    ( \sup_{ 0 \le t \le T} \left(  \| E(t)   \|_\infty +  \| B(t)   \|_\infty  \right)  + g + E_e + |B_e| )\right).
\end{split} 
\Ee
Taking $0 < T \ll1 $, we conclude \eqref{EBlinftybdd}.

\end{proof}

Next we estimate the derivatives of the fields.

\begin{lemma} \label{EBW1inftylemma}
With the formula $E(t,x)$ as in \eqref{Eesttat0pos}--\eqref{Eest3bdrycontri}, and $B(t,x)$ as in \eqref{Besttat0pos}--\eqref{Bestbdrycontri}, there exists a $T \ll 1 $ such that for any $t \in [0, T]$,
\Be \label{nablaxparaE}
\begin{split}
\|  \nabla_{x_\parallel}  E(t) \|_\infty + \|  \nabla_{x_\parallel}  B(t) \|_\infty \lesssim & \| E_0 \|_{C^2} + \| B_0 \|_{C^2} + \sup_{0 \le t \le T} \left(   \| \langle v \rangle^{4 + \delta }   \nabla_{x_\parallel} f(t) \|_\infty  \right)
\\ & +  \sup_{0 \le t \le T}  \| \langle v \rangle^{4+\delta} f(t) \|_\infty   ,
\end{split}
\Ee
\Be \label{lambdanablaxE}
\begin{split}
\|  \p_{x_3}  E(t) \|_\infty + \|  \p_{x_3}  B(t) \|_\infty \lesssim & \| E_0 \|_{C^2} + \| B_0 \|_{C^2} + \sup_{0 \le t \le T} \left(  \| \langle v \rangle^{5 + \delta }  \alpha \p_{x_3} f(t) \|_\infty +  \| \langle v \rangle^{4 + \delta }   \nabla_{x_\parallel} f(t) \|_\infty  \right)
\\ & +  \sup_{0 \le t \le T}  \| \langle v \rangle^{4+\delta} f(t) \|_\infty,
\end{split}
\Ee
and
\Be \label{ptEBest}
\begin{split}
\| \p_t E (t) \|_\infty + \| \p_t B (t) \|_\infty \lesssim &  \| E_0 \|_{C^2}+ \| B_0 \|_{C^2} +  \sup_{ 0 \le t \le T} \left( \| \langle v \rangle^{4+\delta} \nabla_{x_\parallel} f(t) \|_\infty + \|  \langle v \rangle^{5+\delta} \alpha \p_{x_3} f(t) \|_\infty \right) 
\\ &  + \sup_{0 \le t \le T} \| \langle v \rangle^{4+\delta} f(t) \|_\infty.
\end{split}
\Ee


\end{lemma}

\begin{proof}

We take derivative to $\frac{ \p}{\p_{x_k } }  E_i(t,x)$ in \eqref{Eesttat0pos}-\eqref{Eest3bdrycontri} and estimate each term.

First, by using the change of variables $z = y - x$ and spherical coordinate for $z$, we have
\Be \label{Eestat00}
\begin{split}
 \eqref{Eesttat0pos} = &  \frac{1}{4\pi t^2 } \int_{ \{ |z| = t , x_3 + z_3 > 0 \} }    \left( t \p_t  E_{0,i}( x + z  ) +  E_{0,i} (x+z) + \nabla  E_{0,i} ( x+ z) \cdot z \right) dS_z
 \\ = &  \frac{1}{4\pi t^2 } \int_{ \{ t \cos \phi > - x_3 \} }  \int_0^{2 \pi }  \left( t \p_t  E_{0,i}( x + z  ) +  E_{0,i} (x+z) + \nabla  E_{0,i} ( x+ z) \cdot z \right) t^2 \sin \phi d\theta d\phi.
 \end{split}
\Ee
Thus
\Be
\begin{split}
\frac{ \p}{\p_{x_k } }  \eqref{Eesttat0pos}  = &   \frac{1}{4\pi t^2 } \int_{ \{ |z| = t , x_3 + z_3 > 0 \} }    \left( t \p_{x_k } \p_t  E_{0,i}( x + z  ) +  \p_{x_k }  E_{0,i} (x+z) +  \nabla  \p_{x_k }  E_{0,i} ( x+ z) \cdot z \right) dS_z
\\ & + \frac{1}{4 \pi t^2 }  \delta_{k3 } \int_0^{2\pi } (  t \p_t  E_{0,i}( x_\parallel + z_\parallel, 0  ) +  E_{0,i} (x_\parallel + z_\parallel, 0 ) + \nabla  E_{0,i} ( x_\parallel + z_\parallel, 0 ) \cdot z )
\\ & \quad \quad \quad \quad \quad \quad \times     \left( \frac{d}{d x_3} \cos^{-1} \left( \frac{-x_3}{t} \right) \right) \left(  \sin \left( \cos^{-1} \left( \frac{-x_3}{t} \right) \right)  \right) t^2 d \theta
\\ = &   \frac{1}{4\pi t^2 } \int_{ \{ |z| = t , x_3 + z_3 > 0 \} }    \left( t \p_{x_k } \p_t  E_{0,i}( x + z  ) +  \p_{x_k }  E_{0,i} (x+z) +  \nabla  \p_{x_k }  E_{0,i} ( x+ z) \cdot z \right) dS_z
\\ & + \frac{1}{4 \pi t}  \delta_{k3 } \int_0^{2\pi } (  t \p_t  E_{0,i}( x_\parallel + z_\parallel , 0  ) +  E_{0,i} (x_\parallel + z_\parallel , 0) + \nabla  E_{0,i} ( x_\parallel + z_\parallel , 0) \cdot z )  d\theta.
\end{split}
\Ee
For $i = 1,2$, $E_{0,i} (  x_\parallel + z_\parallel , 0 ) = 0$, thus we have
\[
| \frac{ \p}{\p_{x_k } }  \eqref{Eesttat0pos}_{i=1,2}  | \lesssim  t \| \nabla_x \p_t E_0 \|_\infty + \| \nabla_x E_0 \|_\infty +  t \| \nabla_x ^2 E_0 \|_\infty + \| \p_t E_0 \|_\infty + \| \nabla_x E_0 \|_\infty  \lesssim \| E_0 \|_{C^2 }.
\] 
And we apply the same estimate for $  \frac{ \p}{\p_{x_k } }  \eqref{Eesttat0neg}_{i=1,2} $ to obtain
\[
| \frac{ \p}{\p_{x_k } }  \eqref{Eesttat0pos}_{i=1,2} |  + |  \frac{ \p}{\p_{x_k } }  \eqref{Eesttat0neg}_{i=1,2}  | \lesssim  \| E_0 \|_{C^2 }.
\]
For $i=3$, we use the cancellation for $ \frac{ \p}{\p_{x_k } }  \eqref{Eesttat0pos}_{i=3} + \frac{ \p}{\p_{x_k } }  \eqref{Eesttat0neg}_{i=3}  $  at $y_3 = 0$ to get
\Be \label{pxE03posneg}
\begin{split}
& \frac{ \p}{\p_{x_k } }  \eqref{Eesttat0pos}_{i=3} + \frac{ \p}{\p_{x_k } }  \eqref{Eesttat0neg}_{i=3} 
\\ =  &  \frac{1}{4\pi t^2 } \int_{ \{ |z| = t , x_3 + z_3 > 0 \} }    \left( t \p_{x_k } \p_t  E_{0,3}( x + z  ) +  \p_{x_k }  E_{0,3} (x+z) +  \nabla  \p_{x_k }  E_{0,3} ( x+ z) \cdot z \right) dS_z
\\ & + \frac{1}{4\pi t^2 } \int_{ \{ |z| = t , x_3 + z_3 < 0 \} }   \iota_k \left( t \p_{x_k } \p_t  E_{0,3}( x + z  ) +  \p_{x_k }  E_{0,3} (x+z) +  \nabla_\parallel  \p_{x_k }  E_{0,3} ( x+ z) \cdot z_\parallel - \p_{x_3}  \p_{x_k }  E_{0,3} ( x+ z) \cdot z_3  \right) dS_z
\\ & +  \frac{2}{4\pi t^2 }  \delta_{k 3} \int_{0}^{2 \pi } \left(   \p_{x_3} E_{0, 3 } ( x_\parallel + z_\parallel, 0 ) \right) \frac{-x_3 }{t} t^2 d\theta.
\end{split}
\Ee
Thus $| \frac{ \p}{\p_{x_k } }  \eqref{Eesttat0pos}_{i=3} |  + |  \frac{ \p}{\p_{x_k } }  \eqref{Eesttat0neg}_{i=3}  | \lesssim  \| E_0 \|_{C^2 }$, and therefore
\Be \label{Eestat01}
| \frac{ \p}{\p_{x_k } }  \eqref{Eesttat0pos} |  + |  \frac{ \p}{\p_{x_k } }  \eqref{Eesttat0neg}  | \lesssim  \| E_0 \|_{C^2 }.
\Ee

Next, using the change of variables $ z = y-x$ we have
\Be \label{Eestbulkpos1}
\begin{split}
\frac{ \p}{\p_{x_k } }  \eqref{Eestbulkpos}  =&  \int_{ \{ |z| < t  \} \cap \{z_3 > -x_3 \} } \int_{\R^3} \frac{ (|\hat{v}|^2-1 )(\hat{v}_i +  \o_i ) }{|z|^2 (1+  \hat v \cdot \o )^2} \p_{x_k}f (t-|z| ,z+x ,v) \dd v \dd z  
 \\& + \delta_{k3} \int_{ \{ |z| < t  \} \cap \{z_3 = -x_3 \} }   \int_{\R^3}\frac{ (|\hat{v}|^2-1  )(\hat{v}_i +  \o_i ) }{|z|^2 (1+  \hat v \cdot \o )^2} f(t-|z|,z+x ,v) \dd v \dd z_\parallel
 \\ = & \underbrace{ \int_{ \{ |y-x| < t  \} \cap \{y_3 > 0 \} } \int_{\R^3} \frac{ (|\hat{v}|^2-1 )(\hat{v}_i +  \o_i ) }{|y-x|^2 (1+  \hat v \cdot \o )^2} \p_{x_k}f (t-|y-x| ,y ,v) \dd v \dd y  }_{\eqref{Eestbulkpos1}_1}
 \\& + \underbrace{ \delta_{k3} \int_{ \{  (| y_\parallel - x_\parallel |^2 + |x_3|^2 )^{1/2} < t  \}  }   \int_{\R^3}\frac{ (|\hat{v}|^2-1  )(\hat{v}_i +  \o_i ) }{ ( |y_\parallel - x_\parallel  |^2 + x_3^2 ) (1+  \hat v \cdot \o )^2} f(t-|y-x|, y_\parallel, 0  ,v) \dd v \dd y_\parallel}_{\eqref{Eestbulkpos1}_2}.
\end{split}
\Ee
Using \eqref{vkerneldecay}, we have for $k =1,2$,
\Be \label{Eestbulkpospara2}
\begin{split}
& \int_{ \{ |y-x| < t  \} \cap \{y_3 > 0 \} } \int_{\R^3} \frac{ (|\hat{v}|^2-1 )(\hat{v}_i +  \o_i ) }{|y-x|^2 (1+  \hat v \cdot \o )^2} \p_{x_k}f (t-|y-x| ,y ,v) \dd v \dd y
 \\ \lesssim & \sup_{0 \le t \le T}   \| \langle v \rangle^{4+\delta}   \nabla_{x_\parallel} f (t) \|_\infty  \int_{ \{ |y-x| < t  \} } \int_{\R^3} \frac{ 1 }{|y-x|^2}  ( 1 + |v| )^{-3-\delta} \dd v \dd y
 \\ \lesssim &  \sup_{0 \le t \le T}   \| \langle v \rangle^{4+\delta}   \nabla_{x_\parallel} f (t) \|_\infty.
\end{split}
\Ee
For $k= 3$, we have for any $ 1 < p <  \frac{3}{2}$, from \eqref{vkerneldecay}, and Lemma \ref{1alphaintv} which will be proved in the next section,
\Be \label{Eestbulkpos2}
\begin{split}
  & \int_{ \{ |y-x| < t  \} \cap \{y_3 > 0 \} } \int_{\R^3} \frac{ (|\hat{v}|^2-1 )(\hat{v}_i +  \o_i ) }{|y-x|^2 (1+  \hat v \cdot \o )^2} \p_{x_3}f (t-|y-x| ,y ,v) \dd v \dd y
 \\ \lesssim & \sup_{0 \le t \le T} \| \langle v \rangle^{5+\delta}  \alpha \p_{x_3} f (t) \|_\infty   \int_{ \{ |y-x| < t  \} \cap \{y_3 > 0 \} } \int_{\R^3} \frac{1}{|y-x |^2 } \frac{ ( 1 + |v| )^{-4-\delta} }{ \alpha(t- |y-x |, y, v ) } dv dy
 \\ \lesssim &   \sup_{0 \le t \le T}  \left( \|  \langle v \rangle^{5+\delta}  \alpha \p_{x_3} f (t) \|_\infty  \right)   \int_{ \{ |y-x| < t  \} \cap \{y_3 > 0 \} }  \frac{1}{|y-x|^2}\left(1 +   \ln (1 + \frac{1}{y_3}) \right)  dy
 \\ \lesssim &  \sup_{0 \le t \le T}  \left( \|  \langle v \rangle^{5+\delta}  \alpha \p_{x_3} f (t) \|_\infty   \right)  \int_{ \{ |y-x| < t  \} \cap \{y_3 > 0 \} }  \left( \frac{1}{|y-x|^{2p } } + | \ln ( 1 + \frac{1}{y_3} ) |^{\frac{p}{p-1} }  \right) dy
 \\ \lesssim  &   \sup_{0 \le t \le T}  \|  \langle v \rangle^{5+\delta}  \alpha \p_{x_3} f (t) \|_\infty .
 \end{split}
\Ee
We leave the estimate of $\eqref{Eestbulkpos1}_2$ together with the estimate of $\p_{x_k} \eqref{Eestbdrypos}$ later.

Next, from \eqref{EestSposibp}, and using the change of variables $z = y -x $ and taking $\frac{\p}{\p_{x_k } } $ derivative to \eqref{EestSpos} we have
\Be \label{EestSpos1}
\begin{split}
\frac{ \p }{\p_{x_k } } \eqref{EestSpos} = &  \int_{B(x;t) \cap \{y_3 > 0\}}  \int_{\R^3}   \mathcal S^E_i (v ,\o ) \cdot ( \p_{x_k } E +   \hat v  \times  \p_{x_k } B)   f (t-|y-x|,y ,v) \dd v \frac{ \dd y}{|y-x|}
\\  & +  \int_{B(x;t) \cap \{y_3 > 0\}}  \int_{\R^3}   \mathcal S^E_i (v ,\o ) \cdot (  E +   \hat v  \times ( B + B_{\text{ext}}) - g\mathbf e_3 ) \p_{x_k}  f (t-|y-x|,y ,v) \dd v \frac{ \dd y}{|y-x|}
\\ & + \delta_{k3} \int_{B(x;t) \cap \{y_3 = 0\}}  \int_{\R^3}   \mathcal S^E_i (v ,\o ) \cdot (  E +   \hat v  \times ( B + B_{\text{ext}}) - g\mathbf e_3)  f (t-|y-x|,y_\parallel, 0 ,v) \dd v \frac{ \dd y_\parallel }{|y-x|}.
\end{split}
\Ee
Thus for $k=1,2$, from \eqref{calSiest} we have,
\Be \label{EestSpospara2}
\begin{split}
& |  \int_{B(x;t) \cap \{y_3 > 0\}}  \int_{\R^3}   \mathcal S^E_i (v ,\o ) \cdot ( \p_{x_k } E +   \hat v  \times  \p_{x_k } B)   f (t-|y-x|,y ,v) \dd v \frac{ \dd y}{|y-x|} | 
\\ \lesssim &  \sup_{0 \le t \le T} \| ( 1 + |v|^{4 +\delta } )  f(t) \|_\infty  \int_{B(x;t) \cap \{y_3 > 0\}}   \int_{\mathbb R^3 }   | ( \p_{x_k } E +   \hat v  \times  \p_{x_k } B) |  \left( \frac{1}{ 1 + |v|^{3 + \delta } } \right)  dv  \frac{ \dd y}{|y-x|} 
\\ \lesssim &   \sup_{0 \le t \le T} \|  ( 1 + |v|^{4 +\delta } )  f(t) \|_\infty \left(    \sup_{0 \le t \le T} \|   \nabla_{x_\parallel} E(t)  \|_\infty +  \sup_{0 \le t \le T} \|   \nabla_{x_\parallel} B(t)  \|_\infty   \right)  \times \int_{B(x;t) } \frac{dy}{|y-x| } 
\\ \lesssim & t^2  \sup_{0 \le t \le T} \|  ( 1 + |v|^{4 +\delta } )  f(t) \|_\infty \left(    \sup_{0 \le t \le T} \|   \nabla_{x_\parallel} E(t)  \|_\infty +  \sup_{0 \le t \le T} \|   \nabla_{x_\parallel} B(t)  \|_\infty   \right).
\end{split}
\Ee
Similarly, for $k=3$, 
\Be \label{EestSpos2}
\begin{split}
 & |  \int_{B(x;t) \cap \{y_3 > 0\}}  \int_{\R^3}   \mathcal S^E_i (v ,\o ) \cdot ( \p_{x_k } E +   \hat v  \times  \p_{x_3 } B)   f (t-|y-x|,y ,v) \dd v \frac{ \dd y}{|y-x|} | 
\\ \lesssim & t^{2 }  \sup_{0 \le t \le T} \| ( 1 + |v|^{4 +\delta } )  f(t) \|_\infty \left(    \sup_{0 \le t \le T} \|  \p_{x_3} E(t)  \|_\infty +  \sup_{0 \le t \le T} \|   \p_{x_3} B(t)  \|_\infty   \right),
\end{split}
\Ee
for any $0 < \delta \ll 1 $. Thus for $ t \ll 1$, the terms from \eqref{EestSpospara2}, \eqref{EestSpos2} will be absorbed into the LHS. From \eqref{calSiest}, we have for $k=1,2$,
\Be \label{EestSpospara3}
\begin{split}
& | \int_{B(x;t) \cap \{y_3 > 0\}}  \int_{\R^3}   \mathcal S^E_i (v ,\o ) \cdot (  E + E_{\text{ext}} +   \hat v  \times ( B + B_{\text{ext}}) - g\mathbf e_3) \p_{x_k}  f (t-|y-x|,y ,v) \dd v \frac{ \dd y}{|y-x|} |
\\ \lesssim &  \left(    \sup_{0 \le t \le T} \| E(t)  \|_\infty +  \sup_{0 \le t \le T} \|  B(t)  \|_\infty  + |B_e| + g  \right)   \sup_{0 \le t \le T}  \left(   \| \langle v \rangle^{4+\delta}   \nabla_{x_\parallel} f (t) \|_\infty  \right)
\\ & \quad \quad \times   \int_{B(x;t) \cap \{y_3 > 0\}}  \int_{\R^3} \frac{(1+ |v|^{-3 - \delta } )}{|y-x|} dv dy 
\\ \lesssim &  \left(    \sup_{0 \le t \le T} \| E(t)  \|_\infty +  \sup_{0 \le t \le T} \|  B(t)  \|_\infty + |B_e| +g   \right)   \sup_{0 \le t \le T}  \left(   \| \langle v \rangle^{4+\delta}   \nabla_{x_\parallel} f (t) \|_\infty  \right).
\end{split}
\Ee
And, for $k=3$,
\Be \label{EestSpos3}
\begin{split}
& | \int_{B(x;t) \cap \{y_3 > 0\}}  \int_{\R^3}   \mathcal S^E_i (v ,\o ) \cdot (  E + E_{\text{ext}} +   \hat v  \times ( B + B_{\text{ext}}) - g \mathbf e_3) \p_{x_3}  f (t-|y-x|,y ,v) \dd v \frac{ \dd y}{|y-x|} |
\\ \lesssim &  \left(    \sup_{0 \le t \le T} \| E(t)  \|_\infty +  \sup_{0 \le t \le T} \|  B(t)  \|_\infty  +|B_e| + E_e +g  \right)   \sup_{0 \le t \le T}  \left( \|  \langle v \rangle^{5+\delta}  \alpha \p_{x_3} f (t) \|_\infty   \right) 
\\  & \quad \quad \quad \times \int_{B(x;t) \cap \{y_3 > 0\}} \int_{\R^3}   \frac{1}{ |y-x| } \left( \frac{ \langle v \rangle^{ -4- \delta} }{ \alpha(t- |y-x |, y, v ) }  \right) dv  dy 
\\ \lesssim &  \left(    \sup_{0 \le t \le T} \| E(t)  \|_\infty +  \sup_{0 \le t \le T} \|  B(t)  \|_\infty + |B_e| +E_e+g   \right)   
\\ & \times \sup_{0 \le t \le T}  \left( \|  \langle v \rangle^{5+\delta}  \alpha \p_{x_3} f (t) \|_\infty  \right)  \int_{B(x;t) \cap \{y_3 > 0\}}   \frac{1}{ |y-x| } | \ln (y_3 ) |   dy 
\\ \lesssim &  \left(    \sup_{0 \le t \le T} \| E(t)  \|_\infty +  \sup_{0 \le t \le T} \|  B(t)  \|_\infty +|B_e| + E_e +g   \right)    \sup_{0 \le t \le T}  \left( \|  \langle v \rangle^{5+\delta}  \alpha \p_{x_3} f (t) \|_\infty   \right)
\\ &\quad \quad \times \int_{B(x;t) \cap \{y_3 > 0\}}  \left(  \frac{1}{ |y-x|^2  } + \left( \ln (y_3 ) \right)^2 \right)   dy 
\\ \lesssim & \left(    \sup_{0 \le t \le T} \| E(t)  \|_\infty +  \sup_{0 \le t \le T} \|  B(t)  \|_\infty   \right)   \sup_{0 \le t \le T}  \left( \|  \langle v \rangle^{5+\delta}  \alpha \p_{x_3} f (t) \|_\infty   \right).
\end{split}
\Ee
And
\Be \label{EestSpos4}
\begin{split}
& | \int_{B(x;t) \cap \{y_3 = 0\}}  \int_{\R^3}   \mathcal S^E_i (v ,\o ) \cdot (  E + E_{\text{ext}} +   \hat v  \times ( B + B_{\text{ext}}) - g \mathbf e_3)  f (t-|y-x|,y_\parallel, 0 ,v) \dd v \frac{ \dd y_\parallel }{|y-x|} | 
\\ \lesssim &   \left(    \sup_{0 \le t \le T} \| E(t)  \|_\infty +  \sup_{0 \le t \le T} \|  B(t)  \|_\infty + |B_e| + E_e +g   \right)  \sup_{0 \le t \le T} \| \langle v \rangle^{4+\delta} f(t) \|_\infty 
\\ & \quad \quad \times  \int_{ \{  (| y_\parallel - x_\parallel |^2 + |x_3|^2 ) < t^2  \}  }  \frac{1}{ ( |y_\parallel - x_\parallel  |^2 + x_3^2 )^{1/2}  } dy_\parallel
\\ \lesssim &   \left(    \sup_{0 \le t \le T} \| E(t)  \|_\infty +  \sup_{0 \le t \le T} \|  B(t)  \|_\infty + |B_e| +E_e + g   \right)  \sup_{0 \le t \le T} \| ( 1 + |v|^{4 + \delta }  ) f(t) \|_\infty. 
\end{split}
\Ee
Thus combining \eqref{EestSpos1},  \eqref{EestSpospara2}, \eqref{EestSpos2}, \eqref{EestSpospara3}, \eqref{EestSpos3}, \eqref{EestSpos4}, we get
\[
\begin{split}
| \nabla_{x_\parallel} \eqref{EestSpos} |  \lesssim &  t^{2}  \left(    \sup_{0 \le t \le T} \|  \nabla_{x_\parallel} E(t)  \|_\infty    +  \sup_{0 \le t \le T} \|  \nabla_{x_\parallel} B(t)  \|_\infty   \right)  +  \| \langle v \rangle^{4+\delta}   \nabla_{x_\parallel} f (t) \|_\infty  
\\ | \frac{ \p }{\p_{x_3 } } \eqref{EestSpos} |  \lesssim &  t^{2}  \left(    \sup_{0 \le t \le T} \|  \p_{x_3} E(t)  \|_\infty +  \sup_{0 \le t \le T} \|  \p_{x_3} B(t)  \|_\infty   \right) 
\\ & +   \sup_{0 \le t \le T}  \left( \|  \langle v \rangle^{5+\delta}  \alpha \p_{x_3} f (t) \|_\infty   \right) +   \sup_{0 \le t \le T} \|  (1 + |v|^{4 + \delta }) f(t) \|_\infty.
\end{split}
\]
By the same argument we get the same estimate for $ | \frac{ \p}{\p_{x_k } }  \eqref{EestSneg}  |  $. Therefore 
\Be \label{EestSfinal}
\begin{split}
| \nabla_{x_\parallel} \eqref{EestSpos} | + | \nabla_{x_\parallel} \eqref{EestSneg} |  \lesssim &  t^{2}  \left(    \sup_{0 \le t \le T} \|  \nabla_{x_\parallel} E(t)  \|_\infty    +  \sup_{0 \le t \le T} \|  \nabla_{x_\parallel} B(t)  \|_\infty   \right)  +  \| \langle v \rangle^{4+\delta}   \nabla_{x_\parallel} f (t) \|_\infty  
\\ | \frac{ \p }{\p_{x_3 } } \eqref{EestSpos} | + | \frac{ \p }{\p_{x_3 } } \eqref{EestSneg} |  \lesssim & t^{2}  \left(    \sup_{0 \le t \le T} \|  \p_{x_3} E(t)  \|_\infty +  \sup_{0 \le t \le T} \|  \p_{x_3} B(t)  \|_\infty   \right) 
\\ & +   \sup_{0 \le t \le T}  \left( \|  \langle v \rangle^{5+\delta}  \alpha \p_{x_3} f (t) \|_\infty   \right) +   \sup_{0 \le t \le T} \|  (1 + |v|^{4 + \delta }) f(t) \|_\infty.
\end{split}
\Ee

Next, using the change of variables $z_\parallel = y_\parallel - x_\parallel $ we have
\Be \label{Eestbdrypos1}
\begin{split}
& \frac{ \p }{\p_{x_k } } \eqref{Eestbdrypos}
\\ = &  \frac{ \p }{\p_{x_k } } \left(   \int_{ \sqrt{t^2 - |z_\parallel |^2 } > x_3}  
\int_{\R^3} \left( \delta_{i3 } -  \frac{(\o_i + \hat{v}_i)\hat{v}_3}{1+ \hat v \cdot \o } \right) f (t- ( |z_\parallel |^2 + x_3^2 )^{1/2} ,z_\parallel + x_\parallel, 0 ,v) \dd v\frac{ \dd z_\parallel}{(|z_\parallel |^2 + x_3^2)^{1/2}} \right)
\\ = &  ( 1 - \delta_{k3} )   \int_{ \sqrt{t^2 - |z_\parallel |^2 } > x_3}  \int_{\R^3}   \left( \delta_{i3 } -\frac{(\o_i + \hat{v}_i)\hat{v}_3}{1+ \hat v \cdot \o } \right)  \p_{x_k } f (t- ( |z_\parallel |^2 + x_3^2 )^{1/2} ,z_\parallel + x_\parallel, 0 ,v) \dd v\frac{ \dd z_\parallel}{(|z_\parallel |^2 + x_3^2)^{1/2}}
\\ & + \delta_{k3 }  \int_{ \sqrt{t^2 - |z_\parallel |^2 } > x_3}  \int_{\R^3}  \p_{x_3} \left( \frac{ 1}{(|z_\parallel |^2 + x_3^2)^{1/2}} \left( \delta_{i3 } -\frac{(\o_i + \hat{v}_i)\hat{v}_3}{1+ \hat v \cdot \o } \right)  \right)   f (t- ( |z_\parallel |^2 + x_3^2 )^{1/2} ,z_\parallel + x_\parallel, 0 ,v) \dd v \   \dd z_\parallel
\\ & + \delta_{k3}   \int_{ \sqrt{t^2 - |z_\parallel |^2 } > x_3}  \int_{\R^3} \left( \delta_{i3} - \frac{(\o_i + \hat{v}_i)\hat{v}_3}{1+ \hat v \cdot \o } \right)  \p_{t } f (t- ( |z_\parallel |^2 + x_3^2 )^{1/2} ,z_\parallel + x_\parallel, 0 ,v) \dd v\frac{ - x_3 }{(|z_\parallel |^2 + x_3^2)} \dd z_\parallel
\\ & -   \delta_{k3}   \int_{ \sqrt{t^2 - |z_\parallel |^2 } = x_3}   \int_{\R^3} \left( \delta_{i3} \frac{(\o_i + \hat{v}_i)\hat{v}_3}{1+ \hat v \cdot \o }  \right) f (0 ,z_\parallel + x_\parallel, 0 ,v) \dd v  \frac{x_3 }{\sqrt{t^2 - x_3^2} } \left( \frac{ z_\parallel }{\sqrt{t^2-x^2}} \cdot \frac{z_\parallel}{ | z_\parallel | } \right)  \frac{ \dd S_{ z_\parallel } }{t}.
\end{split} 
\Ee 
The first term is only contribute as the tangential derivative, from \eqref{vdecaybasic2},
\Be \label{Eestbdrypos2}
\begin{split}
&  |   ( 1 - \delta_{k3} )  \int_{ \sqrt{t^2 - |z_\parallel |^2 } > x_3}  \int_{\R^3} \left( \delta_{i3 } -\frac{(\o_i + \hat{v}_i)\hat{v}_3}{1+ \hat v \cdot \o } \right) \p_{x_k } f (t- ( |z_\parallel | + x_3^2 )^{1/2} ,z_\parallel + x_\parallel, 0 ,v) \dd v\frac{ \dd z_\parallel}{(|z_\parallel |^2 + x_3^2)^{1/2}} | 
 \\ \lesssim &   \sup_{0 \le t \le T} \| ( 1 + |v|^{4 + \delta} )  \nabla_{x_\parallel } f(t) \|_\infty  \int_{  \sqrt{t^2 - |z_\parallel |^2 } > x_3}  \int_{\mathbb R^3} \frac{ ( 1 + |v|^{4 + \delta} )^{-1}}{(|z_\parallel |^2 + x_3^2)^{1/2} }  dv d z_\parallel 
 \\ \lesssim &   \sup_{0 \le t \le T} \| ( 1 + |v|^{4 + \delta} )  \nabla_{x_\parallel } f(t) \|_\infty  \int_{r^2 + x_3^2 < t^2 }  \frac{ r }{ (r^2 + x_3^2 )^{1/2} } dr  \lesssim   \sup_{0 \le t \le T} \| ( 1 + |v|^{4 + \delta} ) \nabla_{x_\parallel } f(t) \|_\infty.
 \end{split}
\Ee
For the second term, recall $\eqref{Eestbulkpos1}_2$ using the identity (\cite{CKKRM})
\[
\frac{ (|\hat{v}|^2-1  )(\hat{v}_i +  \o_i ) }{ ( |y_\parallel - x_\parallel  |^2 + x_3^2 ) (1+  \hat v \cdot \o )^2} = \left. \sum_{j=1}^3 \frac{\p}{\p y_j}\left[\frac{1}{|y-x|}\left( \delta_{ij} - \frac{(\o_i + \hat{v}_i)\hat{v}_j}{1+\hat{v} \cdot \o} \right)\right] \right|_{y_3 = 0 },
\]
we have
\Be \label{Eestbdrypos2.5}
\begin{split}
& | \eqref{Eestbulkpos1}_2 +   \int_{ \sqrt{t^2 - |z_\parallel |^2 } > x_3}  \int_{\R^3}   \p_{x_3} \left( \frac{ 1}{(|z_\parallel |^2 + x_3^2)^{1/2}}  \left( \delta_{i3 } -\frac{(\o_i + \hat{v}_i)\hat{v}_3}{1+ \hat v \cdot \o } \right)  \right) f (t- ( |z_\parallel |^2 + x_3^2 )^{1/2} ,z_\parallel + x_\parallel, 0 ,v) \dd v \   \dd z_\parallel | 
\\ = & |   \delta_{k3}  \int_{ \{  (| y_\parallel - x_\parallel |^2 + |x_3|^2 )^{1/2} < t  \}  }   \int_{\R^3} \left. \sum_{j=1}^3 \frac{\p}{\p y_j}\left[\frac{1}{|y-x|}\left( \delta_{ij} - \frac{(\o_i + \hat{v}_i)\hat{v}_j}{1+\hat{v} \cdot \o} \right)\right]  f (t-|y-x|,y ,v) \right|_{y_3 = 0 } \dd v \dd y_\parallel  
\\ & + \delta_{k3}    \int_{  \{ | y_\parallel - x_\parallel |^2 + |x_3|^2 )^{1/2} < t \} }  \int_{\R^3}   \p_{x_3} \left( \frac{ \left( \delta_{i3 } -\frac{(\o_i + \hat{v}_i)\hat{v}_3}{1+ \hat v \cdot \o } \right)}{(|y_\parallel - x_\parallel  |^2 + x_3^2)^{1/2}}   \right) f (t- ( |y_\parallel - x_\parallel |^2 + x_3^2 )^{1/2} ,y_\parallel, 0 ,v) \dd v \   \dd y_\parallel  |
\\ = &   |   \delta_{k3}  \int_{ \{  (| y_\parallel - x_\parallel |^2 + |x_3|^2 )^{1/2} < t  \}  }   \int_{\R^3} \left. \sum_{j=1}^2 \frac{\p}{\p y_j}\left[\frac{1}{|y-x|}\left( \delta_{ij} - \frac{(\o_i + \hat{v}_i)\hat{v}_j}{1+\hat{v} \cdot \o} \right)\right]  f (t-|y-x|,y ,v) \right|_{y_3 = 0 } \dd v \dd y_\parallel  
\\ & +   \delta_{k3}  \int_{ \{  (| y_\parallel - x_\parallel |^2 + |x_3|^2 )^{1/2} < t  \}  }   \int_{\R^3} \left.  \left(  \frac{\p}{\p y_3}\left[\frac{1}{|y-x|}\left( \delta_{i3} - \frac{(\o_i + \hat{v}_i)\hat{v}_j}{1+\hat{v} \cdot \o} \right)\right]  f (t-|y-x|,y ,v) \right) \right|_{y_3 = 0 } \dd v \dd y_\parallel  
\\ & +  \delta_{k3}    \int_{  \{ | y_\parallel - x_\parallel |^2 + |x_3|^2 )^{1/2} < t \} }  \int_{\R^3}   \p_{x_3} \left( \frac{ 1}{(|y_\parallel - x_\parallel  |^2 + x_3^2)^{1/2}}  \left( \delta_{i3 } -\frac{(\o_i + \hat{v}_i)\hat{v}_3}{1+ \hat v \cdot \o } \right)  \right)
\\ & \quad \quad \times f (t- ( |y_\parallel - x_\parallel |^2 + x_3^2 )^{1/2} ,y_\parallel, 0 ,v) \dd v \   \dd y_\parallel  |
\\ =  &  |   \delta_{k3}  \int_{ \{  (| y_\parallel - x_\parallel |^2 + |x_3|^2 )^{1/2} < t  \}  }   \int_{\R^3} \sum_{j=1}^2 \frac{\p}{\p y_j}\left[\frac{1}{(|y_\parallel-x_\parallel |^2 + x_3^2 )^{1/2}}\left( \delta_{ij} - \frac{(\o_i + \hat{v}_i)\hat{v}_j}{1+\hat{v} \cdot \o} \right)\right]  
\\ & \quad \quad \times f (t-(|y_\parallel-x_\parallel |^2 + x_3^2 )^{1/2},y_\parallel, 0  ,v) \dd v \dd y_\parallel   |,
\end{split}
\Ee
where we've used the cancellation $  \left.  \frac{\p}{\p y_3}\left[\frac{1}{|y-x|}\left( \delta_{i3} - \frac{(\o_i + \hat{v}_i)\hat{v}_j}{1+\hat{v} \cdot \o} \right]   \right) \right|_{y_3 = 0 }  =  -  \p_{x_3} \left( \frac{ 1}{(|y_\parallel - x_\parallel  |^2 + x_3^2)^{1/2}}  \left( \delta_{i3 } -\frac{(\o_i + \hat{v}_i)\hat{v}_3}{1+ \hat v \cdot \o } \right)  \right) $. Thus from integration by parts and \eqref{vdecaybasic}, 
\Be \label{Eestbdrypos2.6}
\begin{split}
 &  |   \delta_{k3}  \int_{ \{  (| y_\parallel - x_\parallel |^2 + |x_3|^2 )^{1/2} < t  \}  }   \int_{\R^3}  \sum_{j=1}^2 \frac{\p}{\p y_j}\left[\frac{1}{|y-x|}\left( \delta_{ij} - \frac{(\o_i + \hat{v}_i)\hat{v}_j}{1+\hat{v} \cdot \o} \right)\right]  f (t-|y-x|,y ,v) \dd v \dd y_\parallel   |
 \\ \lesssim &  |   - \int_{ \{  (| y_\parallel - x_\parallel |^2 + |x_3|^2 )^{1/2} < t  \}  }  \left[\frac{| |y_\parallel -x_\parallel | }{(|y_\parallel -x_\parallel |^2 + x_3^2 ) }\left( \delta_{ij} - \frac{(\o_i + \hat{v}_i)\hat{v}_j}{1+\hat{v} \cdot \o} \right)\right] \p_t f (t-|y-x|,y ,v) \dd v \dd y_\parallel   
 \\ & -   \int_{ \{  (| y_\parallel - x_\parallel |^2 + |x_3|^2 )^{1/2} < t  \}  }  \left[\frac{1 }{(|y_\parallel -x_\parallel |^2 + x_3^2 )^{1/2}  }\left( \delta_{ij} - \frac{(\o_i + \hat{v}_i)\hat{v}_j}{1+\hat{v} \cdot \o} \right)\right] \nabla_{x_\parallel}  f (t-|y-x|,y ,v) \dd v \dd y_\parallel 
 \\ & +  \int_{ \{  (| y_\parallel - x_\parallel |^2 + |x_3|^2 )^{1/2} = t  \}  } \sum_{j=1}^2 \frac{\o_j }{t } \left( \delta_{ij} - \frac{(\o_i + \hat{v}_i)\hat{v}_j}{1+\hat{v} \cdot \o} \right)   f (0,y ,v) \dd v \dd S_{y_\parallel } | 
 \\ \lesssim  & |   \int_{ \{  (| y_\parallel - x_\parallel |^2 + |x_3|^2 )^{1/2} < t  \}  }  \left[\frac{| |y_\parallel -x_\parallel | }{(|y_\parallel -x_\parallel |^2 + x_3^2 ) }\left( \delta_{ij} - \frac{(\o_i + \hat{v}_i)\hat{v}_j}{1+\hat{v} \cdot \o} \right)\right] \p_t f (t-|y-x|,y ,v) \dd v \dd y_\parallel  | 
 \\ + & \sup_{ 0 \le t \le T }  \| ( 1 + |v|^{4 + \delta} ) \nabla_{x_\parallel} f(t) \|_\infty \int_{ |z_\parallel |^2 + x_3^2 < t^2 }     \frac{1}{ ( |z_\parallel |^2 + x_3^2 )^{1/2} }  \int_{\mathbb R^3 } \frac{1}{ 1 + |v|^{3 + \delta }  } dv d z_\parallel 
 \\ & +   \| ( 1 +|v|^{4 + \delta } ) f_0 \|_\infty   \int_{\mathbb R^3 } \frac{1}{ 1 + |v|^{3 + \delta }  } dv 
 \\ \lesssim & |   \int_{ \{  (| y_\parallel - x_\parallel |^2 + |x_3|^2 )^{1/2} < t  \}  }  \left[\frac{| |y_\parallel -x_\parallel | }{(|y_\parallel -x_\parallel |^2 + x_3^2 ) }\left( \delta_{ij} - \frac{(\o_i + \hat{v}_i)\hat{v}_j}{1+\hat{v} \cdot \o} \right)\right] \p_t f (t-|y-x|,y ,v) \dd v \dd y_\parallel  | 
 \\ & +   \sup_{0 \le t \le T}    \| ( 1 + |v|^{4 + \delta} ) \nabla_{x_\parallel} f(t) \|_\infty   + \| ( 1 +|v|^{4 + \delta } ) f_0 \|_\infty.
\end{split}
\Ee

Now from \eqref{VMfrakF1}, we write $\p_t f = -  \hat v \cdot \nabla_x f -(E + E_{\text{ext}}  + \hat v \times ( B + B_{\text{ext}}) - g \mathbf{e}_3 ) \cdot \nabla_v f $. Then using $\alpha  = \hat v_3$ on $\p \O$, integration by parts in $v$, and that $ \nabla_v \cdot ( \hat v \times B) =0$, we get
\Be \label{Eestbdrypos3}
\begin{split}
& |   \int_{ \sqrt{t^2 - |z_\parallel |^2 } > x_3}  \int_{\R^3} \frac{(\o_i + \hat{v}_i)\hat{v}_3}{1+ \hat v \cdot \o }  \p_{t } f (t- ( |z_\parallel | + x_3^2 )^{1/2} ,z_\parallel + x_\parallel, 0 ,v) \dd v \frac{ x_3 }{(|z_\parallel |^2 + x_3^2)} \dd z_\parallel | 
\\ \le  &  |   \int_{ \sqrt{t^2 - |z_\parallel |^2 } > x_3}  \int_{\R^3} \frac{(\o_i + \hat{v}_i)\hat{v}_3}{1+ \hat v \cdot \o } \left(  \hat v \cdot \nabla_x f (t- ( |z_\parallel | + x_3^2 )^{1/2} ,z_\parallel + x_\parallel, 0 ,v) \right) \dd v \frac{ x_3 }{(|z_\parallel |^2 + x_3^2)} \dd z_\parallel |
\\ & +   |   \int_{ \sqrt{t^2 - |z_\parallel |^2 } > x_3}  \int_{\R^3} \frac{(\o_i + \hat{v}_i)\hat{v}_3}{1+ \hat v \cdot \o } 
\\ & \quad \quad  \times \left( (E + E_{\text{ext}} + \hat v \times ( B + B_{\text{ext}}) - g \mathbf{e}_3 ) \cdot \nabla_v f (t- ( |z_\parallel | + x_3^2 )^{1/2} ,z_\parallel + x_\parallel, 0 ,v) \right) \dd v  \frac{ x_3 }{(|z_\parallel |^2 + x_3^2)} \dd z_\parallel |
\\ \lesssim &   \sup_{0 \le t \le T} \left(  \| ( 1 + |v|^{5 + \delta} ) \alpha \p_{x_3 } f(t) \|_\infty +  \| ( 1 + |v|^{4 + \delta} ) \nabla_{x_\parallel} f(t) \|_\infty \right) \int_{ \sqrt{t^2 - |z_\parallel |^2 } > x_3} \int_{\mathbb R^3} \langle v \rangle^{-4-\delta}   \frac{ x_3 }{(|z_\parallel |^2 + x_3^2)}  dv \dd z_\parallel
\\ & +   |   \int_{ \sqrt{t^2 - |z_\parallel |^2 } > x_3}  \int_{\R^3}  \nabla_v   \left( \frac{(\o_i + \hat{v}_i)\hat{v}_3}{1+ \hat v \cdot \o } \right) (E + E_{\text{ext}} + \hat v \times B - g \mathbf{e}_3 )   f (t- ( |z_\parallel | + x_3^2 )^{1/2} ,z_\parallel + x_\parallel, 0 ,v)  \dd v
\\ & \quad  \times  \frac{ x_3 }{(|z_\parallel |^2 + x_3^2)} \dd z_\parallel |
\\ \lesssim &   \sup_{0 \le t \le T} \| ( 1 + |v|^{5 + \delta} ) \left(  \| ( 1 + |v|^{5 + \delta} ) \alpha \p_{x_3 } f(t) \|_\infty +  \| ( 1 + |v|^{4 + \delta} ) \nabla_{x_\parallel} f(t) \|_\infty \right)  
\\ & + \left(  \sup_{0 \le t \le T} \| ( 1 + |v|^{4 + \delta} )f(t) \|_\infty  \right) \int_{ \sqrt{t^2 - |z_\parallel |^2 } > x_3} \int_{\mathbb R^3} (1 + |v|^{3 + \delta})^{-1}   \frac{ x_3 }{(|z_\parallel |^2 + x_3^2)}  dv \dd z_\parallel
\\ \lesssim &   \sup_{0 \le t \le T}   \left(  \| ( 1 + |v|^{5 + \delta} ) \alpha \p_{x_3 } f(t) \|_\infty +  \| ( 1 + |v|^{4 + \delta} ) \nabla_{x_\parallel} f(t) \|_\infty \right)  +  \sup_{0 \le t \le T} \| ( 1 + |v|^{4 + \delta} )f(t) \|_\infty , 
\end{split}
\Ee
where we've used from \eqref{calSi}, \eqref{calSiest}, and \eqref{vdecaybasic2} that
\[
\left|  \nabla_v   \left( \frac{(\o_i + \hat{v}_i)\hat{v}_3}{1+ \hat v \cdot  \o } \right) \right|  = \left|  \mathcal S^E_i( \o, \hat v ) \hat v_3 +  \frac{(\o_i + \hat{v}_i)}{1+ \hat v \cdot \o }  \nabla_v \hat v_3 \right| \le 14 \sqrt{1 + |v|^2}.
\]
By the same argument we have
\Be \label{Eestbdrypos3.5}
\begin{split}
& |   \int_{ \{  (| y_\parallel - x_\parallel |^2 + |x_3|^2 )^{1/2} < t  \}  }  \left[\frac{| |y_\parallel -x_\parallel | }{(|y_\parallel -x_\parallel |^2 + x_3^2 ) }\left( \delta_{ij} - \frac{(\o_i + \hat{v}_i)\hat{v}_j}{1+\hat{v} \cdot \o} \right)\right] \p_t f (t-|y-x|,y ,v) \dd v \dd y_\parallel  |
\\ \lesssim &    \sup_{0 \le t \le T}   \left(  \| ( 1 + |v|^{5 + \delta} ) \alpha \p_{x_3 } f(t) \|_\infty +  \| ( 1 + |v|^{4 + \delta} ) \nabla_{x_\parallel} f(t) \|_\infty \right)  +  \sup_{0 \le t \le T} \| ( 1 + |v|^{4 + \delta} )f(t) \|_\infty.
\end{split}
\Ee

We also have
\Be \label{Eestbdrypos4}
\begin{split}
& |   \int_{ \sqrt{t^2 - |z_\parallel |^2 } = x_3}   \int_{\R^3} \frac{(\o_i + \hat{v}_i)\hat{v}_3}{1+ \hat v \cdot \o } f (0 ,z_\parallel + x_\parallel, 0 ,v) \dd v  \frac{x_3 }{\sqrt{t^2 - x_3^2} } \left( \frac{ z_\parallel }{\sqrt{t^2 - x_3^2}} \cdot \frac{z_\parallel}{ | z_\parallel | } \right)  \frac{ \dd S_{ z_\parallel } }{t} | 
\\  \lesssim &   \|  \langle v \rangle^{4+\delta} f_0 \|_\infty   \int_{ |z_\parallel | = \sqrt{t^2 -x_3^2 } }  \int_{\mathbb R^3} (1 + |v|^{3 + \delta})^{-1} \frac{x_3}{ \sqrt{ t^2 - x_3^2 } } dv  \frac{ \dd S_{ z_\parallel } }{t} 
\\ \lesssim &   \| \langle v \rangle^{4+\delta} f_0 \|_\infty \left(  \frac{x_3}{t }  \right) \lesssim   \sup_{0 \le t \le T} \| \langle v \rangle^{4+\delta} f(t) \|_\infty.
\end{split}
\Ee
Thus from \eqref{Eestbdrypos1}, \eqref{Eestbdrypos2},  \eqref{Eestbdrypos2.5}, \eqref{Eestbdrypos2.6}, \eqref{Eestbdrypos3}, \eqref{Eestbdrypos3.5}, \eqref{Eestbdrypos4},  and together with \eqref{Eestbulkpos1}--\eqref{Eestbulkpos2}, we have
\[
\begin{split}
| \nabla_{x_\parallel}  \eqref{Eestbulkpos}  +  \nabla_{x_\parallel}  \eqref{Eestbdrypos} | \lesssim  &  \sup_{0 \le t \le T}  \left\{  \| \langle v \rangle^{4+\delta}  \nabla_{x_\parallel}  f(t) \|_\infty \right\},
\\ | \frac{ \p}{\p_{x_3 } }  \eqref{Eestbulkpos} +  \frac{ \p}{\p_{x_3 } }  \eqref{Eestbdrypos} | \lesssim  &  \sup_{0 \le t \le T}  \left\{ \|\langle v \rangle^{4+\delta}  f(t) \|_\infty   +  \|  \langle v \rangle^{5+\delta} \alpha \p_{x_3 } f(t) \|_\infty + \| \langle v \rangle^{4+\delta}  \nabla_{x_\parallel}  f(t) \|_\infty \right\}.
\end{split}
\]
By the same argument we get the same estimate for $  \frac{ \p}{\p_{x_k} }  \eqref{Eestbulkneg} + \frac{\p}{\p_{x_k } } \eqref{Eestbdryneg} $. Thus
\Be \label{Eestbdryfinal}
\begin{split}
 & |  \nabla_{x_\parallel}  \eqref{Eestbulkpos} + \nabla_{x_\parallel}  \eqref{Eestbdrypos} |+ |  \nabla_{x_\parallel}  \eqref{Eestbulkneg} + \nabla_{x_\parallel}  \eqref{Eestbdryneg} | \lesssim   \sup_{0 \le t \le T}  \left\{  \| \langle v \rangle^{4+\delta}  \nabla_{x_\parallel}  f(t) \|_\infty \right\},
\\ &  | \frac{ \p}{\p_{x_3 } }  \eqref{Eestbulkpos} + \frac{ \p}{\p_{x_3 } }  \eqref{Eestbdrypos} | +  |  \frac{ \p}{\p_{x_3 } }  \eqref{Eestbulkneg} + \frac{ \p}{\p_{x_3 } }  \eqref{Eestbdryneg} | 
\\ \lesssim  &  \sup_{0 \le t \le T}  \left\{ \|\langle v \rangle^{4+\delta}  f(t) \|_\infty   +  \|  \langle v \rangle^{5+\delta} \alpha \p_{x_3 } f(t) \|_\infty + \| \langle v \rangle^{4+\delta}  \nabla_{x_\parallel}  f(t) \|_\infty \right\}.
\end{split}
\Ee
Next, by using the change of variables $z = y - x$ and spherical coordinate for $z$, we have
\Be \label{Eestinitalpos0}
\begin{split}
 \eqref{Eestinitialpos} =  &   \int_{|z| = t , \ z_3>-x_3}  \sum_j \o_j  \left(\delta_{ij} - \frac{(\o_i + \hat{v}_i)\hat{v}_j}{1+ \hat v \cdot \o }\right) f(0, z+x ,v) \dd v \frac{\dd S_z}{t}
 \\ = &   \int_{  \ t \cos \phi >-x_3}  \int_0^{2\pi}  \sum_j \o_j  \left(\delta_{ij} - \frac{(\o_i + \hat{v}_i)\hat{v}_j}{1+ \hat v \cdot \o }\right) f(0, z+x ,v) \dd v \frac{ t^2 \sin \phi \, \dd \theta \dd \phi }{t}.
\end{split}
\Ee
Thus
\Be \label{pEestinitialpos}
\begin{split}
& \frac{ \p}{\p_{x_k } }   \eqref{Eestinitialpos}  
\\ = &-   \int_{  \ t \cos \phi >-x_3}  \int_0^{2\pi}  \sum_j \o_j  \left(\delta_{ij} - \frac{(\o_i + \hat{v}_i)\hat{v}_j}{1+ \hat v \cdot \o }\right) \p_{x_k}  f(0, z+x ,v)   ( t  \sin \phi )  \,  \dd v \dd \theta \dd \phi 
\\ &  - \delta_{k3}   \int_0^{2\pi}  \sum_j \o_j  \left(\delta_{ij} - \frac{(\o_i + \hat{v}_i)\hat{v}_j}{1+ \hat v \cdot \o }\right)  f(0, z_\parallel +x_\parallel, 0 ,v) \left( \frac{d}{d x_3} \cos^{-1} \left( \frac{-x_3}{t} \right) \right) \left( t \sin \left( \cos^{-1} \left( \frac{-x_3}{t} \right) \right) \right)  \,  \dd v \dd \theta
\\  =  & -  \int_{  \ t \cos \phi >-x_3}  \int_0^{2\pi}  \sum_j \o_j  \left(\delta_{ij} - \frac{(\o_i + \hat{v}_i)\hat{v}_j}{1+ \hat v \cdot \o }\right) \p_{x_k}  f(0, z+x ,v)   ( t  \sin \phi )  \,  \dd v \dd \theta \dd \phi 
\\ &  - \delta_{k3}   \int_0^{2\pi}  \sum_j \o_j  \left(\delta_{ij} - \frac{(\o_i + \hat{v}_i)\hat{v}_j}{1+ \hat v \cdot \o }\right)   f(0, z_\parallel +x_\parallel, 0 ,v) \frac{-1}{\sqrt{ 1- \left( \frac{x_3}{t} \right)^2  } } \frac{-1}{t}   \left( t \sqrt{ 1- \left( \frac{x_3}{t} \right)^2  }  \right)  \,  \dd v \dd \theta.
\end{split}
\Ee
So from \eqref{vdecaybasic2}, for $k=1,2$,
\Be
\begin{split}
& |   \int_{  \ t \cos \phi >-x_3}  \int_0^{2\pi}  \sum_j \o_j  \left(\delta_{ij} - \frac{(\o_i + \hat{v}_i)\hat{v}_j}{1+ \hat v \cdot \o }\right) \p_{x_k}  f(0, z+x ,v)   ( t  \sin \phi )  \,  \dd v \dd \theta \dd \phi | 
\\ \lesssim &   \| \langle v \rangle^{4+\delta} \nabla_{x_\parallel} f_0 \|_\infty   \int_{  \ t \cos \phi >-x_3}  \int_0^{2\pi}  \int_{\mathbb R^3}  \langle v \rangle^{-4-\delta} ( t  \sin \phi )  \,  \dd v \dd \theta \dd \phi 
\\ \lesssim &   \| \langle v \rangle^{4+\delta} \nabla_{x_\parallel} f_0 \|_\infty.
\end{split}
\Ee
And for $k=3$,
\Be
\begin{split}
& |   \int_{  \ t \cos \phi >-x_3}  \int_0^{2\pi}  \sum_j \o_j  \left(\delta_{ij} - \frac{(\o_i + \hat{v}_i)\hat{v}_j}{1+ \hat v \cdot \o }\right) \p_{x_3}  f(0, z+x ,v)   ( t  \sin \phi )  \,  \dd v \dd \theta \dd \phi | 
\\ \lesssim &  \| \langle v \rangle^{5+\delta} \alpha \p_{x_3} f_0 \|_\infty    \int_{  \ t \cos \phi >-x_3}  \int_0^{2\pi}  \int_{\mathbb R^3}  \frac{\langle v \rangle^{-4-\delta} }{\alpha(0, z +x ,v )}  ( t  \sin \phi )  \,  \dd v \dd \theta \dd \phi 
\\ \lesssim &  \| \langle v \rangle^{5+\delta} \alpha \p_{x_3} f_0 \|_\infty    \int_{  \ t \cos \phi >-x_3}  t \sin \phi |  \ln ( t \cos \phi + x_3 )  | d\phi 
\\ \lesssim &  \| \langle v \rangle^{5+\delta} \alpha \p_{x_3} f_0 \|_\infty  \int_0^{t + x_3} | \ln (s) | ds  \lesssim    \| \langle v \rangle^{5+\delta} \alpha f_0 \|_\infty,
\end{split}
\Ee
and
\Be
\begin{split}
 |  \int_0^{2\pi}  \sum_j \o_j &  \left(\delta_{ij} - \frac{(\o_i + \hat{v}_i)\hat{v}_j}{1+ \hat v \cdot \o }\right)  f(0, z_\parallel +x_\parallel, 0 ,v) \frac{-1}{\sqrt{ 1- \left( \frac{x_3}{t} \right)^2  } } \frac{-1}{t}   \left( t \sqrt{ 1- \left( \frac{x_3}{t} \right)^2  }  \right)  \,  \dd v \dd \theta |
\\ =  &  |  \int_0^{2\pi}  \sum_j \o_j  \left(\delta_{ij} - \frac{(\o_i + \hat{v}_i)\hat{v}_j}{1+ \hat v \cdot \o }\right)   f(0, z_\parallel +x_\parallel, 0 ,v) \,  \dd v \dd \theta |  \lesssim \|  \langle v \rangle^{4+\delta} f(0) \|_\infty.
\end{split}
\Ee
Therefore, we have
\[
\begin{split}
| \nabla_{x_\parallel}  \eqref{Eestinitialpos} | \lesssim &   \| \langle v \rangle^{4+\delta} \nabla_{x_\parallel} f_0 \|_\infty,
\\ | \frac{ \p}{\p_{x_3 } }   \eqref{Eestinitialpos} | \lesssim &  \| \langle v \rangle^{5+\delta} \alpha \p_{x_3} f_0 \|_\infty  + \| \langle v \rangle^{4+\delta} f(0) \|_\infty.
\end{split}
\]
And by the same argument we have the same estimate for $\frac{ \p}{\p_{x_k } }   \eqref{Eestinitialneg}$. Thus
\Be \label{Eestinitialfinal}
\begin{split}
| \nabla_{x_\parallel}  \eqref{Eestinitialpos} | + | \nabla_{x_\parallel}  \eqref{Eestinitialneg} | \lesssim &   \| \langle v \rangle^{4+\delta} \nabla_{x_\parallel} f_0 \|_\infty,
\\ | \frac{ \p}{\p_{x_3 } }   \eqref{Eestinitialpos} | +  | \frac{ \p}{\p_{x_3 } }   \eqref{Eestinitialneg} | \lesssim &  \| \langle v \rangle^{5+\delta} \alpha \p_{x_3} f_0 \|_\infty  + \| \langle v \rangle^{4+\delta} f(0) \|_\infty.
\end{split}
\Ee
Finally, we estimate $\frac{\p}{\p x_k } \eqref{Eest3bdrycontri} $. We have
\[
\begin{split}
& | \frac{\p}{\p x_k }  \eqref{Eest3bdrycontri} |  
\\ & \le 2  |  \frac{\p}{\p x_k }   \int_{ \sqrt{ |z_\parallel |^2 + x_3^2 } < t }  \int_{\mathbb R^3 }  \frac{  f (t -  \sqrt{ |z_\parallel |^2 + x_3^2 }, x_\parallel + z_\parallel, 0, v )}{ \sqrt{ |z_\parallel |^2 + x_3^2 } }  dv d z_\parallel |
\\ & \lesssim  ( 1 - \delta_{k3} )   \int_{ \sqrt{t^2 - |z_\parallel |^2 } > x_3}  \int_{\R^3}    |  \p_{x_k } f (t- ( |z_\parallel |^2 + x_3^2 )^{1/2} ,z_\parallel + x_\parallel, 0 ,v) |  \dd v\frac{ \dd z_\parallel}{(|z_\parallel |^2 + x_3^2)^{1/2}}
\\ & + \delta_{k3 }  \int_{ \sqrt{t^2 - |z_\parallel |^2 } > x_3}  \int_{\R^3} |   f (t- ( |z_\parallel |^2 + x_3^2 )^{1/2} ,z_\parallel + x_\parallel, 0 ,v) \dd v \  \p_{x_3} \left( \frac{ 1}{(|z_\parallel |^2 + x_3^2)^{1/2}} \right)  |  \dd z_\parallel
\\ & + \delta_{k3}   \int_{ \sqrt{t^2 - |z_\parallel |^2 } > x_3}  \int_{\R^3}  |   \p_{t } f (t- ( |z_\parallel |^2 + x_3^2 )^{1/2} ,z_\parallel + x_\parallel, 0 ,v) \dd v\frac{ - x_3 }{(|z_\parallel |^2 + x_3^2)} | \dd z_\parallel
\\ & -   \delta_{k3}   \int_{ 0}^{2 \pi }    \int_{\R^3} | \frac{ |  f (0 ,z_\parallel + x_\parallel, 0 ,v) | }{t} \dd v  \frac{x_3 }{\sqrt{t^2 - x_3^2} } \sqrt{t^2 - x_3^2}  d\theta.
\end{split}
\]
Similar to the estimate in \eqref{Eestbdrypos1}-\eqref{Eestbdrypos4}, we get
\[
\int_{ \sqrt{t^2 - |z_\parallel |^2 } > x_3}  \int_{\R^3}    |  \p_{x_k } f (t- ( |z_\parallel |^2 + x_3^2 )^{1/2} ,z_\parallel + x_\parallel, 0 ,v) |  \dd v\frac{ \dd z_\parallel}{(|z_\parallel |^2 + x_3^2)^{1/2}} \lesssim    \sup_{0 \le t \le T} \| ( 1 + |v|^{4 + \delta} )  \nabla_{x_\parallel } f(t) \|_\infty,
\]
\[
\begin{split}
 & \int_{ \sqrt{t^2 - |z_\parallel |^2 } > x_3}  \int_{\R^3} |   f (t- ( |z_\parallel |^2 + x_3^2 )^{1/2} ,z_\parallel + x_\parallel, 0 ,v) \dd v \  \p_{x_3} \left( \frac{ 1}{(|z_\parallel |^2 + x_3^2)^{1/2}} \right)  |  \dd z_\parallel
 \\ &  +  \int_{ 0}^{2 \pi }    \int_{\R^3} | \frac{ |  f (0 ,z_\parallel + x_\parallel, 0 ,v) | }{t} \dd v  \frac{x_3 }{\sqrt{t^2 - x_3^2} } \sqrt{t^2 - x_3^2}  d\theta \lesssim \sup_{0 \le t \le T} \| \langle v \rangle^{4+\delta} f(t) \|_\infty.
\end{split}
\]
And
\[
\begin{split}
     \int_{ \sqrt{t^2 - |z_\parallel |^2 } > x_3}  & \int_{\R^3}  |   \p_{t } f (t- ( |z_\parallel |^2 + x_3^2 )^{1/2} ,z_\parallel + x_\parallel, 0 ,v) \dd v\frac{ - x_3 }{(|z_\parallel |^2 + x_3^2)} | \dd z_\parallel 
\\   \lesssim &    \left(  \| ( 1 + |v|^{5 + \delta} ) \alpha \p_{x_3 } f(t) \|_\infty +  \| ( 1 + |v|^{4 + \delta} ) \nabla_{x_\parallel} f(t) \|_\infty \right)  +  \sup_{0 \le t \le T} \| ( 1 + |v|^{4 + \delta} )f(t) \|_\infty.
  \end{split}
\]
Therefore
\Be \label{E3estbdryfinal}
\begin{split}
| \nabla_{x_\parallel}  \eqref{Eest3bdrycontri} |   \lesssim &  \sup_{0 \le t \le T}  \left\{  \| \langle v \rangle^{4+\delta}  \nabla_{x_\parallel}  f(t) \|_\infty \right\},
 \\ | \frac{\p}{\p x_3}  \eqref{Eest3bdrycontri} |   \lesssim &  \sup_{0 \le t \le T}  \left\{ \|\langle v \rangle^{4+\delta}  f(t) \|_\infty   +  \|  \langle v \rangle^{5+\delta} \alpha \nabla_{x } f(t) \|_\infty  + \| \langle v \rangle^{4+\delta}  \nabla_{x_\parallel}  f(t) \|_\infty \right\}.
 \end{split}
\Ee

Next, we estimate the $\p_{x_k}$ derivatives to $B$. Using the same argument as in \eqref{Eestat00}--\eqref{Eestat01}, we get
\Be \label{Bestat01}
| \frac{ \p}{\p_{x_k } }  \eqref{Besttat0pos} |  + |  \frac{ \p}{\p_{x_k } }  \eqref{Besttat0neg}  | \lesssim  \| E_0 \|_{C^2 }.
\Ee
Next, using the decay of the kernel in \eqref{vdecayBbulk}, and following the same argument as in \eqref{Eestbulkpos1}--\eqref{Eestbulkpos2}, and \eqref{Eestbdrypos1}--\eqref{Eestbdryfinal},
 we obtain
\Be \label{Bestbulkfinal}
\begin{split}
& | \nabla_{x_\parallel}  \eqref{Bestbulkpos}  + \nabla_{x_\parallel}  \eqref{Bestbdrypos} | + | \nabla_{x_\parallel}  \eqref{Bestbulkneg} + \nabla_{x_\parallel}  \eqref{Bestbdryneg}  | \lesssim  \sup_{0 \le t \le T} \| \langle v \rangle^{4+\delta}   \nabla_{x_\parallel} f (t) \|_\infty 
\\ &  | \frac{ \p}{\p_{x_3 } }  \eqref{Bestbulkpos}  +  \frac{ \p}{\p_{x_3 } }  \eqref{Bestbdrypos}   | + | \frac{ \p}{\p_{x_3 } }  \eqref{Bestbulkneg} + \frac{ \p}{\p_{x_3 } }  \eqref{Bestbdryneg} |
\\  \lesssim &  \sup_{0 \le t \le T} \| ( 1 + |v|^{5 + \delta } )  \alpha \p_{x_3} f (t) \|_\infty +   \sup_{0 \le t \le T} \| \langle v \rangle^{4+\delta}   f(t) \|_\infty | .
\end{split}
\Ee 
Next, from \eqref{BestSposibp} we have
\Be \label{BestSposibp}
\begin{split}
 \frac{ \p}{\p x_k } \eqref{BestSpos} =    \frac{ \p}{\p x_k } \left(  \int_{B(x;t) \cap \{y_3 > 0\}}  \int_{\R^3}  \mathcal S^B_i ( v ,\o ) \cdot (E + E_{\text{ext}} + \hat v \times ( B + B_{\text{ext}}) - g\mathbf e_3 )  
 f (t-|y-x|,y ,v) \dd v \frac{ \dd y}{|y-x|} \right),
\end{split}
\Ee
where $\mathcal S^B_i (v, \o )$ as in \eqref{calBSirep}. Then using the bound for $\mathcal S^B_i (v, \o )$ in \eqref{calBSiest} and applying the same argument as in \eqref{EestSpos1}--\eqref{EestSfinal}, we obtain
\Be \label{BestSfinal}
\begin{split}
| \nabla_{x_\parallel} \eqref{BestSpos} | + | \nabla_{x_\parallel} \eqref{BestSneg} |  \lesssim &  t^{2}  \left(    \sup_{0 \le t \le T} \|  \nabla_{x_\parallel} E(t)  \|_\infty    +  \sup_{0 \le t \le T} \|  \nabla_{x_\parallel} B(t)  \|_\infty   \right)  +  \| \langle v \rangle^{4+\delta}   \nabla_{x_\parallel} f (t) \|_\infty  
\\ | \frac{ \p }{\p_{x_3 } } \eqref{BestSpos} | + | \frac{ \p }{\p_{x_3 } } \eqref{BestSneg} |  \lesssim & t^{2}  \left(    \sup_{0 \le t \le T} \|  \p_{x_3} E(t)  \|_\infty +  \sup_{0 \le t \le T} \|  \p_{x_3} B(t)  \|_\infty   \right) 
\\ & +   \sup_{0 \le t \le T}  \left( \|  \langle v \rangle^{5+\delta}  \alpha \p_{x_3} f (t) \|_\infty   \right) +   \sup_{0 \le t \le T} \|  (1 + |v|^{4 + \delta }) f(t) \|_\infty.
\end{split}
\Ee
Next, using the same argument as in \eqref{Eestinitalpos0}--\eqref{Eestinitialfinal} together with \eqref{vdecaybasic3}, we have
\Be \label{Bestinitialfinal}
\begin{split}
| \nabla_{x_\parallel}  \eqref{Bestinitialpos} | + | \nabla_{x_\parallel}  \eqref{Bestinitialneg} | \lesssim &   \| \langle v \rangle^{4+\delta} \nabla_{x_\parallel} f_0 \|_\infty,
\\ | \frac{ \p}{\p_{x_3 } }   \eqref{Bestinitialpos} | +  | \frac{ \p}{\p_{x_3 } }   \eqref{Bestinitialneg} | \lesssim &  \| \langle v \rangle^{5+\delta} \alpha \p_{x_3} f_0 \|_\infty  + \| \langle v \rangle^{4+\delta} f(0) \|_\infty.
\end{split}
\Ee
Finally, similar \eqref{E3estbdryfinal}, we have
\Be \label{Bestbdryfinal}
\begin{split}
| \nabla_{x_\parallel}  \eqref{Bestbdrycontri} |   \lesssim &  \sup_{0 \le t \le T}  \left\{  \| \langle v \rangle^{4+\delta}  \nabla_{x_\parallel}  f(t) \|_\infty \right\},
 \\ | \frac{\p}{\p x_3}  \eqref{Bestbdrycontri} |   \lesssim &  \sup_{0 \le t \le T}  \big\{ \|\langle v \rangle^{4+\delta}  f(t) \|_\infty ( 1+ \sup_{0 \le t \le T} \left( \| E(t) \|_\infty + \| B(t) \|_\infty \right) ) 
 \\ & \quad \quad \quad \quad   +  \|  \langle v \rangle^{5+\delta} \alpha \nabla_{x } f(t) \|_\infty  + \| \langle v \rangle^{4+\delta}  \nabla_{x_\parallel}  f(t) \|_\infty \big\}.
 \end{split}
\Ee

Collecting \eqref{Eestbulkpos1}, \eqref{EestSfinal}, \eqref{Eestbdryfinal}, \eqref{Eestinitialfinal}, and \eqref{E3estbdryfinal}, and \eqref{Eestat01}--\eqref{Bestbdryfinal}, and letting $T \ll 1$, we get
\Be \label{dxEBfinal}
\begin{split}
\| \nabla_{x_\parallel} E \|_\infty + \| \nabla_{x_\parallel} B \|_\infty \lesssim & \|E_0\|_{C^2} + \|B_0\|_{C^2}  +  \sup_{0 \le t \le T}  \langle v \rangle^{4+\delta}  \nabla_{x_\parallel} f (t) \|_\infty .
\\ \|   \p_{x_3} E \|_\infty + \|  \p_{x_3} B \|_\infty \lesssim & \|E_0\|_{C^2} + \|B_0\|_{C^2}   +  \sup_{0 \le t \le T} \| \left(  \langle v \rangle^{5+\delta} \alpha \p_{x_3} f (t) \|_\infty + \langle v \rangle^{4+\delta}  \nabla_{x_\parallel} f (t) \|_\infty \right) 
\\ & +   \sup_{0 \le t \le T}  \| \langle v \rangle^{4+\delta} f(t) \|_\infty  .
\end{split}
\Ee
This concludes \eqref{nablaxparaE} and \eqref{lambdanablaxE}.

For $\p_t E$, and $\p_t B$, from the Maxwell equations \eqref{Maxwell}, we have
\[
\| \p_t E \|_\infty \le \| \nabla_x B \|_\infty +   \| \langle v \rangle^{4+\delta} f \|_\infty, \  \| \p_t B \|_\infty \le \| \nabla_x E \|_\infty.
\]
So from \eqref{nablaxparaE}, \eqref{lambdanablaxE}, we get \eqref{ptEBest}.

\end{proof}

\section{Estimates on trajectories}

We have the following crucial lemma:

\begin{lemma}[Velocity lemma] \label{vlemma}
Let $\alpha$ be defined as in \eqref{alphadef}. Suppose
\Be \label{vlfieldass}
\begin{split}
\sup_{ 0 \le t \le T} & \left(   \| E(t) \|_\infty +  \| B(t) \|_\infty  \right.
\\  &  \left.  + \|  \p_t E_3(t) \|_\infty  +\|  \p_t    (\hat v \times B)_3(t) \|_\infty  + \|   \nabla_x  E_3 (t) \|_\infty + \|   \nabla_x  (\hat v \times B)_3 (t) \|_\infty \right) + g + B_e  < C.
\end{split}
 \Ee
And for all $t, x_\parallel$,
\Be \label{signcondition}
g - E_e  -  E_3(t,x_\parallel, 0 )  -  (\hat v \times B)_3(t,x_\parallel, 0 ) >  c_0  \text{ for some } c_0 > 0.
 \Ee
Then for any $(t,x,v) \in (0, T) \times \O \times \mathbb R^3$, with the trajectory $X(s;t,x,v)$ and $V(s;t,x,v)$ satisfies \eqref{HamiltonODE},
\Be \label{alphaest}
e^{-10\frac{C}{c_0}|t-s| } \alpha(t,x,v) \le \alpha(s, X(s;t,x,v) , V(s;t,x,v) ) \le e^{10\frac{C}{c_0}|t-s| } \alpha(t,x,v).
\Ee
\end{lemma}

\begin{proof}
Note that 
\Bes
\frac{\p \hat{v}_3}{\p v_3 } =  \frac{1}{\langle v \rangle} - \frac{ ( \hat{v}_3)^2 }{\langle v \rangle}, \ \nabla_v  \left( \frac{1}{\langle v \rangle } \right) = - \frac{\hat v }{ \langle v \rangle^2 }.
\Ees
By direct computation,
\begin{align}
[ \p_t & + \hat v \cdot \nabla_x + \mathfrak F \cdot \nabla_v ] (\alpha^2 )  \notag
 \\ \label{diffalpha2main} =  &  - 2 \frac{ \hat v_3}{\langle v \rangle}    \mathfrak F_3(t, x_\parallel, 0,v  )  + 2 \frac{\hat v_3}{\langle v \rangle}  \mathfrak F_3(t,x, v ) 
\\  \notag & +  2 \left( \p_t \mathfrak F_3(t,x_\parallel, 0 ,v) \right) \frac{x_3} {\langle v \rangle }  + 2 \hat v_3 x_3 
 - ( 2 \hat v_\parallel \cdot \nabla_{x_\parallel} \mathfrak F_3(t, x_\parallel, 0,v  )  ) \frac{x_3}{\langle v \rangle}
\\  \notag &   - 2  \frac{ (\hat v_3)^3}{\langle v \rangle } \mathfrak F_3(t,x, v )  -2  (\hat v_3)^2 \frac{ \hat v_\parallel}{\langle v \rangle} \cdot \mathfrak F_\parallel (t,x,v) - 2 \frac{x_3}{\langle v \rangle } \mathfrak F(t,x,v) \cdot \nabla_v \mathfrak F_3(t,x_\parallel, 0,v)  
\\   \label{diffalph2arest} & + 2 x_3 \frac{ \hat v }{\langle v \rangle^2}  \cdot \mathfrak F(t,x,v) \mathfrak F_3(t,x_\parallel, 0, v ).
\end{align}
Using the fundamental theorem of calculus
\Be \notag
\eqref{diffalpha2main} = 2 \frac{ \hat v_3}{\langle v \rangle}  \left(   \int_0^{x_3}   \p_{x_3} \mathfrak F_3 (t,x_\parallel, s, v )   ds \right).
\Ee
Since $\mathfrak F_3 =   E_3 + E_e + \hat v _1 B_2 - \hat v_2 B_1  - g $, and since $\hat v \cdot \mathfrak F = \hat v \cdot ( E  + E_{\text{ext}}+ \hat v \times (B_3 + \Bex ) -  g \mathbf e_3 ) = \hat v \cdot E - ( g - E_e) \hat v_3 $, we have
\Be \label{diffalpha3}
\begin{split}
[ \p_t & + \hat v \cdot \nabla_x + \mathfrak F \cdot \nabla_v ] (\alpha^2 ) 
\\ =  & 2 \frac{ \hat v_3}{\langle v \rangle}  \left(   \int_0^{x_3}   \p_{x_3} E_3 (t,x_\parallel, s )  + \hat v_1 \p_{x_3} B_2( t, x_\parallel, s) - \hat v_2   \p_{x_3} B_1 (t,x_\parallel, s )  ds \right)
\\ & + 2 \frac{x_3}{\langle v \rangle} \left( ( \p_{t} - \hat v_\parallel \cdot \nabla_{x_\parallel }  ) \left( E_3 (t,x_\parallel, 0 )  + \hat v_1 B_2( t, x_\parallel, 0) - \hat v_2  B_1 (t,x_\parallel, 0) \right) \right) + 2 \hat v_3 x_3 
\\ & - 2 \frac{ (\hat v_3 )^2}{\langle v \rangle} \left( \hat v \cdot E(t,x,v) - (g - E_e) \hat v_3  \right)
\\ & - 2\frac{x_3}{\langle v \rangle} \left( E + E_{\text{ext}}+ (\hat v \times  (B + B_{\text{ext}}) - g \mathbf e_3 ) \right) \cdot \nabla_v (   \hat v _1 B_2(t,x_\parallel,0) - \hat v_2 B_1(t,x_\parallel,0) ) 
\\ & + 2 x_3 \frac{ \hat v  \cdot E - (g - E_e) \hat v_3 }{\langle v \rangle^2}    \left( E_3(t,x_\parallel,0 )  + \hat v _1 B_2 (t,x_\parallel,0 ) - \hat v_2 B_1 (t,x_\parallel,0 )  - g \right)
\\ &  \underbrace{ - 2 \frac{x_3}{\langle v \rangle} ( \hat v \times \Bex  ) \cdot  \nabla_v (   \hat v _1 B_2(t,x_\parallel,0) - \hat v_2 B_1(t,x_\parallel,0) ) }_{\eqref{diffalpha3}_1}.
\end{split}
\Ee
Now from the assumptions \eqref{vlfieldass} and \eqref{signcondition}, all the terms on the RHS of \eqref{diffalpha3} except $\eqref{diffalpha3}_1$ can be bounded by
\Be \label{diffalpha31}
\frac{C_1}{c_0} \left( c_0   \frac{ x_3 }{ \langle v \rangle} +  (x_3)^2 +  (\hat v_3)^2 \right),
\Ee
where $C_1 = \sup_{ 0 \le t \le T} \left(   \| E(t) \|_\infty +  \| B(t) \|_\infty   + \|     E_3 (t) \|_{W^{1,\infty} } + \|   B_1 (t) \|_{W^{1,\infty} } + \| B_2(t) \|_{W^{1,\infty} } \right) +  E_e + g $. And from direct computation,
\[
\begin{split}
\eqref{diffalpha3}_1 = &  - 2 B_e \frac{x_3}{ \langle v \rangle^2 } \begin{bmatrix} \hat v_2 \\ - \hat v_1 \\ 0 \end{bmatrix}  \cdot \left( \begin{bmatrix} 1 - \hat v_1^2 \\ - \hat v_1 \hat v_2 \\ - \hat v_1 \hat v_3 \end{bmatrix} B_2 (t,x_\parallel, 0 ) -  \begin{bmatrix}  - \hat v_2 \hat v_1 \\  1 - \hat v_2 ^2 \\ - \hat v_2 \hat v_3 \end{bmatrix}  B_1(t,x_\parallel, 0 ) \right)
\\ = &  - 2 B_e \frac{x_3}{ \langle v \rangle^2 } \left( \hat v_2 B_2 (t,x_\parallel, 0 ) + \hat v_1 B_1(t,x_\parallel, 0 )  ,\right) 
\end{split}
\]
thus from \eqref{signcondition} 
\Be \label{diffalpha32}
|  \eqref{diffalpha3}_1 | \le  \frac{B_e}{c_0}  c_0  \frac{x_3}{\langle v \rangle }.
\Ee
Combining \eqref{diffalpha31} and \eqref{diffalpha32}, we get
%
%
%
%
%
\Be \label{diffalpha21}
\begin{split}
| [ \p_t & + \hat v \cdot \nabla_x + \mathfrak F \cdot \nabla_v ] (\alpha^2 )  | 
 \le 10 (\frac{C_1 +  B_e}{c_0} ) c_0 \frac{ x_3 }{ \langle v \rangle} + 8 C_1 \left(  (x_3)^2  +  (\hat v_3)^2   \right).
\end{split}
\Ee
From the expression of $\alpha$ in \eqref{alphadef} and the assumption \eqref{signcondition}, this yields
\Be \label{tranalpha}
| [ \p_t + \hat v \cdot \nabla_x + \mathfrak F \cdot \nabla_v ]  (\alpha^2 ) | \le 20 (\frac{C_1 + B_e}{c_0} )  \alpha^2,
\Ee
Thus along the characteristics, by the Gr\"owall's inequality we get
\Be
e^{ - 20 (\frac{C_1 + B_e}{c_0} )|t-s| } \alpha^2(t,x,v) \le \alpha^2(s, X(s;t,x,v) , V(s;t,x,v) ) \le e^{20 (\frac{C_1 + B_e}{c_0} )|t-s| }  \alpha^2(t,x,v) .
\Ee
Taking square root we get \eqref{alphaest}.
\end{proof}

\begin{lemma} \label{1alphaintv}
Let $\alpha$ be defined as in \eqref{alphadef}. Then for any $(t,x) \in [0, T) \times \O$, we have
\Be \label{alphavint}
\int_{\mathbb R^3} \frac{\mathbf 1_{|v| \le M } }{ \alpha(t, x, v ) } dv  \le 4 M^3 \ln \left( 1 + \frac{1}{x_3} \right),
\Ee
and
\Be \label{alphavdecayint}
\int_{\mathbb R^3 } \frac{1}{ 1 + |v|^{4 +  \delta} }  \frac{ 1 }{ \alpha(t, x, v )} dv \le  C_\delta  \ln \left( 1+ \frac{1}{x_3} \right).
\Ee
\end{lemma}

\begin{proof}
From \eqref{alphadef} we have
\[
\begin{split}
\int_{\mathbb R^3} \frac{\mathbf 1_{|v| \le M } }{ \alpha(t, x, v ) } dv = &  \int_{ |v| \le M  } \left( (x_3)^2+(\hat{v}_{3})^2 -2\left( E_3(t,x_\parallel, 0 ) + E_e+ (\hat v \times B)_3(t,x_\parallel, 0 )  -  g \right) \frac{x_3}{\langle v \rangle} \right)^{-1/2} dv
\\  \le &  \int_{|v| \le M }   \frac{2}{ x_3 + \frac{| v_3 | }{\langle v \rangle } } dv 
\\ \le &  \int_{ |v_3| \le M }   \frac{2 M^2 }{ x_3 + \frac{| v_3 | }{ M } } dv_3 
\\ = & 4 M^3  \ln \left( x_3 +  \frac{v_3}{M}  \right) \big \rvert_0^M = 4M^3 \ln \left( 1 + \frac{1}{x_3} \right).
\end{split}
\]
Now, for \eqref{alphavdecayint}, we have
\Be \label{alphavdecayintest1}
\begin{split}
\int_{\mathbb R^3 } \frac{1}{ 1 + |v|^{4 +  \delta} }  \frac{ 1 }{ \alpha(t, x, v )} dv \le &  \int_{\mathbb R^3 }  \frac{1}{ 1 + |v|^{4 +  \delta} }  \frac{2}{x_3 + \frac{| v_3 | }{\langle v \rangle } } dv
\\ \lesssim &   \int_{\mathbb R^3 }  \frac{1}{ \left( 1 + |v|^{3 +  \delta}  \right) \left(x_3 + |v_3| \right) } dv.
\end{split}
\Ee
Using the spherical coordinate $v = (r, \theta, \phi )$ we have $dv = r^2 \sin \phi \, dr d\theta d\phi$, $ v_3 = r \cos \phi$, and
\Be \label{alphavdecayintest2}
\begin{split}
&  \int_{\mathbb R^3 }  \frac{1}{ \left( 1 + |v|^{3 +  \delta}  \right) \left(x_3 + |v_3| \right) } dv  
 \\ = & 4 \pi \int_0^\infty \int_0^{\pi /2 }  \frac{r^2 \sin \phi }{ ( 1 + r^{3 + \delta } )( x_3 + r \cos \phi )}  d \phi dr
 \\ = & 4 \pi \int_0^1 \int_0^{\pi /2 }  \frac{r^2 \sin \phi }{ ( 1 + r^{3 + \delta } )( x_3 + r \cos \phi )}  d \phi dr + 4 \pi \int_1^\infty \int_0^{\pi /2 }  \frac{r^2 \sin \phi }{ ( 1 + r^{3 + \delta } )( x_3 + r \cos \phi )}  d \phi dr.
 \end{split}
\Ee
Using change of variables $ r \cos \phi = u  $, $ - r \sin \phi   \, d \phi = du$, we have
\Be \label{alphavdecayintest3}
\begin{split}
& \int_0^1 \int_0^{\pi /2 }  \frac{r^2 \sin \phi }{ ( 1 + r^{3 + \delta } )( x_3 + r \cos \phi )}  d \phi dr
\\ \le &  \int_0^1 \int_0^{\pi /2 }  \frac{r^2 \sin \phi }{ x_3 + r \cos \phi }  d \phi dr
\\ = &  \int_0^1\int_0^{r }  \frac{r}{ x_3 +  u}  d u dr = \int_0^1 r \ln \left(1 + \frac{r}{x_3} \right) dr < \ln \left(1 + \frac{1}{x_3} \right).
\end{split}
\Ee
And using change of variables $ \cos \phi = u$, $- \sin \phi \, d\phi = du$, we have
\Be \label{alphavdecayintest4}
\begin{split}
&  \int_1^\infty \int_0^{\pi /2 }  \frac{r^2 \sin \phi }{ ( 1 + r^{3 + \delta } )( x_3 + r \cos \phi )}  d \phi dr
 \\ \le & C_\delta \int_0^{\pi /2 } \frac{\sin \phi }{x_3 + \cos \phi } d\phi 
 \\  = &  C_\delta \int_0^1 \frac{1}{x_3 + u } du = C_\delta \ln \left( 1+ \frac{1}{x_3} \right).
\end{split}
\Ee
Combining \eqref{alphavdecayintest1}, \eqref{alphavdecayintest2}, \eqref{alphavdecayintest3}, and \eqref{alphavdecayintest4} we conclude \eqref{alphavdecayint}.

\end{proof}

%
%

We have the following estimate on the backward exit time $\tb$ for the trajectory.

\begin{lemma}
Let $(t,x,v) \in (0, T) \times \O \times \mathbb R^3$, and the trajectory $X(s;t,x,v)$ and $V(s;t,x,v)$ satisfies \eqref{HamiltonODE}. Extending $E(t ) = E_0$, $B(t) = B_0$ for $t < 0 $. Suppose for all $t,x,v$, 
\Be \label{gbig10}
 g - E_e  -  E_3(t,x_\parallel, 0 )  -  (\hat v \times B)_3(t,x_\parallel, 0 ) >  c_0,
\Ee
then there exists a $C$ depending on $T$, $g$, $B_e$, $\| E \|_{W^{1,\infty}((0,T) \times \O )} $, $\| B \|_{W^{1,\infty}((0,T) \times \O )} $ such that 
\Be \label{tbbdvb}
\frac{ \tb(t,x,v) }{ \sup_{t-\tb < s < t } \sqrt{ 1 +  |V(s)|^2  }} \le \frac{C}{c_0}    {\hat \vb}_{,3}.
\Ee
If $t -\tb(t,x,v) \le 0$, then
\Be \label{tbdV0}
\frac{ t }{ \sup_{0  \le  s \le t } \sqrt{ 1 +  |V(s)|^2  }} \le \frac{C}{c_0}   \alpha(0,X(0),V(0)).
\Ee
%
\end{lemma}
\begin{proof}
We first prove \eqref{tbbdvb}. For any $(t,x,v) \in (0, T) \times \O \times \mathbb R^3$ with $t -\tb(t,x,v) > 0$,  we have
\[
 {\vb}_{,3} - v_3 = \int_{t-\tb}^t - \mathfrak F_3(s, X(s;t,x,v), V(s;t,x,v) ) ds.
\]
From \eqref{gbig10} this implies
\Be \label{tbvbbd1}
\int_{t-\tb}^t c_0 ds  <  \int_{t-\tb}^t - \mathfrak F_3(s, X(s;t,x,v), V(s;t,x,v) ) ds \le  | {\vb}_{,3}| + | v_3 | 
\Ee
On the other hand, from \eqref{alphaest},
\Be \label{tbvbbd2}
 | {\vb}_{,3}| + | v_3 |  = \langle \vb \rangle | \hat{{\vb}}_{,3}|  + \langle v  \rangle | \hat v_3 |  \le \sup_{t-\tb < s < t } \langle V(s) \rangle \left(  \hat{{\vb}}_{,3} + \alpha (t,x,v) \right) < C  \sup_{t-\tb < s < t } \langle V(s) \rangle \hat{{\vb}}_{,3}.
\Ee
Combining \eqref{tbvbbd1} and \eqref{tbvbbd2} we get
\[
\tb c_0 < C  \sup_{t-\tb < s < t } \langle V(s) \rangle  \hat{{\vb}}_{,3} .
\]
This implies \eqref{tbbdvb}.

For $t-\tb(t,x,v) \le 0$, using the same argument we have, 
\[
\begin{split}
c_0 t < \int_0^t - \mathfrak F_3(s,X(s) ,V(s) ) ds \le & |V_3(0) | + |v_3| 
\end{split}
\]
and from \eqref{alphaest},
\[
 |V_3(0) | + |v_3|   \le \sup_{t-\tb < s < t } \langle V(s) \rangle \left(  \alpha(0,X(0), V(0) ) + \alpha (t,x,v) \right) < C \sup_{t-\tb < s < t } \langle V(s) \rangle  \alpha(0,X(0), V(0) ),
\]
thus we get
\[
c_0 t < C \sup_{t-\tb < s < t } \langle V(s) \rangle  \alpha(0,X(0), V(0) ),
\]
and this yields \eqref{tbdV0}.

\end{proof}

\begin{lemma}
Suppose 
\[
\sup_{0 \le t \le T} \| \nabla_x E(t) \|_{\infty } + \sup_{0 \le t \le T} \| \nabla_x B(t) \|_{\infty  } < \infty.
\]
Then for any $s, t \in (0,T)$, we have
\Be \label{pxviXVest}
\begin{split}
 & |  \p_{x_i} X(s;t,x,v) |  \lesssim e^{C_1 |t-s| } ,
\\ &   | \p_{x_i} V (s;t,x,v)| \lesssim  (t-s)  e^{C_1 |t-s| }  ,
\\ & |  \p_{v_i} X(s;t,x,v) |  \lesssim \frac{|t-s|}{\langle V(s) \rangle }  e^{C_1 |t-s| }  ,
\\ &  | \p_{v_i} V (s;t,x,v)| \lesssim  e^{C_1 |t-s| },
\end{split} 
\Ee
and for $i \neq j $, 
\Be \label{pxiXj} 
\begin{split}
 |  \p_{x_i} X_j(s;t,x,v) |  \lesssim e^{C_1 |t-s| }  \frac{ |t-s |^2}{\langle V(s) \rangle } .
\end{split}
\Ee
where $C_1 =  \left( \sup_{0 \le t \le T} \left( \|  \nabla_x E(t) \|_\infty + \|  \nabla_x B(t) \|_\infty  +   \| E(t) \|_\infty +  \| B(t) \|_\infty  \right)  + g +|B_e|\right)^2$.
\end{lemma}
\begin{proof}
The expressions of $X(s;t,x,v)$ and $V(s;t,x,v) $ are
\Be
\begin{split}
X(s;t,x,v) =  & x-(t-s) \hat{v} + \int^t_s\int^t_\tau \hat{\mathfrak F } _{}(\tau^\prime, X (\tau^\prime), V(\tau^\prime) ) \dd \tau^\prime \dd \tau,
\\ V(s;t,x,v) = & v - \int_s^t \mathfrak F (\tau, X(\tau) , V(\tau) ) \dd \tau.
\end{split}
\Ee

We denote  
\Be \label{hatF}
\begin{split}
\frac{d}{ds} \hat{V}(s)&= \frac{1}{\sqrt{1+ |V(s)|^2}} \frac{d}{ds} V(s)
- \frac{1}{\sqrt{1+ |V(s)|^2}}\hat{V}(s) \cdot \frac{d}{ds}V(s)  \hat{V}(s) \\
&=\frac{1}{\sqrt{1+ |V(s)|^2}} \mathfrak F (s,X(s),V(s))
- \frac{1}{\sqrt{1+ |V(s)|^2}}\hat{V}(s) \cdot \mathfrak F  (s,X(s),V(s))  \hat{V}(s) 
\\ &:= \hat{\mathfrak F }_{} (s,X(s),V(s)).
\end{split}\Ee

By direct computation we get
\Be \label{pxipviXV}
\begin{split}
\p_{x_i} X(s;t,x,v) =  & e_i + \int^t_s\int^t_\tau [ \nabla_x \hat{\mathfrak F } _{}(\tau^\prime ) \cdot \p_{x_i} X(\tau' ) +  \nabla_v \hat{\mathfrak F } _{}(\tau^\prime ) \cdot \p_{x_i} V(\tau' ) ]  \dd \tau^\prime \dd \tau,
\\ \p_{v_i} X(s;t,x,v) =  & -(t-s) \p_{v_i} \hat v + \int^t_s\int^t_\tau [ \nabla_x \hat{\mathfrak F } _{}(\tau^\prime ) \cdot \p_{v_i} X(\tau' ) +  \nabla_v \hat{\mathfrak F } _{}(\tau^\prime ) \cdot \p_{v_i} V(\tau' ) ]  \dd \tau^\prime \dd \tau,
\\ \p_{x_i} V(s;t,x,v) = &  - \int_s^t  [ \nabla_x {\mathfrak F } _{}(\tau^\prime ) \cdot \p_{x_i} X(\tau ) +  \nabla_v {\mathfrak F } _{}(\tau^\prime ) \cdot \p_{x_i} V(\tau ) ] \dd \tau
\\ \p_{v_i} V(s;t,x,v) = & e_i  - \int_s^t  [ \nabla_x {\mathfrak F } _{}(\tau^\prime ) \cdot \p_{v_i} X(\tau ) +  \nabla_v {\mathfrak F } _{}(\tau^\prime ) \cdot \p_{v_i} V(\tau ) ] \dd \tau.
\end{split}
\Ee
Since $\mathfrak F = E + E_{\text{ext}} + \hat v \times ( B +B_{\text{ext}}) - g \mathbf e_3$, we have
\Be \label{nablaxfrakF}
| \nabla_x  \mathfrak F |  \le  | \nabla_{x} E | +  | \nabla_{x} B | .
\Ee
And since $\p_{v_i} \hat v = \frac{ e_i}{\langle v \rangle } - \frac{\hat v_i \hat  v }{\langle v \rangle} $,
\Be \label{nablavfrakF}
| \nabla_v  \mathfrak F |  \lesssim \frac{1}{   \sqrt{ 1 + |v|^2 } } (| B | + |B_e|) .
\Ee
From  the expression of $\hat{\mathfrak F}$ in \eqref{hatF}, we have from \eqref{nablaxfrakF},
\Be \label{nablaxhatF}
| \nabla_x \hat{\mathfrak F} | \le \frac{2}{ \sqrt{ 1 + |v|^2 }} | \nabla_x \mathfrak F | \lesssim  \frac{1}{ \sqrt{ 1 + |v|^2 }} \left(  | \nabla_{x} E | +  | \nabla_{x} B | \right),
\Ee
and from \eqref{nablavfrakF},
\Be \label{nablavhatF}
\begin{split}
| \nabla_v \hat{\mathfrak F} | \le &  \left|  \nabla_v( \frac{1}{ \sqrt{1 +|v|^2}} )(\mathfrak F - \hat v \cdot \mathfrak F \cdot \hat v )  \right| +  \left| \frac{1}{ \sqrt{1 +|v|^2}} \left( \nabla_v \mathfrak F - \nabla_v ( \hat v \cdot \mathfrak F \cdot \hat v ) \right)  \right|
\\ \lesssim &   \frac{1}{ 1 +|v|^2} | \mathfrak F | + \frac{1}{ \sqrt{ 1 +|v|^2}}| \nabla_v  \mathfrak F |
\\ \lesssim  &  \frac{1}{ 1 +|v|^2}  \left( | E | + |B| + g + |B_e| \right)
\end{split}
\Ee
From \eqref{pxipviXV} and Fubini's theorem,
\[
\begin{split}
\p_{x_i} X(s) = &  e_i + \int_s^t \int_s^{\tau'} [ \nabla_x \hat{\mathfrak F } _{}(\tau^\prime ) \cdot \p_{x_i} X(\tau' ) +  \nabla_v \hat{\mathfrak F } _{}(\tau^\prime ) \cdot \p_{x_i} V(\tau' ) ]  \dd \tau \dd \tau'
\\  =  &  e_i + \int_s^t (\tau' -s ) [ \nabla_x \hat{\mathfrak F } _{}(\tau^\prime ) \cdot \p_{x_i} X(\tau' ) +  \nabla_v \hat{\mathfrak F } _{}(\tau^\prime ) \cdot \p_{x_i} V(\tau' ) ] \dd \tau'.
\end{split} 
\]
Therefore, from \eqref{nablaxhatF} and \eqref{nablavhatF} we have
\Be \label{pxiXsest1}
\begin{split}
| \p_{x_i} X(s)  |  \le &  1 + (t-s) \int_s^t \left(  | \nabla_x \hat{\mathfrak F } (\tau ) |  |  \p_{x_i} X(\tau ) |  +   | \nabla_v \hat{\mathfrak F } (\tau ) |  |  \p_{x_i} V(\tau )|  \right) \dd \tau
\\ \lesssim & 1 + (t-s) \left( \|  \nabla_x E \|_\infty + \|  \nabla_x B \|_\infty  \right)  \int_s^t  \frac{1}{\sqrt{ 1 +|V(\tau) |^2 }  }  |  \p_{x_i} X(\tau ) |  d\tau
\\ & \quad + (t-s) \left( \| E \|_\infty +  \| B \|_\infty + g  +|B_e| \right)  \int_s^t  \frac{1}{ 1 +|V(\tau) |^2 }  |  \p_{x_i} V(\tau ) |  d\tau.
\end{split}
\Ee
Thus 
\Be \label{pxiXsest2}
\begin{split}
& \langle V(s) \rangle | \p_{x_i} X(s)  | 
\\   \lesssim & \langle V(s) \rangle + (t-s) \left( \|  \nabla_x E \|_\infty + \|  \nabla_x B \|_\infty   \right) \frac{ \sup_{0 \le s \le t } \langle V(s) \rangle}{ \inf_{0 \le s \le t } \langle V(s) \rangle } \int_s^t  \frac{1}{\sqrt{ 1 +|V(\tau) |^2 }  }  \langle V(\tau) \rangle  |  \p_{x_i} X(\tau ) |  d\tau
\\ & \quad + (t-s) \left( \| E \|_\infty +  \| B \|_\infty + g +|B_e|  \right)  \frac{ \sup_{0 \le s \le t } \langle V(s) \rangle}{ \inf_{0 \le s \le t } \langle V(s) \rangle }  \int_s^t  \frac{1}{ \sqrt{ 1 +|V(\tau) |^2 }}  |  \p_{x_i} V(\tau ) |  d\tau
\\ \lesssim &   \langle V(s) \rangle + C_1 \int_s^t   \frac{1}{\sqrt{ 1 +|V(\tau) |^2 } }  \left(  \langle V(\tau) \rangle  |  \p_{x_i} X(\tau ) | +   |  \p_{x_i} V(\tau ) |  \right)  d\tau,
\end{split}
\Ee
where $C_1 =  \left( \sup_{0 \le t \le T} \left( \|  \nabla_x E(t) \|_\infty + \|  \nabla_x B(t) \|_\infty  +   \| E(t) \|_\infty +  \| B(t) \|_\infty  \right) + g + | B_e| \right)^2 $. From \eqref{pxipviXV}, \eqref{nablaxfrakF}, and \eqref{nablavfrakF}, 
\Be \label{pxiVsest2}
\begin{split}
| \p_{x_i} V(s)  | \lesssim  &   \left( \|  \nabla_x E \|_\infty + \|  \nabla_x B \|_\infty  \right)  \int_s^t   \frac{1}{\sqrt{ 1 +|V(\tau) |^2 }  }  \langle V(\tau) \rangle |  \p_{x_i} X(\tau ) |   d \tau
\\ &  +  (  \| B \|_\infty  + |B_e| ) \int_s^t   \frac{1}{\sqrt{ 1 +|V(\tau) |^2 } }  |  \p_{x_i} V(\tau ) |  \dd \tau
\\ \lesssim &  C_1 \int_s^t   \frac{1}{\sqrt{ 1 +|V(\tau) |^2 } }  \left(  \langle V(\tau) \rangle  |  \p_{x_i} X(\tau ) | +   |  \p_{x_i} V(\tau ) |  \right)  d\tau.
\end{split}
\Ee
Combine \eqref{pxiXsest2} and \eqref{pxiVsest2} we have
\Be \label{pxiXspxiVsest}
\langle V(s) \rangle | \p_{x_i} X(s)  | + | \p_{x_i} V(s)  | \lesssim  \langle V(s) \rangle +  C_1 \int_s^t   \frac{1}{\sqrt{ 1 +|V(\tau) |^2 } }  \left(  \langle V(\tau) \rangle  |  \p_{x_i} X(\tau ) | +   |  \p_{x_i} V(\tau ) |  \right)  d\tau.
\Ee
So from Gronwall's inequality,
\Be \label{pxiXspxiVsest2}
\langle V(s) \rangle | \p_{x_i} X(s)  | + | \p_{x_i} V(s)  | \lesssim  \langle V(s) \rangle e^{C_1 \int_s^t  \frac{1}{\sqrt{ 1 +|V(\tau) |^2 }  }  d\tau }  \lesssim e^{C_1|t-s| }  \langle V(s) \rangle,
\Ee
Next, using the same argument as \eqref{nablaxfrakF}--\eqref{pxiXspxiVsest}, from \eqref{pxipviXV} we get
\[ 
\begin{split}
\langle V(s) \rangle | \p_{v_i} X(s)  | + | \p_{v_i} V(s)  | \lesssim  1 +  C_1 \int_s^t   \frac{1}{\sqrt{ 1 +|V(\tau) |^2 } }  \left(  \langle V(\tau) \rangle  |  \p_{x_i} X(\tau ) | +   |  \p_{x_i} V(\tau ) |  \right)  d\tau.
\end{split}
 \]
Again by Gronwall's inequality,
\Be 
\langle V(s) \rangle | \p_{v_i} X(s)  | + | \p_{v_i} V(s)  | \lesssim  e^{C_1 |t-s|  } .
\Ee
Thus,
\Be \label{pviXspviVsest2}
|  \p_{v_i} X(s;t,x,v) |  \lesssim \frac{e^{C_1 |t-s|} }{ \langle V(s) \rangle  } ,  \,  | \p_{v_i} V (s;t,x,v)| \lesssim  e^{C_1 |t-s|} .
\Ee
Now plug \eqref{pxiXspxiVsest2}, \eqref{pviXspviVsest2} back to \eqref{pxipviXV} and using \eqref{nablaxhatF}, \eqref{nablavhatF}, and \eqref{pviXspviVsest2}, we have
\Be \label{pviXspviVsest3}
\begin{split}
 | \p_{v_i} X (s;t,x,v)| & \lesssim   \frac{ |t-s|}{\langle  v \rangle } +  e^{C_1 |t-s| }  \int_s^t \int_\tau^t  \frac{1}{  \langle V(\tau' ) \rangle^2 }  d\tau' d\tau 
 \\ & \lesssim \frac{|t-s|}{\langle V(s) \rangle }  e^{C_1 |t-s| },
\end{split}
\Ee
and
\Be
\begin{split}
| \p_{x_i} V(s;t,x,v)  |  & \lesssim  e^{C_1 |t-s| }  \int_s^t \frac{1}{ \langle V(\tau) \rangle }  \langle V(\tau) \rangle  \dd \tau \lesssim |t-s|   e^{C_1 |t-s| }
\end{split}
\Ee
From \eqref{pxiXspxiVsest2}, \eqref{pviXspviVsest2}, and \eqref{pviXspviVsest3} we conclude \eqref{pxviXVest}. Finally, for $i \neq j$, from \eqref{pxipviXV}, \eqref{nablaxhatF}, and \eqref{nablavhatF},
\[
| \p_{x_i } X_j (s;t,x,v) | \lesssim e^{C_1 |t-s| }  \int_s^t \int_\tau^t  \frac{1}{ \langle V(\tau' )   \rangle }   d\tau' d\tau \lesssim e^{C_1 |t-s| }  \frac{ |t-s |^2}{\langle V(s) \rangle }.
\]

%
%
%
\end{proof}

\section{$W^{1,\infty}$ estimate of inflow problem }
In this section, we prove an a priori estimate for the inflow problem \eqref{VMfrakF1}, \eqref{inflow}.
From \eqref{HamiltonODE}, we have
\Be \label{Frep}
f(t,x,v) = \mathbf 1_{\tb \ge t } f( 0, X(0), V(0)) +  \mathbf 1_{\tb < t } g(t - \tb, \xb, \vb ) .
\Ee

From \eqref{HamiltonODE}, we have
\Be \label{Xiformula}
\begin{split}
X_{i}(s;t,x,v) &= x_i - \int^t_s \hat{V}_i(\tau;t,x,v) \dd \tau \\
&=x_i - \int^t_s \left\{ \hat{v}_i - \int^t_\tau \hat{\mathfrak F }_i (\tau^\prime) \dd \tau^\prime\right\} \dd \tau\\
&= x_i-(t-s) \hat{v}_i + \int^t_s\int^t_\tau \hat{\mathfrak F } _{i}(\tau^\prime, X (\tau^\prime), V(\tau^\prime) ) \dd \tau^\prime \dd \tau.
\end{split}
\Ee
Set $s=t-\tb$ so that $X_3(t-\tb;t,x,v)=0$. Then 
 \Be
\begin{split}
\tb \hat{v}_3&=x_3+  \int^{t}_{t-\tb} \int^t_\tau \hat{\mathfrak F } _{3}(\tau^\prime) \dd \tau^\prime \dd \tau\\
\end{split}
\Ee
By taking derivatives we obtain
\Be
\begin{split}
\p_x x_3+ \int^{t}_{t-\tb} \int^t_\tau\p_x [\hat{\mathfrak F } _{3}(\tau^\prime, X(\tau^\prime), V(\tau^\prime))] \dd \tau^\prime \dd \tau 
=&
\Big(
\hat{v}_3 - \int^{t}_{t-\tb} \hat{\mathfrak F }_{3} (\tau^\prime) \dd \tau^\prime
\Big) \p_x \tb 
\\
=& \ { \hat \vb}_{,3} \p_x \tb,
\end{split}
\Ee
and hence 
\Be \label{pxitb}
\begin{split}
\p_{x_i} \tb  = &  \frac{1}{{ \hat \vb}_{,3}  } \left\{ \p_{x_i} x_3+ \int^{t}_{t-\tb} \int^t_\tau \p_{x_i} [\hat{\mathfrak F } _{3}(\tau^\prime, X(\tau^\prime), V(\tau^\prime))] \dd \tau^\prime \dd \tau \right\}
\\ = &  \frac{1}{{ \hat \vb}_{,3}  } \left\{ \p_{x_i} x_3+ \int^{t}_{t-\tb} \int^t_\tau  [ \nabla_x \hat{\mathfrak F } _{3}(\tau' ) \cdot \p_{x_i} X(\tau') + \nabla_v \hat{\mathfrak F } _{3}(\tau' ) \cdot \p_{x_i} V(\tau') ] \dd \tau^\prime \dd \tau \right\} .
\end{split}
\Ee

Similarly, 
\Be
\begin{split}
\p_{v_i} x_3 - \tb \p_{v_i} \hat v_3 + \int^{t}_{t-\tb} \int^t_\tau\p_{v_i} [\hat{\mathfrak F } _{3}(\tau^\prime, X(\tau^\prime), V(\tau^\prime))] \dd \tau^\prime \dd \tau 
=&
\Big(
\hat{v}_3 - \int^{t}_{t-\tb} \hat{\mathfrak F }_{3} (\tau^\prime) \dd \tau^\prime
\Big) \p_{v_i} \tb 
\\
=& \ { \hat \vb}_{,3} \p_{v_i} \tb,
\end{split}
\Ee
Thus
\Be \label{pvitb}
 \begin{split}
\p_{v_i} \tb  = &  \frac{1}{{ \hat \vb}_{,3}  } \left\{ - \tb \p_{v_i} \hat v_3 + \int^{t}_{t-\tb} \int^t_\tau \p_{v_i} [\hat{\mathfrak F } _{3}(\tau^\prime, X(\tau^\prime), V(\tau^\prime))] \dd \tau^\prime \dd \tau \right\}
\\ = &  \frac{1}{{ \hat \vb}_{,3}  } \left\{  - \tb \p_{v_i} \hat v_3 +\int^{t}_{t-\tb} \int^t_\tau  [ \nabla_x \hat{\mathfrak F } _{3}(\tau' ) \cdot \p_{v_i} X(\tau') + \nabla_v \hat{\mathfrak F } _{3}(\tau' ) \cdot \p_{v_i} V(\tau') ] \dd \tau^\prime \dd \tau \right\} .
\end{split}
\Ee

And we have
\Be \label{pxbvb}
\begin{split}
\p_{x_i} \xb = & e_i  -  (\p_{x_i} \tb ) \hat \vb  +\int^{t}_{t-\tb} \int^t_\tau  [ \nabla_x \hat{\mathfrak F } _{3}(\tau' ) \cdot \p_{x_i} X(\tau') + \nabla_v \hat{\mathfrak F } _{3}(\tau' ) \cdot \p_{x_i} V(\tau') ] \dd \tau^\prime \dd \tau
\\  \p_{x_i} \vb  = & -  (\p_{x_i } \tb ) \mathfrak F(t-\tb,\xb,\tb) - \int_{t-\tb}^t  [ \nabla_x {\mathfrak F } _{3}(\tau) \cdot \p_{x_i} X(\tau) + \nabla_v {\mathfrak F } _{3}(\tau ) \cdot \p_{x_i} V(\tau) ] \dd \tau
\\ \p_{v_i} \xb = &  -  (\p_{v_i} \hat v )  \tb - ( \p_{v_i} \tb  ) \hat \vb  +\int^{t}_{t-\tb} \int^t_\tau  [ \nabla_x \hat{\mathfrak F } _{3}(\tau' ) \cdot \p_{v_i} X(\tau') + \nabla_v \hat{\mathfrak F } _{3}(\tau' ) \cdot \p_{v_i} V(\tau') ] \dd \tau^\prime \dd \tau
\\  \p_{v_i} \vb  = & e_i -  (\p_{v_i } \tb ) \mathfrak F(t-\tb,\xb,\tb) - \int_{t-\tb}^t  [ \nabla_x {\mathfrak F } _{3}(\tau) \cdot \p_{v_i} X(\tau) + \nabla_v {\mathfrak F } _{3}(\tau ) \cdot \p_{v_i} V(\tau) ] \dd \tau.
\end{split}
\Ee

We have the following calculations for the derivatives of $f$ in \eqref{Frep}.
\Be \label{pxiF}
\begin{split}
 & \p_{x_i} f(t,x,v) 
\\ =  &  \mathbf 1_{\tb > t }  \{ \nabla_x  f_0( X(0), V(0)) \cdot \p_{x_i} X(0) +  \nabla_v  f_0(X(0), V(0)) \cdot \p_{x_i} V(0)  \} 
\\ & + \mathbf 1_{\tb <  t } \{-  \p_t g(t - \tb, \xb, \vb ) \p_{x_i} \tb +\nabla_x g(t - \tb, \xb, \vb ) \p_{x_i} \xb +\nabla_v g(t - \tb, \xb, \vb ) \p_{x_i} \vb  \}
\\ = &  \mathbf 1_{\tb > t }  \{ \nabla_x  f_0( X(0), V(0)) \cdot \p_{x_i} X(0) +  \nabla_v  f_0(X(0), V(0)) \cdot \p_{x_i} V(0)  \}
\\ & +  \mathbf 1_{\tb <  t } \bigg(  -  \p_t g(t - \tb, \xb, \vb ) \frac{1}{{ \hat \vb}_{,3}  } \left\{ \p_{x_i} x_3+ \int^{t}_{t-\tb} \int^t_\tau  [ \nabla_x \hat{\mathfrak F } _{3}(\tau' ) \cdot \p_{x_i} X(\tau') + \nabla_v \hat{\mathfrak F } _{3}(\tau' ) \cdot \p_{x_i} V(\tau') ] \dd \tau^\prime \dd \tau \right\}
\\ & \quad   + \nabla_x g(t - \tb, \xb, \vb ) \cdot  \left\{ e_i  -  (\p_{x_i} \tb ) \hat \vb  +\int^{t}_{t-\tb} \int^t_\tau  [ \nabla_x \hat{\mathfrak F } _{3}(\tau' ) \cdot \p_{x_i} X(\tau') + \nabla_v \hat{\mathfrak F } _{3}(\tau' ) \cdot \p_{x_i} V(\tau') ] \dd \tau^\prime \dd \tau \right\}
\\ & \quad  +  \nabla_v g(t - \tb, \xb, \vb ) \cdot \left\{   -  (\p_{x_i } \tb ) \mathfrak F(t-\tb,\xb,\tb) - \int_{t-\tb}^t  [ \nabla_x {\mathfrak F } _{3}(\tau) \cdot \p_{v_i} X(\tau) + \nabla_v {\mathfrak F } _{3}(\tau ) \cdot \p_{v_i} V(\tau) ] \dd \tau \right\}
\bigg) 
,
\end{split}
\Ee
and 
\Be \label{pviF}
\begin{split}
 & \p_{v_i} f(t,x,v) 
\\ =  &  \mathbf 1_{\tb > t }  \{ \nabla_x  f_0( X(0), V(0)) \cdot \p_{v_i} X(0) +  \nabla_v  f_0(X(0), V(0)) \cdot \p_{v_i} V(0)  \} 
\\ & + \mathbf 1_{\tb <  t } \{-  \p_t g(t - \tb, \xb, \vb ) \p_{v_i} \tb +\nabla_x g(t - \tb, \xb, \vb ) \p_{v_i} \xb +\nabla_v g(t - \tb, \xb, \vb ) \p_{v_i} \vb  \}
\\ = &   \mathbf 1_{\tb > t }  \{ \nabla_x  f_0( X(0), V(0)) \cdot \p_{v_i} X(0) +  \nabla_v  f_0(X(0), V(0)) \cdot \p_{v_i} V(0)  \} 
\\ & +  \mathbf 1_{\tb <  t } \bigg(  -  \p_t g(t - \tb, \xb, \vb ) \frac{1}{{ \hat \vb}_{,3}  } \left\{  - \tb \p_{v_i} \hat v_3 +\int^{t}_{t-\tb} \int^t_\tau  [ \nabla_x \hat{\mathfrak F } _{3}(\tau' ) \cdot \p_{v_i} X(\tau') + \nabla_v \hat{\mathfrak F } _{3}(\tau' ) \cdot \p_{v_i} V(\tau') ] \dd \tau^\prime \dd \tau \right\}  
\\ & \quad   + \nabla_x g(t - \tb, \xb, \vb ) \cdot  \left\{ -  (\p_{v_i} \hat v )  \tb - ( \p_{v_i} \tb  ) \hat \vb + +\int^{t}_{t-\tb} \int^t_\tau  [ \nabla_x \hat{\mathfrak F } _{3}(\tau' ) \cdot \p_{v_i} X(\tau') + \nabla_v \hat{\mathfrak F } _{3}(\tau' ) \cdot \p_{v_i} V(\tau') ] \dd \tau^\prime \dd \tau \right\}
\\ & \quad  +  \nabla_v g(t - \tb, \xb, \vb ) \cdot \left\{   e_i -  (\p_{v_i } \tb ) \mathfrak f(t-\tb,\xb,\tb) - \int_{t-\tb}^t  [ \nabla_x {\mathfrak F } _{3}(\tau) \cdot \p_{v_i} X(\tau) + \nabla_v {\mathfrak F } _{3}(\tau ) \cdot \p_{v_i} V(\tau) ] \dd \tau \right\}
\bigg) 
\end{split}
\Ee

\begin{proposition} \label{inflowprop}
Let $(f,E,B)$ be a solution of \eqref{VMfrakF1}, \eqref{inflow}, \eqref{Maxwell}. Suppose the fields satisfies \eqref{gbig10}, and
\[
\sup_{0 \le t \le T}  \left(\| \nabla_{x} E(t)  \|_\infty + \| \nabla_{x} B(t)  \|_\infty \right) < \infty.
\]
And assume that for $\delta > 0$, 
\[
\begin{split}
 \| \langle v \rangle^{5 + \delta }   \nabla_{x_\parallel} f_0 \|_\infty +  \| \langle v \rangle^{5 + \delta }   \alpha \p_{x_3} f_0 \|_\infty + \| \langle v \rangle^{5 + \delta }    \nabla_v f_0 \|_\infty  & < \infty,
\\   \| \langle v \rangle^{5 + \delta }   \p_t g \|_\infty +   \| \langle v \rangle^{5 + \delta }   \nabla_{x_\parallel} g \|_\infty +  \| \langle v \rangle^{5 + \delta }   \nabla_{v} g \|_\infty    & < \infty.
\end{split}
\]
then for $0 < T \ll 1$ small enough, we have
\Be \label{alphavnablafbd}
\begin{split}
& \sup_{0 \le t \le T} \left( \| \langle v \rangle^{4 + \delta }   \nabla_{x_\parallel} f(t) \|_\infty  + \|  \langle v \rangle^{5 + \delta }  \alpha \p_{x_3} f(t) \|_\infty  +  \|  \langle v \rangle^{5 + \delta }  \nabla_{v} f(t) \|_\infty \right) 
<  \infty.
\end{split} 
\Ee

\end{proposition}
\begin{proof}
For notational simplicity, we assume that the lower order terms of $E$ and $B$ are smaller than the higher order terms:
\Be
\begin{split}
& \sup_{0 \le t \le T} \| E(t) \|_\infty + g + |B_e| \lesssim  \sup_{0 \le t \le T } \| \nabla_{x} E(t) \|_\infty ,
\\ & \sup_{0 \le t \le T} \| B(t) \|_\infty + g + |B_e| \lesssim \sup_{0 \le t \le T } \| \nabla_{x} B(t) \|_\infty .
\end{split} 
\Ee

From \eqref{pxiF}, \eqref{pviF}, we have
\Be \label{nablaxfest}
\begin{split}
 & |  \p_{x_i}  f(t,x,v)  |
\\ \le &  \mathbf 1_{\tb > t }  \{ |  \nabla_x  f_0( X(0), V(0)) | |  \p_{x_i}  X(0)  | +  |  \nabla_v  f_0(X(0), V(0))| |  \p_{x_i} V(0)  | \}
\\ & +  \mathbf 1_{\tb <  t } \bigg(  |   \p_t g(t - \tb, \xb, \vb ) | \left|  \frac{1}{{ \hat \vb}_{,3}  }   \left\{ \p_{x_i} x_3+ \int^{t}_{t-\tb} \int^t_\tau  [ \nabla_x \hat{\mathfrak F } _{3}(\tau' ) \cdot \p_{x_i} X(\tau') + \nabla_v \hat{\mathfrak F } _{3}(\tau' ) \cdot \p_{x_i} V(\tau') ] \dd \tau^\prime \dd \tau \right\} \right|
\\ & \quad   + |  \nabla_x g(t - \tb, \xb, \vb )  | \left|  e_i  -  (\p_{x_i} \tb ) \hat \vb  +\int^{t}_{t-\tb} \int^t_\tau  [ \nabla_x \hat{\mathfrak F } _{3}(\tau' ) \cdot \p_{x_i} X(\tau') + \nabla_v \hat{\mathfrak F } _{3}(\tau' ) \cdot \p_{x_i} V(\tau') ] \dd \tau^\prime \dd \tau \right|
\\ & \quad  + |   \nabla_v g(t - \tb, \xb, \vb )|  \left|  -  (\p_{x_i } \tb ) \mathfrak F(t-\tb,\xb,\tb) - \int_{t-\tb}^t  [ \nabla_x {\mathfrak F } _{3}(\tau) \cdot \p_{x_i} X(\tau) + \nabla_v {\mathfrak F } _{3}(\tau ) \cdot \p_{x_i} V(\tau) ] \dd \tau \right|
\bigg).
\end{split}
\Ee
\Be \label{nablavfest}
\begin{split}
 & |  \p_{v_i}  f(t,x,v)  |
\\ \le &  \mathbf 1_{\tb > t }  \{ |  \nabla_x  f_0( X(0), V(0)) | |  \p_{v_i}  X(0)  | +  |  \nabla_v  f_0(X(0), V(0))| |  \p_{v_i} V(0)  | \}
\\ & +  \mathbf 1_{\tb <  t } \bigg(  |   \p_t g(t - \tb, \xb, \vb ) | \left|  \frac{1}{{ \hat \vb}_{,3}  }   \left\{ - \tb \p_{v_i} \hat v_3 + \int^{t}_{t-\tb} \int^t_\tau  [ \nabla_x \hat{\mathfrak F } _{3}(\tau' ) \cdot \p_{v_i} X(\tau') + \nabla_v \hat{\mathfrak F } _{3}(\tau' ) \cdot \p_{v_i} V(\tau') ] \dd \tau^\prime \dd \tau \right\} \right|
\\ & \quad   + |  \nabla_x g(t - \tb, \xb, \vb )  | \left|  -(\p_{v_i} \hat v ) \tb  -  (\p_{v_i} \tb ) \hat \vb  +\int^{t}_{t-\tb} \int^t_\tau  [ \nabla_x \hat{\mathfrak F } _{3}(\tau' ) \cdot \p_{v_i} X(\tau') + \nabla_v \hat{\mathfrak F } _{3}(\tau' ) \cdot \p_{v_i} V(\tau') ] \dd \tau^\prime \dd \tau \right|
\\ & \quad  + |   \nabla_v g(t - \tb, \xb, \vb )|  \left| e_i  -  (\p_{v_i } \tb ) \mathfrak F(t-\tb,\xb,\tb) - \int_{t-\tb}^t  [ \nabla_x {\mathfrak F } _{3}(\tau) \cdot \p_{v_i} X(\tau) + \nabla_v {\mathfrak F } _{3}(\tau ) \cdot \p_{v_i} V(\tau) ] \dd \tau \right|
\bigg).
\end{split}
\Ee
From \eqref{nablaxhatF}, \eqref{nablavhatF}, \eqref{pxviXVest}, \eqref{pxitb}, and \eqref{pvitb}, we have
\Be \label{pxitbest}
\begin{split}
| \p_{x_i} \tb | \le &   \frac{1}{{ \hat \vb}_{,3}  } \left\{ | \p_{x_i} x_3  | + \int^{t}_{t-\tb} \int^t_\tau  |  \nabla_x \hat{\mathfrak F } _{3}(\tau' ) \cdot \p_{x_i} X(\tau') + \nabla_v \hat{\mathfrak F } _{3}(\tau' ) \cdot \p_{x_i} V(\tau') | \dd \tau^\prime \dd \tau \right\}
\\ \le &  \frac{1}{{ \hat \vb}_{,3}  } \left\{ | \p_{x_i} x_3 |  + \int^{t}_{t-\tb} \int^t_\tau  |  \nabla_x \hat{\mathfrak F } _{3}(\tau' )| | \p_{x_i} X(\tau') |  + |  \nabla_v \hat{\mathfrak F } _{3}(\tau' ) | | \p_{x_i} V(\tau') | \dd \tau^\prime \dd \tau \right\}
\\ \lesssim &   \frac{1}{{ \hat \vb}_{,3}  }  \left( | \p_{x_i} x_3 | + \frac{  C  \tb^2}{\langle v \rangle}  ( \|  \nabla_{x} E \|_{L^\infty_{t,x} } +  \|  \nabla_{x} B \|_{L^\infty_{t,x} } )  \right),
\end{split}
\Ee
and
\Be \label{pvitbest}
\begin{split}
| \p_{v_i} \tb | \le &   \frac{1}{{ \hat \vb}_{,3}  } \left\{ | \tb \p_{v_i} \hat v_3  | + \int^{t}_{t-\tb} \int^t_\tau  |  \nabla_x \hat{\mathfrak F } _{3}(\tau' ) \cdot \p_{v_i} X(\tau') + \nabla_v \hat{\mathfrak F } _{3}(\tau' ) \cdot \p_{v_i} V(\tau') | \dd \tau^\prime \dd \tau \right\}
\\ \le &  \frac{1}{{ \hat \vb}_{,3}  } \left\{  | \tb \p_{v_i} \hat v_3  |   + \int^{t}_{t-\tb} \int^t_\tau  |  \nabla_x \hat{\mathfrak F } _{3}(\tau' )| | \p_{v_i} X(\tau') |  + |  \nabla_v \hat{\mathfrak F } _{3}(\tau' ) | | \p_{v_i} V(\tau') | \dd \tau^\prime \dd \tau \right\}
\\ \lesssim &   \frac{1}{{ \hat \vb}_{,3}  }  \left(  \frac{\tb}{\langle v \rangle } +  \frac{ C \tb^2}{\langle v \rangle^2}   ( \|  \nabla_{x} E \|_{L^\infty_{t,x} } +  \|  \nabla_{x} B \|_{L^\infty_{t,x} }  ) \right).
\end{split}
\Ee

Thus from \eqref{nablaxfest} and \eqref{pxitbest}, for $i = 1,2$, 
\begin{align}
 & \notag  | \langle v \rangle^{4+\delta} \p_{x_i}  f(t,x,v)  |
\\ \label{vnablaparaf0} & \le      \langle v \rangle^{4+\delta} \big(  |  \nabla_{x_\parallel}  f_0( X(0), V(0)) | |  \p_{x_i}  X_\parallel (0)  |+ |  \p_{x_3}  f_0( X(0), V(0)) | |  \p_{x_i}  X_3 (0)  | 
\\ \notag &\quad \quad \quad \quad \quad \quad \quad +  |  \nabla_v  f_0(X(0), V(0))| |  \nabla_x V(0)  | \big)
\\ \label{vnablaparaptg} & +    \langle v \rangle^{4+\delta}   |   \p_t g(t - \tb, \xb, \vb ) |   \frac{C \tb^2}{{ \hat \vb}_{,3} \langle v \rangle  }    (  \|  \nabla_{x} E \|_{L^\infty_{t,x} } +  \|  \nabla_{x} B \|_{L^\infty_{t,x} }  )
\\ \label{vnablaparaxg} &    +     \langle v \rangle^{4+\delta}   |  \nabla_x g(t - \tb, \xb, \vb )  | \left( 1 + \left( \frac{1}{{ \hat \vb}_{,3}  } + 1 \right) \frac{ C \tb^2}{\langle v \rangle}   (  \|  \nabla_{x} E \|_{L^\infty_{t,x} } +  \|  \nabla_{x} B \|_{L^\infty_{t,x} }  ) \right)
\\ \notag & +   \langle v \rangle^{4+\delta}  |   \nabla_v g(t - \tb, \xb, \vb )|  C  \left( \frac{\tb^2}{{ \hat \vb}_{,3}  \langle v \rangle } |  \mathfrak F(t-\tb,\xb,\tb)  | + \tb \right)   (  \|  \nabla_{x} E \|_{L^\infty_{t,x} } +  \|  \nabla_{x} B \|_{L^\infty_{t,x} } )  .
\end{align}
From \eqref{pxviXVest}, \eqref{pxipviXV}, and \eqref{tbdV0}, we have
\Be  \label{nablaparaf0est}  
\begin{split}
| \eqref{vnablaparaf0} | \le  & C \langle v \rangle^{4+\delta} \left(  |  \nabla_{x_\parallel}  f_0( X(0), V(0)) | +  t |  \p_{x_3}  f_0( X(0), V(0)) |  +  (1+|v|) |  \nabla_v  f_0(X(0), V(0))| \right)
\\ \le & ( C^2 +1 )  \left(  \| \langle v \rangle^{5+\delta}  \alpha  \p_{x_3} f_0 \|_{L^\infty_{x} } + \| \langle v \rangle^{4+\delta}  \nabla_{x_\parallel} f_0 \|_{L^\infty_{x} }   + \| \langle v \rangle^{5+\delta}   \nabla_v f_0 \|_{L^\infty_{x} } \right).
\end{split}
\Ee
And
\Be \label{nablaparagest}
\begin{split}
& |\eqref{vnablaparaptg}| +| \eqref{vnablaparaxg} | +  | \eqref{vnablaparaxg} |
\\  \le &       ( C + \tb C   (  \|  \nabla_{x} E \|_{L^\infty_{t,x} } +  \|  \nabla_{x} B \|_{L^\infty_{t,x} } ) ) 
\\ & \times  \left( \| \langle v \rangle^{4+\delta}   \p_t g \|_{L^\infty_{t,x} } + \| \langle v \rangle^{4+\delta}   \nabla_x g \|_{L^\infty_{t,x} }  + \| \langle v \rangle^{4+\delta}   \nabla_v g \|_{L^\infty_{t,x} } \right) .
\end{split}
\Ee

And from \eqref{nablaxfest} and \eqref{pxitbest}, for $i=3$,
\begin{align}
 & \notag  | \langle v \rangle^{5+\delta}  \alpha  \p_{x_3}  f(t,x,v)  |
\\ \label{alphavnablaf0} & \le      \langle v \rangle^{5+\delta} \alpha(t,x,v) \big(  |  \nabla_{x_\parallel}  f_0( X(0), V(0)) | |  \p_{x_3}  X_\parallel (0)  |+ |  \p_{x_3}  f_0( X(0), V(0)) | |  \p_{x_3}  X_3 (0)  | 
\\ \notag &\quad \quad \quad \quad \quad \quad \quad \quad \quad \quad \quad +  |  \nabla_v  f_0(X(0), V(0))| |  \nabla_x V(0)  | \big)
\\ \label{alphavptg} & +    \langle v \rangle^{5+\delta}    |   \p_t g(t - \tb, \xb, \vb ) |   \frac{\alpha(t,x,v)}{{ \hat \vb}_{,3}  }   ( 1 +  \frac{ C \tb^2}{\langle v \rangle }    (  \|  \nabla_{x} E \|_{L^\infty_{t,x} } +  \|  \nabla_{x} B \|_{L^\infty_{t,x} }) )
\\ \label{alphavnablaxg} &    +     \langle v \rangle^{5+\delta}    |  \nabla_x g(t - \tb, \xb, \vb )  | \frac{\alpha(t,x,v)}{{ \hat \vb}_{,3}  }   ( 1 +  \frac{ C \tb^2}{\langle v \rangle }  (  \|  \nabla_{x} E \|_{L^\infty_{t,x} } +  \|  \nabla_{x} B \|_{L^\infty_{t,x} }  ) )
\\  \label{alphavnablavg} &   +   \langle v \rangle^{5+\delta}  |   \nabla_v g(t - \tb, \xb, \vb )|  \frac{\alpha(t,x,v)}{{ \hat \vb}_{,3}  } |  \mathfrak F(t-\tb,\xb,\tb)  |  ( 1 +  \tb C   (  \|  \nabla_{x} E \|_{L^\infty_{t,x} } +  \|  \nabla_{x} B \|_{L^\infty_{t,x} }  ) ) .
\end{align}
Now from \eqref{pxviXVest} and the velocity lemma \eqref{alphaest},
\Be \label{alphavnablaf0est}
\begin{split}
|\eqref{alphavnablaf0} | \le &   C  \frac{ \langle v \rangle^{5+\delta}  \alpha(t,x,v) }{   ( 1 +|V(0)| )^{5 + \delta } \alpha(0,X(0),V(0))} \big( \| \langle v \rangle^{5+\delta}  \alpha  \p_{x_3} f_0 \|_{L^\infty_{x} } + \| \langle v \rangle^{5+\delta}  \nabla_{x_\parallel} f_0 \|_{L^\infty_{x} }  
\\ & \quad \quad \quad \quad \quad \quad \quad \quad \quad \quad \quad \quad \quad \quad \quad \quad + \| \langle v \rangle^{5+\delta}  \alpha \nabla_v f_0 \|_{L^\infty_{x} } \big)
\\ \le & C  \left( \| \langle v \rangle^{5+\delta}  \alpha  \p_{x_3} f_0 \|_{L^\infty_{x} } + \| \langle v \rangle^{5+\delta}  \nabla_{x_\parallel} f_0 \|_{L^\infty_{x} }  + \| \langle v \rangle^{5+\delta}  \alpha \nabla_v f(0) \|_{L^\infty_{x} } \right),
\end{split} 
\Ee
and
\Be \label{alphavptgxgvgest}
\begin{split}
& |\eqref{alphavptg} | + |\eqref{alphavnablaxg} | + |\eqref{alphavnablavg} |
\\  \le &    \frac{ \langle v \rangle^{5+\delta}  \alpha(t,x,v) }{   ( 1 + |\vb|  )^{5 + \delta }  \hat{ {\vb}_{,3} } } ( 1 + \tb C   (  \|  \nabla_{x} E \|_{L^\infty_{t,x} } +  \|  \nabla_{x} B \|_{L^\infty_{t,x} }  ) )
\\ & \times  \left( \| \langle v \rangle^{5+\delta}   \p_t g \|_{L^\infty_{t,x} } + \| \langle v \rangle^{5+\delta}   \nabla_x g \|_{L^\infty_{t,x} }  + \| \langle v \rangle^{5+\delta}   \nabla_v g \|_{L^\infty_{t,x} } \right)
\\  \le &       ( C + \tb C   (   \|  \nabla_{x} E \|_{L^\infty_{t,x} } +  \|  \nabla_{x} B \|_{L^\infty_{t,x} } ) )
 \\ & \times  \left( \| \langle v \rangle^{5+\delta}   \p_t g \|_{L^\infty_{t,x} } + \| \langle v \rangle^{5+\delta}   \nabla_x g \|_{L^\infty_{t,x} }  + \| \langle v \rangle^{5+\delta}   \nabla_v g \|_{L^\infty_{t,x} } \right).
\end{split} 
\Ee

Also, similarly from \eqref{nablavfest} and \eqref{pvitbest},
\begin{align}
 & \notag  | \langle v \rangle^{5+\delta}  \p_{v_i}  f(t,x,v)  |
\\ \notag & \le      \langle v \rangle^{5+\delta}  \big(  |  \nabla_{x_\parallel}  f_0( X(0), V(0)) | |  \p_{v_i}  X_\parallel (0)  |+ |  \p_{x_3}  f_0( X(0), V(0)) | |  \p_{v_i}  X_3 (0)  | 
\\ \notag &\quad \quad \quad \quad \quad \quad \quad +  |  \nabla_v  f_0(X(0), V(0))| |  \nabla_{v_i} V(0)  | \big)
\\ \notag & +    \langle v \rangle^{5+\delta}    |   \p_t g(t - \tb, \xb, \vb ) |   \frac{\tb}{{ \hat \vb}_{,3} \langle v \rangle  }    \times ( 1 + \frac{\tb}{\langle v \rangle} C   (  \|  \nabla_{x} E \|_{L^\infty_{t,x} } +  \|  \nabla_{x} B \|_{L^\infty_{t,x} } ) )
\\ \notag &    +     \langle v \rangle^{5+\delta}    |  \nabla_x g(t - \tb, \xb, \vb )  | \left( \frac{\tb}{\langle v \rangle} + \left( \frac{1}{{ \hat \vb}_{,3}  } + 1 \right) \frac{\tb^2}{\langle v \rangle}  C (    \|  \nabla_{x} E \|_{L^\infty_{t,x} } +  \|  \nabla_{x} B \|_{L^\infty_{t,x} }  ) \right)
\\ \notag &   +   \langle v \rangle^{5+\delta}  |   \nabla_v g(t - \tb, \xb, \vb )|  \left(( 1 +  \frac{\tb}{{ \hat \vb}_{,3}  \langle v \rangle } |  \mathfrak F(t-\tb,\xb,\tb)  |  + \tb  C   (  \|  \nabla_{x} E \|_{L^\infty_{t,x} } +  \|  \nabla_{x} B \|_{L^\infty_{t,x} } ) \right) 
\\ \notag  & \le   ( C_2^2 +1 )  \left(  \| \langle v \rangle^{5+\delta}  \alpha  \p_{x_3} f_0 \|_{L^\infty_{x} } + \| \langle v \rangle^{4+\delta}  \nabla_{x_\parallel} f_0 \|_{L^\infty_{x} }   + \| \langle v \rangle^{5+\delta}   \nabla_v f_0 \|_{L^\infty_{x} } \right)
\\  \notag &  \quad +      ( C + \tb C   (  \|  \nabla_{x} E \|_{L^\infty_{t,x} } +  \|  \nabla_{x} B \|_{L^\infty_{t,x} } ) )
 \\ \label{dvfestfinal} & \quad \quad \times  \left( \| \langle v \rangle^{5+\delta}   \p_t g \|_{L^\infty_{t,x} } + \| \langle v \rangle^{5+\delta}   \nabla_x g \|_{L^\infty_{t,x} }  + \| \langle v \rangle^{5+\delta}   \nabla_v g \|_{L^\infty_{t,x} } \right) .
\end{align}

Now from \eqref{lambdanablaxE} we have 
\[ \begin{split}
&  \|  \nabla_{x} E \|_{L^\infty_{t,x} } +  \|  \nabla_{x} B \|_{L^\infty_{t,x} }  
  \\ & \lesssim  \|  \langle v \rangle^{4+\delta} \nabla_{x_\parallel}  f(t,x,v)  \|_{L^{\infty}_{t,x} }  + \|  \langle v \rangle^{5+\delta}  \alpha \p_{x_3}  f(t,x,v)  \|_{L^{\infty}_{t,x} }  +C,
\end{split}  \]
 thus combining \eqref{nablaparaf0est}, \eqref{nablaparagest}, \eqref{alphavnablaf0est},  \eqref{alphavptgxgvgest}, and \eqref{dvfestfinal}, and by choosing $0 < T \ll 1 $ small enough 
we have
\Be \label{dxfestfinalin}
\begin{split}
 & \|  \langle v \rangle^{4+\delta} \nabla_{x_\parallel}  f(t,x,v)  \|_{L^{\infty}_{t,x} }  + \|  \langle v \rangle^{5+\delta}  \alpha \p_{x_3}  f(t,x,v)  \|_{L^{\infty}_{t,x} }  +   \|  \langle v \rangle^{5+\delta}  \nabla_{v}  f(t,x,v)  \|_{L^{\infty}_{t,x} }
 \\  \le &   2  \left( \| \langle v \rangle^{5+\delta}  \alpha \p_{x_3} f_0 \|_{L^\infty_{x} } + \| \langle v \rangle^{5+\delta}   \nabla_{x_\parallel} f_0 \|_{L^\infty_{x} }  + \| \langle v \rangle^{5+\delta}   \nabla_v f_0 \|_{L^\infty_{x} } \right) 
 \\ & + C_1    \left( \| \langle v \rangle^{5+\delta}   \p_t g \|_{L^\infty_{t,x} } + \| \langle v \rangle^{5+\delta}   \nabla_x g \|_{L^\infty_{t,x} }  + \| \langle v \rangle^{5+\delta}   \nabla_v g \|_{L^\infty_{t,x} } \right)
 \\ & + \frac{1}{2} \left(  \|  \langle v \rangle^{4+\delta} \nabla_{x_\parallel}  f(t,x,v)  \|_{L^{\infty}_{t,x} }  + \|  \langle v \rangle^{5+\delta}  \alpha \p_{x_3}  f(t,x,v)  \|_{L^{\infty}_{t,x} } \right) < \infty.
 \end{split}
\Ee
This conclude \eqref{alphavnablafbd}.

\end{proof}

We state and prove a variation of Ukai's trace theorem in \cite{Ukai}.
\begin{lemma} \label{traceUkai}
Suppose $f \in L^\infty( (0,T) \times \O \times \mathbb R^3 )$, and  $\mathfrak F \in W^{1,\infty}((0,T) \times \mathbb R^3 ) $ satisfy
\Be \label{transopass}
\p_t f + \hat v \cdot \nabla_x f + \mathfrak F \cdot \nabla_v f  = h \in L^\infty((0,T) \times \O \times \mathbb R^3 ).
\Ee
Then $f \in L^\infty((0,T) \times (\gamma \setminus \gamma_0 ) ) $, and
\Be \label{tracelinftyUkai}
 \sup_{0 \le t \le T } \| f(t) \|_{L^\infty( \gamma \setminus \gamma_0 )  } \le  \sup_{0 \le t \le T } \| f(t) \|_{L^\infty( \O \times \mathbb R^3 )  } .
 \Ee
\end{lemma}
\begin{proof}
Denote the characteristics $X(s;t,x,v), V(s;t,x,v)$ which solves
\Be \label{charaholder}
\begin{split}
\frac{d}{ds} X(s;t,x,v) = & \hat V(s;t,x,v),
\\ \frac{d}{ds} V(s;t,x,v) = & \mathfrak F(s, X(s;t,x,v) ,V(s;t,x,v) ),
\end{split}
\Ee 
and $X(t;t,x,v) = x $, $V(t;t,x,v) = v$. Then since $\mathfrak F \in W^{1,\infty}((0,T) \times \mathbb R^3 ) $, the characteristics \eqref{charaholder} is H\"older continuous. From \eqref{transopass}, for almost every $(t,x,v) \in (0,T) \times \gamma_+$, and $ \max\{ 0, t-\tb(t,x,v) \} $,
\[
f(t,x,v) = f(s, X(s;t,x,v) , V(s;t,x,v) ) + \int_s^t h(\tau; X(\tau;t,x,v) , V(\tau;t,x,v)  ) d\tau.
\]
Thus
\[
\sup_{0 < t < T} \| f(t)  \|_{L^\infty(\gamma_+ ) } \le  \sup_{0 < t < T} \| f(t)  \|_{L^\infty(\O \times \mathbb R^3 ) } + (t-s) \| h \|_{L^\infty((0,T) \times \O \times \mathbb R^3 ) }.
\]
Since $ t -s  > 0$ can be arbitrarily small, we have
\[
\sup_{0 < t < T} \| f(t)  \|_{L^\infty(\gamma_+ ) } \le  \sup_{0 < t < T} \| f(t)  \|_{L^\infty(\O \times \mathbb R^3 ) }.
\]
Now, for $(x,v) \in \gamma_-$ and $s \in (t, \max \{ T, \tf(t,x,v) \} ) $, we have
\[
f(t,x,v) = f(s, X(s;t,x,v) , V(s;t,x,v) ) - \int_t^s h(\tau; X(\tau;t,x,v) , V(\tau;t,x,v)  ) d\tau.
\]
Using the same argument we get
\[
\sup_{0 < t < T} \| f(t)  \|_{L^\infty(\gamma_- ) } \le  \sup_{0 < t < T} \| f(t)  \|_{L^\infty(\O \times \mathbb R^3 ) }.
\]
This proves \eqref{tracelinftyUkai}.

\end{proof}

Next, we prove a trace theorem for the derivatives of $f$.

\begin{lemma} \label{tracepf}
Let $(f,E,B)$ be a solution of \eqref{VMfrakF1}, \eqref{inflow}, \eqref{Maxwell}. Suppose
\Be \label{nablaEBassin}
\sup_{0 \le t \le T}  \left(\| \nabla_{x} E(t)  \|_\infty + \| \nabla_{x} B(t)  \|_\infty \right) < \infty,
\Ee
\Be
 \| \langle v \rangle^{5 + \delta }   \p_t g \|_\infty +   \| \langle v \rangle^{5 + \delta }   \nabla_{x_\parallel} g \|_\infty +  \| \langle v \rangle^{5 + \delta }   \nabla_{v} g \|_\infty    < \infty,
\Ee
and
\Be \label{diffusefinalin}
 \sup_{0 \le t \le T} \left( \| \langle v \rangle^{4 + \delta }   \nabla_{x_\parallel} f(t) \|_\infty  + \|  \langle v \rangle^{5 + \delta }  \alpha \p_{x_3} f(t) \|_\infty  +  \|  \langle v \rangle^{5 + \delta }  \nabla_{v} f(t) \|_\infty \right) < \infty.
 \Ee
Then
\Be \label{pfgammapbd}
 \sup_{0 \le t \le T}  \left(  \| \langle v \rangle^{4 + \delta }  \nabla_{x_\parallel} f (t) \|_{L^\infty(\gamma \setminus \gamma_0) } +     \| \langle v \rangle^{5+\delta} \alpha \p_{x_3} f  (t,x,v)  \|_{L^\infty(\gamma \setminus \gamma_0 )  }  +  \| \langle v \rangle^{5 + \delta }  \nabla_{v} f (t) \|_{L^\infty(\gamma \setminus \gamma_0) }    \right) < \infty
 \Ee
\end{lemma}

\begin{proof}
The proof uses similar argument as Ukai's proof of a trace theorem in \cite{Ukai}. Next, notice that for $ \p_{\mathbf e} \in \{ \nabla_{x_\parallel} ,\nabla_v \} $, and $p = 4 + \delta, 5+ \delta$, we have
\Be
\p_t  ( \langle v \rangle^{p} \p_{\mathbf e} f )  + \hat v \cdot \nabla_x (  \langle v \rangle^{p}\p_{\mathbf e} f  )+ \mathfrak F \cdot \nabla_v  (  \langle v \rangle^{p}\p_{\mathbf e} f ) = -  \langle v \rangle^{p}\p_{\mathbf e} \hat v \cdot \nabla_x f -  \langle v \rangle^{p} \p_{\mathbf e } \mathfrak F \cdot \nabla_v f -  \mathfrak F \cdot \nabla_v (  \langle v \rangle^{p} ) \p_{\mathbf e} f.
\Ee
Then for almost every $(x,v) \in \gamma_+$, and $ s \in ( \max\{ 0 ,  t- \tb(t,x,v) \} , t ) $, we have
\Be \label{pefexp1}
\begin{split}
\langle v \rangle^{p} \p_{\mathbf e} f (t,x,v) = &   \langle V(s;t,x,v) \rangle^{p}\p_{\mathbf e} f (s,X(s;t,x,v) ,V(s;t,x,v))
\\ &  + \int_s^t \left(  -  \langle v \rangle^{p} \p_{\mathbf e} \hat v \cdot \nabla_x f -  \langle v \rangle^{p} \p_{\mathbf e} \mathfrak F \cdot \nabla_v f  -  \mathfrak F \cdot \nabla_v (  \langle v \rangle^{p} ) \p_{\mathbf e} f \right) (\tau, X(\tau;t,x,v), V(\tau;t,x,v) ) d \tau.
\end{split}
\Ee
Thus, from \eqref{nablaEBassin} and \eqref{diffusefinalin},
\Be \label{pefbd1}
\begin{split}
& | \langle v \rangle^{p}  \p_{\mathbf e} f (t,x,v) | 
\\  \le & \sup_{0 \le s \le t} \| \langle v\rangle^p \p_{\mathbf e} f(s) \|_{\infty}  + \frac{ C (t-s) }{\alpha(t,x,v) }  \bigg( (  \sup_{0 \le t \le T}  (\| \nabla_{x} E(t)  \|_\infty + \| \nabla_{x} B(t) \|_\infty ) + g + |B_e| )
\\ & \quad \quad \quad \quad \times \sup_{ 0 \le s \le t } \left( \| \langle v \rangle^{4 + \delta}\nabla_{x_\parallel } f(t) \|_\infty +  \|  \langle v \rangle^{5 + \delta} \alpha \p_{x_3} f(t) \|_\infty   +  \|  \langle v \rangle^{5 + \delta}  \nabla_v f(t) \|_\infty \right) \bigg),
\end{split}
\Ee
since we can choose $s $ close enough to $t$ such that 
\[
\begin{split}
&  \frac{ C (t-s) }{\alpha(t,x,v) }  \bigg(  ( \sup_{0 \le t \le T}  (\| \nabla_{x} E(t)  \|_\infty + \| \nabla_{x} B(t) \|_\infty ) + g + |B_e|)
\\ & \quad \quad \quad \quad \times \sup_{ 0 \le s \le t } \left( \| \langle v \rangle^{4 + \delta}\nabla_{x_\parallel } f(t) \|_\infty +  \|  \langle v \rangle^{5 + \delta} \alpha \p_{x_3} f(t) \|_\infty   +  \|  \langle v \rangle^{5 + \delta}  \nabla_v f(t) \|_\infty \right) \bigg)< \e \ll 1,
\end{split}
\]
we have
\[
 | \langle v \rangle^{p}  \p_{\mathbf e} f (t,x,v) |  \le  \sup_{0 \le s \le t} \| \langle v\rangle^p \p_{\mathbf e} f(s) \|_{\infty}  + \e,
\]
for any $\e > 0$. Therefore, we get
\Be \label{xparavfbdpin}
\begin{split}
& \sup_{0 \le t \le T}  \| \langle v \rangle^{4 + \delta }  \nabla_{x_\parallel} f (t) \|_{L^\infty(\gamma_+) }  \le \sup_{0 \le t \le T}   \| \langle v \rangle^{4 + \delta }  \nabla_{x_\parallel} f (t) \|_\infty, 
\\ &  \sup_{0 \le t \le T}   \| \langle v \rangle^{5 + \delta }  \nabla_{v} f (t) \|_{L^\infty(\gamma_+) }  \le  \sup_{0 \le t \le T}   \| \langle v \rangle^{5 + \delta }  \nabla_{v} f (t) \|_\infty.
\end{split}
\Ee
Similarly, since
\[
\begin{split}
& (\p_t + \hat v \cdot \nabla_x +\mathfrak F \cdot \nabla_v )   ( \langle v \rangle^{5+\delta} \alpha \p_{x_3} f ) 
\\  & =  -  \langle v \rangle^{5+\delta} \alpha \p_{x_3 } \mathfrak F \cdot \nabla_v f -  \mathfrak F \cdot \nabla_v (  \langle v \rangle^{5+\delta} ) \alpha \p_{x_3} f - [(\p_t + \hat v \cdot \nabla_x + \mathfrak F \cdot \nabla_v ) \alpha ]  \langle v \rangle^{5+\delta} \p_{x_3} f
\\ & = : G_{\alpha} (t,x,v).
\end{split}
\] 
Then for almost every $(x,v) \in \gamma_+$, and $ s \in ( \max\{ 0 ,  t- \tb(t,x,v) \} , t ) $, we have
\Be \label{ap3fexp}
\begin{split}
& \langle v \rangle^{5+\delta} \alpha \p_{x_3} f  (t,x,v)  
\\ = & \langle V(s;t,x,v) \rangle^{5 +\delta}  \alpha \p_{x_3} f (s, X(s;t,x,v) , V(s;t,x,v) )  + \int_s^t G_\alpha(\tau , X(\tau; t,x,v) , V(\tau; t,x,v) ) d\tau.
\end{split}
\Ee
Since
\[
\begin{split}
| G_\alpha (t,x,v) | \le &   C \big(  (( \sup_{0 \le t \le T}  \| \nabla_{x} E(t)  \|_\infty + \| \nabla_{x} B(t) \|_\infty ) + g + |B_e|)
\\ & \quad \quad \quad \quad \times \sup_{ 0 \le s \le t } \left( \| \langle v \rangle^{4 + \delta}\nabla_{x_\parallel } f(t) \|_\infty +  \|  \langle v \rangle^{5 + \delta} \alpha \p_{x_3} f(t) \|_\infty   +  \|  \langle v \rangle^{5 + \delta}  \nabla_v f(t) \|_\infty \right) \big),
\end{split}
\]
using the same argument as \eqref{pefbd1}--\eqref{xparavfbdpin} , we obtain 
\Be \label{x3fbdpin}
\sup_{0 \le t \le T}   \| \langle v \rangle^{5+\delta} \alpha \p_{x_3} f  (t,x,v)  \|_{L^\infty(\gamma_+ ) } \le  \sup_{0 \le t \le T}     \| \langle v \rangle^{5+\delta} \alpha \p_{x_3} f  (t,x,v)  \|_\infty.
\Ee

Now, for $(x,v) \in \gamma_-$, and any $s \in (t, \max \{ T, \tf(t,x,v) \} ) $ we have the same formula \eqref{pefexp1} and \eqref{ap3fexp} for $ \langle v \rangle^p \p \mathbf e f $ and $\langle v \rangle ^{5 +\delta}\alpha \p_{x_3}  f $ respectively. Therefore by the same argument, we get
\Be \label{xparavfbdmin}
\begin{split}
& \sup_{0 \le t \le T}  \| \langle v \rangle^{4 + \delta }  \nabla_{x_\parallel} f (t) \|_{L^\infty(\gamma_-) }  \le \sup_{0 \le t \le T}   \| \langle v \rangle^{4 + \delta }  \nabla_{x_\parallel} f (t) \|_\infty, 
\\ &  \sup_{0 \le t \le T}   \| \langle v \rangle^{5 + \delta }  \nabla_{v} f (t) \|_{L^\infty(\gamma_-) }  \le  \sup_{0 \le t \le T}   \| \langle v \rangle^{5 + \delta }  \nabla_{v} f (t) \|_\infty,
\\  & \sup_{0 \le t \le T}   \| \langle v \rangle^{5+\delta} \alpha \p_{x_3} f  (t,x,v)  \|_{L^\infty(\gamma_-) } \le  \sup_{0 \le t \le T}     \| \langle v \rangle^{5+\delta} \alpha \p_{x_3} f  (t,x,v)  \|_\infty.
\end{split}
\Ee

Combining \eqref{xparavfbdpin}, \eqref{x3fbdpin}, and \eqref{xparavfbdmin}, we conclude \eqref{pfgammapbd}.
\end{proof}

\section{Local existence}
We prove the local existence for RVM system with the inflow boundary condition in this section. We recursively define a sequence of functions:
\[
f^0(t,x,v) = f_0(x,v), \ E^0(t,x) = E_0(t,x), \ B^0(t,x) = B_0(x).
\]
For $\ell \ge 1$, let $f^\ell$ be the solution of
\Be \label{fellseqin}
\begin{split}
\p_t f^\ell + \hat v \cdot \nabla_x f^\ell +  \mathfrak F^{\ell-1}   \cdot \nabla_v f^\ell = &  0, \text{ where }   \mathfrak F^{\ell-1} =  E^{\ell-1} + E_{\text{ext}} + \hat v \times ( B^{\ell-1}  + B_{\text{ext}} ) - g \mathbf e_3,
\\f^\ell(0,x,v)  = &  f_0(x,v) ,
\\ f^\ell(t,x,v) |_{\gamma_-} = &  g(t,x,v).
\end{split}
\Ee
Let $\rho^\ell = \int_{\mathbb R^3 } f^\ell dv, j^\ell = \int_{\mathbb R^3 } \hat v  f^\ell dv$. Let 
\Be \label{ElBlin}
E^\ell = \eqref{Eesttat0pos} + \dots + \eqref{Eest3bdrycontri}, \  B^\ell = \eqref{Besttat0pos} + \dots + \eqref{Bestbdrycontri}, \text{ with } f \text{ changes to } f^\ell. 
\Ee
We prove several uniform-in-$\ell$ bounds for the sequence before passing the limit.

\begin{lemma}
Suppose $f_0$ satisfies \eqref{f0bdd}, $E_0$, $B_0$ satisfy \eqref{E0B0g}, \eqref{E0B0bdd}, then there exits $M_1, M_2$, and $c_0$, such that for $0 < T \ll 1 $, 
\Be \label{fellboundin}
\begin{split}
 \sup_\ell \sup_{0 \le t \le T} \left( \| \langle v \rangle^{4 + \delta } f^\ell(t) \|_{L^\infty(\bar \O \times \mathbb R^3)}    \right) < & M_1, \ 
\\  \sup_\ell \sup_{0 \le t \le T} \left( \| E^\ell (t) \|_\infty + \| B^\ell (t) \|_\infty \right) + |B_e| + E_e + g < & M_2,
\\ \inf_{\ell} \inf_{ t ,x_\parallel} \left( g - E_e  -  E^\ell_3(t,x_\parallel, 0 )  -  (\hat v \times B^\ell)_3(t,x_\parallel, 0 ) \right) > &  c_0.
\end{split}
\Ee
\end{lemma}

\begin{proof}
Let $\ell \ge 1$. By induction hypothesis we assume that
\Be \label{inductfEBin}
\begin{split}
\sup_{ 0 \le i \le \ell }  \sup_{0 \le t \le T} \left( \| \langle v \rangle^{4 + \delta} f^{\ell-i}(t) \|_{L^\infty(\bar \O \times \mathbb R^3)}   \right) < & M_1,
\\  \sup_{ 0 \le i \le \ell }  \sup_{0 \le t \le T} \left(  \| E^{\ell-i} (t) \|_\infty + \| B^{\ell-i} (t) \|_\infty \right) + |B_e| + E_e + g < & M_2.
\end{split}
\Ee

Denote the characteristics $(X^\ell, V^\ell)$ which solves
\Be\label{XV_ell}
\begin{split}
\frac{d}{ds}X^\ell(s;t,x,v) &= \hat V^\ell(s;t,x,v),\\
\frac{d}{ds} V^\ell(s;t,x,v) &= \mathfrak F^\ell (s, X^\ell(s;t,x,v),V^\ell(s;t,x,v) ).\end{split}
\Ee

First, we note that for any $ 0 s < t$, since
\[
V^{\ell } (s;  t,x,v) = v - \int_{s}^{ t} \mathfrak F^{\ell } (\tau, X^{\ell}(\tau), V^{\ell}(\tau)  ) d\tau,
\]
from \eqref{inductfEBin}, we have
\[
|v| - (t-s)M_2 \le | V^{\ell } (s;   t, x, v) | \le |v | +  (t- s)M_2.
\]
Thus
\Be \label{Vsvin}
\left( 1 + (t-s)(M_2+g) \right)^{-1} \langle v \rangle  \le \langle V^{\ell } (s;   t, x, v ) \rangle \le \left( 1 + (t-s)(M_2+g) \right) \langle v \rangle.
\Ee
			
From \eqref{fellseqin} we have for any $(t,x,v) \in (0,T) \times \bar \O \times \mathbb R^3$, 
\Be
\begin{split}
f^{\ell+1}(t,x,v) = & \mathbf 1_{ {\tb}^\ell \le 0} f^{\ell+1}( 0, X^\ell(0), V^\ell(0)) +  \mathbf 1_{{\tb}^\ell \ge 0 } g ({\tb}^\ell , X^\ell({\tb}^\ell; t,x,v), V^\ell({\tb}^\ell;t,x,v) ).
\end{split}
\Ee
And \eqref{Vsvin} gives
\Be \label{fnmiterate2in}
\begin{split}
& \langle v \rangle^{4 + \delta }  |f^{\ell+1}(t,x,v) | 
\\  \le &   \mathbf 1_{ {\tb}^\ell \le 0} (1+ T (M_2+g))^{4 + \delta} | \langle V^\ell(0) \rangle^{4 + \delta }  f^{\ell+1}( 0, X^\ell(0), V^\ell(0)) |  
\\ &  +   \mathbf 1_{ {\tb}^\ell \ge 0 }    (1+ T (M_2+g))^{4 + \delta}   \langle V^\ell( {\tb}^\ell) \rangle^{4 + \delta}   g ({\tb}^\ell , X^\ell({\tb}^\ell; t,x,v), V^\ell({\tb}^\ell;t,x,v) ).
\end{split}
\Ee
Thus, by choosing $T \ll 1 $ we have
\Be \label{fell1finalin}
\sup_{0 \le t \le T} \|  \langle v \rangle^{4 + \delta } f^{\ell+1}(t) \|_{L^\infty(\bar \O \times \mathbb R^3 ) }   < M_1.
\Ee
Now from \eqref{ElBlin}, using the same argument in the proof of Lemma \ref{EBlinflemma}, we have
\Be \label{BEellinftyestinflow1}
 \begin{split}
& \sup_{0 \le t \le T} \| E^{\ell+1} (t) \|_\infty + \sup_{0 \le t \le T} \| B^{\ell+1} (t) \|_\infty
\\  \le &  C( \| E_0 \|_{\infty } +  \| B_0 \|_{\infty }  )  + C T ( \| E_0 \|_{C^1 } +  \| B_0 \|_{C^1 }  )
\\ &+   C T \sup_{ 0 \le t \le T} \|  \langle v \rangle^{4+\delta} f^\ell (t) \|_\infty \left( 1  +  T  (    \sup_{ 0 \le t \le T} \left(  \| E^\ell(t)   \|_\infty +  \| B^\ell(t)   \|_\infty  \right) + g + |B_e|  ) \right)
\\ \le  & C( \| E_0 \|_{C^1 } +  \| B_0 \|_{C^1 }  ) +   C T M_1 \left( 1  +  T     (M_2+g + |B_e|) \right).
\end{split} 
\Ee
Letting $M_2  = (C+1) ( \| E_0 \|_{C^1 } +  \| B_0 \|_{C^1 }  )  + |B_e| + E_e + g $ and $T \ll 1 $, we get
\Be \label{EBlifinalin}
  \sup_{0 \le t \le T} \| E^{\ell+1} (t) \|_\infty + \sup_{0 \le t \le T} \| B^{\ell+1} (t) \|_\infty + |B_e| + E_e +  g< M_2. 
\Ee
Next, from \eqref{E0B0g} and the proof of Lemma \ref{EBlinflemma}, by letting $c_0 = \frac{c_1}{2}$ in \eqref{E0B0g}, we get
\[
\begin{split}
 \inf_{ t ,x_\parallel} \left( g - E_e  -  E^{\ell+1}_3(t,x_\parallel, 0 )  -  (\hat v \times B^{\ell+1})_3(t,x_\parallel, 0 ) \right) > &  2 c_0 -  C T M_1 \left( 1  +  T     (M_2+g + |B_e|) \right).
\end{split}
\]
By choosing $T \ll 1$ small enough, we have
\Be \label{BEellinftyestinflow4}
 \inf_{ t ,x_\parallel} \left( g - E_e  -  E^{\ell+1}_3(t,x_\parallel, 0 )  -  (\hat v \times B^{\ell+1})_3(t,x_\parallel, 0 ) \right) >   c_0.
\Ee
Combining with \eqref{fell1finalin} and \eqref{EBlifinalin}, we conclude \eqref{fellboundin} by induction. 

\end{proof}

Next, we consider the derivative of the sequences. Define $\alpha^{\ell} $ as
\Be \label{alphan}
\begin{split}
\alpha^\ell(t, x, v ) =   \sqrt{(x_3)^2+(\hat{v}_{3})^2 -2\left( E^\ell_3(t,x_\parallel, 0 ) + E_e + (\hat v \times B^\ell)_3(t,x_\parallel, 0 )  -  g \right) \frac{x_3}{\langle v \rangle}}.
\end{split}
\Ee
We have the following estimate.
\begin{lemma}
Suppose $f_0$ satisfies \eqref{f0bdd}, \eqref{inflowdata}, $E_0$, $B_0$ satisfy \eqref{E0B0bdd}, then there exits $M_3, M_4$ such that for $0 < T \ll 1 $, 
\Be \label{flElBldsqbdin}
\begin{split}
 &\sup_\ell \sup_{0 \le t \le T} \left( \| \langle v \rangle^{4 + \delta } \nabla_{x_\parallel } f^\ell(t) \|_\infty  + \| \langle v \rangle^{5 + \delta }  \alpha^{\ell-1} \p_{x_3 } f^\ell(t) \|_\infty + \|  \langle v \rangle^{4 + \delta }  \nabla_v f^\ell(t) \|_\infty   \right) 
 \\& + \sup_\ell \sup_{0 \le t \le T} \left( \| \langle v \rangle^{4 + \delta } \nabla_{x_\parallel } f^\ell(t) \|_{L^\infty(\gamma \setminus \gamma_0)}  + \| \langle v \rangle^{5 + \delta }  \alpha^{\ell-1} \p_{x_3 } f^\ell(t) \|_{L^\infty(\gamma \setminus \gamma_0)} + \|  \langle v \rangle^{4 + \delta }  \nabla_v f^\ell(t) \|_{L^\infty(\gamma \setminus \gamma_0)}   \right) <  M_3 ,
\\ & \sup_\ell \sup_{0 \le t \le T} \left( \| \p_t E^\ell(t) \|_\infty +  \| \p_t B^\ell(t) \|_\infty + \| \nabla_{x } E^\ell(t) \|_\infty +\| \nabla_{x } B^\ell(t) \|_\infty  \right) < M_4.
\end{split} 
\Ee
\end{lemma}
\begin{proof}
The proof is essentially the same as the proof of Proposition \ref{inflowprop}. 
From the uniform estimate \eqref{fellboundin}, and from the velocity lemma (Lemma \ref{vlemma}), we have for some $C>0$,
\Be
e^{-C|t-s| } \alpha^\ell(t,x,v) \le \alpha^\ell(s, X^\ell(s;t,x,v) , V^\ell(s;t,x,v) ) \le e^{C|t-s| } \alpha^\ell(t,x,v), \ \text{ for all } \ell.
\Ee
Therefore, following the same proof of Proposition \ref{inflowprop} and Lemma \ref{tracepf}, we get
\Be \label{}
\begin{split}
& \sup_{0 \le t \le T} \left( \| \langle v \rangle^{4 +\delta } \nabla_{x_\parallel} f^{\ell+1} (t)  \|_\infty  + \|  \langle v \rangle^{5 +\delta }  \alpha^\ell \p_{x_3} f ^{\ell+1} (t) \|_\infty  + \| \langle v \rangle^{5 +\delta } \nabla_{v} f^{\ell+1} (t)  \|_\infty \right)
\\ + & \sup_{0 \le t \le T} \left( \| \langle v \rangle^{4 + \delta } \nabla_{x_\parallel } f^\ell(t) \|_{L^\infty(\gamma \setminus \gamma_0)}  + \| \langle v \rangle^{5 + \delta }  \alpha^{\ell-1} \p_{x_3 } f^\ell(t) \|_{L^\infty(\gamma \setminus \gamma_0)} + \|  \langle v \rangle^{4 + \delta }  \nabla_v f^\ell(t) \|_{L^\infty(\gamma \setminus \gamma_0)}   \right) 
\\ \le & C_k \bigg( \left( \| \langle v \rangle^{4 +\delta }  \nabla_{x_\parallel}   f_0 \|_\infty +  \| \| \langle v \rangle^{5 +\delta }   \alpha \p_{x_3}   f_0 \|_\infty +  \| \| \langle v \rangle^{5 +\delta }  \nabla_{v}   f_0 \|_\infty  \right)  
\\  & +  \sup_{0 \le t \le T} \left( \| \langle v \rangle^{4 + \delta } \nabla_{x_\parallel } g (t) \|_{L^\infty(\gamma_-)}  + \| \langle v \rangle^{5 + \delta }  \nabla_v g(t) \|_{L^\infty(\gamma_-} + \|  \langle v \rangle^{5 + \delta }  \p_t g(t) \|_{L^\infty(\gamma_-)}   \right)  \bigg).
\end{split}
\Ee
Thus, by choosing $M_3 \gg 1$, we conclude
\Be \label{dfelluniin}
\begin{split}
& \sup_\ell \sup_{0 \le t \le T} \left( \| \langle v \rangle^{4 + \delta } \nabla_{x_\parallel } f^\ell(t) \|_\infty  + \| \langle v \rangle^{5 + \delta }  \alpha^{\ell-1} \p_{x_3 } f^\ell(t) \|_\infty + \| \langle v \rangle^{4 + \delta }  \nabla_v f^\ell(t) \|_\infty   \right)  \\& + \sup_\ell \sup_{0 \le t \le T} \left( \| \langle v \rangle^{4 + \delta } \nabla_{x_\parallel } f^\ell(t) \|_{L^\infty(\gamma \setminus \gamma_0)}  + \| \langle v \rangle^{5 + \delta }  \alpha^{\ell-1} \p_{x_3 } f^\ell(t) \|_{L^\infty(\gamma \setminus \gamma_0)} + \|  \langle v \rangle^{4 + \delta }  \nabla_v f^\ell(t) \|_{L^\infty(\gamma \setminus \gamma_0)}   \right)<  M_3.
\end{split}
\Ee
From this, we use the same argument to get \eqref{dxEBfinal} in the proof of Lemma \ref{EBW1inftylemma} and obtain
\[
\begin{split}
 \sup_{0 \le t \le T} & \left( \| \p_t E^{\ell+1}(t) \|_\infty +  \| \p_t B^{\ell+1}(t) \|_\infty + \| \nabla_{x } E^{\ell+1}(t) \|_\infty +\| \nabla_{x } B^{\ell+1}(t) \|_\infty   \right)
\\ \le & TC \sup_{1 \le i \le \ell}  \sup_{0 \le  t \le T}    \left( \| \p_t E^{i}(t) \|_\infty +  \| \p_t B^{i}(t) \|_\infty + \| \nabla_{x } E^{i}(t) \|_\infty +\| \nabla_{x } B^{i}(t) \|_\infty   \right)  
\\ & C   \left( \| E_0 \|_{C^2 } + \| B_0 \|_{C_2}  \right) +  C \sup_{0 \le t \le T} \left( \| \langle v \rangle ^{4 + \delta } \nabla_{x_\parallel } f^{\ell +1}(t) \|_\infty  + \| \langle v \rangle^{\ell  + \delta }  \alpha^{n} \p_{x_3 } f^{\ell +1}(t) \|_\infty   \right)  
\\ &  +  C \sup_{0 \le t \le T} \left( \| \langle v \rangle^{4 + \delta } \nabla_{x_\parallel } f^\ell(t) \|_{L^\infty(\gamma \setminus \gamma_0)}  + \| \langle v \rangle^{5 + \delta }  \alpha^{\ell-1} \p_{x_3 } f^\ell(t) \|_{L^\infty(\gamma \setminus \gamma_0)} \right) 
\\ & + C  \sup_{0 \le t \le T} \left( \| \langle v \rangle^{4  + \delta } f^{\ell+1}(t) \|_\infty  + \| E^{\ell +1} (t) \|_\infty + \| B^{\ell +1} (t) \|_\infty \right).
\end{split}
\]
From \eqref{fellboundin} and \eqref{dfelluniin}, this gives
\[
\begin{split}
\sup_{\ell} \sup_{0 \le t \le T} & \left( \| \p_t E^{\ell}(t) \|_\infty +  \| \p_t B^{\ell}(t) \|_\infty + \| \nabla_{x } E^{\ell}(t) \|_\infty +\| \nabla_{x } B^{\ell}(t) \|_\infty   \right)
\\ \le & TC \sup_{ \ell}  \sup_{0 \le  t \le T}    \left( \| \p_t E^{\ell}(t) \|_\infty +  \| \p_t B^{\ell}(t) \|_\infty + \| \nabla_{x } E^{\ell}(t) \|_\infty +\| \nabla_{x } B^{\ell}(t) \|_\infty   \right)  
\\ & + C \left( \| E_0 \|_{C^2 } + \| B_0 \|_{C_2}  \right)+ C(M_1 +M_2 +M_3).
\end{split}
\]
Therefore, by choosing $M_4 \gg 1$ and $T \ll 1$, we get
\Be
\sup_{\ell} \sup_{0 \le t \le T}  \left( \| \p_t E^{\ell}(t) \|_\infty +  \| \p_t B^{\ell}(t) \|_\infty + \| \nabla_{x } E^{\ell}(t) \|_\infty +\| \nabla_{x } B^{\ell}(t) \|_\infty   \right) < M_4.
\Ee
Together with \eqref{dfelluniin}, we conclude \eqref{flElBldsqbdin}.
\end{proof}

We have the following trace properties for $E^\ell$ and $B^\ell$:

\begin{lemma} \label{EBelltrace}
Suppose $E^\ell, B^\ell$ satisfies \eqref{flElBldsqbdin}. Then for any $\ell \ge 1$, $0 < t < T$,
\Be \label{EBpO}
E^\ell(t,\cdot ,0) \in L^\infty(\p \O ), \ B^\ell(t,\cdot,0) \in L^\infty(\p \O ).
\Ee
\end{lemma}
\begin{proof}
From \eqref{flElBldsqbdin} we have
\[
E^\ell(t,x) \in W^{1,\infty}((0,T) \times \O), \  B^\ell (t,x) \in W^{1,\infty}((0,T) \times \O),
\]
in particular, from the Morrey's inequality, $E(t)$, $B(t)$ are Lipschitz continuous on $\O$. Now pick any $x_\parallel \in \mathbb R^2$, and $0< x_3 < 1$, from the fundamental theorem of calculus, we have
\[
\begin{split}
E^\ell (t,x_\parallel, 0 ) = & E^\ell(t,x_\parallel, x_3) -\int_0^{x_3} \p_{3} E^\ell(t,x_\parallel, y ) dy,
\\ B^\ell(t,x_\parallel, 0 ) = & B^\ell(t,x_\parallel, x_3) -\int_0^{x_3} \p_{3} B^\ell(t,x_\parallel, y ) dy.
\end{split}
\]
Therefore 
\[
\begin{split}
& \| E^\ell(t,\cdot , 0 ) \|_{L^\infty(\p \O ) } \le \| E^\ell(t) \|_{L^\infty(\O ) } + x_3  \|   \p_{3} E^\ell(t ) \|_{L^\infty(\O ) } ,   
\\ & \| B^\ell(t,\cdot, 0 ) \|_{L^\infty(\p \O ) } \le \| B^\ell(t) \|_{L^\infty(\O ) } +  x_3 \|   \p_{3} B^\ell(t ) \|_{L^\infty(\O ) } ,
\end{split}
\]
for any $x_3 > 0 $. Thus
\[
 \| E^\ell(t,\cdot , 0 ) \|_{L^\infty(\p \O ) } \le \| E^\ell(t) \|_{L^\infty(\O ) } < \infty, \ \| B^\ell(t,\cdot, 0 ) \|_{L^\infty(\p \O ) } \le \| B^\ell(t) \|_{L^\infty(\O ) } < \infty, 
\]
and we conclude \eqref{EBpO}.
\end{proof}

Next, we prove the strong convergence of the sequence $f^\ell$.

\begin{lemma} \label{fEBsollemmain}
Suppose $f_0, g$ satisfies \eqref{f0bdd}, \eqref{inflowdata}, $E_0$, $B_0$ satisfy \eqref{E0B0g}, \eqref{E0B0bdd}. There exists functions $(f,E,B)$ with $  \langle v \rangle^{4 +\delta }  f(t,x,v) \in L^\infty( (0, T) ; L^\infty( \bar \O \times \mathbb R^3 ) )  $, and $(E,B) \in L^\infty((0,T) ; L^\infty( \O) \cap  L^\infty( \p \O ) )$, such that as $\ell \to \infty$, 
\Be \label{EnBnconvergein}
 \sup_{0 \le t \le T} \left(  \|  E^\ell (t) - E(t) \|_{L^\infty( \O)} +  \|  E^\ell (t) - E(t) \|_{L^\infty( \p \O)} +  \|  B^\ell (t) - B(t) \|_{L^\infty( \O)} +  \|  B^\ell (t) - B(t) \|_{L^\infty( \p \O)} \right)   \to 0, 
\Ee
and
\Be \label{fnconvergein}
 \sup_{0 \le t \le T}  \| \langle v \rangle^{4 +\ell } f^\ell(t) -  \langle v \rangle^{4 +\delta }  f (t) \|_{L^\infty(\bar \O \times \mathbb R^3)}  \to 0.
\Ee
Moreover, $(f,E,B)$ is a (weak) solution of the system \eqref{VMfrakF1}-\eqref{rhoJ1}, and \eqref{inflow}.
\end{lemma}
\begin{proof}
Let $m > n  \ge 1$. Note that $f^m - f^n $ satisfies $(f^m - f^n ) |_{t = 0 } = 0 $ and 
\[
(f^m- f^n )|_{\gamma_-} = 0.
\] 

The equation for $f^m - f^n $ is
\[
\p_t(f^m- f^n ) + \hat v \cdot \nabla_x (f^m - f^n ) + \mathfrak F^{m-1} \cdot \nabla_v (f^{m} - f^n ) = - ( \mathfrak F^{m-1} - \mathfrak F^{n-1} ) \cdot \nabla_v f^n.
\]
Thus, for any $(t,x,v) \in (0,T) \times \bar \O \times \mathbb R^3$,  using \eqref{Vsvin}, we get
\[
\begin{split}
& |  \langle v \rangle^{4 + \delta } (f^m - f^n)(t,x,v) | 
\\  \le & C_1  \int_{\max \{ {\tb}^{m-1} , 0 \} }^t |    \langle V^{m-1}(s) \rangle^{4 + \delta }  ( \mathfrak F^{m-1} - \mathfrak F^{n-1} ) \cdot \nabla_v f^n  )(s, X^{m-1}(s), V^{m-1}(s) ) | ds 
\end{split}
\]
where $C_{1} =  (1+ T (M_2+g) )^{4 + \delta}$. Then from \eqref{flElBldsqbdin},
\Be \label{fellnmiterate1in}
\begin{split}
\| \langle  v \rangle^{4 + \delta} f^m(t) -  \langle  v \rangle^{4 + \delta} f^n)(t) \|_\infty  \le & C_1 \left( \sup_{\ell}  \sup_{0 \le s \le t } \| \langle  v \rangle^{4 + \delta} \nabla_v f^\ell(s) \|_\infty \right) \int_0^t   \|   \mathfrak F^{m-1} (s) - \mathfrak F^{n-1}(s)  \|_\infty ds
\\ \le  & C_2  \int_0^t  \|   \mathfrak F^{m-1} (s) - \mathfrak F^{n-1}(s)  \|_\infty ds,
\end{split}
\Ee
where $C_2 = C_1M_3$.

Now, 
from \eqref{fellseqin} and using the same argument as Lemma \ref{EBlinflemma} with \eqref{fellnmiterate1in}, we have
\Be \label{iterate2in}
\begin{split}
 \|   \mathfrak F^{m-1} (s) - \mathfrak F^{n-1}(s)  \|_\infty   \le &   \| E^{n-1}(s) - E^{m-1}(s) \|_\infty + \| B^{n-1}(s) - B^{m-1}(s)  \|_\infty
\\ \le & C \left(  \sup_{0 \le s' \le s } \| \langle v \rangle^{4 + \delta } (f^{n-1} - f^{m-1} )(s' ) \|_\infty +  \int_0^s \|   \mathfrak F^{m-2} (s') - \mathfrak F^{n-2 }(s')  \|_\infty ds'  \right)
\\ \le & C \int_0^s   \|   \mathfrak F^{m-2} (s') - \mathfrak F^{n-2}(s')  \|_\infty ds'
\\ \le & C \int_0^s \left(  \| E^{m-2}(s') - E^{n-2}(s') \|_\infty + \| B^{m-2}(s') - B^{n-2}(s')  \|_\infty \right) ds'.
\end{split}
\Ee
Iteration of \eqref{iterate2in} and using \eqref{fellboundin} yields
\[
\begin{split}
&  \| E^{m}(t) - E^{n}(t) \|_\infty + \| B^{m}(t) - B^{n}(t)  \|_\infty
\\ \le & C^2 \int_0^t \int_0^s   \left(  \| E^{m-2}(s') - E^{n-2}(s') \|_\infty + \| B^{m-2}(s') - B^{n-2}(s')  \|_\infty \right) ds' ds
 \\   =  & C^2 \int_0^t \tau    \left(  \| E^{m-2}(\tau) - E^{n-2}(\tau) \|_\infty + \| B^{m-2}(\tau) - B^{n-2}(\tau)  \|_\infty \right) d\tau
 \\ \le & C^l \int_0^t \frac{ \tau^{l-1}}{ (l-1)!} \left(  \| E^{m-i}(\tau) - E^{n-k}(\tau) \|_\infty + \| B^{m-k}(\tau) - B^{n-k}(\tau)  \|_\infty \right) d\tau
 \\ \le  &M_2 \frac{C^l t^l}{l!}.
\end{split}
\]
%
Thus the sequences $E^\ell$, $B^\ell$ are Cauchy in $L^\infty((0,T) \times \O ) $, moreover, from Lemma \ref{EBelltrace}, $E^\ell, B^\ell \in L^\infty([0,T] \times \p \O )$. Therefore, there exists functions $E,B \in L^\infty((0,T) ; L^\infty( \O) \cap  L^\infty( \p \O ) )$, such that 
\Be \label{EnBncovin}
E^\ell \to E, B^\ell \to B \text{ in } L^\infty((0,T) \times  \O ) \cap  L^\infty((0,T) \times \p \O ) .
\Ee
This proves \eqref{EnBnconvergein}. Also, from \eqref{fellnmiterate1in}, \eqref{iterate2in},
\[
\| \langle v\rangle^{4 + \delta} f^m(t) -  \langle v \rangle^{4 + \delta} f^n)(t) \|_{L^\infty((0,T) \times \bar \O ) }   \le M_2 \frac{C^{l-1} t^{l-1}}{(l-2)!},
\]
therefore we get \eqref{fnconvergein}.

Now, take any $\phi(t,x,v) \in C_c^\infty( [0,T) \times \bar \O \times \mathbb R^3$ with $\text{supp } \phi   \subset \{ [0, T) \times \bar \O \times \mathbb R^3 \} \setminus \{ (0 \times \gamma ) \cup (0,T) \times \gamma_0 \} $, from \eqref{fellseqin}, we have
\Be \label{weakfellVMin}
\begin{split}
& \int_{\O \times \mathbb R^3 } f_0 \phi (0) dv dt +  \int_0^T \int_{\O \times \mathbb R^3}  f^\ell \left(  \p_t \phi + \hat v \cdot \nabla_x \phi +   \mathfrak F^{\ell-1}   \cdot \nabla_v \phi \right) dv dx dt
\\ = & \int_0^T \int_{\gamma_+} \phi f^\ell \hat v_3 dv dS_x + \int_0^T \int_{\gamma_-} \phi g \hat v_3 dv dS_x.
\end{split}
\Ee
Because of the strong convergence \eqref{EnBnconvergein}, \eqref{fnconvergein}, we have that as $\ell \to \infty$, each term in \eqref{weakfellVMin} goes to the corresponding terms with $f^\ell$ replaced by $f$ and $\mathfrak F^\ell$ replaced by $\mathfrak F$. Therefore we conclude that $(f,E,B)$ satisfy \eqref{weakf}.

Next, from Proposition \ref{Eiform} and Proposition \ref{Biform}, we have that $E^\ell$ and $B^\ell$ are (weak) solutions to the wave equations with the initial data, boundary condition and forcing term in \eqref{E12solin}-\eqref{B12solin}, with $\rho, J$ changed to $\rho^\ell ,J^\ell$. Then from \eqref{flElBldsqbdin} and Lemma \ref{wavetoMax}, we have 
\Be \label{Maxwellell}
\begin{split}
\p_t E^\ell & = \nabla_x \times B^\ell - 4 \pi J^\ell, \, \nabla_x \cdot E^\ell = 4\pi \rho^\ell,
\\ \p_t B^\ell & = - \nabla_x \times E^\ell, \, \nabla_x \cdot B^\ell = 0,
\end{split}
\Ee
with
\Be \label{EellBellDin}
E^\ell_1 = E^\ell_2 = 0, B^\ell_3 = 0, \ \text{ on } \p \O.
\Ee
Clearly, \eqref{nablaEBweak} is satisfied. Now, for any test functions $\Psi(t,x)  \in C_c^\infty([0,T) \times \bar \O ; \mathbb R^3 ), \    \Phi(t,x) \in C_c^\infty([0,T) \times \O ; \mathbb R^3 )$, from \eqref{Maxwellell} and \eqref{EellBellDin}, we have
\Be \label{EBellweak1}
\int_0^T \int_\O E^\ell \cdot \p_t \Psi dx dt - \int_\O \Psi(0,x) \cdot E_0 dx  = - \int_0^T \int_\O (\nabla_x \times \Psi) \cdot B^\ell dx dt + 4\pi \int_0^T \int_\O \Psi \cdot J^\ell dx dt,
\Ee
and
\Be \label{EBellweak2}
\int_0^T \int_\O B^\ell \cdot \p_t \Phi dx dt + \int_\O \Phi(0,x) \cdot B_0 dx = \int_0^T \int_\O (\nabla_x \times \Phi) \cdot E^\ell dx dt,
\Ee

Then from the strong convergence \eqref{fnconvergein}, \eqref{EnBnconvergein}, we can pass $\ell \to \infty$ and deduce that each term in \eqref{EBellweak1} and \eqref{EBellweak2} converges to the corresponding term with $E^\ell$, $B^\ell$, $J^\ell$ replace by $E$, $B$, and $J$ respectively. Therefore, $(f,E,B)$ satisfy \eqref{Maxweak1} and \eqref{Maxweak2}.
So we conclude that $(f,E,B)$ is a (weak) solution of the RVM system \eqref{VMfrakF1}-\eqref{rhoJ1} with inflow BC \eqref{inflow}.
\end{proof}

In the next lemma, we consider the regularity of the solution.

\begin{lemma} \label{fEBregin}
Let $\alpha(t,x,v)$ be defined as in \eqref{alphadef}. The solution $(f,E,B)$ obtained in Lemma \ref{fEBsollemma} satisfies
\Be \label{pfbdlimitin}
 \| \langle v \rangle^{4 + \delta } \nabla_{x_\parallel } f(t) \|_\infty  + \| \langle v \rangle^{5 + \delta }  \alpha^{} \p_{x_3 } f(t) \|_\infty + \|  \langle v \rangle^{4 + \delta }  \nabla_v f(t) \|_\infty < \infty,
\Ee
and
\Be \label{pEBbdlimitin}
\| \p_t E(t) \|_\infty +  \| \p_t B(t) \|_\infty + \| \nabla_{x } E(t) \|_\infty +\| \nabla_{x } B(t) \|_\infty   < \infty.
\Ee
\end{lemma}
\begin{proof}
From the $L^\infty$ strong convergence \eqref{EnBnconvergein}, and the uniform-in-$\ell$ bound \eqref{flElBldsqbdin}, we can pass the limit up to subsequence if necessary and get the weak$-*$ convergence
\Be \label{dEnBncovin}
\p_t E^\ell  \overset{\ast}{\rightharpoonup} \p_t E , \ \nabla_x E^\ell  \overset{\ast}{\rightharpoonup} \nabla_x E, \  \p_t B^\ell  \overset{\ast}{\rightharpoonup} \p_t B, \  \nabla_x B^\ell  \overset{\ast}{\rightharpoonup} \nabla_x B \text{ in } L^\infty((0,T) \times \O ),
\Ee
and
\Be \label{dfnconvergein}
\langle v \rangle ^{4 + \delta} \nabla_{x_\parallel } f^\ell   \overset{\ast}{\rightharpoonup}  \langle v \rangle^{4 + \delta}\nabla_{x_\parallel } f, \  \langle v \rangle ^{4 + \delta} \nabla_{v } f^\ell   \overset{\ast}{\rightharpoonup}  \langle v \rangle^{4 + \delta}\nabla_{v } f \text{ in } L^\infty((0,T) \times \O \times \mathbb R^3 ).
\Ee
We also claim
\Be \label{ap3fellcovin}
\langle v \rangle ^{5 + \delta}  \alpha^{\ell-1} \p_{x_3}  f^\ell   \overset{\ast}{\rightharpoonup}  \langle v \rangle^{5 + \delta} \alpha \p_{x_3} f \text{ in } L^\infty((0,T) \times \O \times \mathbb R^3 ).
\Ee
For any test function $\phi \in C^\infty_c((0,T) \times \O \times \mathbb R^3 ) $, we have
\begin{eqnarray}
&  & \notag \int_0^t \iint_{\O \times \mathbb R^3} ( \langle v \rangle^{5 + \delta} \alpha^{\ell-1} \p_{x_3}  f^\ell  - \langle v \rangle^{5 + \delta} \alpha \p_{x_3}  f ) \phi dv dx dt  
\\ \label{alphanpfn1in} = & &-   \int_0^t \iint_{\O \times \mathbb R^3}  ( \langle v \rangle^{5 + \delta} \alpha^{\ell-1}  f^\ell  - \langle v \rangle^{5 + \delta} \alpha f ) \p_{x_3} \phi dv dx dt
\\ \label{alphanpfn2in} & & -    \int_0^t \iint_{\O \times \mathbb R^3}  ( \langle v \rangle^{5 + \delta} \p_{x_3} \alpha^{\ell-1}  f^\ell  - \langle v \rangle^{5 + \delta} \p_{x_3} \alpha^{\ell-1} f )  \phi dv dx dt
\\ \label{alphanpfn3in} &  &-    \int_0^t \iint_{\O \times \mathbb R^3}  ( \langle v \rangle^{5 + \delta} \p_{x_3} \alpha^{\ell-1}  f  - \langle v \rangle^{5 + \delta} \p_{x_3} \alpha f )  \phi dv dx dt.
\end{eqnarray}
From \eqref{alphan} and \eqref{EnBnconvergein} we have 
\Be \label{alphancovin}
\| \alpha^\ell - \alpha \|_{L^\infty((0,T) \times \O \times \mathbb R^3)} \to 0 \text{ as } \ell \to \infty.
\Ee
Thus, together with \eqref{fnconvergein}, we have $\eqref{alphanpfn1in}  \to 0$ as $n \to \infty$. Next, note that
\[
 \p_{x_3} \alpha^\ell  =  \frac{x_3 -  \left( E^\ell_3(t,x_\parallel, 0 ) + E_e + (\hat v \times B^\ell)_3(t,x_\parallel, 0 )  -  g \right) \frac{1}{\langle v \rangle} }{\alpha^\ell}.
\]
For $(t,x,v) \in$ supp $\phi$, $x_3 > 0$ for some $c > 0$. Thus, $\frac{1}{\alpha^\ell(t,x,v) } < \frac{1}{c}$, so 
\[
|  \p_{x_3} \alpha^\ell  \mathbf 1_{\text{supp}( \phi ) }(t,x,v)  | < \frac{ C M_2}{c}
\]
From \eqref{fnconvergein}, this yields
\[
| \eqref{alphanpfn2in} | \le  \frac{ C M_2}{c}  \int_0^t \iint_{\O \times \mathbb R^3} |  \langle v \rangle^{5 + \delta} (f^\ell - f ) \phi  |  dv dx dt \to 0.
\]
For \eqref{alphanpfn3in}, since
\[
\begin{split}
& \p_{x_3} \alpha^\ell - \p_{x_3} \alpha 
\\  = & \frac{x_3 -  \left( E^\ell_3(t,x_\parallel, 0 ) + E_e + (\hat v \times B^\ell)_3(t,x_\parallel, 0 )  -  g \right) \frac{1}{\langle v \rangle} }{\alpha^\ell} -   \frac{x_3 -  \left( E_3(t,x_\parallel, 0 ) + E_e + (\hat v \times B)_3(t,x_\parallel, 0 )  -  g \right) \frac{1}{\langle v \rangle} }{\alpha}
\\  = &  \frac{ -  \left( (E^\ell-E)_3(t,x_\parallel, 0 ) + (\hat v \times( B^\ell - B))_3(t,x_\parallel, 0 )   \right) \frac{1}{\langle v \rangle} }{\alpha^\ell} +   \\ & + \left( x_3 -  \left( E_3(t,x_\parallel, 0 ) + E_e + (\hat v \times B)_3(t,x_\parallel, 0 )  -  g \right) \right) \frac{1}{\langle v \rangle}\frac{ \alpha - \alpha^\ell }{\alpha^\ell \alpha}.
\end{split}
\]
Again, for $(t,x,v) \in$ supp $\phi$, $x_3 > 0$ for some $c > 0$. Thus $\frac{1}{\alpha^\ell(t,x,v)} < \frac{1}{c}$,  $\frac{1}{\alpha(t,x,v)} < \frac{1}{c}$. So from \eqref{EnBnconvergein}, \eqref{alphancovin}, we have
\Be \label{px3alphacov3}
\begin{split}
& \|  (\p_{x_3}  \alpha^\ell - \p_{x_3} \alpha ) \mathbf 1_{\text{supp} (\phi)  }(t,x,v) \|_{L^\infty((0,T) \times \O \times \mathbb R^3)} 
\\ & \le \frac{1}{c} \sup_{0 \le t \le T} \left(  \| E(t) - E^\ell(t) \|_\infty + \| B(t) - B^\ell(t) \|_\infty \right) + \frac{CM_2}{c^2}  \| \alpha^\ell - \alpha \|_{L^\infty((0,T) \times \O \times \mathbb R^3)} \to 0 \text{ as } \ell \to \infty.
\end{split}
\Ee
Thus, we have $\eqref{alphanpfn3in} \to 0$, and this gives \eqref{ap3fellcovin}.

Therefore, from using the weak lower semi-continuity of the weak-$*$ convergence \eqref{dEnBncovin}, \eqref{dfnconvergein}, \eqref{ap3fellcovin}, and the uniform-in-$\ell$ bound \eqref{flElBldsqbdin}, we conclude \eqref{pfbdlimitin}, \eqref{pEBbdlimitin}.

\end{proof}

Next, we prove the uniqueness of the solutions of the RVM system \eqref{VMfrakF1}--\eqref{rhoJ1}, \eqref{inflow}.

\begin{lemma} \label{VMuniqlemmain}
Suppose  $(f,E_f, B_f)$ and $(g, E_g, B_g)$ are solutions to the VM system \eqref{VMfrakF1}--\eqref{rhoJ1}, \eqref{inflow} with $f(0) = g(0)$, $E_f(0) = E_g(0)$, $B_f(0) = B_g(0)$, and that 
\[
E_f, B_f, E_g, B_g \in W^{1,\infty}((0,T) \times \O ), \ \nabla_x \rho_{f}, \nabla_x J_f,  \p_t J_f , \nabla_x \rho_{g}, \nabla_x J_g,  \p_t J_g \in  L^\infty((0,T); L_{\text{loc}}^p(\O)) \text{ for some } p>1.
\]
And
\Be \label{dvfgbdin}
\sup_{0 < t < T} \| \langle v \rangle^{5+ \delta} \nabla_v f(t) \|_\infty <\infty, \sup_{0 < t < T} \| \langle v \rangle^{5+ \delta} \nabla_v g(t) \|_\infty <\infty.
\Ee
Then $f = g, E_f = E_g, B_f = B_g$.
\end{lemma}
\begin{proof}
The difference function $f-g $ satisfies
\Be \label{fminusgeqin}
\begin{split}
(\p_t + \hat v \cdot \nabla_x + \mathfrak F_f \cdot \nabla_v)(f-g) = (\mathfrak F_g - \mathfrak F_f ) \cdot \nabla_v g, 
\\ (f-g)(0) = 0, \, (f- g )|_{\gamma_-  } =  0,
\end{split} 
\Ee
where
\[
\mathfrak F_f = E_f + E_{\text{ext}} + \hat v \times ( B_f + B_{\text{ext}}) - g \mathbf e_3 , \, \mathfrak F_g = E_g + E_{\text{ext}} + \hat v \times ( B_g + B_{\text{ext}}) - g \mathbf e_3, 
\]
so
\Be \label{mathfrakFfgin}
\mathfrak F_g - \mathfrak F_f = E_f - E_g + \hat v \times (B_f - B_g ).
\Ee
From Lemma \ref{Maxtowave} we have $E_{f,1} - E_{g,1} , E_{f,2} - E_{g,2}, B_{f,3} - B_{g,3}$ solve the wave equation with the Dirichlet boundary condition \eqref{waveD} in the sense of \eqref{waveD_weak} with \begin{align}
u_0 = 0, \  u_1 = 0 , \ G = -4\pi \p_{x_i} (\rho_f - \rho_g) - 4 \pi \p_t (J_{f,i} - J_{g,i} ), \ g = 0 , \ \ \text{for} \  E_{f,i} - E_{g,i},  i =1,2, \label{E12solin} \\
 u_0 = 0, \ u_1 = 0, \   G =  4 \pi (\nabla_x \times (J_f -J_g) )_3, \  g = 0, \ \ \text{for} \ B_{f,3}- B_{g,3},  \label{B3solin}
 \end{align} 
respectively. And $E_{f,3} - E_{g,3}, B_{f,1} - B_{g,1}, B_{f,2}- B_{g,2}$ solve the wave equation with the Neumann boundary condition \eqref{waveNeu}  in the sense of \eqref{waveinner} \text{ with }
\begin{align}
u_0 = 0, \  u_1 = 0 , \ G = -4\pi \p_{x_3} ( \rho_f - \rho_g) - 4 \pi \p_t (J_{f,3} - J_{g,3} ) , \ g = - 4\pi (\rho_f - \rho_g), \ \ \text{for} \  E_{f,3} - E_{g,3}, \label{E3solin} \\
 u_0 = 0, \ u_1 = 0, \   G =  4 \pi (\nabla_x \times (J_f - J_g) )_i, \  g = (-1)^{i+1} 4 \pi (J_{f,{\underline i}} - J_{g, \underline i } ), \ \ \text{for} \ B_{f,i} - B_{j,i}, \ i=1,2, \label{B12solin}
 \end{align} 
respectively. Therefore, from Lemma \ref{wavesol} and Lemma \ref{wavesolD}, we know that $E_f - E_g$  and $B_f - B_g$ would have the form of
\Be \label{EBdiffformin}
\begin{split}
& E_f - E_g = \eqref{Eesttat0pos} + \dots + \eqref{Eest3bdrycontri}, \  B_f -B_g = \eqref{Besttat0pos} + \dots + \eqref{Bestbdrycontri},
\\ & \text{ with } E_0, B_0 \text{ changes to } 0, \text{ and } f \text{ changes to } f -g.
\end{split}
\Ee 

Now consider the characteristics
\[
\begin{split}
\dot X_f(s;t,x,v) = & \hat V_f(s;t,x,v) ,
\\ \dot V_f(s;t,x,v) = & \mathfrak F_f(s, X_f(s;t,x,v), V_f(s;t,x,v) ) .
\end{split}
\]
Then from \eqref{fminusgeqin}, same as \eqref{fnmiterate2in}, we obtain
\Be \label{fnmiterate2final}
\begin{split}
& |  \langle v \rangle^{4 + \delta } (f - g)(t,x,v) | 
\\  \le &   C_{1}    \int_{0 }^{t }    |   \langle V_f(s) \rangle^{4 + \delta }  ( \mathfrak F_g - \mathfrak F_g ) \cdot \nabla_v f^{}  )(s, X_f^{}(s), V_f^{}(s) ) | ds 
\end{split}
\Ee
So using  \eqref{flElBldsqbdin}, we have
\Be \label{fgdiffrepin}
\sup_{ 0 \le s \le t } \| \langle v \rangle^{4 + \delta} (f-g)(s) \|_\infty \le C   \int^t_0  \|  (\mathfrak F_g - \mathfrak F_f )(s) \|_\infty \|\langle v \rangle^{4 + \delta} \nabla_v g (s) \|_\infty  ds.
\Ee
Now, from \eqref{EBdiffformin} and the estimate in Lemma \ref{EBlinflemma}, we have
\Be \label{FgFfdiffin}
\begin{split}
\|  (\mathfrak F_g - \mathfrak F_f )(s) \|_\infty \le &  \| (E_f - E_g )(s) \|_\infty + \| (B_f - B_g )(s) \|_\infty
\\ \le & C \sup_{0 \le s' \le s } \| \langle v \rangle^{4 + \delta} (f-g )(s' ) \|_\infty,
\end{split}
\Ee
and from the assumption \eqref{dvfgbdin}, $\sup_{0 \le s \le t }   \| \langle v \rangle^{4+\delta} \nabla_v g (s) \|_\infty < C$. Therefore from \eqref{fgdiffrepin} and \eqref{FgFfdiffin}, we have
\Be
\sup_{0 \le s \le t } \|  \langle v \rangle^{4 +\delta} (f-g)(s) \|_\infty \le C' \int^t_{0 } \sup_{0 \le s' \le s } \|  \langle v \rangle^{4 +\delta}(f-g )(s' ) \|_\infty  ds.
\Ee
Therefore from Gronwall
\[
\sup_{0 \le s' \le t } \|  \langle v \rangle^{4 +\delta} (f-g)(s') \|_\infty \le   e^{C't}  \|  \langle v \rangle^{4 +\delta} (f-g)(0) \|_\infty = 0.
\]
Therefore we conclude that the solutions to  \eqref{VMfrakF1}--\eqref{rhoJ1}, \eqref{inflow}, is unique.
\end{proof}

\begin{proof}[proof of Theorem \ref{main1}]
Using the sequence $f^\ell, E^\ell , B^\ell$ constructed in \eqref{fellseqin}, \eqref{ElBlin}, we have from Lemma \ref{fEBsollemmain} that the limit $(f,E,B)$ is a solution to the VM system \eqref{VMfrakF1}--\eqref{rhoJ1}, \eqref{inflow}. This proves the existence. From Lemma \ref{fEBregin}, we have the regularity estimate \eqref{inflowfreg}, \eqref{inflowEBreg}. And from Lemma \ref{VMuniqlemmain}, we conclude the uniqueness.
\end{proof}

\section{Diffuse BC} \label{DiffuseSec}


In \eqref{diffuseBC}, we denote $ \mu = \frac{1}{ (2 \pi )^{3/2} } e^{-\frac{|v|^2}{2} } $. And let the constant $c_{\mu}$ be such that $c_{\mu} \int_{v_3 > 0 } \hat v_{3} \mu (v) dv  = 1$.

We first prove an a priori estimate for diffuse BC.
\begin{proposition} \label{diffuseprop}
Let $(f,E,B)$ be a solution of \eqref{VMfrakF1}--\eqref{rhoJ1}, \eqref{diffuseBC}. Suppose the fields satisfies \eqref{signcondition},
and
\Be \label{nablaEBass}
\sup_{0 \le t \le T}  \left(\| \nabla_{x} E(t)  \|_\infty + \| \nabla_{x} B(t)  \|_\infty \right) < \infty.
\Ee
Assume that for $\delta > 0$, $ \langle v \rangle^{4 + \delta}   \nabla_{x_\parallel} f,  \langle v \rangle^{5 + \delta}  \alpha \p_{x_3} f, \langle v \rangle^{5 + \delta}   \nabla_{v} f \in L^\infty((0,T) \times \O \times \mathbb R^3 ) $, 
then for $0 < T \ll 1$, there exists a $C> 0$ such that 
\Be \label{diffusedfbd}
\begin{split}
& \sup_{0 \le t \le T} \left( \| \langle v \rangle^{4 + \delta}   \nabla_{x_\parallel} f(t) \|_\infty  + \| \langle v \rangle^{5+\delta}  \alpha \p_{x_3} f(t) \|_\infty  +  \| \langle v \rangle^{5+\delta}   \nabla_{v} f(t) \|_\infty \right) 
\\ &+ \sup_{0 \le t \le T} \left( \|  \langle v \rangle^{4 + \delta}  \nabla_{x_\parallel} f(t) \|_{L^\infty(\gamma \setminus \gamma_0)}  + \|  \langle v \rangle^{5 + \delta}  \alpha \p_{x_3} f(t) \|_{L^\infty(\gamma \setminus \gamma_0)}  +  \|  \langle v \rangle^{5 + \delta}   \nabla_{v} f(t) \|_{L^\infty(\gamma \setminus \gamma_0)} \right) 
\\ <  & C \left(  \| \langle v \rangle ^{5 + \delta }   \nabla_{x_\parallel} f_0 \|_\infty +  \| \langle v \rangle^{5 + \delta }   \alpha \p_{x_3} f_0 \|_\infty + \| \langle v \rangle^{5 + \delta }    \nabla_v f_0 \|_\infty \right).
\end{split}
\Ee
\end{proposition}

\begin{proof}
For any $(t,x,v) \in (0,T) \times \O \times \mathbb R^3$, from \eqref{nablaxfest} and \eqref{nablavfest}, we have
\Be
\begin{split}
& | \p_{x_i } f(t,x,v) |
\\ \lesssim &    \mathbf 1_{\tb > t }  \{ |  \nabla_x  f_0( X(0), V(0)) | |  \p_{x_i}  X(0)  | +  |  \nabla_v  f_0(X(0), V(0))| |  \p_{x_i} V(0)  | \}
\\ & +  \mathbf 1_{\tb <  t } \left( |  \p_t f(t-\tb, \xb, \vb ) | \frac{\delta_{i3} + \tb }{ \hat {\vb}_{,3} } +  |\nabla_{x_\parallel } f(t-\tb, \xb, \vb ) |  \frac{\delta_{i3} + \tb }{ \hat {\vb}_{,3} }  +   |\nabla_{v } f(t-\tb, \xb, \vb ) |  \left( \frac{\delta_{i3} + \tb }{ \hat {\vb}_{,3} }  + \langle v \rangle \right)  \right),
\end{split}
\Ee
and
\Be
\begin{split}
& | \p_{v_i } f(t,x,v) |
\\ \lesssim &    \mathbf 1_{\tb > t }  \{ |  \nabla_x  f_0( X(0), V(0)) | |  \p_{v_i}  X(0)  | +  |  \nabla_v  f_0(X(0), V(0))| |  \p_{v_i} V(0)  | \}
\\ & +  \mathbf 1_{\tb <  t } \left( |  \p_t f(t-\tb, \xb, \vb ) | \frac{ \tb }{ \hat {\vb}_{,3} } +  |\nabla_{x_\parallel } f(t-\tb, \xb, \vb ) |  \frac{ \tb }{ \hat {\vb}_{,3} }  +   |\nabla_{v } f(t-\tb, \xb, \vb ) |  \left( \frac{ \tb }{ \hat {\vb}_{,3} }  + \langle v \rangle \right)  \right)
\end{split}
\Ee
Now, using the boundary condition \eqref{diffuseBC} and equation \eqref{VMfrakF1}, we have
\[
\begin{split}
\nabla_{x_\parallel } f(t-\tb, \xb, \vb ) = & c_\mu \mu (\vb)   \int_{u_3 < 0 }  - \nabla_{x_\parallel }  f(t-\tb, \xb ,  u ) \hat u _3 du,
\\ \nabla_{v } f(t-\tb, \xb, \vb )  = &  c_\mu \vb \mu (\vb)   \int_{u_3 < 0 }   f(t-\tb, \xb ,  u ) \hat u _3 du,
\\ \p_t f(t-\tb, \xb, \vb )  =  & c_\mu \mu(\vb) \int_{u_3 < 0 }  - \p_t  f(t -\tb,\xb ,  u ) \hat u _3 du
 \\ = &  c_\mu \mu(\vb) \int_{u_3 < 0 }  \left( \hat u \cdot \nabla_x  f(t -\tb,\xb ,  u ) + \mathfrak F \cdot \nabla_v f (t-\tb, \xb, u ) \right)  \hat u _3  du 
  \\ = &  c_\mu \mu(\vb) \int_{u_3 < 0 }   \hat u \cdot \nabla_x  f(t -\tb,\xb ,  u ) \hat u_3 +  f (t-\tb, \xb, u )  \mathfrak F \cdot      \nabla_u(  \hat  u _3 )  du .
\end{split}
\]
Therefore, for $i = 1,2$, 
\Be \label{dxparallelfest1}
\begin{split}
& | \p_{x_i } f(t,x,v)  | 
\\ \lesssim &    \mathbf 1_{\tb > t }  \{ |  \nabla_x  f_0( X(0), V(0)) | |  \p_{x_i}  X(0)  | +  |  \nabla_v  f_0(X(0), V(0))| |  \p_{x_i} V(0)  | \}
\\ & +  \mathbf 1_{\tb <  t } \left(  c_\mu \mu(\vb)  \langle v \rangle^2  \int_{u_3 < 0 } \left( | \nabla_x f (t-\tb, \xb, u ) | \hat u_3   + f(t-\tb, \xb, u ) \right) du   \right)
  \\ \lesssim &  \mathbf 1_{\tb > t }  \{ |  \nabla_{x_\parallel}   f_0( X(0), V(0)) | |\p_{x_i } X_\parallel (0 ) | +  |  \p_{x_3}   f_0( X(0), V(0)) | |\nabla_{x_\parallel} X_3 (0 ) |  +  |  \nabla_v  f_0(X(0), V(0))| |\p_{x_i } V(0 ) | \}
\\ & +  \mathbf 1_{\tb <  t } \left(  c_\mu \mu(\vb)    \langle v \rangle^2  \int_{u_3 < 0 } \left( | \nabla_x f (t-\tb, \xb, u ) | \hat u_3   + f(t-\tb, \xb, u ) \right) du   \right)  ,
\end{split}
\Ee
where we've used \eqref{tbbdvb}. And for $i =3$, we have
\Be \label{dx3fest1}
\begin{split}
& | \p_{x_3 } f(t,x,v)  | 
\\ \lesssim &    \mathbf 1_{\tb > t }  \{  |  \nabla_{x_\parallel}   f_0( X(0), V(0)) | |\p_{x_3 } X_\parallel (0 ) | +  |  \p_{x_3}   f_0( X(0), V(0)) | |\nabla_{x_3} X_3 (0 ) |  +  |  \nabla_v  f_0(X(0), V(0))| |\p_{x_i } V(0 ) | \}
\\ & +  \mathbf 1_{\tb <  t } \left(  c_\mu \frac{1}{ \hat {\vb}_{,3} } \mu(\vb)  \langle v \rangle^2  \int_{u_3 < 0 } \left( | \nabla_x f (t-\tb, \xb, u ) | \hat u_3   + f(t-\tb, \xb, u ) \right) du   \right).
\end{split}
\Ee
Also,
\Be \label{dvfest1}
\begin{split}
& | \p_{v_i } f(t,x,v)  | 
\\  \lesssim &  \mathbf 1_{\tb > t }  \{ |  \nabla_{x_\parallel}   f_0( X(0), V(0)) | |\p_{v_i } X_\parallel (0 ) | +  |  \p_{x_3}   f_0( X(0), V(0)) | |\nabla_{v_i} X_3 (0 ) |  +  |  \nabla_v  f_0(X(0), V(0))| |\p_{v_i } V(0 ) | \}
\\ & +  \mathbf 1_{\tb <  t } \left(  c_\mu \mu(\vb) |\vb|  \langle v \rangle^2  \int_{u_3 < 0 } \left( | \nabla_x f (t-\tb, \xb, u ) | \hat u_3   + f(t-\tb, \xb, u ) \right) du   \right) .
\end{split}
\Ee

Let $(x,v) \notin \gamma_0$ and $(t^0,x^0,v^0) = (t,x,v)$. For the characteristic
\Be \label{diffusecycles0}
\begin{split}
\frac{d}{ds}X(s;t,x,v) &= \hat V(s;t,x,v),\\
\frac{d}{ds} V(s;t,x,v) &= \mathfrak F (s, X(s;t,x,v),V(s;t,x,v) ),
\end{split}
\Ee
we define the stochastic (diffuse) cycles as 
\begin{equation} \label{diffusecycles1} \begin{split}
& t^1 = t - \tb(t,x,v), \, x^1 = \xb(t,x,v) = X(t - \tb(t,x,v);t,x,v), 
\\ & v_b^0 = V(t - \tb(t,x,v);t,x,v) = \vb(t,x,v),
\end{split} \end{equation}
 and $v^1 \in \mathbb R^3$ with $n(x^1) \cdot v^1 > 0$. 
For $l\ge1$, define
\Be \begin{split}
& t^{l+1} = t^l - \tb(t^l,x^l,v^l), x^{l + 1 } = \xb(t^l,x^l,v^l), 
\\ & v_b^l = \vb(t^l,x^l,v^l), 
\end{split} 
\Ee
and $v^{l+1} \in \mathbb R^3 \text{ with } n(x^{l+1}) \cdot v^{l+1} > 0$.
Also, define 
\Be  \label{diffusecycles2}
X^l(s) = X(s;t^l,x^l,v^l), \, V^l(s) = V(s;t^l,x^l,v^l),
\Ee
 so $X(s) = X^0(s), V(s) = V^0(s)$.
 
 Expanding $\nabla_x f(t^1, x^1 , v^1 )  + f(t^1,x^1,v^1)$ in \eqref{dx3fest1} again, we get for $i=1,2$,
\[ 
\begin{split}
& | \p_{x_i } f(t,x,v)  | 
\\ \lesssim &    \mathbf 1_{t^1 < 0 }  \{ |  \nabla_x  f_0( X(0), V(0)) | |\p_{x_i } X(0 ) |  +  |  \nabla_v  f_0(X(0), V(0))| |\p_{x_i } V(0 ) | \}
\\ & +  \mathbf 1_{  t^2  < 0 < t^1 } \bigg\{  c_\mu \mu(\vb)  \langle v \rangle^2  \int_{v^1_3 < 0 } \big( \left( | \nabla_x f (0, X^1(0) , V^1(0) ) +  \langle v^1 \rangle  | \nabla_v f (0, X^1(0) , V^1(0) ) | \right) | \hat v^1_3  
\\ & \quad \quad \quad \quad \quad \quad \quad \quad \quad \quad \quad \quad  + f(0 , X^1(0) , V^1(0) ) \big) d v^1   \bigg \}
\\ & +  \mathbf 1_{   t^2 > 0 } \left(  c_\mu \mu(\vb)  \langle v \rangle^2  \int_{v^1_3 < 0 } \left(  c_\mu \mu(\vb^1) \langle v_1 \rangle^2 \frac{\hat v^1_3 }{\hat \vb_{,3}^1} \int_{v^2_3 < 0 } | \left(  \nabla_{x_\parallel} f(t^2,x^2, v^2 ) | \hat v^2_3   + f(t^2, x^2, v^2 )  \right) dv^2 \right) dv^1  \right),
\end{split}
\]
keep doing the expansion we get for $\ell > 1$,
\Be \label{dxparadiffusef2}
\begin{split}
& | \nabla_{x_\parallel} f(t,x,v)  | 
\\ \lesssim &  \mathbf 1_{t^1 < 0 }  \{ |  \nabla_{x_\parallel}   f_0( X(0), V(0)) | |\nabla_{x_\parallel} X_\parallel (0 ) | +  |  \p_{x_3}   f_0( X(0), V(0)) | |\nabla_{x_\parallel} X_3 (0 ) |  +  |  \nabla_v  f_0(X(0), V(0))| |\nabla_{x_\parallel} V(0 ) | \}
\\  &+   \mu (\vb)  \langle  v \rangle^2 \int_{\prod_{j=1}^{l-1} \mathcal V_j} \sum_{i=1}^{l-1} \textbf{1}_{\{t^{i+1} < 0 < t^i \}} \bigg( \left( |  \nabla_x f^{}  (0,X^{i}(0), V^{i}(0)) | + \langle v^i \rangle   |  \nabla_v f^{}  (0,X^{i}(0), V^{i}(0)) |  \right) \hat v^i_3
\\ &  \quad \quad \quad \quad \quad \quad \quad \quad \quad \quad \quad \quad \quad \quad \quad \quad + f(0 , X^i(0) , V^i(0) )   \bigg)  \, d \Sigma_{i}^{l-1}
\\ & +   \mu (\vb)  \langle v \rangle^2  \int_{\prod_{j=1}^{l-1} \mathcal V_j}  \textbf{1}_{\{t^{l} > 0 \}} \int_{\mathcal V_l}  \left(  | \nabla_{x_\parallel}  f^{} (t^l,x^l, v^l ) | \hat v^l_3 + f(t^l, x^l, v^l )  \right)  d v^l  d \Sigma_{l-1}^{l-1},
\end{split}
\Ee
where $\mathcal V_j = \{ v^j \in \mathbb R^3:  v^J_3 < 0 \}$,
and 
\[ \begin{split}
d \Sigma_i^{l -1 } = & \{\prod_{j=i+1}^{l-1}  \mu(v^j) c_\mu |  \hat v_3^j | dv^j \} 
 \{\prod_{j=1}^{i-1} c_\mu \mu(\vb^j ) \langle v^j \rangle^2    \frac{\hat v^J_3 }{\hat \vb_{,3}^j} d v^j\},
\end{split} \]
where $c_\mu$ is the constant that $c_\mu  \int_{\mathbb R^3 }  \mu(v^j)  | \hat v_3^j  | dv^j = 1$. Similarly, we get
\Be
\begin{split}
& | \p_{x_3 } f(t,x,v)  | 
\\ \lesssim &  \mathbf 1_{t^1 < 0 }  \{ |  \nabla_{x_\parallel}   f_0( X(0), V(0)) | |\p_{x_3 } X_\parallel (0 ) | +  |  \p_{x_3}   f_0( X(0), V(0)) | | \p_{x_3} X_3 (0 ) |  +  |  \nabla_v  f_0(X(0), V(0))| |\p_{x_3 } V(0 ) | \}
\\ & +   \frac{\mu (\vb)}{ \hat {\vb}_{,3} }   \langle  v \rangle^2 \int_{\prod_{j=1}^{l-1} \mathcal V_j} \sum_{i=1}^{l-1} \textbf{1}_{\{t^{i+1} < 0 < t^i \}} \bigg(  \left( |  \nabla_x f^{}  (0,X^{i}(0), V^{i}(0)) | + \langle v^i \rangle   |  \nabla_v f^{}  (0,X^{i}(0), V^{i}(0)) | \right) \hat v^i_3
\\ &  \quad \quad \quad \quad \quad \quad \quad \quad \quad \quad \quad \quad \quad \quad \quad \quad + f(0 , X^i(0) , V^i(0) )   \bigg)  \, d \Sigma_{i}^{l-1}
\\ & +  \frac{\mu (\vb)}{ \hat {\vb}_{,3} }  \langle v \rangle^2  \int_{\prod_{j=1}^{l-1} \mathcal V_j}  \textbf{1}_{\{t^{l} > 0 \}}  \int_{\mathcal V_l}  \left(  | \nabla_x  f^{} (t^l,x^l, v^\ell ) | \hat v^\ell_3 + f(t^\ell, x^\ell, v^\ell )  \right) d v^l d \Sigma_{l-1}^{l-1}.
\end{split}
\Ee
And
\Be \label{dvdiffusef2}
\begin{split}
& | \nabla_{v} f(t,x,v)  | 
\\ \lesssim &  \mathbf 1_{t^1 < 0 }  \{ |  \nabla_{x_\parallel}   f_0( X(0), V(0)) | |\nabla_{v} X_\parallel (0 ) | +  |  \p_{x_3}   f_0( X(0), V(0)) | |\nabla_{v} X_3 (0 ) |  +  |  \nabla_v  f_0(X(0), V(0))| |\nabla_{v} V(0 ) | \}
\\  &+   \mu (\vb) |\vb|  \langle  v \rangle^2 \int_{\prod_{j=1}^{l-1} \mathcal V_j} \sum_{i=1}^{l-1} \textbf{1}_{\{t^{i+1} < 0 < t^i \}} \bigg( \left( |  \nabla_x f^{}  (0,X^{i}(0), V^{i}(0)) | + \langle v^i \rangle   |  \nabla_v f^{}  (0,X^{i}(0), V^{i}(0)) |  \right) \hat v^i_3
\\ &  \quad \quad \quad \quad \quad \quad \quad \quad \quad \quad \quad \quad \quad \quad \quad \quad + f(0 , X^i(0) , V^i(0) )   \bigg)  \, d \Sigma_{i}^{l-1}
\\ & +   \mu (\vb) |\vb|  \langle v \rangle^2  \int_{\prod_{j=1}^{l-1} \mathcal V_j}  \textbf{1}_{\{t^{l} > 0 \}} \int_{\mathcal V_l}  \left(  | \nabla_x  f^{} (t^l,x^l, v^l ) | \hat v^l_3 + f(t^l, x^l, v^l )  \right)  d v^l  d \Sigma_{l-1}^{l-1},
\end{split}
\Ee

Next, we claim that there exists $l_0 \gg 1$ such that for $l \ge l_0$,  we have
\Be \label{trajexpansionendterm}
\int_{\prod_{j=1}^{l-1} \mathcal V_j} \textbf{1}_{\{t^{l}(t,x,v,v^1,...,v^{l-1} ) >0 \}} \, d \, \Sigma_{l-1}^{l-1} \lesssim \left( \frac{1}{2} \right)^{l}.
\Ee
Since
\[
|\vb^j|^2 \lesssim |v^j |^2 + (t^j - t^{j+1} ) ^2  ( \| E \|_\infty^2 + \| B \|_\infty^2 ), \, \langle \vb^j \rangle \lesssim \langle v^j \rangle + (t^j - t^{j+1} ) ( \| E \|_\infty + \| B \|_\infty ) ,
\]
and using \eqref{alphaest}, we have for some fixed constant $C_{0 } > 0$, 
\[ \begin{split}
 \, d \, \Sigma_{l-1}^{l-1} 
 \le (C_{0 })^l \prod_{j=1}^{l-1} \sqrt{ \mu (v_j) } \langle v^j \rangle^2 dv^j.
\end{split} \]
Choose a sufficiently small $\delta = \delta (C_0) > 0$. Define
\[
\mathcal V_j^\delta = \{ v^j \in \mathcal V_j: v^J_3 \ge \delta
\},
\]
where we have $\int_{\mathcal V_j \setminus \mathcal V_j^\delta} C_0  \sqrt{ \mu (v_j) } \langle v^j \rangle^2 dv^j \lesssim \delta$.

On the other hand if $v^j \in \mathcal V_j^\delta$,  we have from \eqref{Xiformula}
\[
\begin{split}
| (t^j - t^{j+1} ) \hat{v}_3^j  | = &  |  \int^{t^j }_{ t^{j+1} }\int^{t^j}_s \hat{\mathfrak F } _{3}(\tau, X^j (\tau), V^j(\tau) ) \dd \tau \dd s | 
\\ \le &   \int^{t^j }_{ t^{j+1} }\int^{t^j}_s \frac{2 g }{ \langle V^j(\tau ) \rangle }  d\tau ds   \le  \frac{(t^j - t^{j+1} )^2  g }{ \min_{ 0 \le \tau \le T} \langle V^j(\tau) \rangle }.
\end{split}
\]
Thus
\[
|t^j - t^{j+1} | \ge \frac{ v^J_3 }{g} \frac{ \min_{ 0 \le \tau \le T} \langle V^j(\tau) \rangle}{ \max_{ 0 \le \tau \le T} \langle V^j(\tau) \rangle }  \gtrsim   v^J_3 \ge \delta.
\]
Now if $t^l \ge 0$ then there are at most $ \left[ \frac{ C_\Omega }{ \delta }  \right] + 1$ numbers of $v^m \in \mathcal V^\delta_m$ for $1 \le m \le l-1$. Equivalently there are at least $l -2 - \left[ \frac{ C_\Omega }{ \delta }  \right]$ numbers of $v^{m_i} \in \mathcal V_{m_i} \setminus \mathcal V_{m_i }^\delta$. Therefore we have:
\Be \label{tailtermsmall} \begin{split}
\int_{\prod_{j=1}^{l-1} \mathcal V_j} & \textbf{1}_{\{t^{l}(t,x,v,v^1,...,v^{l-1} ) >0 \}} \, d \, \Sigma_{l-1}^{l-1}
\\ \le & \sum_{m = 1}^{  \left[ \frac{ C_\Omega }{ \delta }  \right] +1} \int_{ \left\{ \parbox{15em}{there are exactly $m$ of $v^{m_i} \in \mathcal V_{m_i}^\delta $ and $l -1 -m $ of $v^{m_i} \in \mathcal V_{m_i } \setminus \mathcal V^\delta_{m_i}$ } \right\} } \prod_{j =1}^{l-1} C_0  \sqrt{ \mu (v_j) } \langle v^j \rangle^2  d v^j
\\ \le & \sum_{m = 1}^{  \left[ \frac{ C_\Omega }{ \delta }  \right]+1} {l-1  \choose m} \left\{ \int_{\mathcal V} C_0  \sqrt{ \mu (v) } \langle v \rangle^2 dv \right\}^m \left\{ \int_{\mathcal V \setminus \mathcal V^\delta} C_0 \sqrt{ \mu (v) } \langle v \rangle^2 dv \right\}^{l-1-m}
\\ \le & \left(  \left[ \frac{ C_\Omega }{ \delta}  \right] + 1 \right) ( l -1)^{ \left[ \frac{ C_\Omega }{ \delta }  \right] + 1} ( \delta)^{l - 2 -  \left[ \frac{ C_\Omega }{ \delta }  \right] } \left\{ \int_{\mathcal V} C_0 \sqrt{ \mu (v) } \langle v \rangle^2 dv \right\}^{ \left[ \frac{ C_\Omega }{ \delta }  \right]+1}
\\ \le & C \delta^{l/2} \le C (\frac{1}{2})^{l},
\end{split} \Ee
if $l \gg 1$, say $ l = 2 \left(  \left[ \frac{ C_\Omega }{ \delta }  \right] +1 \right) ^2$.


Therefore, from \eqref{tbdV0},  \eqref{pxviXVest}, \eqref{pxipviXV}, \eqref{dxparadiffusef2},\eqref{trajexpansionendterm}, and \eqref{tailtermsmall} we have
\Be \label{dxparadiffusef3}
\begin{split}
& \langle v \rangle^{4 +\delta } | \nabla_{x_\parallel} f(t,x,v)  | 
\\ \lesssim & \langle v \rangle^{4 +\delta } \{ |  \nabla_{x_\parallel}   f_0( X(0), V(0)) | + t  |  \p_{x_3}   f_0( X(0), V(0)) |   + \langle v \rangle  |  \nabla_v  f_0(X(0), V(0))| \}
\\  &+   \langle v \rangle^{4 +\delta }  \mu (\vb)  \langle  v \rangle^2 \int_{\prod_{j=1}^{l-1} \mathcal V_j} \sum_{i=1}^{l-1} \textbf{1}_{\{t^{i+1} < 0 < t^i \}} \bigg( \left( |  \nabla_x f^{}  (0,X^{i}(0), V^{i}(0)) | + \langle v^i \rangle   |  \nabla_v f^{}  (0,X^{i}(0), V^{i}(0)) |  \right) \hat v^i_3
\\ &  \quad \quad \quad \quad \quad \quad \quad \quad \quad \quad \quad \quad \quad \quad \quad \quad + f(0 , X^i(0) , V^i(0) )   \bigg)  \, d \Sigma_{i}^{l-1}
\\ & +  \langle v \rangle^{4 +\delta }   \mu (\vb)  \langle v \rangle^2  \int_{\prod_{j=1}^{l-1} \mathcal V_j}  \textbf{1}_{\{t^{l} > 0 \}} \int_{\mathcal V_l}  \left(  | \nabla_{x_\parallel}  f^{} (t^l,x^l, v^l ) | \hat v^l_3 + f(t^l, x^l, v^l )  \right)  d v^l  d \Sigma_{l-1}^{l-1}
\\ \lesssim & C_l \left(  \| (1 +|v|^{4 + \delta } )  \nabla_{x_\parallel}   f_0 \|_\infty +  \| (1 +|v|^{4 + \delta } )  \alpha \p_{x_3}   f_0 \|_\infty +  \| (1 +|v|^{5 + \delta } )  \nabla_{v}   f_0 \|_\infty  \right)  
\\ & + C \left( \frac{1}{2} \right)^l \sup_{0 \le t \le T }   \| \langle v \rangle^{4+\delta}   \nabla_{x_\parallel} f(t) \|_\infty ,
\end{split}
\Ee
and similarly, 
\Be \label{dx3diffusef3}
\begin{split}
& \langle v \rangle^{5 +\delta } | \alpha \p_{x_3} f(t,x,v)  | 
\\ \lesssim & \langle v \rangle^{5 +\delta } \alpha(t,x,v) \{ |  \nabla_{x_\parallel}   f_0( X(0), V(0)) | +  |  \p_{x_3}   f_0( X(0), V(0)) |   + \langle v \rangle  |  \nabla_v  f_0(X(0), V(0))| \}
\\  &+   \langle v \rangle^{5 +\delta }   \frac{\alpha(t,x,v) \mu (\vb)}{ \hat {\vb}_{,3} }   \langle  v \rangle^2 \int_{\prod_{j=1}^{l-1} \mathcal V_j} \sum_{i=1}^{l-1} \textbf{1}_{\{t^{i+1} < 0 < t^i \}} \bigg( \left( |  \nabla_x f^{}  (0,X^{i}(0), V^{i}(0)) | + \langle v^i \rangle   |  \nabla_v f^{}  (0,X^{i}(0), V^{i}(0)) |  \right) \hat v^i_3
\\ &  \quad \quad \quad \quad \quad \quad \quad \quad \quad \quad \quad \quad \quad \quad \quad \quad + f(0 , X^i(0) , V^i(0) )   \bigg)  \, d \Sigma_{i}^{l-1}
\\ & +  \langle v \rangle^{5 +\delta }    \frac{\alpha(t,x,v) \mu (\vb)}{ \hat {\vb}_{,3} }  \langle v \rangle^2  \int_{\prod_{j=1}^{l-1} \mathcal V_j}  \textbf{1}_{\{t^{l} > 0 \}} \int_{\mathcal V_l}  \left(  | \nabla_x  f^{} (t^l,x^l, v^l ) | \hat v^l_3 + f(t^l, x^l, v^l )  \right)  d v^l  d \Sigma_{l-1}^{l-1}
\\ \lesssim & C_l \left(  \| (1 +|v|^{5 + \delta } )  \nabla_{x_\parallel}   f_0 \|_\infty +  \| (1 +|v|^{5 + \delta } )  \alpha \p_{x_3}   f_0 \|_\infty +  \| (1 +|v|^{5 + \delta } )  \nabla_{v}   f_0 \|_\infty  \right)  
\\ & + C \left( \frac{1}{2} \right)^l \sup_{0 \le t \le T }   \| \langle v \rangle^{5+\delta} \alpha  \nabla_x f(t) \|_\infty ,
\end{split}
\Ee
and
\Be \label{dvdiffusef4}
\begin{split}
& \langle v \rangle^{5 +\delta } | \nabla_{v} f(t,x,v)  | 
\\ \lesssim & C_l \left(  \| (1 +|v|^{4 + \delta } )  \nabla_{x_\parallel}   f_0 \|_\infty +  \| (1 +|v|^{4 + \delta } )  \alpha \p_{x_3}   f_0 \|_\infty +  \| (1 +|v|^{5 + \delta } )  \nabla_{v}   f_0 \|_\infty  \right)  
\\ & + C \left( \frac{1}{2} \right)^l \sup_{0 \le t \le T }   \| \langle v \rangle^{5+\delta} \alpha  \nabla_x f(t) \|_\infty ,
\end{split}
\Ee
where we've used \eqref{alphaest}. Adding \eqref{dxparadiffusef2}, \eqref{dx3diffusef3}, and \eqref{dvdiffusef4} and choosing $l \gg 1$, we get for a large $C > 0$, 
\Be \label{diffusefinal}
\begin{split}
& \sup_{ 0 \le t \le T}  \left( \| \langle v \rangle^{4 + \delta}\nabla_{x_\parallel } f(t) \|_\infty +  \|  \langle v \rangle^{5 + \delta} \alpha \p_{x_3} f(t) \|_\infty   +  \|  \langle v \rangle^{5 + \delta}  \nabla_v f(t) \|_\infty \right)
\\ <  & C \left(   \|  \langle v \rangle^{5 + \delta}  \nabla_{x_\parallel}   f_0 \|_\infty +  \|  \langle v \rangle^{5 + \delta}  \alpha \p_{x_3}   f_0 \|_\infty +  \| \langle v \rangle^{5 + \delta}  \nabla_{v}   f_0 \|_\infty \right).
\end{split} 
\Ee

Next, using the same argument in Lemma \ref{tracepf}, we obtain
\Be \label{xparavfbdp}
\begin{split}
& \sup_{0 \le t \le T}  \| \langle v \rangle^{4 + \delta }  \nabla_{x_\parallel} f (t) \|_{L^\infty(\gamma \setminus \gamma_0) }  \le \sup_{0 \le t \le T}   \| \langle v \rangle^{4 + \delta }  \nabla_{x_\parallel} f (t) \|_\infty, 
\\ &  \sup_{0 \le t \le T}   \| \langle v \rangle^{5 + \delta }  \nabla_{v} f (t) \|_{L^\infty(\gamma \setminus \gamma_0) }  \le  \sup_{0 \le t \le T}   \| \langle v \rangle^{5 + \delta }  \nabla_{v} f (t) \|_\infty,
\\ &  \sup_{0 \le t \le T}   \| \langle v \rangle^{5+\delta} \alpha \p_{x_3} f  (t,x,v)  \|_{L^\infty(\gamma \setminus \gamma_0 ) } \le  \sup_{0 \le t \le T}     \| \langle v \rangle^{5+\delta} \alpha \p_{x_3} f  (t,x,v)  \|_\infty.
\end{split}
\Ee
Together \eqref{diffusefinal}, we conclude \eqref{diffusedfbd}.

\end{proof}

In order to construct a solution to the system \eqref{VMfrakF1}--\eqref{rhoJ1}, \eqref{diffuseBC}, we define a sequence of functions:
\[
f^0(t,x,v) = f_0(x,v), \ E^0(t,x) = E_0(t,x), \ B^0(t,x) = B_0(x).
\]
For $\ell \ge 1$, let $f^\ell$ be the solution of
\Be \label{fellseq}
\begin{split}
\p_t f^\ell + \hat v \cdot \nabla_x f^\ell +  \mathfrak F^{\ell-1}   \cdot \nabla_v f^\ell = &  0, \text{ where }   \mathfrak F^{\ell-1} =  E^{\ell-1} + E_{\text{ext}} + \hat v \times ( B^{\ell-1} + B_{\text{ext}} ) - g \mathbf e_3,
\\f^\ell(0,x,v)  = &  f_0(x,v) ,
\\ f^\ell(t,x,v) |_{\gamma_-} = &  c_\mu \mu(v) \int_{u_3 < 0 }  - f^{\ell-1} (t,x,  u ) \hat u _3 du.
\end{split}
\Ee
Let $\rho^\ell = \int_{\mathbb R^3 } f^\ell dv, j^\ell = \int_{\mathbb R^3 } \hat v  f^\ell dv$. Let 
\Be \label{ElBl}
E^\ell = \eqref{Eesttat0pos} + \dots + \eqref{Eest3bdrycontri}, \  B^\ell = \eqref{Besttat0pos} + \dots + \eqref{Bestbdrycontri}, \text{ with } f \text{ changes to } f^\ell. 
\Ee
And let
\Be \label{Fell}
\mathfrak F^\ell = E^\ell + E_{\text{ext}} - \hat v \times ( B^\ell + B_{\text{ext}} )  - g \mathbf e_3. 
\Ee
We prove several uniform-in-$\ell$ bounds for the sequence before passing the limit.

\begin{lemma}
Suppose $f_0$ satisfies \eqref{f0bdd}, $E_0$, $B_0$ satisfy \eqref{E0B0g}, \eqref{E0B0bdd}, then there exits $M_1, M_2$, and $c_0$, such that for $0 < T \ll 1 $, 
\Be \label{fellbound}
\begin{split}
 \sup_\ell \sup_{0 \le t \le T} \left( \| \langle v \rangle^{4 + \delta } f^\ell(t) \|_{L^\infty(\bar \O \times \mathbb R^3)}    \right) < & M_1, \ 
\\  \sup_\ell \sup_{0 \le t \le T} \left( \| E^\ell (t) \|_\infty + \| B^\ell (t) \|_\infty \right) + |B_e| + E_e + g < & M_2,
\\ \inf_{\ell} \inf_{ t ,x_\parallel} \left( g - E_e  -  E^\ell_3(t,x_\parallel, 0 )  -  (\hat v \times B^\ell)_3(t,x_\parallel, 0 ) \right) > &  c_0.
\end{split}
\Ee
\end{lemma}
\begin{proof}
Let $\ell \ge 1$. By induction hypothesis we assume that
\Be \label{inductfEB}
\begin{split}
\sup_{ 0 \le i \le \ell }  \sup_{0 \le t \le T} \left( \| \langle v \rangle^{4 + \delta} f^{\ell-i}(t) \|_{L^\infty(\bar \O \times \mathbb R^3)}   \right) < & M_1,
\\  \sup_{ 0 \le i \le \ell }  \sup_{0 \le t \le T} \left(  \| E^{\ell-i} (t) \|_\infty + \| B^{\ell-i} (t) \|_\infty \right) + |B_e| + E_e + g< & M_2.
\end{split}
\Ee

Let the characteristics $(X^\ell, V^\ell)$ be defined as in \eqref{XV_ell}. We define the stochastic cycles:
		\Be
		\begin{split}\label{cycle}
			t^{\ell}_1 (t,x,v)&:= 
			\sup\{ s<t:
			X^\ell(s;t,x,v) \in \p\O
			\}
			,\\
			x^\ell_1 (t,x,v ) &:= X^\ell (t^{\ell}_1 (t,x,v);t,x,v)
			,\\
			t^{\ell-1}_2 (t,x,v, v_1) &:= \sup\{ s<t^\ell_1:
			X^{\ell-1}(s;t^{\ell}_1 (t,x,v),x^{\ell}_1 (t,x,v),v_1) \in \p\O
			\}
			,\\
			x^{\ell-1}_2 (t,x,v, v_1 ) &:= X^{\ell-1} (t^{\ell-1}_2 (t,x,v,v_1);t^\ell_1(t,x,v),x^\ell_1(t,x,v),v_1)
			,\\
		\end{split}
		\Ee
		and inductively 
		\Be
		\begin{split}\label{cycle_ellin}
			& t^{\ell-(k-1)}_k (t,x,v, v_1, \cdots, v_{k-1})  \\
			&:=     \sup\big\{ s<t^{\ell-(k-2)}_{k-1} 
			:
			X^{\ell-1}(s;t_{k-1}^{\ell - (k-2)}  , x_{k-1}^{\ell - (k-2)} ,v_{k-1}) \in \p\O
			\big\},\\
			& x_k^{\ell - (k-1)} (t,x,v, v_1, \cdots, v_{k-1})\\
			&:= X^{\ell- (k-2)} (t_k^{\ell- (k-1)}; t_{k-1}^{\ell- (k-2)},x_{k-1}^{\ell- (k-2)} , v_{k-1})
			.
		\end{split}
		\Ee
		Here,
		\Be\begin{split}\notag
			t^{\ell-(i-1)}_{i } &:= t^{\ell-(i-1)}_{i }
			(t,x,v,v_1, \cdots, v_{i-1}),\\
			x^{\ell-(i-1)}_{i } &:= x^{\ell-(i-1)}_{i }
			(t,x,v,v_1, \cdots, v_{i-1}).\end{split}\Ee
			
First, we note that for any $t^{\ell-i}_{i+1} \le s < t^{\ell-(i-1)}_{i}$, since
\[
V^{\ell - i} (s;   t^{\ell-(i-1)}_{i}, x^{\ell-(i-1)}_{i }, v_i ) = v_i - \int_{s}^{ t^{\ell-(i-1)}_{i} } \mathfrak F^{\ell - i} (\tau, X^{\ell-i}(\tau), V^{\ell-i}(\tau)  ) d\tau,
\]
from \eqref{inductfEB}, we have
\[
|v_i | - (t^{\ell-(i-1)}_{i} - t^{\ell-i}_{i+1})M_2 \le | V^{\ell - i} (s;   t^{\ell-(i-1)}_{i}, x^{\ell-(i-1)}_{i }, v_i ) | \le |v_i | +  (t^{\ell-(i-1)}_{i} - t^{\ell-i}_{i+1})M_2.
\]
Thus
\Be \label{Vsv}
\left( 1 + (t^{\ell-(i-1)}_{i} - t^{\ell-i}_{i+1})M_2 \right)^{-1} \langle v \rangle  \le \langle V^{\ell - i} (s;   t^{\ell-(i-1)}_{i}, x^{\ell-(i-1)}_{i }, v_i ) \rangle \le \left( 1 + (t^{\ell-(i-1)}_{i} - t^{\ell-i}_{i+1})M_2 \right) \langle v \rangle.
\Ee
			
From \eqref{fellseq} we have for any $(t,x,v) \in (0,T) \times \bar \O \times \mathbb R^3$, 
\Be
\begin{split}
f^{\ell+1}(t,x,v) = & \mathbf 1_{t^\ell_1 \le 0} f^{\ell+1}( 0, X^\ell(0), V^\ell(0)) +  \mathbf 1_{t_1^\ell \ge 0 } f^{\ell+1}(t_1^\ell , X^\ell(t_1^\ell; t,x,v), V^\ell(t_1^\ell;t,x,v) ) 
\\ = &  \mathbf 1_{t^\ell_1 \le 0} f^{\ell+1}( 0, X^\ell(0), V^\ell(0))   - \mathbf 1_{t_1^\ell \ge 0 } c_\mu \mu(V^\ell(t^\ell_{1}) ) \int_{v_{1,3} < 0 } f^\ell (t^\ell_1, x^\ell_1, v_1 ) \hat{v}_{1,3}  \dd v_1.
\end{split}
\Ee
And \eqref{Vsv} gives
\Be
\begin{split}
& \langle v \rangle^{4 + \delta }  |f^{\ell+1}(t,x,v) | 
\\  \le &   \mathbf 1_{t^\ell_1 \le 0} (1+ T (M_2+g))^{4 + \delta} | \langle V^\ell(0) \rangle^{4 + \delta }  f^{\ell+1}( 0, X^\ell(0), V^\ell(0)) |  
\\ &  +   \mathbf 1_{t_1^\ell \ge 0 }   c_\mu  (1+ T (M_2+g))^{4 + \delta}  | \langle V^\ell(t^\ell_{1}) \rangle^{4 + \delta}  \mu(V^\ell(t^\ell_{1}) ) \int_{v_{1,3} < 0 } \langle v_1 \rangle^{4 +\delta}  f^\ell (t^\ell_1, x^\ell_1, v_1 ) \frac{\hat{v}_{1,3}}{\langle v_1 \rangle^{4 + \delta } }   \dd v_1 | .
\end{split}
\Ee

Then inductively, we obtain
\Be \label{vfinftyinduc}
\begin{split}
&  \langle v \rangle^{4 + \delta }  | f^{\ell+1}(t,x,v) | \le  \mathbf 1_{t^\ell_1 \le 0}  (1+t(M_2+g))^{4 + \delta} | |   \langle V^\ell(0) \rangle^{4 + \delta } f^{\ell+1}( 0, X^\ell(0), V^\ell(0))  |
\\  &+  (1+ T (M_2+g))^{4 + \delta}   \int_{\prod_{j=1}^{k-1} \mathcal V_j} \sum_{i=1}^{k-1} \textbf{1}_{\{t^{\ell-i}_{i+1} \le 0 < t^{\ell-(i-1)}_{i} \}}   |  \langle V^{\ell - i } (0; v_i) \rangle^{4 + \delta }  f^{\ell - (i-1)}(0 , X^{\ell-i}(0; v_i) , V^{\ell-i}(0;v_i) )   |  \, d \Sigma_{i}^{k-1}
\\ & +     (1+ T (M_2+g))^{4 + \delta}  \int_{\prod_{j=1}^{k-1} \mathcal V_j}  \textbf{1}_{\{t_k^{\ell - (k-1)} > 0 \}} \int_{\mathcal V_k}  |  f^{\ell - k } ( t_k^{\ell - (k-1) }, x_k^{\ell-(k-1) }, v_{k} ) |  d v_k d \Sigma_{k-1}^{k-1},
\end{split}
\Ee
where
\[
\begin{split}
X^{\ell-i }(0;v_i) = & X^{\ell-i}(0; t^{\ell - (i-1) }_i, x^{\ell-(i-1) }_i, v_i),
\\ V^{\ell-i }(0;v_i) = & V^{\ell-i}(0; t^{\ell - (i-1) }_i, x^{\ell-(i-1) }_i, v_i),
\end{split}
\]
$\mathcal V_j = \{ v_j \in \mathbb R^3:  v_{j,3} < 0 \}$,
and 
\Be \label{Nujsigmaprod}
\begin{split}
d \Sigma_i^{k -1 } = &  \{\prod_{j=1}^{i-1} c_\mu (1+ T (M_2+g) )^{4 + \delta} \mu(V^{\ell-(j-1)}(t^{\ell- (j-1)}_{j }  ) )     \frac{\hat v_{j,3} \langle  V^{\ell-(j-1)}(t^{\ell- (j-1)}_{j }  ) \rangle^{4 + \delta} }{ {\hat  V}^{\ell-(j-1)}_3(t^{\ell- (j-1)}_{j }  ) \langle v_j \rangle^{4 + \delta} } d v_j\}  \{\prod_{j=i+1}^{k-1}  \mu(v_j) c_\mu |  \hat v_{j,3} | dv_j \}. 
\end{split}
\Ee
From the same argument as in \eqref{trajexpansionendterm}-\eqref{tailtermsmall}, we get there exists $k_0 \gg 1$ such that for $k \ge k_0$,
\Be \label{1over2k}
 \int_{\prod_{j=1}^{k-1} \mathcal V_j}  \textbf{1}_{\{t_k^{\ell - (k-1)} > 0 \}} d \Sigma_{k-1}^{k-1} \le \left( \frac{1}{2} \right)^{k }.
\Ee
Thus, from \eqref{vfinftyinduc}, \eqref{1over2k}, we have
\Be
\begin{split}
& \sup_{0 \le t \le T} \|   \langle v \rangle^{4 + \delta }  f^{\ell+1}(t) \|_{L^\infty(\bar \O \times \mathbb R^3)}
\\ &  \le k (1 + TM_2)^{4 + \delta }  \| \langle v \rangle^{4 + \delta } f_0 \|_\infty  +   (1 + T(M_2+g))^{4 + \delta }  \left( \frac{1}{2} \right)^{k }  \sup_{0 \le t \le T} \| \langle v \rangle^{4 + \delta} f^{\ell-i}(t) \|_{L^\infty(\bar \O \times \mathbb R^3)}.
\end{split}
\Ee
By choosing $M_1 \gg 1$ and then $ T \ll 1$, we get 
\Be \label{fell1final}
 \sup_{0 \le t \le T} \| \langle v \rangle^{4 + \delta } f^{\ell+1} (t) \|_{L^\infty(\bar \O \times \mathbb R^3)} < M_1. 
\Ee
Now from \eqref{ElBl} and \eqref{E0B0g}, using the same argument as \eqref{BEellinftyestinflow1}--\eqref{BEellinftyestinflow4}, we get
\Be \label{EBlifinal}
  \sup_{0 \le t \le T} \| E^{\ell+1} (t) \|_\infty + \sup_{0 \le t \le T} \| B^{\ell+1} (t) \|_\infty + |B_e| +E_e + g < M_2,
\Ee
and
\Be
 \inf_{ t ,x_\parallel} \left( g - E_e  -  E^{\ell+1}_3(t,x_\parallel, 0 )  -  (\hat v \times B^{\ell+1})_3(t,x_\parallel, 0 ) \right) >   c_0.
\Ee
Thus we conclude \eqref{fellbound} by induction. 
\end{proof}

Next, we consider the derivative of the sequences. Define $\alpha^{\ell} $ as in \eqref{alphan}.
We have the following estimate.
\begin{lemma}
Suppose $f_0$ satisfies \eqref{f0bdd}, $E_0$, $B_0$ satisfy \eqref{E0B0bdd}, then there exits $M_3, M_4$ such that for $0 < T \ll 1 $, 
\Be \label{flElBldsqbd}
\begin{split}
 &\sup_\ell \sup_{0 \le t \le T} \left( \| \langle v \rangle^{4 + \delta } \nabla_{x_\parallel } f^\ell(t) \|_\infty  + \| \langle v \rangle^{5 + \delta }  \alpha^{\ell-1} \p_{x_3 } f^\ell(t) \|_\infty + \|  \langle v \rangle^{4 + \delta }  \nabla_v f^\ell(t) \|_\infty   \right) 
 \\& + \sup_\ell \sup_{0 \le t \le T} \left( \| \langle v \rangle^{4 + \delta } \nabla_{x_\parallel } f^\ell(t) \|_{L^\infty(\gamma \setminus \gamma_0)}  + \| \langle v \rangle^{5 + \delta }  \alpha^{\ell-1} \p_{x_3 } f^\ell(t) \|_{L^\infty(\gamma \setminus \gamma_0)} + \|  \langle v \rangle^{4 + \delta }  \nabla_v f^\ell(t) \|_{L^\infty(\gamma \setminus \gamma_0)}   \right) <  M_3 ,
\\ & \sup_\ell \sup_{0 \le t \le T} \left( \| \p_t E^\ell(t) \|_\infty +  \| \p_t B^\ell(t) \|_\infty + \| \nabla_{x } E^\ell(t) \|_\infty +\| \nabla_{x } B^\ell(t) \|_\infty  \right) < M_4.
\end{split} 
\Ee
\end{lemma}
\begin{proof}
The proof is essentially the same as the proof of Proposition \ref{diffuseprop}. The only difference is that instead of using the stochastic cycles \eqref{diffusecycles0}-\eqref{diffusecycles2} that flows under fixed $E(t,x), B(t,x)$, we use the \eqref{cycle}-\eqref{cycle_ellin} that flows with a different $E^\ell(t,x), B^\ell(t,x)$ after each bounce. 

From the uniform estimate \eqref{fellbound}, and from the velocity lemma (Lemma \ref{vlemma}), we have for some $C>0$,
\Be
e^{-C|t-s| } \alpha^\ell(t,x,v) \le \alpha^\ell(s, X^\ell(s;t,x,v) , V^\ell(s;t,x,v) ) \le e^{C|t-s| } \alpha^\ell(t,x,v), \ \text{ for all } \ell.
\Ee
Therefore, following the same proof of Proposition \ref{diffuseprop} we get
\Be \label{}
\begin{split}
& \sup_{0 \le t \le T} \left( \| \langle v \rangle^{4 +\delta } \nabla_{x_\parallel} f^{\ell+1} (t)  \|_\infty  + \|  \langle v \rangle^{5 +\delta }  \alpha^\ell \p_{x_3} f ^{\ell+1} (t) \|_\infty  + \| \langle v \rangle^{5 +\delta } \nabla_{v} f^{\ell+1} (t)  \|_\infty \right)
\\ + & \sup_{0 \le t \le T} \left( \| \langle v \rangle^{4 + \delta } \nabla_{x_\parallel } f^\ell(t) \|_{L^\infty(\gamma \setminus \gamma_0)}  + \| \langle v \rangle^{5 + \delta }  \alpha^{\ell-1} \p_{x_3 } f^\ell(t) \|_{L^\infty(\gamma \setminus \gamma_0)} + \|  \langle v \rangle^{4 + \delta }  \nabla_v f^\ell(t) \|_{L^\infty(\gamma \setminus \gamma_0)}   \right) 
\\ \le & C_k \left( \| \langle v \rangle^{4 +\delta }  \nabla_{x_\parallel}   f_0 \|_\infty +  \| \| \langle v \rangle^{4 +\delta }   \alpha \p_{x_3}   f_0 \|_\infty +  \| \| \langle v \rangle^{5 +\delta }  \nabla_{v}   f_0 \|_\infty  \right)  
\\ & + C \left( \frac{1}{2} \right)^k  \sup_{0 \le i \le \ell }  \bigg( \sup_{0 \le t \le T } \left(   \| \langle v\rangle^{4 +\delta}   \nabla_{x_\parallel} f^i(t) \|_\infty +  \| \langle v\rangle^{5 +\delta} \alpha  \p_{x_3} f^i(t) \|_\infty +  \| \langle v \rangle^{5 +\delta } \nabla_{v} f^i(t)  \|_\infty \right) 
\\  & + \sup_{0 \le i \le \ell} \sup_{0 \le t \le T} \left( \| \langle v \rangle^{4 + \delta } \nabla_{x_\parallel } f^i(t) \|_{L^\infty(\gamma \setminus \gamma_0)}  + \| \langle v \rangle^{5 + \delta }  \alpha^{i-1} \p_{x_3 } f^i(t) \|_{L^\infty(\gamma \setminus \gamma_0)} + \|  \langle v \rangle^{4 + \delta }  \nabla_v f^i(t) \|_{L^\infty(\gamma \setminus \gamma_0)}   \right)  \bigg).
\end{split}
\Ee
Thus, by choosing $k \gg 1 $ and $M_3 \gg 1$, we conclude
\Be \label{dfelluni}
\begin{split}
& \sup_\ell \sup_{0 \le t \le T} \left( \| \langle v \rangle^{4 + \delta } \nabla_{x_\parallel } f^\ell(t) \|_\infty  + \| \langle v \rangle^{5 + \delta }  \alpha^{\ell-1} \p_{x_3 } f^\ell(t) \|_\infty + \| \langle v \rangle^{4 + \delta }  \nabla_v f^\ell(t) \|_\infty   \right)  \\& + \sup_\ell \sup_{0 \le t \le T} \left( \| \langle v \rangle^{4 + \delta } \nabla_{x_\parallel } f^\ell(t) \|_{L^\infty(\gamma \setminus \gamma_0)}  + \| \langle v \rangle^{5 + \delta }  \alpha^{\ell-1} \p_{x_3 } f^\ell(t) \|_{L^\infty(\gamma \setminus \gamma_0)} + \|  \langle v \rangle^{4 + \delta }  \nabla_v f^\ell(t) \|_{L^\infty(\gamma \setminus \gamma_0)}   \right)<  M_3.
\end{split}
\Ee
From this, we use the same argument to get \eqref{dxEBfinal} in the proof of Lemma \ref{EBW1inftylemma} and obtain
\[
\begin{split}
 \sup_{0 \le t \le T} & \left( \| \p_t E^{\ell+1}(t) \|_\infty +  \| \p_t B^{\ell+1}(t) \|_\infty + \| \nabla_{x } E^{\ell+1}(t) \|_\infty +\| \nabla_{x } B^{\ell+1}(t) \|_\infty   \right)
\\ \le & TC \sup_{1 \le i \le \ell}  \sup_{0 \le  t \le T}    \left( \| \p_t E^{i}(t) \|_\infty +  \| \p_t B^{i}(t) \|_\infty + \| \nabla_{x } E^{i}(t) \|_\infty +\| \nabla_{x } B^{i}(t) \|_\infty   \right)  
\\ & C   \left( \| E_0 \|_{C^2 } + \| B_0 \|_{C_2}  \right) +  C \sup_{0 \le t \le T} \left( \| \langle v \rangle ^{4 + \delta } \nabla_{x_\parallel } f^{\ell +1}(t) \|_\infty  + \| \langle v \rangle^{\ell  + \delta }  \alpha^{n} \p_{x_3 } f^{\ell +1}(t) \|_\infty   \right)  
\\ &  +  C \sup_{0 \le t \le T} \left( \| \langle v \rangle^{4 + \delta } \nabla_{x_\parallel } f^\ell(t) \|_{L^\infty(\gamma \setminus \gamma_0)}  + \| \langle v \rangle^{5 + \delta }  \alpha^{\ell-1} \p_{x_3 } f^\ell(t) \|_{L^\infty(\gamma \setminus \gamma_0)} \right) 
\\ & + C  \sup_{0 \le t \le T} \left( \| \langle v \rangle^{4  + \delta } f^{\ell+1}(t) \|_\infty  + \| E^{\ell +1} (t) \|_\infty + \| B^{\ell +1} (t) \|_\infty \right).
\end{split}
\]
From \eqref{fellbound} and \eqref{dfelluni}, this gives
\[
\begin{split}
\sup_{\ell} \sup_{0 \le t \le T} & \left( \| \p_t E^{\ell}(t) \|_\infty +  \| \p_t B^{\ell}(t) \|_\infty + \| \nabla_{x } E^{\ell}(t) \|_\infty +\| \nabla_{x } B^{\ell}(t) \|_\infty   \right)
\\ \le & TC \sup_{ \ell}  \sup_{0 \le  t \le T}    \left( \| \p_t E^{\ell}(t) \|_\infty +  \| \p_t B^{\ell}(t) \|_\infty + \| \nabla_{x } E^{\ell}(t) \|_\infty +\| \nabla_{x } B^{\ell}(t) \|_\infty   \right)  
\\ & + C \left( \| E_0 \|_{C^2 } + \| B_0 \|_{C_2}  \right)+ C(M_1 +M_2 +M_3).
\end{split}
\]
Therefore, by choosing $M_4 \gg 1$ and $T \ll 1$, we get
\Be
\sup_{\ell} \sup_{0 \le t \le T}  \left( \| \p_t E^{\ell}(t) \|_\infty +  \| \p_t B^{\ell}(t) \|_\infty + \| \nabla_{x } E^{\ell}(t) \|_\infty +\| \nabla_{x } B^{\ell}(t) \|_\infty   \right) < M_4.
\Ee
Together with \eqref{dfelluni}, we conclude \eqref{flElBldsqbd}.
\end{proof}

Note that from \eqref{flElBldsqbd}, using the argument in Lemma \ref{EBelltrace} we have $E^\ell |_{\p \O}, B^\ell  |_{\p \O} \in L^\infty((0,T) \times \p \O ) $ for all $\ell$. Next, we prove the strong convergence of the sequence $f^\ell$.

\begin{lemma} \label{fEBsollemma}
Suppose $f_0$ satisfies \eqref{f0bdd}, $E_0$, $B_0$ satisfy \eqref{E0B0g}, \eqref{E0B0bdd}. There exists functions $(f,E,B)$ with $  \langle v \rangle^{4 +\delta }  f(t,x,v) \in L^\infty( (0, T) ; L^\infty( \bar \O \times \mathbb R^3 ) )  $, and $(E,B) \in L^\infty((0,T) ; L^\infty( \O) \cap  L^\infty( \p \O ) )$, such that as $\ell \to \infty$, 
\Be \label{EnBnconverge}
 \sup_{0 \le t \le T} \left(  \|  E^\ell (t) - E(t) \|_{L^\infty( \O)} +  \|  E^\ell (t) - E(t) \|_{L^\infty( \p \O)} +  \|  B^\ell (t) - B(t) \|_{L^\infty( \O)} +  \|  B^\ell (t) - B(t) \|_{L^\infty( \p \O)} \right)   \to 0, 
\Ee
and
\Be \label{fnconverge}
 \sup_{0 \le t \le T}  \| \langle v \rangle^{4 +\ell } f^\ell(t) -  \langle v \rangle^{4 +\delta }  f (t) \|_{L^\infty(\bar \O \times \mathbb R^3)}  \to 0.
\Ee
Moreover, $(f,E,B)$ is a (weak) solution of the system \eqref{VMfrakF1}--\eqref{rhoJ1}, and \eqref{diffuseBC}.
\end{lemma}
\begin{proof}
Let $m > n  \ge 1$. Note that $f^m - f^n $ satisfies $(f^m - f^n ) |_{t = 0 } = 0 $ and 
\[
(f^m- f^n )|_{\gamma_-} = c_\mu \mu \int_{\gamma_+} (f^{m-1} - f^{n-1} )(t,x,u ) \hat u_3 du.
\] 

The equation for $f^m - f^n $ is
\[
\p_t(f^m- f^n ) + \hat v \cdot \nabla_x (f^m - f^n ) + \mathfrak F^{m-1} \cdot \nabla_v (f^{m} - f^n ) = - ( \mathfrak F^{m-1} - \mathfrak F^{n-1} ) \cdot \nabla_v f^n.
\]
Thus, for any $(t,x,v) \in (0,T) \times \bar \O \times \mathbb R^3$,  using \eqref{Vsv}, we get
\[
\begin{split}
& |  \langle v \rangle^{4 + \delta } (f^m - f^n)(t,x,v) | 
\\  \le & C_1  \int_{\max \{ t^{m-1}_1 , 0 \} }^t |    \langle V^{m-1}(s) \rangle^{4 + \delta }  ( \mathfrak F^{m-1} - \mathfrak F^{n-1} ) \cdot \nabla_v f^n  )(s, X^{m-1}(s), V^{m-1}(s) ) | ds 
\\ & + \mathbf 1_{t^{m-1}_1 > 0 } C_{1} c_\mu V^{m-1}(t_1^m) \mu (V^m(t_1^m ) ) \int_{v_{1,3} < 0 } |(f^{m-1} - f^{n-1} )(t_1^{m-1}, x_1^{m-1} , v_1 ) \hat v_{1,3} | dv_1,
\end{split}
\]
where $C_{1} =  (1+ T (M_2+g) )^{4 + \delta}$ . Doing this inductively, we obtain
\Be \label{fnmiterate2}
\begin{split}
& |  \langle v \rangle^{4 + \delta } (f^m - f^n)(t,x,v) | 
\\  \le &   C_{1}  \int_{\prod_{j=1}^{k-1} \mathcal V_j} \sum_{ i =1}^{k-1}  \int_{\max \{ t^{m-i}_{i} , 0 \} }^{t^{m -(i-1) }_{ i-1} }    \mathbf{1}_{ \{  t^{m -i }_i \le 0 < t^{m -(i-1) }_{ i-1} \} }
\\ & \quad \quad \quad \quad  \times |   \langle V^{m-i}(s) \rangle^{4 + \delta }  ( \mathfrak F^{m-i} - \mathfrak F^{n-i} ) \cdot \nabla_v f^{n-(i-1)}  )(s, X^{m-i}(s), V^{m-i}(s) ) | ds  d \Sigma_{i}^{k-1}
\\ &  +  C_1 \int_{\prod_{j=1}^{k-1} \mathcal V_j}  \textbf{1}_{\{t_{k-1}^{m - (k-1)} > 0 \}} \int_{\mathcal V_k}  |  (f^{m - k } - f^{n-k} )( t_k^{m - (k-1) }, x_k^{m-(k-1) }, v_{k} ) |  d v_k d \Sigma_{k-1}^{k-1}.
\end{split}
\Ee
Where, $\mathcal V_j$ and $\Sigma_{i}^{k-1}$ are in \eqref{Nujsigmaprod}. Then from \eqref{1over2k} and \eqref{flElBldsqbd}, by fixing $k \gg 1$,we get
\Be \label{fellnmiterate1}
\begin{split}
\| \langle  v \rangle^{4 + \delta} f^m(t) -  \langle  v \rangle^{4 + \delta} f^n)(t) \|_\infty  \le & C_k C_1 \left( \sup_{\ell}  \sup_{0 \le s \le t } \| \langle  v \rangle^{4 + \delta} \nabla_v f^\ell(s) \|_\infty \right) \int_0^t \sup_{1 \le i \le k}  \|   \mathfrak F^{m-i} (s) - \mathfrak F^{n-i}(s)  \|_\infty ds
\\ \le  & C_2  \int_0^t  \sup_{1 \le i \le k}  \|   \mathfrak F^{m-i} (s) - \mathfrak F^{n-i}(s)  \|_\infty ds,
\end{split}
\Ee
where $C_2 = C_k C_1M_3$.

Now, 
from \eqref{Fell} and using the same argument as Lemma \ref{EBlinflemma} with \eqref{fellnmiterate1}, we have
\Be \label{iterate2}
\begin{split}
 \|   \mathfrak F^{m-i} (s) - \mathfrak F^{n-i}(s)  \|_\infty   \le &   \| E^{n-i}(s) - E^{m-i}(s) \|_\infty + \| B^{n-i}(s) - B^{m-i}(s)  \|_\infty
\\ \le & C \left(  \sup_{0 \le s' \le s } \| \langle v \rangle^{4 + \delta } (f^{n-i} - f^{m-i} )(s' ) \|_\infty +  \int_0^s \|   \mathfrak F^{m-i-1} (s') - \mathfrak F^{n-i-1 }(s')  \|_\infty ds'  \right)
\\ \le & C \int_0^s  \sup_{1 \le i_1 \le k } \|   \mathfrak F^{m-i - i_1} (s') - \mathfrak F^{n-i - i_1}(s')  \|_\infty ds'
\\ \le & C \int_0^s \sup_{1 \le i \le 2k}  \left(  \| E^{m-i}(s') - E^{n-i}(s') \|_\infty + \| B^{m-i}(s') - B^{n-i}(s')  \|_\infty \right) ds'.
\end{split}
\Ee
Iteration of \eqref{iterate2} and using \eqref{fellbound} yields
\[
\begin{split}
&  \| E^{m}(t) - E^{n}(t) \|_\infty + \| B^{m}(t) - B^{n}(t)  \|_\infty
\\ \le & C^2 \int_0^t \int_0^s  \sup_{1 \le i \le 2k}  \left(  \| E^{m-i}(s') - E^{n-i}(s') \|_\infty + \| B^{m-i}(s') - B^{n-i}(s')  \|_\infty \right) ds' ds
 \\   =  & C^2 \int_0^t \tau \sup_{1 \le i \le 2k }   \left(  \| E^{m-i}(\tau) - E^{n-i}(\tau) \|_\infty + \| B^{m-i}(\tau) - B^{n-i}(\tau)  \|_\infty \right) d\tau
 \\ \le & C^l \int_0^t \frac{ \tau^{l-1}}{ (l-1)!} \sup_{1 \le i \le lk}  \left(  \| E^{m-i}(\tau) - E^{n-i}(\tau) \|_\infty + \| B^{m-i}(\tau) - B^{n-i}(\tau)  \|_\infty \right) d\tau
 \\ \le  &M_2 \frac{C^l t^l}{l!}.
\end{split}
\]
%
Thus the sequences $E^\ell$, $B^\ell$ are Cauchy in $L^\infty((0,T) \times \O ) $, moreover, from Lemma \ref{EBelltrace}, $E^\ell, B^\ell \in L^\infty([0,T] \times \p \O )$. Therefore, there exists functions $E,B \in L^\infty((0,T) ; L^\infty( \O) \cap  L^\infty( \p \O ) )$, such that 
\Be \label{EnBncov}
E^\ell \to E, B^\ell \to B \text{ in } L^\infty((0,T) \times  \O ) \cap  L^\infty((0,T) \times \p \O ) .
\Ee
This proves \eqref{EnBnconverge}. Also, from \eqref{fellnmiterate1}, \eqref{iterate2},
\[
\| \langle v\rangle^{4 + \delta} f^m(t) -  \langle v \rangle^{4 + \delta} f^n)(t) \|_{L^\infty((0,T) \times \bar \O ) }   \le M_2 \frac{C^{l-1} t^{l-1}}{(l-2)!},
\]
therefore we get \eqref{fnconverge}.

Now, take any $\phi(t,x,v) \in C_c^\infty( [0,T) \times \bar \O \times \mathbb R^3$ with $\text{supp } \phi   \subset \{ [0, T) \times \bar \O \times \mathbb R^3 \} \setminus \{ (0 \times \gamma ) \cup (0,T) \times \gamma_0 \} $, from \eqref{fellseq}, we have
\Be \label{weakfellVM}
\begin{split}
& \int_{\O \times \mathbb R^3 } f_0 \phi (0) dv dt +  \int_0^T \int_{\O \times \mathbb R^3}  f^\ell \left(  \p_t \phi + \hat v \cdot \nabla_x \phi +   \mathfrak F^{\ell-1}   \cdot \nabla_v \phi \right) dv dx dt
\\ = & \int_0^T \int_{\gamma_+} \phi f^\ell \hat v_3 dv dS_x +  \int_0^T \int_{\gamma_+ }  \left( - c_\mu   \int_{u_3 > 0 }  \mu(u) \phi (t,x,u )\hat u_3 du  \right) \hat v_3   f^\ell  \, dv dS_x.
\end{split}
\Ee
Because of the strong convergence \eqref{EnBnconverge}, \eqref{fnconverge}, we have that as $\ell \to \infty$, each term in \eqref{weakfellVM} goes to the corresponding terms with $f^\ell$ replaced by $f$ and $\mathfrak F^\ell$ replaced by $\mathfrak F$. Therefore we conclude that $(f,E,B)$ satisfy \eqref{weakf}.

Next, using the same argument as in \eqref{Maxwellell}-\eqref{EBellweak2}, we get $(f,E,B)$ satisfy \eqref{Maxweak1} and \eqref{Maxweak2}. 
Therefore, we conclude that $(f,E,B)$ is a (weak) solution of the RVM system \eqref{VMfrakF1}--\eqref{rhoJ1} with diffuse BC \eqref{diffuseBC}.
\end{proof}

In the next lemma, we consider the regularity of the solution.

\begin{lemma} \label{fEBreg}
Let $\alpha(t,x,v)$ be defined as in \eqref{alphadef}. The solution $(f,E,B)$ obtained in Lemma \ref{fEBsollemma} satisfies
\Be \label{pfbdlimit}
 \| \langle v \rangle^{4 + \delta } \nabla_{x_\parallel } f(t) \|_\infty  + \| \langle v \rangle^{5 + \delta }  \alpha^{} \p_{x_3 } f(t) \|_\infty + \|  \langle v \rangle^{4 + \delta }  \nabla_v f(t) \|_\infty < \infty,
\Ee
and
\Be \label{pEBbdlimit}
\| \p_t E(t) \|_\infty +  \| \p_t B(t) \|_\infty + \| \nabla_{x } E(t) \|_\infty +\| \nabla_{x } B(t) \|_\infty   < \infty.
\Ee
\end{lemma}
\begin{proof}
From the $L^\infty$ strong convergence \eqref{EnBnconverge}, and the uniform-in-$\ell$ bound \eqref{flElBldsqbd}, we can pass the limit up to subsequence if necessary and get the weak$-*$ convergence
\Be \label{dEnBncov}
\p_t E^\ell  \overset{\ast}{\rightharpoonup} \p_t E , \ \nabla_x E^\ell  \overset{\ast}{\rightharpoonup} \nabla_x E, \  \p_t B^\ell  \overset{\ast}{\rightharpoonup} \p_t B, \  \nabla_x B^\ell  \overset{\ast}{\rightharpoonup} \nabla_x B \text{ in } L^\infty((0,T) \times \O ),
\Ee
and
\Be \label{dfnconverge}
\langle v \rangle ^{4 + \delta} \nabla_{x_\parallel } f^\ell   \overset{\ast}{\rightharpoonup}  \langle v \rangle^{4 + \delta}\nabla_{x_\parallel } f, \  \langle v \rangle ^{4 + \delta} \nabla_{v } f^\ell   \overset{\ast}{\rightharpoonup}  \langle v \rangle^{4 + \delta}\nabla_{v } f \text{ in } L^\infty((0,T) \times \O \times \mathbb R^3 ).
\Ee
Then using the same argument as in \eqref{ap3fellcovin}--\eqref{px3alphacov3}, we also have
\Be \label{ap3fellcov}
\langle v \rangle ^{5 + \delta}  \alpha^{\ell-1} \p_{x_3}  f^\ell   \overset{\ast}{\rightharpoonup}  \langle v \rangle^{4 + \delta} \alpha \p_{x_3} f \text{ in } L^\infty((0,T) \times \O \times \mathbb R^3 ).
\Ee

Therefore, from using the weak lower semi-continuity of the weak-$*$ convergence \eqref{dEnBncov}, \eqref{dfnconverge}, \eqref{ap3fellcov}, and the uniform-in-$\ell$ bound \eqref{flElBldsqbd}, we conclude \eqref{pfbdlimit}, \eqref{pEBbdlimit}.

\end{proof}

Next, we prove the uniqueness of the solutions of the RVM system \eqref{VMfrakF1}--\eqref{rhoJ1}, \eqref{diffuseBC}.

\begin{lemma} \label{VMuniqlemma}
Suppose  $(f,E_f, B_f)$ and $(g, E_g, B_g)$ are solutions to the VM system \eqref{VMfrakF1}--\eqref{rhoJ1}, \eqref{diffuseBC} with $f(0) = g(0)$, $E_f(0) = E_g(0)$, $B_f(0) = B_g(0)$, and that 
\[
E_f, B_f, E_g, B_g \in W^{1,\infty}((0,T) \times \O ), \ \nabla_x \rho_{f}, \nabla_x J_f,  \p_t J_f , \nabla_x \rho_{g}, \nabla_x J_g,  \p_t J_g \in  L^\infty((0,T); L_{\text{loc}}^p(\O)) \text{ for some } p>1.
\]
And
\Be \label{dvfgbd}
\sup_{0 < t < T} \| \langle v \rangle^{5+ \delta} \nabla_v f(t) \|_\infty <\infty, \sup_{0 < t < T} \| \langle v \rangle^{5+ \delta} \nabla_v g(t) \|_\infty <\infty.
\Ee
Then $f = g, E_f = E_g, B_f = B_g$.
\end{lemma}
\begin{proof}
The difference function $f-g $ satisfies
\Be \label{fminusgeq}
\begin{split}
(\p_t + \hat v \cdot \nabla_x + \mathfrak F_f \cdot \nabla_v)(f-g) = (\mathfrak F_g - \mathfrak F_f ) \cdot \nabla_v g 
\\ (f-g)(0) = 0, \, (f- g )|_{\gamma_-  } =  c_\mu \mu(v) \int_{u_3 < 0 }  -  (f - g) (t,x,  u ) \hat u _3 du,
\end{split} 
\Ee
where
\[
\mathfrak F_f = E_f + E_{\text{ext}} + \hat v \times ( B_f + B_{\text{ext}}) - g \mathbf e_3 , \, \mathfrak F_g = E_g +E_{\text{ext}} + \hat v \times ( B_g + B_{\text{ext} } ) - g \mathbf e_3, 
\]
so
\Be \label{mathfrakFfg}
\mathfrak F_g - \mathfrak F_f = E_f - E_g + \hat v \times (B_f - B_g ).
\Ee
From Lemma \ref{Maxtowave} we have $E_{f,1} - E_{g,1} , E_{f,2} - E_{g,2}, B_{f,3} - B_{g,3}$ solve the wave equation with the Dirichlet boundary condition \eqref{waveD} in the sense of \eqref{waveD_weak} with \begin{align}
u_0 = 0, \  u_1 = 0 , \ G = -4\pi \p_{x_i} (\rho_f - \rho_g) - 4 \pi \p_t (J_{f,i} - J_{g,i} ), \ g = 0 , \ \ \text{for} \  E_{f,i} - E_{g,i},  i =1,2, \label{E12sol_A} \\
 u_0 = 0, \ u_1 = 0, \   G =  4 \pi (\nabla_x \times (J_f -J_g) )_3, \  g = 0, \ \ \text{for} \ B_{f,3}- B_{g,3},  \label{B3sol_A}
 \end{align} 
respectively. And $E_{f,3} - E_{g,3}, B_{f,1} - B_{g,1}, B_{f,2}- B_{g,2}$ solve the wave equation with the Neumann boundary condition \eqref{waveNeu}  in the sense of \eqref{waveinner} \text{ with }
\begin{align}
u_0 = 0, \  u_1 = 0 , \ G = -4\pi \p_{x_3} ( \rho_f - \rho_g) - 4 \pi \p_t (J_{f,3} - J_{g,3} ) , \ g = - 4\pi (\rho_f - \rho_g), \ \ \text{for} \  E_{f,3} - E_{g,3}, \label{E3sol_A} \\
 u_0 = 0, \ u_1 = 0, \   G =  4 \pi (\nabla_x \times (J_f - J_g) )_i, \  g = (-1)^{i+1} 4 \pi (J_{f,{\underline i}} - J_{g, \underline i } ), \ \ \text{for} \ B_{f,i} - B_{j,i}, \ i=1,2, \label{B12sol_A}
 \end{align} 
respectively. Therefore, from Lemma \ref{wavesol} and Lemma \ref{wavesolD}, we know that $E_f - E_g$  and $B_f - B_g$ would have the form of
\Be \label{EBdiffform}
\begin{split}
& E_f - E_g = \eqref{Eesttat0pos} + \dots + \eqref{Eest3bdrycontri}, \  B_f -B_g = \eqref{Besttat0pos} + \dots + \eqref{Bestbdrycontri},
\\ & \text{ with } E_0, B_0 \text{ changes to } 0, \text{ and } f \text{ changes to } f -g.
\end{split}
\Ee 

Now consider the characteristics
\[
\begin{split}
\dot X_f(s;t,x,v) = & \hat V_f(s;t,x,v) ,
\\ \dot V_f(s;t,x,v) = & \mathfrak F_f(s, X_f(s;t,x,v), V_f(s;t,x,v) ) .
\end{split}
\]
Then from \eqref{fminusgeq}, same as \eqref{fnmiterate2}, we obtain
\Be \label{fnmiterate2final}
\begin{split}
& |  \langle v \rangle^{4 + \delta } (f - g)(t,x,v) | 
\\  \le &   C_{1}  \int_{\prod_{j=1}^{k-1} \mathcal V_j} \sum_{ i =1}^{k-1}  \int_{\max \{ t^{}_{i} , 0 \} }^{t^{ }_{ i-1} }    \mathbf{1}_{ \{  t^{}_i \le 0 < t^{ }_{ i-1} \} }
\\ & \quad \quad \quad \quad  \times |   \langle V_f(s) \rangle^{4 + \delta }  ( \mathfrak F_g - \mathfrak F_f ) \cdot \nabla_v f^{}  )(s, X_f^{}(s), V_f^{}(s) ) | ds  d \Sigma_{i}^{k-1}
\\ &  +  C_1 \int_{\prod_{j=1}^{k-1} \mathcal V_j}  \textbf{1}_{\{t_{}^{m - (k-1)} > 0 \}} \int_{\mathcal V_k}  |  (f - g )( t_k^{}, x_k^{ }, v_{k} ) |  d v_k d \Sigma_{k-1}^{k-1}.
\end{split}
\Ee
So using \eqref{tailtermsmall} and \eqref{1over2k}, we have
\Be \label{fgdiffrep}
\sup_{ 0 \le s \le t } \| \langle v \rangle^{5 + \delta} (f-g)(s) \|_\infty \le C   \int^t_0  \|  (\mathfrak F_g - \mathfrak F_f )(s) \|_\infty \|\langle v \rangle^{5 + \delta} \nabla_v g (s) \|_\infty  ds.
\Ee
Now, from \eqref{EBdiffform} and the estimate in Lemma \ref{EBlinflemma}, we have
\Be \label{FgFfdiff}
\begin{split}
\|  (\mathfrak F_g - \mathfrak F_f )(s) \|_\infty \le &  \| (E_f - E_g )(s) \|_\infty + \| (B_f - B_g )(s) \|_\infty
\\ \le & C \sup_{0 \le s' \le s } \| \langle v \rangle^{5 + \delta} (f-g )(s' ) \|_\infty,
\end{split}
\Ee
and from the assumption \eqref{dvfgbd}, $\sup_{0 \le s \le t }   \|(1 + |v |^{5 + \delta } ) \nabla_v g (s) \|_\infty < C$. Therefore from \eqref{fgdiffrep} and \eqref{FgFfdiff}, we have
\Be
\sup_{0 \le s \le t } \|  \langle v \rangle^{5 +\delta} (f-g)(s) \|_\infty \le C' \int^t_{0 } \sup_{0 \le s' \le s } \|  \langle v \rangle^{5 +\delta}(f-g )(s' ) \|_\infty  ds.
\Ee
Therefore from Gronwall
\[
\sup_{0 \le s' \le t } \|  \langle v \rangle^{5 +\delta} (f-g)(s') \|_\infty \le   e^{C't}  \|  \langle v \rangle^{5 +\delta} (f-g)(0) \|_\infty = 0.
\]
Therefore we conclude that the solutions to \eqref{VMfrakF1}--\eqref{rhoJ1}, \eqref{diffuseBC} is unique.
\end{proof}

\begin{proof}[proof of Theorem \ref{main2}]
Using the sequence $f^\ell, E^\ell , B^\ell$ constructed in \eqref{fellseq}, \eqref{ElBl}, we have from Lemma \ref{fEBsollemma} that the limit $(f,E,B)$ is a solution to the VM system \eqref{VMfrakF1}--\eqref{rhoJ1}, \eqref{diffuseBC}. This proves the existence. From Lemma \ref{fEBreg}, we have the regularity estimate \eqref{inflowfreg}, \eqref{inflowEBreg}. And from Lemma \ref{VMuniqlemma}, we conclude the uniqueness.
\end{proof}

\section{Specular BC} \label{chapspec}
In this section we consider the solution $f$ of the Vlasov-Maxwell system \eqref{VMfrakF1} satisfies the specular reflection boundary condition \eqref{spec}. We have the following a priori estimate for $f$.
\begin{proposition} \label{specBCprop}
Let $(f,E,B)$ be a solution of \eqref{VMfrakF1}--\eqref{rhoJ1}, \eqref{spec}. Suppose the fields satisfies \eqref{gbig10}, and
\Be \label{pEBassspec}
\sup_{0 \le t \le T}  \left(\| \nabla_{x} E(t)  \|_\infty + \| \nabla_{x} B(t)  \|_\infty \right) < \infty.
\Ee
Assume that for some $\delta > 0$, and some $C > 0 $ such that 
\[
\begin{split}
 \| \langle v \rangle^{5 + \delta }    e^{\frac{C}{\sqrt{ \alpha \langle v \rangle } } }   \nabla_{x} f_0 \|_\infty +  \| \langle v \rangle^{5 + \delta }    e^{\frac{C}{\sqrt{ \alpha \langle v \rangle } } }      \nabla_v f_0 \|_\infty  < \infty,
\end{split}
\]
then there exists a $0 < T \ll 1$ small enough such that
\Be \label{specdfbd}
\begin{split}
& \sup_{0 \le t \le T} \left( \|  \langle v \rangle^{4+\delta}   \nabla_{x} f(t) \|_\infty   +  \| \langle v \rangle^{4+\delta}   \nabla_{v} f(t) \|_\infty \right) 
\\ & + \sup_{0 \le t \le T} \left( \|  \langle v \rangle^{4+\delta}   \nabla_{x} f(t) \|_{L^\infty(\gamma \setminus \gamma_0 ) }   +  \| \langle v \rangle^{4+\delta}   \nabla_{v} f(t) \|_{L^\infty(\gamma \setminus \gamma_0 ) }  \right)  <  \infty.
\end{split}
\Ee
\end{proposition}

Let $(t,x,v) \in (0,T) \times \bar \O \times \mathbb R^3$. Recall the definition of $\tb(t,x,v), \xb(t,x,v), \vb(t,x,v)$ in \eqref{tb}. Now let $(t^{0}, x^{0}, v^{0}) = (t,x,v).$ We define the specular cycles, for $\ell\geq 0,$
\[
(t^{\ell+1}, x^{\ell+1},v^{\ell+1}) = (t^{\ell}-t_{\mathbf{b}}(t^\ell, x^{\ell}, v^{\ell}), x_{\mathbf{b}}(t^\ell,x^{\ell},v^{\ell}),   v_{\mathbf b}(t^\ell, x^{\ell}, v^{\ell}) - 2 v_{\mathbf b , 3}(t^\ell, x^{\ell}, v^{\ell}) \mathbf e_3 ).
\]
And we define the generalized characteristics for the specular BC as
\begin{equation} \label{cycles} 
\begin{split}
X_{\mathbf{cl}}(s;t,x,v)   \ = \ \sum_{\ell} \mathbf{1}_{[t^{\ell+1},t^{%
\ell})}(s) X(s;t^\ell, x^\ell , v^\ell ),  \ \ 
V_{\mathbf{cl}}(s;t,x,v)   \ = \ \sum_{\ell} \mathbf{1}_{[t^{\ell+1},t^{%
\ell})}(s)  V(s;t^\ell, x^\ell , v^\ell ).
\end{split}%
\end{equation}

The key to prove Proposition \ref{specBCprop} is the following estimate for the derivative of the characteristics under the specular reflection.

\begin{lemma} \label{dXVcl}
For any $(t,x,v) \in (0,T) \times \bar \O \times \mathbb R^3$, and $0 \le s \le t $, let $\p_{\mathbf e } \in \{ \nabla_x , \nabla_v \} $, then for some $C_1 \gg 1$, we have
\begin{equation}\label{lemma_Dxv}
\begin{split}
| \p_{\mathbf e } X_{\mathbf{cl}}(s;t,x,v)| & \le C_1 \langle v \rangle  e^{\frac{C_1}{\sqrt{ \alpha(t,x,v) \langle v \rangle } } }   , 
\\ | \p_{\mathbf e } V_{\mathbf{cl}}(s;t,x,v)| &  \le  C_1 \langle v \rangle  e^{\frac{C_1}{\sqrt{ \alpha(t,x,v) \langle v \rangle } } }.
\end{split}
\end{equation}
\end{lemma}

%

\begin{proof}
We need to estimate along the bounces:
\begin{equation}\label{chain}
\begin{split}
& \frac{\partial ( X_{\mathbf{cl}}(s;t,x,v),V_{\mathbf{cl}}(s;t,x,v))}{\partial (x,v)}\\
 = &\underbrace{\frac{\partial ( X_{\mathbf{cl}}(s), V_{\mathbf{cl}}(s))}{\partial (t^{\ell_{*}},  {x}_{\parallel_{\ell_{*}}}^{\ell_{*}}, {v}_{3_{\ell_{*}}}^{\ell_{*}}, {v}_{\parallel_{\ell_{*}}}^{\ell_{*}})}  }_{\text{from the last bounce to the }s-\text{plane}} 
\times \underbrace{\prod_{\ell=1}^{\ell_*}  \frac{\partial (t^{\ell+1}, {x}_{\parallel}^{\ell+1}, {v}_{3 }^{\ell+1},{v}_{\parallel }^{\ell+1})}{\partial (t^{\ell}, {x}_{\parallel }^{\ell}, {v}_{3 }^{\ell},{v}_{\parallel }^{\ell})} }_{\text{ intermediate groups}}
 \times  \underbrace{ \frac{\partial (t^{1},  {x}_{\parallel_{1}}^{1}, {v}_{3_{1}}^{1}, {v}_{\parallel_{1}}^{1})}{\partial (x,v)}}_{\text{from the } t-\text{plane to the first bounce} }.
\end{split}
\end{equation}
%
%
%
%
%
%
%

We first find out the matrix of derivatives in the intermediate groups from the $\ell$-th bounce to the $(\ell + 1)$-th bounce:
\Be \label{Jell1matrix}
\begin{split}
J^{\ell+1}_{\ell}&:= \frac{\partial (t^{\ell+1}, {x}_{\parallel}^{\ell+1}, {v}_{3 }^{\ell+1},{v}_{\parallel }^{\ell+1})}{\partial (t^{\ell}, {x}_{\parallel }^{\ell}, {v}_{3 }^{\ell},{v}_{\parallel }^{\ell})}.
\end{split}
\Ee
We have
\Be \label{ellellplus1}
(t^\ell - t^{\ell +1 } ) \hat v_3^\ell = \int_{t^{\ell +1}}^{t^\ell } \int_s^{t^\ell } \hat{\mathfrak F}_3(\tau ) d\tau ds.
\Ee
Taking $\frac{\p}{\p t^\ell }  \eqref{ellellplus1} $ gives
\[ \begin{split}
(1 - \frac{\p t^{\ell+1} }{\p t^\ell }  ) \hat v_3^\ell = -  \frac{\p t^{\ell+1} }{\p t^\ell }    \int_{t^{\ell+1}}^{t^\ell } \hat{\mathfrak F}_3(\tau ) d\tau  + \int_{t^{\ell +1}}^{t^\ell }  \hat{\mathfrak F}_3( t^\ell )  ds + \int_{t^{\ell +1}}^{t^\ell } \int_s^{t^\ell } \frac{ \p  \hat{\mathfrak F}_3(\tau )}{ \p_{t^\ell } }  d\tau ds, 
\end{split}
\]
so
\[ \begin{split}
- \hat v_3^{\ell +1 } \frac{\p t^{\ell+1} }{\p t^\ell } = & - ( \hat v_3^\ell -  \int_{t^{\ell +1}}^{t^\ell }  \hat{\mathfrak F}_3( t^\ell )  ds )  + \int_{t^{\ell +1}}^{t^\ell } \int_s^{t^\ell } \frac{ \p  \hat{\mathfrak F}_3(\tau )}{ \p_{t^\ell } }  d\tau ds
\\  = &  - \hat v_3^{\ell+1}  + \int_{t^{\ell+1}}^{t^\ell }  (  \hat{\mathfrak F}_3( t^\ell )  -   \hat{\mathfrak F}_3( s )  ) ds  + \int_{t^{\ell +1}}^{t^\ell } \int_s^{t^\ell } \frac{ \p  \hat{\mathfrak F}_3(\tau )}{ \p_{t^\ell } }  d\tau ds
\\  = &  - \hat v_3^{\ell+1}   + \int_{t^{\ell +1}}^{t^\ell } \int_s^{t^\ell } \left(  \frac{ \p  \hat{\mathfrak F}_3(\tau )}{ \p_{\tau} }   + \frac{ \p  \hat{\mathfrak F}_3(\tau )}{ \p_{t^\ell } } \right)  d\tau ds  , 
\end{split}
\]
thus
\Be \label{ptltlplus1}
\frac{\p t^{\ell+1} }{\p t^\ell }  = 1 - \frac{1}{ \hat v_3^{\ell +1 } }   \int_{t^{\ell +1}}^{t^\ell } \int_s^{t^\ell }  \left(  \frac{ \p  \hat{\mathfrak F}_3(\tau )}{ \p_{\tau} }   + \frac{ \p  \hat{\mathfrak F}_3(\tau )}{ \p_{t^\ell } } \right) d\tau ds.
\Ee
Taking $\frac{\p}{\p t^{\ell } } $ derivative to
\Be \label{xell1rep}
x_{\parallel}^{\ell + 1 } = x_\parallel^\ell - (t^\ell - t^{\ell +1 } ) \hat v_\parallel^\ell + \int_{t^{\ell + 1 }}^{t^\ell } \int_s^{t^\ell } \hat{ \mathfrak F}_\parallel (\tau ) d\tau ds,
\Ee
we get
\Be
\begin{split}
\frac{ \p x_\parallel^{\ell+1}}{\p t^\ell } = & - ( 1 - \frac{\p t^{\ell +1}}{\p t^\ell } ) \hat v_\parallel^\ell -  \frac{\p t^{\ell +1}}{\p t^\ell }  \int_{t^{\ell+1}}^{t^\ell }  \hat {\mathfrak F}_\parallel(\tau) d\tau +  \int_{t^{\ell+1}}^{t^\ell }  \hat {\mathfrak F}_\parallel(t^\ell) d\tau + \int_{t^{\ell +1}}^{t^\ell } \int_s^{t^\ell } \frac{ \p  \hat{\mathfrak F}_\parallel (\tau )}{ \p_{t^\ell } }  d\tau ds
\\ & = - \hat v_\parallel^\ell + \frac{\p t^{\ell +1}}{\p t^\ell }  \hat v_\parallel^{\ell + 1 }   +  \int_{t^{\ell+1}}^{t^\ell }  \hat {\mathfrak F}_\parallel(t^\ell) d\tau + \int_{t^{\ell +1}}^{t^\ell } \int_s^{t^\ell } \frac{ \p  \hat{\mathfrak F}_\parallel (\tau )}{ \p_{t^\ell } }  d\tau ds
\\ & =  - \hat v_\parallel^\ell +   \hat v_\parallel^{\ell + 1 } - \frac{\hat v_\parallel^{\ell+1} }{ \hat v_3^{\ell +1 } }   \int_{t^{\ell +1}}^{t^\ell } \int_s^{t^\ell }  \left(  \frac{ \p  \hat{\mathfrak F}_3(\tau )}{ \p_{\tau} }   + \frac{ \p  \hat{\mathfrak F}_3(\tau )}{ \p_{t^\ell } } \right) d\tau ds  +  \int_{t^{\ell+1}}^{t^\ell }  \hat {\mathfrak F}_\parallel(t^\ell) d\tau + \int_{t^{\ell +1}}^{t^\ell } \int_s^{t^\ell } \frac{ \p  \hat{\mathfrak F}_\parallel (\tau )}{ \p_{t^\ell } }  d\tau ds
\\ & = \int_{t^{\ell +1}}^{t^\ell } \int_s^{t^\ell } \left(  \frac{ \p  \hat{\mathfrak F}_\parallel (\tau )}{ \p_{\tau} }   + \frac{ \p  \hat{\mathfrak F}_\parallel (\tau )}{ \p_{t^\ell } } \right)  d\tau ds  - \frac{\hat v_\parallel^{\ell+1} }{ \hat v_3^{\ell +1 } }   \int_{t^{\ell +1}}^{t^\ell } \int_s^{t^\ell }  \left(  \frac{ \p  \hat{\mathfrak F}_3(\tau )}{ \p_{\tau} }   + \frac{ \p  \hat{\mathfrak F}_3(\tau )}{ \p_{t^\ell } } \right) d\tau ds.
\end{split}
\Ee
And taking $\frac{\p}{\p t^{\ell } } $ derivative to
\Be \label{vell1rep}
v_\parallel^{\ell+1} = v_\parallel^\ell - \int_{t^{\ell+1}}^{t^\ell} \mathfrak F_\parallel(s) ds,
\Ee
we get
\Be
\begin{split}
\frac{ \p v_\parallel^{\ell+1}}{\p t^\ell } = & - \mathfrak F_\parallel(t^\ell) + \frac{\p t^{\ell +1}}{\p t^\ell }   \mathfrak F_\parallel(t^{\ell+1}) - \int_{t^{\ell+1}}^{t^\ell} \frac{ \p \mathfrak F_\parallel(s) }{\p t^\ell }  ds
\\ = &   - \mathfrak F_\parallel(t^\ell) + F_\parallel(t^{\ell+1}) -  \frac{ \mathfrak F_\parallel(t^{\ell+1} )  }{ \hat v_3^{\ell +1 } }   \int_{t^{\ell +1}}^{t^\ell } \int_s^{t^\ell }  \left(  \frac{ \p  \hat{\mathfrak F}_3(\tau )}{ \p_{\tau} }   + \frac{ \p  \hat{\mathfrak F}_3(\tau )}{ \p_{t^\ell } } \right) d\tau ds - \int_{t^{\ell+1}}^{t^\ell} \frac{ \p \mathfrak F_\parallel(s) }{\p t^\ell }  ds
\\ = &  -  \frac{ \mathfrak F_\parallel(t^{\ell+1} )  }{ \hat v_3^{\ell +1 } }   \int_{t^{\ell +1}}^{t^\ell } \int_s^{t^\ell }  \left(  \frac{ \p  \hat{\mathfrak F}_3(\tau )}{ \p_{\tau} }   + \frac{ \p  \hat{\mathfrak F}_3(\tau )}{ \p_{t^\ell } } \right) d\tau ds - \int_{t^{\ell+1}}^{t^\ell }  \left(  \frac{ \p  {\mathfrak F}_\parallel (s )}{ \p_{s} }   + \frac{ \p  {\mathfrak F}_\parallel (s )}{ \p_{t^\ell } } \right)  ds.
\end{split}
\Ee
Similarly, taking  taking $\frac{\p}{\p t^{\ell } } $ derivative to
\Be \label{v3ell1rep}
v_3^{\ell+1} = -v_3^\ell - \int_{t^{\ell+1}}^{t^\ell} \mathfrak F_3(s) ds,
\Ee
we get
\Be \label{v3tlderi}
\begin{split}
\frac{ \p v_3^{\ell+1}}{\p t^\ell } = & - \mathfrak F_3(t^\ell) + \frac{\p t^{\ell +1}}{\p t^\ell }   \mathfrak F_3(t^{\ell+1}) - \int_{t^{\ell+1}}^{t^\ell} \frac{ \p \mathfrak F_3(s) }{\p t^\ell }  ds
\\ = &   - \mathfrak F_3(t^\ell) + F_3(t^{\ell+1}) -  \frac{ \mathfrak F_3(t^{\ell+1} )  }{ \hat v_3^{\ell +1 } }   \int_{t^{\ell +1}}^{t^\ell } \int_s^{t^\ell }  \left(  \frac{ \p  \hat{\mathfrak F}_3(\tau )}{ \p_{\tau} }   + \frac{ \p  \hat{\mathfrak F}_3(\tau )}{ \p_{t^\ell } } \right) d\tau ds - \int_{t^{\ell+1}}^{t^\ell} \frac{ \p \mathfrak F_3(s) }{\p t^\ell }  ds
\\ = &  -  \frac{ \mathfrak F_3(t^{\ell+1} )  }{ \hat v_3^{\ell +1 } }   \int_{t^{\ell +1}}^{t^\ell } \int_s^{t^\ell }  \left(  \frac{ \p  \hat{\mathfrak F}_3(\tau )}{ \p_{\tau} }   + \frac{ \p  \hat{\mathfrak F}_3(\tau )}{ \p_{t^\ell } } \right) d\tau ds - \int_{t^{\ell+1}}^{t^\ell }  \left(  \frac{ \p  {\mathfrak F}_3 (s )}{ \p_{s} }   + \frac{ \p  {\mathfrak F}_3 (s )}{ \p_{t^\ell } } \right)  ds.
\end{split}
\Ee

Now let's calculate the matrix $J_\ell^{\ell+1}$ in \eqref{Jell1matrix}. Taking $\p_{\mathbf e } \in \{ \p_{x_\parallel^\ell } , \p_{v_3^\ell }, \p_{v_\parallel^\ell } \}$ derivatives to \eqref{ellellplus1} we get
\Be \label{partialtell1}
\begin{split}
\frac{ \p t^{\ell+1}}{\p x_\parallel^\ell }  = &  - \frac{1}{\hat v_3^{\ell +1} }  \int_{t^{\ell +1}}^{t^\ell } \int_s^{t^\ell } \frac{ \p  \hat{\mathfrak F}_3  (\tau )  }{\p_{x_\parallel^\ell } }d\tau ds, 
\\ \frac{ \p t^{\ell+1}}{\p v_3^\ell }  = &   \frac{t^\ell - t^{\ell +1 } }{\hat v_3^{\ell +1 } } \frac{ \p \hat v_3^\ell }{\p v_3^\ell } +  \frac{1}{\hat v_3^{\ell +1} }  \int_{t^{\ell +1}}^{t^\ell } \int_s^{t^\ell } \frac{ \p  \hat{\mathfrak F}_3  (\tau )  }{\p_{v_3^\ell } }d\tau ds
\\ = &   \frac{(t^\ell - t^{\ell +1 }  )}{\hat v_3^{\ell +1 } } \frac{ ( 1 - (\hat v_3^\ell )^2 )}{\langle v^\ell \rangle } +  \frac{1}{\hat v_3^{\ell +1} }  \int_{t^{\ell +1}}^{t^\ell } \int_s^{t^\ell } \frac{ \p  \hat{\mathfrak F}_3  (\tau )  }{\p_{v_3^\ell } }d\tau ds,
\\ \frac{ \p t^{\ell+1}}{\p v_\parallel^\ell }  = & - \frac{1}{\hat v_3^{\ell +1} }  \int_{t^{\ell +1}}^{t^\ell } \int_s^{t^\ell } \frac{ \p  \hat{\mathfrak F}_3  (\tau )  }{\p_{v_\parallel^\ell } }d\tau ds.
\end{split}
\Ee
Taking $\p_{\mathbf e } \in \{ \p_{x_\parallel^\ell } , \p_{v_3^\ell }, \p_{v_\parallel^\ell } \}$ derivatives to \eqref{xell1rep} we get
\Be \label{partialxell1}
\begin{split}
\frac{ \p x_\parallel^{\ell+1}}{\p {x_\parallel } ^\ell } = & \textbf{Id}_{2,2} +  \frac{ \p t^{\ell+1}}{\p x_\parallel^\ell } \hat v_\parallel^{\ell +1 } +  \int_{t^{\ell +1}}^{t^\ell } \int_s^{t^\ell } \frac{ \p  \hat{\mathfrak F}_\parallel  (\tau )  }{\p_{x_\parallel^\ell } }d\tau ds,
\\ \frac{ \p x_\parallel^{\ell+1}}{\p {v_3 } ^\ell } = &  \frac{ \p t^{\ell+1}}{\p v_3^\ell } \hat v_\parallel^{\ell +1 } +  \int_{t^{\ell +1}}^{t^\ell } \int_s^{t^\ell } \frac{ \p  \hat{\mathfrak F}_\parallel  (\tau )  }{\p_{v_3^\ell } }d\tau ds,
\\ \frac{ \p x_\parallel^{\ell+1}}{\p {v_\parallel } ^\ell } = &  \frac{ \p t^{\ell+1}}{\p v_\parallel^\ell } \hat v_\parallel^{\ell +1 } - (t^\ell - t^{\ell +1 } ) \frac{ \p \hat v_\parallel^\ell }{\p v_\parallel^\ell }  +  \int_{t^{\ell +1}}^{t^\ell } \int_s^{t^\ell } \frac{ \p  \hat{\mathfrak F}_\parallel  (\tau )  }{\p_{v_\parallel^\ell } }d\tau ds
\\ = & \frac{ \p t^{\ell+1}}{\p v_\parallel^\ell } \hat v_\parallel^{\ell +1 } - (t^\ell - t^{\ell +1 } ) \frac{1 - (\hat v_\parallel^\ell ) ^2 }{\langle v^\ell \rangle }  +  \int_{t^{\ell +1}}^{t^\ell } \int_s^{t^\ell } \frac{ \p  \hat{\mathfrak F}_\parallel  (\tau )  }{\p_{v_\parallel^\ell } }d\tau ds .
\end{split}
\Ee
Taking $\p_{\mathbf e } \in \{ \p_{x_\parallel^\ell } , \p_{v_3^\ell }, \p_{v_\parallel^\ell } \}$ derivatives to \eqref{vell1rep} we get
\Be \label{partialvell1}
\begin{split}
\frac{ \p v_\parallel^{\ell+1}}{\p {x_\parallel } ^\ell } = &   \frac{ \p t^{\ell+1}}{\p x_\parallel^\ell }  \mathfrak F_\parallel (t^{\ell +1 } ) -  \int_{t^{\ell +1}}^{t^\ell }  \frac{ \p  {\mathfrak F}_\parallel  (s )  }{\p_{x_\parallel^\ell } } ds,
\\ \frac{ \p v_\parallel^{\ell+1}}{\p {v_3 } ^\ell } = &   \frac{ \p t^{\ell+1}}{\p v_3^\ell }  \mathfrak F_\parallel (t^{\ell +1 } ) -  \int_{t^{\ell +1}}^{t^\ell }  \frac{ \p  {\mathfrak F}_\parallel  (s )  }{\p_{v_3^\ell } } ds,
\\ \frac{ \p v_\parallel^{\ell+1}}{\p {v_\parallel } ^\ell } = &   \textbf{Id}_{2,2} + \frac{ \p t^{\ell+1}}{\p v_\parallel^\ell }  \mathfrak F_\parallel (t^{\ell +1 } ) -  \int_{t^{\ell +1}}^{t^\ell }  \frac{ \p  {\mathfrak F}_\parallel  (s )  }{\p_{v_\parallel^\ell } } ds.
\end{split}
\Ee
And finally, taking $\p_{\mathbf e } \in \{ \p_{x_\parallel^\ell } , \p_{v_3^\ell }, \p_{v_\parallel^\ell } \}$ derivatives to \eqref{v3ell1rep} we get
\Be \label{partialv3ell1}
\begin{split}
\frac{ \p v_3^{\ell+1}}{\p {x_\parallel } ^\ell } = &   \frac{ \p t^{\ell+1}}{\p x_\parallel^\ell }  \mathfrak F_3 (t^{\ell +1 } ) -  \int_{t^{\ell +1}}^{t^\ell }  \frac{ \p  {\mathfrak F}_3  (s )  }{\p_{x_\parallel^\ell } } ds,
\\ \frac{ \p v_3^{\ell+1}}{\p {v_3 } ^\ell } = &  -1 -   \frac{ \p t^{\ell+1}}{\p v_3^\ell }  \mathfrak F_3 (t^{\ell +1 } ) +  \int_{t^{\ell +1}}^{t^\ell }  \frac{ \p  {\mathfrak F}_3  (s )  }{\p_{v_3^\ell } } ds,
\\ \frac{ \p v_3^{\ell+1}}{\p {v_\parallel } ^\ell } = &    \frac{ \p t^{\ell+1}}{\p v_\parallel^\ell }  \mathfrak F_3 (t^{\ell +1 } ) -  \int_{t^{\ell +1}}^{t^\ell }  \frac{ \p  {\mathfrak F}_3  (s )  }{\p_{v_\parallel^\ell } } ds.
\end{split}
\Ee
For the estimate, from \eqref{ptltlplus1} and that $|  \frac{ \p  \hat{\mathfrak F}_3(\tau )}{ \p_{\tau} }   + \frac{ \p  \hat{\mathfrak F}_3(\tau )}{ \p_{t^\ell } } | \lesssim \frac{1}{ \langle v \rangle } $, we obtain
\Be
\begin{split}
| \frac{\p t^{\ell+1} }{\p t^\ell } | & \le  1 + M \frac{ |t^\ell - t^{\ell +1} |^2 }{ \hat v_3^{\ell +1 } \langle v \rangle },
\\ | \frac{\p x_\parallel^{\ell+1} }{\p t^\ell } |  & \le M   \frac{ |t^\ell - t^{\ell +1} |^2 }{ \hat v_3^{\ell +1 } \langle v \rangle },
\\ | \frac{\p v_\parallel^{\ell+1} }{\p t^\ell } |  & \le M(  \frac{ |t^\ell - t^{\ell +1} |^2 }{ \hat v_3^{\ell +1 } \langle v \rangle } + \frac{|t^\ell - t^{\ell +1 } | }{\langle v \rangle } ),
\\ | \frac{\p v_3^{\ell+1} }{\p t^\ell } |  &  \le M(  \frac{ |t^\ell - t^{\ell +1} |^2 }{ \hat v_3^{\ell +1 } \langle v \rangle } + \frac{ |t^\ell - t^{\ell +1 } | }{\langle v \rangle} ).
\end{split}
\Ee

Next, we estimate $\frac{ \p  \hat{\mathfrak F}  (\tau )  }{\p_{x_\parallel^\ell } } $, $\frac{ \p  \hat{\mathfrak F}  (\tau )  }{\p_{v^\ell } }$, $\frac{ \p  {\mathfrak F}  (\tau )  }{\p_{x_\parallel^\ell } } $, $\frac{ \p  {\mathfrak F}  (\tau )  }{\p_{v^\ell } }$. Since
\[
\begin{split}
 \frac{ \p  \hat{\mathfrak F}  (\tau )  }{\p_{x_\parallel^\ell } }  =  \nabla_{x_\parallel}  \hat{\mathfrak F}(\tau ) \cdot \p_{x_\parallel^\ell } X_\parallel (\tau) + \p_{x_3}  \hat{\mathfrak F}(\tau ) \cdot \p_{x_\parallel^\ell } X_3 (\tau) + \nabla_v  \hat{\mathfrak F}(\tau ) \cdot \p_{x_\parallel^\ell }  V (\tau).
\end{split}
\]
From \eqref{pxviXVest}, \eqref{pxiXj}, \eqref{nablaxhatF}, \eqref{nablavhatF}, we have
\Be \notag
\begin{split}
| \frac{ \p  \hat{\mathfrak F}  (\tau )  }{\p_{x_\parallel^\ell } }|  \lesssim \frac{1}{\langle V(\tau ) \rangle^2 } |t^\ell - t^{\ell + 1 } | + \frac{1}{  \langle V(\tau ) \rangle }.
\end{split}
\Ee
And,
\Be \notag
\begin{split}
|  \frac{ \p  \hat{\mathfrak F}  (\tau )  }{\p_{v^\ell } } |   = |    \nabla_{x}  \hat{\mathfrak F}(\tau ) \cdot \p_{v^\ell } X (\tau)    + \nabla_v  \hat{\mathfrak F}(\tau ) \cdot \p_{v^\ell }  V (\tau) | \lesssim \frac{|t^\ell - t^{\ell +1} | }{ \langle V(\tau ) \rangle^2  }  + \frac{1}{ \langle V(\tau ) \rangle^2 }.
\end{split}
\Ee
Similarly, from \eqref{pxviXVest}, \eqref{pxiXj}, \eqref{nablaxfrakF} and \eqref{nablavfrakF},
\[
\begin{split}
|  \frac{ \p  {\mathfrak F}  (\tau )  }{\p_{x_\parallel^\ell } }  |  = &   | \nabla_{x_\parallel}  {\mathfrak F}(\tau ) \cdot \p_{x_\parallel^\ell } X_\parallel (\tau) + \p_{x_3}  {\mathfrak F}(\tau ) \cdot \p_{x_\parallel^\ell } X_3 (\tau) + \nabla_v  {\mathfrak F}(\tau ) \cdot \p_{x_\parallel^\ell }  V (\tau)|  \lesssim  \frac{1}{ \langle V(\tau) \rangle } |t^\ell - t^{\ell + 1 } | + 1,
\\ |  \frac{ \p  {\mathfrak F}  (\tau )  }{\p_{v^\ell } } |   = &  |    \nabla_{x}  {\mathfrak F}(\tau ) \cdot \p_{v^\ell } X (\tau)    + \nabla_v  {\mathfrak F}(\tau ) \cdot \p_{v^\ell }  V (\tau) | \lesssim \frac{|t^\ell - t^{\ell +1} | }{ \langle V(\tau ) \rangle   }  + \frac{1}{ \langle V(\tau ) \rangle }.
\end{split}
\]
Thus, we have
\Be  \label{phatFpxvellest}
\begin{split}
\int_{t^{\ell+1}}^{t^\ell } \int_s^{t^\ell }  | \frac{ \p  \hat{\mathfrak F}  (\tau )  }{\p_{x_\parallel^\ell } }|  d\tau ds  \lesssim & \frac{ |t^\ell - t^{\ell +1 } |^2}{\langle v \rangle}, \ \  \int_{t^{\ell+1}}^{t^\ell } \int_s^{t^\ell }  | \frac{ \p  \hat{\mathfrak F}  (\tau )  }{\p_{v^\ell } }|  d\tau ds  \lesssim \frac{ |t^\ell - t^{\ell +1 } |^2}{\langle v \rangle^2},
\\  \int_{t^{\ell+1}}^{t^\ell }  | \frac{ \p  {\mathfrak F}  (\tau )  }{\p_{x_\parallel^\ell } }|   ds  \lesssim &  |t^\ell - t^{\ell +1 } |  , \ \  \int_{t^{\ell+1}}^{t^\ell }   | \frac{ \p  {\mathfrak F}  (s )  }{\p_{v^\ell } }|   ds  \lesssim \frac{  |t^\ell - t^{\ell +1 } | }{\langle v \rangle} .
\end{split}
\Ee

Thus, from \eqref{partialtell1} and \eqref{phatFpxvellest} , we obtain
\Be
\begin{split}
| \frac{ \p t^{\ell+1}}{\p x_\parallel^\ell }  |  \le  & M  \frac{ |t^\ell - t^{\ell+1} | ^2}{ \hat v_3^{\ell +1 } \langle v \rangle },
\\ |  \frac{ \p t^{\ell+1}}{\p v_3^\ell } |  \le  & M   \frac{ |t^\ell - t^{\ell+1} |}{ \hat v_3^{\ell +1 } \langle v \rangle },
\\ |  \frac{ \p t^{\ell+1}}{\p v_\parallel ^\ell } |  \le  & M  \frac{ |t^\ell - t^{\ell+1} |^2}{ \hat v_3^{\ell +1 } \langle v \rangle^2 }.
\end{split}
\Ee
From  \eqref{partialxell1} and \eqref{phatFpxvellest},
\Be
\begin{split}
| \frac{ \p x_\parallel^{\ell+1}}{\p x_\parallel^\ell }  |  \le  & 1 +  M  \frac{ |t^\ell - t^{\ell+1} | ^2}{ \hat v_3^{\ell +1 } \langle v \rangle },
\\ |  \frac{ \p x_\parallel^{\ell+1}}{\p v_3^\ell } |  \le  & M   \frac{ |t^\ell - t^{\ell+1} |}{ \hat v_3^{\ell +1 } \langle v \rangle },
\\ |  \frac{ \p x_\parallel^{\ell+1}}{\p v_\parallel ^\ell } |  \le  & M \left(  \frac{ |t^\ell - t^{\ell+1} |^2}{ \hat v_3^{\ell +1 } \langle v \rangle^2 } + \frac{|t^\ell - t^{\ell +1} | }{\langle v \rangle } \right) .
\end{split}
\Ee
From  \eqref{partialvell1} and \eqref{phatFpxvellest},
\Be
\begin{split}
| \frac{ \p v_\parallel^{\ell+1}}{\p x_\parallel^\ell }  |  \le  &  M \left( \frac{ |t^\ell - t^{\ell+1} | ^2}{ \hat v_3^{\ell +1 } \langle v \rangle } + |t^\ell - t^{\ell +1} | \right) ,
\\ |  \frac{ \p v_\parallel^{\ell+1}}{\p v_3^\ell } |  \le  & M  \left(  \frac{ |t^\ell - t^{\ell+1} |}{ \hat v_3^{\ell +1 } \langle v \rangle } + \frac{ |t^\ell - t^{\ell +1} | }{\langle v \rangle} \right),
\\ |  \frac{ \p v_\parallel^{\ell+1}}{\p v_\parallel ^\ell } |  \le  & 1 +  M \left(  \frac{ |t^\ell - t^{\ell+1} |^2}{ \hat v_3^{\ell +1 } \langle v \rangle^2 } + \frac{ | t^\ell - t^{\ell+1} |}{\langle v \rangle} \right) .
\end{split}
\Ee
From \eqref{partialv3ell1} and \eqref{phatFpxvellest},
\Be
\begin{split}
| \frac{ \p v_3^{\ell+1}}{\p x_\parallel^\ell }  |  \le  &  M \left( \frac{ |t^\ell - t^{\ell+1} | ^2}{ \hat v_3^{\ell +1 } \langle v \rangle } + |t^\ell - t^{\ell +1} |  \right) ,
\\ |  \frac{ \p v_3^{\ell+1}}{\p v_\parallel ^\ell } |  \le  &   M \left(  \frac{ |t^\ell - t^{\ell+1} |^2}{ \hat v_3^{\ell +1 } \langle v \rangle^2 } + \frac{ | t^\ell - t^{\ell+1} |}{\langle v \rangle}  \right) .
\end{split}
\Ee
Now we estimate $|  \frac{ \p v_3^{\ell+1}}{\p v_3^\ell } |  $. Notice that since
\Be \label{v3elliint}
\begin{split}
0 = x_3^{\ell} =  & x_3^{\ell+1} + \int_{t^{\ell +1}}^{t^\ell } \hat V(s ) ds 
\\ = & (t^\ell - t^{\ell+1} )  \hat v_3^{\ell+1} +  \int_{t^{\ell +1}}^{t^\ell }  \int_{t^{\ell+1}}^s \hat {\mathfrak F}_3 (\tau) d\tau ds 
\\ = &  (t^\ell - t^{\ell+1} )  \hat v_3^{\ell+1} + \frac{(t^\ell - t^{\ell+1} )^2}{2}  \hat {\mathfrak F}_3 (t^{\ell+1} ) +   \int_{t^{\ell +1}}^{t^\ell }  \int_{t^{\ell+1}}^s \int_{t^{\ell+1}}^\tau \frac{d}{d \tau' }  \hat {\mathfrak F}_3 (\tau')  d\tau' d\tau ds,
\end{split}
\Ee
and from \eqref{nablaxhatF}, \eqref{nablavhatF},
\[
|  \frac{d}{d \tau' }  \hat {\mathfrak F}_3 (\tau') | \le \frac{M}{\langle v \rangle } 
\]
where $M$ depends on $\| \nabla_x E \|_\infty$, $\| \nabla_x B \|_\infty$, and $g$. Thus
\Be
1 + \frac{(t^\ell - t^{\ell+1} )}{2 \hat v_3^{\ell+1} }  \hat {\mathfrak F}_3 (t^{\ell+1} )  =  M \frac{ |t^\ell - t^{\ell+1} |^2}{\hat v_3^{\ell+1} \langle v \rangle } .
\Ee
From \eqref{hatF},
\[
\hat {\mathfrak F}_3 (t^{\ell+1} ) = \frac{ 1 - |\hat v_3^{\ell+1} |^2 }{\langle v^{\ell+1} \rangle } \mathfrak F_3(t^{\ell+1} )  + O_{\mathfrak F} (1)  \frac{  \hat v_3^{\ell+1} }{\langle v^{\ell+1} \rangle }.
\]
Therefore,
\Be \label{vell3d1}
| 1 + \frac{(t^\ell - t^{\ell+1} )}{2 \hat v_3^{\ell+1} } \frac{ 1 - |\hat v_3^{\ell+1} |^2 }{\langle v^{\ell+1} \rangle } \mathfrak F_3(t^{\ell+1} ) | \le M \left( \frac{ |t^\ell - t^{\ell+1} |^2}{\hat v_3^{\ell+1} \langle v \rangle }  + |t^\ell - t^{\ell+1} | \right).
\Ee
Also, notice that
\Be \label{vell3d2}
| \frac{ ( 1 - (\hat v_3^\ell )^2 )}{\langle v^\ell \rangle } - \frac{ ( 1 - (\hat v_3^{\ell+1} )^2 )}{\langle v^{\ell+1} \rangle } |  = |  \int_{t^{\ell+1}}^{t^\ell}  \frac{d}{ds} \left( \frac{ ( 1 - (\hat V^\ell_3(s) )^2 )}{\langle  V^\ell(s) \rangle }  \right) ds |   \le  M \frac{ |t^\ell - t^{\ell+1} |}{ \langle v \rangle ^2 }.
\Ee
Therefore from \eqref{phatFpxvellest}, \eqref{phatFpxvellest}, \eqref{vell3d1}, and \eqref{vell3d2}, we have
\Be
\begin{split}
\frac{ \p v_3^{\ell+1}}{\p {v_3 } ^\ell } =   &  -1 -   \frac{ \p t^{\ell+1}}{\p v_3^\ell }  \mathfrak F_3 (t^{\ell +1 } ) +  \int_{t^{\ell +1}}^{t^\ell }  \frac{ \p  {\mathfrak F}_3  (s )  }{\p_{v_3^\ell } } ds
\\  = & -1  - \left(  \frac{(t^\ell - t^{\ell +1 }  )}{\hat v_3^{\ell +1 } } \frac{ ( 1 - (\hat v_3^\ell )^2 )}{\langle v^\ell \rangle } +  \frac{1}{\hat v_3^{\ell +1} }  \int_{t^{\ell +1}}^{t^\ell } \int_s^{t^\ell } \frac{ \p  \hat{\mathfrak F}_3  (\tau )  }{\p_{v_3^\ell } }d\tau ds \right)   \mathfrak F_3 (t^{\ell +1 } )+  \int_{t^{\ell +1}}^{t^\ell }  \frac{ \p  {\mathfrak F}_3  (s )  }{\p_{v_3^\ell } } ds
\\ = & - 1 -   \frac{(t^\ell - t^{\ell +1 }  )}{\hat v_3^{\ell +1 } } \frac{ ( 1 - (\hat v_3^{\ell+1} )^2 )}{\langle v^{\ell+1} \rangle } \mathfrak F_3 (t^{\ell +1 } )  +   \frac{(t^\ell - t^{\ell +1 }  )}{\hat v_3^{\ell +1 } } \left(  \frac{ ( 1 - (\hat v_3^\ell )^2 )}{\langle v^\ell \rangle } - \frac{ ( 1 - (\hat v_3^{\ell+1} )^2 )}{\langle v^{\ell+1} \rangle } \right) \mathfrak F_3 (t^{\ell +1 } ) 
\\ &-  \frac{ \mathfrak F_3 (t^{\ell +1 } )}{\hat v_3^{\ell +1} }  \int_{t^{\ell +1}}^{t^\ell } \int_s^{t^\ell } \frac{ \p  \hat{\mathfrak F}_3  (\tau )  }{\p_{v_3^\ell } }d\tau ds  +  \int_{t^{\ell +1}}^{t^\ell }  \frac{ \p  {\mathfrak F}_3  (s )  }{\p_{v_3^\ell } } ds
\\ = & -1 + 2  + M \left( \frac{ |t^\ell - t^{\ell+1} |^2}{\hat v_3^{\ell+1} \langle v \rangle^2 } +  \frac{  |t^\ell - t^{\ell +1 } | }{\langle v \rangle} \right)
\\ =  & 1  + M \left( \frac{ |t^\ell - t^{\ell+1} |^2}{\hat v_3^{\ell+1} \langle v \rangle^2 } +  \frac{  |t^\ell - t^{\ell +1 } | }{\langle v \rangle} \right).
\end{split}
\Ee
Finally, from \eqref{tbbdvb}, \eqref{alphaest},
\[
|t^\ell - t^{\ell+1} | \lesssim \hat v_3^{\ell+1} \langle v \rangle \lesssim \alpha(t,x,v) \langle v \rangle,
\]
and from \eqref{v3elliint},
\[
 |t^\ell - t^{\ell+1} | \hat v_3^{\ell+1} =   - \int_{t^{\ell +1}}^{t^\ell }  \int_{t^{\ell+1}}^s \hat {\mathfrak F}_3 (\tau) d\tau ds \lesssim \frac{|t^\ell - t^{\ell+1}|^2 g }{\langle v \rangle } ,
\]
so by \eqref{alphaest},
\[
 \alpha(t,x,v) \langle v \rangle \lesssim  \hat v_3^{\ell+1} \langle v \rangle \lesssim   |t^\ell - t^{\ell+1} |.
\]
Therefore, we have for $\ell \ge 1$,  there exists $c, C >0$ depending on $T, \| E \|_\infty, \|B\|_\infty$, and $g$ such that 
\Be \label{tdiffalpha}
c  \alpha(t,x,v) \langle v \rangle \le   |t^\ell - t^{\ell+1} |  \le C \alpha(t,x,v) \langle v \rangle. 
\Ee
Put together the above estimates and using \eqref{tdiffalpha}, we have for some $M>0$, 
\Be
\begin{split}
J^{\ell+1}_{\ell} \le & 
 \left[\begin{array}{c|cc|c|cc}
 1 + M \alpha \langle v \rangle &   M \alpha \langle v \rangle & M \alpha \langle v \rangle &  M &  M \alpha \langle v \rangle  & M  \alpha \langle v \rangle
 \\  \hline  M \alpha \langle v \rangle & 1 +  M \alpha \langle v \rangle &  M \alpha \langle v \rangle & M &  M  \alpha \langle v \rangle  & M \alpha \langle v \rangle 
 \\  M \alpha \langle v \rangle  &    M \alpha \langle v \rangle & 1+  M \alpha \langle v \rangle & M &  M  \alpha \langle v \rangle  & M \alpha \langle v \rangle 
 \\  \hline  M \alpha \langle v \rangle  & M \alpha \langle v \rangle & M \alpha \langle v \rangle  & 1 + M \alpha \langle v \rangle &  M  \alpha \langle v \rangle  & M  \alpha \langle v \rangle 
 \\ \hline  M \alpha \langle v \rangle & M \alpha \langle v \rangle & M \alpha \langle v \rangle & M & 1 +  M  \alpha \langle v \rangle & M  \alpha \langle v \rangle
 \\  M \alpha \langle v \rangle  & M \alpha \langle v \rangle & M \alpha \langle v \rangle & M &   M  \alpha \langle v \rangle & 1+  M  \alpha \langle v \rangle
 \end{array} \right]
 \\ : = & J(\alpha \langle v \rangle) .
\end{split}
\Ee
From diagonalization, we get
\[
J(\alpha \langle v \rangle ) = \mathcal P \Lambda \mathcal P^{-1},
\]
where
\[
\Lambda = \text{diag} \left[ 1, 1, 1, 1,  1 + M \left( \sqrt{\alpha \langle v \rangle ( 4 \alpha \langle v \rangle + 5 )  }  + 3 \alpha \langle v \rangle \right) , 1 + M \left( - \sqrt{\alpha \langle v \rangle ( 4 \alpha \langle v \rangle + 5 )  }  + 3 \alpha \langle v \rangle \right) \right], 
\]
and
\[
\begin{split}
\mathcal P = &
\begin{bmatrix}
-1 & -1 & -1 & -1 & 1 & 1 
\\ 1 & 0 & 0 & 0 & 1 & 1
\\ 0 & 1 & 0 & 0 & 1 & 1
\\ 0 & 0 & 0 & 0 &  \sqrt{\alpha \langle v \rangle ( 4 \alpha \langle v \rangle + 5 )  } - 2 \alpha \langle  v \rangle & - \sqrt{\alpha \langle v \rangle ( 4 \alpha \langle v \rangle + 5 )  } - 2 \alpha \langle  v \rangle 
\\ 0 & 0 & 1 & 0 & 1 & 1
\\ 0 & 0 & 0 & 1 & 1 & 1
\end{bmatrix},
\\ \mathcal P^{-1} = & \begin{bmatrix}
-\frac{1}{5} & \frac{4}{5}  & -\frac{1}{5} & 0 & -\frac{1}{5}  & -\frac{1}{5} 
\\  -\frac{1}{5} & -\frac{1}{5}  & \frac{4}{5} & 0 &  -\frac{1}{5}  & -\frac{1}{5} 
\\ -\frac{1}{5} & -\frac{1}{5}  & -\frac{1}{5}  & 0 & \frac{4}{5}  & -\frac{1}{5}  
\\ -\frac{1}{5} & -\frac{1}{5}  & -\frac{1}{5}  & 0 & -\frac{1}{5}  & \frac{4}{5} 
\\ a &  a  &  a  & b & a & a 
\\ - a  & -a  & -a  & -b & -a & -a
\end{bmatrix},
\end{split}
\]
where
\[
a =   \frac{ 2 \alpha \sqrt{ \langle v \rangle} +  \sqrt{  \alpha( 4 \alpha \langle v \rangle + 5 )} }{10   \sqrt{  \alpha( 4 \alpha \langle v \rangle + 5 )} } , \ b = \frac{1}{2 \sqrt{ \alpha \langle v \rangle (4 \alpha \langle v \rangle + 5 ) }}.
\]
Now from \eqref{tdiffalpha}, the number of bounces
\Be \label{ellstarbd}
 \ell^*(0;t,x,v) \le \frac{T}{c \alpha(t,x,v) \langle v \rangle }.
\Ee
Therefore,
\Be
\begin{split}
  \prod_{\ell = 1 }^{\ell^*(0;t,x,v) } J^{\ell+1}_{\ell }  & \le  J^{\ell^*(0;t,x,v) }  
\le   \widetilde{   \mathcal P  }  \widetilde{ \Lambda} ^{\ell^*}  \widetilde{   \mathcal P^{-1} }
 \le ( 1 + 2 M \sqrt{\alpha \langle v \rangle } )^{\frac{1}{c \alpha \langle v \rangle } }   \widetilde{   \mathcal P  }   \widetilde{   \mathcal P^{-1} },
\end{split}
\Ee
where we use the notation: for a matrix $A$, the entries of a matrix $\widetilde A$ are absolute values of the entries of $A$, i.e. $(\widetilde A)_{ij} = | (A)_{ij} |$. From
\[
( 1 + 2 M \sqrt{\alpha \langle v \rangle } )^{\frac{1}{c \alpha \langle v \rangle } } = \left(  \left( 1 + 2 M \sqrt{\alpha \langle v \rangle } \right)^{\frac{1}{ 2 M \sqrt{ \alpha \langle v \rangle } } } \right)^{  \frac{2M}{c \sqrt{ \alpha \langle v \rangle } } }  \le e^{  \frac{2M}{c \sqrt{ \alpha \langle v \rangle } } },
\]
and that $ \left( \widetilde{   \mathcal P  }   \widetilde{   \mathcal P^{-1} } \right)_{ij} \le  \frac{M}{\sqrt{\alpha \langle v \rangle } } $, we get
\Be \label{midij}
 \left( \prod_{\ell = 1 }^{\ell^*(0;t,x,v) } J^{\ell+1}_{\ell } \right)_{ij}  \le   e^{  \frac{2M}{c \sqrt{ \alpha \langle v \rangle } } }  \left(\widetilde{   \mathcal P  }   \widetilde{   \mathcal P^{-1} } \right)_{ij} \le   \frac{M}{\sqrt{\alpha \langle v \rangle } }  e^{  \frac{2M}{c \sqrt{ \alpha \langle v \rangle } } } .
\Ee
Next, we estimate  $ \frac{\partial ( X_{\mathbf{cl}}(s), V_{\mathbf{cl}}(s))}{\partial (t^{\ell_{*}},  {x}_\parallel^{\ell_{*}}, {v}_{3}^{\ell_{*}}, {v}_{\parallel}^{\ell_{*}})} $ and  $ \frac{\partial (t^{1},  {x}_{\parallel}^{1}, {v}_{3}^{1}, {v}_{\parallel}^{1})}{\partial (x,v)}$. From
\[
\begin{split}
X(s;t^{\ell^*},x^{\ell^*},v^{\ell^*}) =  & x^{\ell^*} - (t^{\ell^*}-s) \hat{v}^{\ell^*} + \int^{t^{\ell^*}}_s\int^{t^{\ell^*}}_\tau \hat{\mathfrak F } _{}(\tau^\prime, X (\tau^\prime), V(\tau^\prime) ) \dd \tau^\prime \dd \tau,
\\ V(s;t^{\ell^*},x^{\ell^*},v^{\ell^*}) = & v^{\ell^*} - \int_s^{t^{\ell^*}} \mathfrak F (\tau, X(\tau) , V(\tau) ) \dd \tau,
\end{split}
\]
we have
\[
\begin{split}
\frac{\p X(s) }{\p t^{\ell^*} } = & -   \hat{v}^{\ell^*} +  \int^{t^{\ell^*}}_s \hat{\mathfrak F } ( t^{\ell^*} ) +   \int^{t^{\ell^*}}_s\int^{t^{\ell^*}}_\tau \frac{\p \hat{\mathfrak F } _{}(\tau^\prime )}{\p t^{\ell^*} }  \dd \tau^\prime \dd \tau
\\ = & - \hat V(s) +  \int^{t^{\ell^*}}_s\int^{t^{\ell^*}}_\tau \frac{\p \hat{\mathfrak F } _{}(\tau^\prime )}{\p t^{\ell^*} }  \dd \tau^\prime \dd \tau
\\ \frac{\p V(s) }{\p t^{\ell^*} } = & - \mathfrak F( t^{\ell^*} ) -  \int_s^{t^{\ell^*}} \frac{ \p \mathfrak F (\tau )}{\p t^{\ell^*} }   \dd \tau.
\end{split}
\]
Therefore, $ | \frac{\p X(s) }{\p t^{\ell^*} }  | \lesssim |  \hat V(s) | \lesssim 1$, and $  \frac{\p V(s) }{\p t^{\ell^*} } \lesssim 1 $. Combine this with \eqref{pxviXVest}, we have
\Be \label{lastij}
  \frac{\partial ( X_{\mathbf{cl}}(s), V_{\mathbf{cl}}(s))}{\partial (t^{\ell_{*}},  {x}_\parallel^{\ell_{*}}, {v}_{3}^{\ell_{*}}, {v}_{\parallel}^{\ell_{*}})}  \le 
  \begin{bmatrix}
 M & M & \frac{|t^{\ell^*}-s| }{\langle v \rangle } 
 \\ M &   |t^{\ell^*}-s| & M
  \end{bmatrix}.
\Ee
Lastly, since $t^1 = \tb(t,x,v)$, $(x^1_\parallel, 0 ) = \xb(t,x,v)$, $(v^1_3, v^1_\parallel) = \vb(t,x,v)$, from \eqref{pxitb} and \eqref{pvitb},
\[ 
| \p_x t^1 | \lesssim \frac{1}{ \hat {v }_3^1} \lesssim \frac{1}{\alpha} , \  | \p_v t^1 | \lesssim \frac{t^1 }{ \hat {v }_3^1 \langle v \rangle } \lesssim M.
\]
And thus from \eqref{pxbvb}, we have
\Be \label{firstij}
 \frac{\partial (t^{1},  {x}_{\parallel}^{1}, {v}_{3}^{1}, {v}_{\parallel}^{1})}{\partial (x,v)}
  \le 
  \begin{bmatrix}
\frac{M}{\alpha} &  M
\\ 1 +  \frac{M}{\alpha} & M
\\  \frac{M}{\alpha} & M
  \end{bmatrix}.
\Ee
Finally, combining \eqref{midij}, \eqref{lastij}, and \eqref{firstij}, we get for some $C_1= C_1(M,c)$, 
\Be
\left| \left( \frac{\partial ( X_{\mathbf{cl}}(s;t,x,v),V_{\mathbf{cl}}(s;t,x,v))}{\partial (x,v)} \right)_{ij} \right| \le \frac{ 36 M^3}{\alpha^{3/2}  \sqrt{ \langle v \rangle } } e^{  \frac{2M}{c \sqrt{ \alpha \langle v \rangle } }} \le C_1 \langle v \rangle  e^{\frac{C_1}{\sqrt{ \alpha \langle v \rangle } } }.  
\Ee
\end{proof}

Proposition \ref{specBCprop} comes as a consequence of the lemma.
\begin{proof}[proof of Proposition \ref{specBCprop}]
For any $(t,x,v) \in (0,T) \times \bar \O \times \mathbb R^3$, let $\p_{\mathbf e } \in \{ \nabla_x , \nabla_v \} $, then
\Be \label{pef}
\begin{split}
\p_{\mathbf e }  f(t,x,v ) & = \p_{\mathbf e} (f(0, X_{\mathbf{cl}}(0;t,x,v) , V_{\mathbf{cl}}(0;t,x,v) ) 
\\ & = \nabla_x f_0 \cdot \p_{\mathbf e } X_{\mathbf{cl}}(0;t,x,v) +  \nabla_v f_0 \cdot \p_{\mathbf e } V_{\mathbf{cl}}(0;t,x,v).
\end{split}
\Ee
Now, from \eqref{alphaest} and \eqref{lemma_Dxv},  write $X(0) = X_{\mathbf{cl}}(0;t,x,v)$, and $V(0) = V_{\mathbf{cl}}(0;t,x,v)$, we have
\Be \label{nablaxfest1}
\begin{split}
& \langle v \rangle^{4+\delta} |  \p_{\mathbf e }  f (t,x,v ) |
\\  \le &  C_1 | \nabla_x f(0,X(0), V(0) )  |   \langle v \rangle^{5+\delta}  e^{\frac{C_1}{\sqrt{ \alpha(t,x,v) \langle v \rangle } } }  +  C_1 | \nabla_v f(0,X(0), V(0) )  |    \langle v \rangle^{5+\delta}  e^{\frac{C_1}{\sqrt{ \alpha(t,x,v) \langle v \rangle } } }
\\ \le &  C | \nabla_x f(0,X(0), V(0) )   \langle |V(0)\rangle^{5+\delta}  e^{\frac{C}{\sqrt{ \alpha(0,X(0),V(0)) \langle V(0) \rangle } } }
\\ &  +  C | \nabla_v f(0,X(0), V(0) )    \langle V(0) \rangle ^{5+\delta}  e^{\frac{C}{\sqrt{ \alpha(0,X(0),V(0)) \langle V(0) \rangle } } }
\\ \le & C \left(  \| \langle v \rangle^{5 + \delta }    e^{\frac{C}{\sqrt{ \alpha \langle v \rangle } } }   \nabla_{x} f_0 \|_\infty +  \| \langle v \rangle^{5 + \delta }    e^{\frac{C}{\sqrt{ \alpha \langle v \rangle } } }      \nabla_v f_0 \|_\infty \right) .
\end{split}
\Ee

Next, using the same argument in Lemma \ref{tracepf}, we obtain
\Be \label{xparavfbdpspec}
\begin{split}
& \sup_{0 \le t \le T}  \| \langle v \rangle^{4 + \delta }  \nabla_{x} f (t) \|_{L^\infty(\gamma \setminus \gamma_0) }  \le \sup_{0 \le t \le T}   \| \langle v \rangle^{4 + \delta }  \nabla_{x} f (t) \|_\infty, 
\\ &  \sup_{0 \le t \le T}   \| \langle v \rangle^{4 + \delta }  \nabla_{v} f (t) \|_{L^\infty(\gamma \setminus \gamma_0) }  \le  \sup_{0 \le t \le T}   \| \langle v \rangle^{4 + \delta }  \nabla_{v} f (t) \|_\infty.
\end{split}
\Ee

Combining with \eqref{nablaxfest1}, we conclude \eqref{specdfbd}.
\end{proof}

We consider the sequence of functions:
\[
f^0(t,x,v) = f_0(x,v), \ E^0(t,x) = E_0(t,x), \ B^0(t,x) = B_0(x).
\]
For $\ell \ge 1$, let $f^\ell$ be the solution of
\Be \label{fellseqspec}
\begin{split}
\p_t f^\ell + \hat v \cdot \nabla_x f^\ell +  \mathfrak F^{\ell-1}   \cdot \nabla_v f^\ell = &  0, \text{ where }   \mathfrak F^{\ell-1} =  E^{\ell-1} + E_{\text{ext}} + \hat v \times (B^{\ell-1} + B_{\text{ext} } ) - g \mathbf e_3,
\\f^\ell(0,x,v)  = &  f_0(x,v) ,
\\  f^\ell(t,x,v) |_{\gamma_-} = &  f^{\ell-1}(t,x,v_\parallel, - v_3 ).
\end{split}
\Ee
Let $\rho^\ell = \int_{\mathbb R^3 } f^\ell dv, J^\ell = \int_{\mathbb R^3 } \hat v  f^\ell dv$. Let 
\Be \label{ElBlspec}
E^\ell = \eqref{Eesttat0pos} + \dots + \eqref{Eest3bdrycontri}, \  B^\ell = \eqref{Besttat0pos} + \dots + \eqref{Bestbdrycontri}, \text{ with } f \text{ changes to } f^\ell. 
\Ee
And let
\Be \label{Fellspec}
\mathfrak F^\ell = E^\ell + E_{\text{ext}} - \hat v \times (B^\ell + B_{\text{ext} } ) - g \mathbf e_3. 
\Ee
We prove several uniform-in-$\ell$ bounds for the sequence before passing the limit.

\begin{lemma} \label{fEBellbdlemspec}
Suppose $f_0$ satisfies \eqref{f0bdd}, $E_0$, $B_0$ satisfy \eqref{E0B0g}, \eqref{E0B0bdd}, then there exits $M_1, M_2$ such that for $0 < T \ll 1 $, 
\Be \label{fellboundspec}
\begin{split}
 \sup_\ell \sup_{0 \le t \le T} \left( \| \langle v \rangle^{4 + \delta } f^\ell(t) \|_{L^\infty(\bar \O \times \mathbb R^3)}    \right) < & M_1, \ 
\\  \sup_\ell \sup_{0 \le t \le T} \left( \| E^\ell (t) \|_\infty + \| B^\ell (t) \|_\infty \right) + |B_e| + E_e +  g < & M_2,
\\ \inf_{\ell} \inf_{ t ,x_\parallel} \left( g - E_e  -  E^\ell_3(t,x_\parallel, 0 )  -  (\hat v \times B^\ell)_3(t,x_\parallel, 0 ) \right) > &  c_0.
\end{split}
\Ee
\end{lemma}
\begin{proof}
By induction hypothesis we assume that
\Be \label{inductfEBspec}
\begin{split}
\sup_{ 0 \le i \le \ell }  \sup_{0 \le t \le T} \left( \| \langle v \rangle^{4 + \delta} f^{i}(t) \|_{L^\infty(\bar \O \times \mathbb R^3)}   \right) < & M_1,
\\  \sup_{ 0 \le i \le \ell }  \sup_{0 \le t \le T} \left(  \| E^{i} (t) \|_\infty + \| B^{i} (t) \|_\infty \right) + |B_e| + E_e +  g < & M_2.
\end{split}
\Ee

Denote the characteristics $(X^\ell, V^\ell)$ which solves
\Be\label{specXV_ell}
\begin{split}
\frac{d}{ds}X^\ell(s;t,x,v) &= \hat V^\ell(s;t,x,v),\\
\frac{d}{ds} V^\ell(s;t,x,v) &= \mathfrak F^\ell (s, X^\ell(s;t,x,v),V^\ell(s;t,x,v) ).\end{split}
\Ee
And define the specular cycles:
		\Be
		\begin{split}\label{speccycle}
			t^{\ell}_1 (t,x,v)&:= 
			\sup\{ s \ge 0 :
			X^\ell(\tau;t,x,v) \in \O \text{ for all } \tau \in (t-s, t ) 
			\}
			,\\
			x^\ell_1 (t,x,v ) &:= X^\ell (t^{\ell}_1 (t,x,v);t,x,v),
\\ v^{\ell}_1(t,x,v) & : = V^\ell(t^\ell_1(t,x,v);t,x,v) - 2  V_3^\ell(t^\ell_1(t,x,v);t,x,v) \mathbf e_3
		\end{split}
		\Ee
		and inductively for $k \ge 2$, 
		\Be
		\begin{split}\label{cycle_ell}
			& t^{\ell-(k-1)}_k (t,x,v)  \\
			&:=     \sup\big\{  s \ge 0 
			:
			X^{\ell-(k-1)}(\tau;t_{k-1}^{\ell - (k-2)}  , x_{k-1}^{\ell - (k-2)} ,v^{\ell - (k-2) }_{k-1}) \in \O  \text{ for all } \tau \in (t-s, t )
			\big\},\\
			& x_k^{\ell - (k-1)} (t,x,v)\\
			&:= X^{\ell- (k-1)} (t_k^{\ell- (k-1)}; t_{k-1}^{\ell- (k-2)},x_{k-1}^{\ell- (k-2)} , v^{\ell - (k-2) } _{k-1}),
			\\ & v_k^{\ell - (k-1)} (t,x,v)\\
			&:= V^{\ell- (k-1)} (t_k^{\ell- (k-1)}; t_{k-1}^{\ell- (k-2)},x_{k-1}^{\ell- (k-2)} , v^{\ell - (k-2) } _{k-1}) - 2 V_3^{\ell- (k-1)} (t_k^{\ell- (k-1)}; t_{k-1}^{\ell- (k-2)},x_{k-1}^{\ell- (k-2)} , v^{\ell - (k-2) } _{k-1}) \mathbf e_3 .
		\end{split}
		\Ee
And we define the generalized characteristics for the specular BC as
\begin{equation} \label{speccycles} 
\begin{split}
& X^\ell_{\mathbf{cl}}(s;t,x,v)   \ = \ \mathbf 1_{[t^\ell_1(t,x,v) , t ) } (s) X^\ell(s;t,x,v)  + \sum_{k \ge 1} \mathbf{1}_{[t^{\ell-k}_{k+1},t^{%
\ell- (k-1)}_{k})}(s) X^{\ell- k}(s;t^{\ell - (k-1)}_{k}, x^{\ell - (k-1) }_{k} , v^{\ell-(k-1)}_{k} ), 
\\ &  
V^\ell_{\mathbf{cl}}(s;t,x,v)   \ = \ \mathbf 1_{[t^\ell_1(t,x,v) , t ) } (s) V^\ell(s;t,x,v)  + \sum_{k \ge 1} \mathbf{1}_{[t^{\ell-k}_{k+1},t^{%
\ell- (k-1)}_{k})}(s) V^{\ell- k}(s;t^{\ell - (k-1)}_{k}, x^{\ell - (k-1) }_{k} , v^{\ell-(k-1)}_{k} ).
\end{split}%
\end{equation}

From \eqref{fellseqspec} and \eqref{speccycle}, for any $(t,x,v) \in (0,T) \times \bar \O \times \mathbb R^3$, let $k$ be such that $ t^{\ell-k}_{k+1}(t,x,v) \le 0 < t^{\ell- (k-1)}_{k}(t,x,v) $, then we have
\Be \label{fellspecto0}
\begin{split}
f^{\ell+1}(t,x,v) 
= & f^{\ell- (k-1) } \left(0, X^{\ell- k }_{}\left(0; t^{\ell- (k-1)}_{k}, x^{\ell- (k-1)}_{k},  v^{\ell- (k-1)}_{k}\right) , V^{\ell-k}_{}\left(0;t^{\ell- (k-1)}_{k}, x^{\ell- (k-1)}_{k},  v^{\ell- (k-1)}_{k}\right)   \right)
\\ = &  f_0 \left( X^{\ell- k }_{}\left(0; t^{\ell- (k-1)}_{k}, x^{\ell- (k-1)}_{k},  v^{\ell- (k-1)}_{k}\right) , V^{\ell-k}_{}\left(0;t^{\ell- (k-1)}_{k}, x^{\ell- (k-1)}_{k},  v^{\ell- (k-1)}_{k}\right)   \right).
\end{split}
\Ee
Thus
\Be \label{specv4f1}
\langle v \rangle^{4 + \delta }  f^{\ell+1}(t,x,v) \le  \frac{ \langle v \rangle ^{4 + \delta}}{  \left( V^{\ell-k}_{}\left(0;t^{\ell- (k-1)}_{k}, x^{\ell- (k-1)}_{k},  v^{\ell- (k-1)}_{k}\right) \right)^{4 + \delta} }   \| \langle v \rangle^{4 + \delta} f_0 \|_\infty.
\Ee
Now, since
\[
\begin{split}
&|  v - V^{\ell-k}_{}\left(0;t^{\ell- (k-1)}_{k}, x^{\ell- (k-1)}_{k},  v^{\ell- (k-1)}_{k}\right)  | 
\\  & = \left|  v  - v^{\ell}_1 + \sum_{i=1}^{k-1} \left(  v^{\ell - (i-1)}_i -  v^{\ell - i}_{i+1} \right) + v^{\ell - ( k-1)}_k - V^{\ell-k}_{}\left(0;t^{\ell- (k-1)}_{k}, x^{\ell- (k-1)}_{k},  v^{\ell- (k-1)}_{k}\right) \right| 
\\ & \le  \int_{t^\ell_1}^t \| \mathfrak F^\ell (s) \|_\infty ds +  \sum_{i=1}^{k-1} \int_{t^{\ell - i }_{i+1}}^{t^{\ell -(i-1) }_i } \| \mathfrak F^{\ell - i } (s) \|_\infty  ds  + \int_{0}^{t^{\ell - (k-1) }_k } \| \mathfrak F^{\ell - k } (s) \|_\infty ds
\\ & \le \int_0^t \sup_{0 \le i \le l } \| \mathfrak F^i(s) \|_\infty ds.
\end{split}
\]
From \eqref{inductfEBspec} we have
\[
|v| \le \left|  V^{\ell-k}_{}\left(0;t^{\ell- (k-1)}_{k}, x^{\ell- (k-1)}_{k},  v^{\ell- (k-1)}_{k}\right)  \right| + t M_2,
\]
and this yields 
\Be \label{vfracVspec}
\begin{split}
\frac{ \langle v \rangle }{  \langle V^{\ell-k}_{}\left(0;t^{\ell- (k-1)}_{k}, x^{\ell- (k-1)}_{k},  v^{\ell- (k-1)}_{k}\right) \rangle }  < 1 + t M_2 . 
\end{split}
\Ee
Thus \eqref{specv4f1} gives
\Be \label{fell1finalspec}
\sup_{0 \le t \le T} \| \langle v \rangle^{4 + \delta }  f^{\ell+1}(t )  \|_{L^\infty(\bar \O \times \mathbb R^3 ) }   < (1 +   T (g+M_2) )  \| \langle v \rangle^{4 + \delta} f_0 \|_\infty < M_1,
\Ee
for $T$ small enough. Now from \eqref{ElBlspec} and \eqref{E0B0g}, using the same argument as \eqref{BEellinftyestinflow1}--\eqref{BEellinftyestinflow4}, we get
\Be \label{EBlifinal}
  \sup_{0 \le t \le T} ( \| E^{\ell+1} (t) \|_\infty + \| B^{\ell+1} (t) \|_\infty) + |B_e| + E_e + g < M_2,
\Ee
and
\Be
 \inf_{ t ,x_\parallel} \left( g - E_e  -  E^{\ell+1}_3(t,x_\parallel, 0 )  -  (\hat v \times B^{\ell+1})_3(t,x_\parallel, 0 ) \right) >   c_0.
\Ee
Thus we conclude \eqref{fellbound} by induction. 

\end{proof}

Next, we define $\alpha^\ell(t,x,v)$ in the same way as in \eqref{alphan}. Then we have

\begin{lemma}
Suppose $f_0$ satisfies \eqref{f0spec}, $E_0$, $B_0$ satisfy \eqref{E0B0g}, \eqref{E0B0bdd}, then there exits $M_3, M_4$ such that for $0 < T \ll 1 $, 
\Be \label{flElBldsqbdspec}
\begin{split}
 &\sup_\ell \sup_{0 \le t \le T} \left( \| \langle v \rangle^{4 + \delta } \nabla_{x } f^\ell(t) \|_\infty  + \|  \langle v \rangle^{4 + \delta }  \nabla_v f^\ell(t) \|_\infty   \right) 
 \\& + \sup_\ell \sup_{0 \le t \le T} \left( \| \langle v \rangle^{4 + \delta } \nabla_{x } f^\ell(t) \|_{L^\infty(\gamma \setminus \gamma_0)}   + \|  \langle v \rangle^{4 + \delta }  \nabla_v f^\ell(t) \|_{L^\infty(\gamma \setminus \gamma_0)}   \right) <  M_3 ,
\\ & \sup_\ell \sup_{0 \le t \le T} \left( \| \p_t E^\ell(t) \|_\infty +  \| \p_t B^\ell(t) \|_\infty + \| \nabla_{x } E^\ell(t) \|_\infty +\| \nabla_{x } B^\ell(t) \|_\infty  \right) < M_4.
\end{split} 
\Ee
And with the $M_2$ as in Lemma \ref{fEBellbdlemspec},
\Be \label{EBellpObd}
\sup_\ell \sup_{0 \le  t \le T} \left( \| E^\ell (t,x ) \|_{L^\infty( \p \O ) } + \| B^\ell (t,x ) \|_{L^\infty( \p \O ) } \right)  < M_2.
\Ee
\end{lemma}

\begin{proof}
Let $\p_{\mathbf e } \in \{ \nabla_x, \nabla_v \}$. From \eqref{fellspecto0} we have
\Be \label{pefell1spec}
\begin{split}
\p_{\mathbf e } f^{\ell+1}(t,x,v)  = & \nabla_x f_0 \cdot \p_{\mathbf e } X^{\ell- k }_{}\left(0; t^{\ell- (k-1)}_{k}(t,x,v), x^{\ell- (k-1)}_{k}(t,x,v),  v^{\ell- (k-1)}_{k}(t,x,v) \right)
\\ & + \nabla_v f_0 \cdot \p_{\mathbf e } V^{\ell- k }_{}\left(0; t^{\ell- (k-1)}_{k}(t,x,v), x^{\ell- (k-1)}_{k}(t,x,v),  v^{\ell- (k-1)}_{k}(t,x,v) \right)
\end{split}
\Ee
Then from \eqref{fellboundspec}, we can use essentially the same argument as the proof of Lemma \ref{dXVcl} to get
\Be \label{dXellkVellkspec}
\begin{split}
& \left| \left( \frac{\partial ( X_{}^{\ell - k}(0;  t^{\ell- (k-1)}_{k}, x^{\ell- (k-1)}_{k},  v^{\ell- (k-1)}_{k} ),V_{}^{\ell - k }(0; t^{\ell- (k-1)}_{k}, x^{\ell- (k-1)}_{k},  v^{\ell- (k-1)}_{k})}{\partial (x,v)} \right)_{ij} \right|
\\ &  \le C_1 \langle  V^{\ell -k }(0)   \rangle  \exp \left({\frac{C_1}{\sqrt{ \alpha^{\ell - k }(0, X^{\ell - k}(0), V^{\ell -k }(0)  )  \langle  V^{\ell -k }(0) \rangle } } } \right),
\end{split}
\Ee
where
\[
X^{\ell - k}(0) = X^{\ell - k}(0; t^{\ell- (k-1)}_{k}, x^{\ell- (k-1)}_{k},  v^{\ell- (k-1)}_{k} ), \  V^{\ell -k }(0)  = V_{}^{\ell - k }(0; t^{\ell- (k-1)}_{k}, x^{\ell- (k-1)}_{k},  v^{\ell- (k-1)}_{k}),
\]
\[
t^{\ell- (k-1)}_{k} =t^{\ell- (k-1)}_{k} (t,x,v), \  x^{\ell- (k-1)}_{k} = x^{\ell- (k-1)}_{k} (t,x,v), \   v^{\ell- (k-1)}_{k} = v^{\ell- (k-1)}_{k} (t,x,v),
\]
and $C_1$ depends on $M_1$, $M_2$.
%
Therefore, from \eqref{pefell1spec} and \eqref{dXellkVellkspec},
\[
\langle v \rangle^{4 + \delta}  \p_{\mathbf e } f^{\ell+1}(t,x,v) \le C_1 \frac{\langle v \rangle^{4 + \delta}}{\langle  V^{\ell -k }(0)   \rangle^{4 + \delta} } \left( \| \langle v \rangle ^{5 + \delta }    e^{\frac{C}{\sqrt{ \alpha \langle v \rangle } } }   \nabla_{x} f_0 \|_\infty +  \| \langle v \rangle^{5 + \delta }    e^{\frac{C}{\sqrt{ \alpha \langle v \rangle } } }      \nabla_v f_0 \|_\infty  \right).
\]
And using \eqref{vfracVspec}, we conclude
\Be \label{pefell1bdend}
\sup_\ell \sup_{0 \le t \le T} \| \langle v \rangle^{4 + \delta}  \p_{\mathbf e } f^{\ell+1}(t) \|_\infty \le C \left( \| \langle v \rangle ^{5 + \delta }    e^{\frac{C}{\sqrt{ \alpha \langle v \rangle } } }   \nabla_{x} f_0 \|_\infty +  \| \langle v \rangle^{5 + \delta }    e^{\frac{C}{\sqrt{ \alpha \langle v \rangle } } }      \nabla_v f_0 \|_\infty  \right) < M_3.
\Ee
Then, from the same argument as in Lemma \ref{tracepf}, we get
\Be \label{pefell1pObdend}
\begin{split}
\sup_\ell \sup_{0 \le t \le T} \| \langle v \rangle^{4 + \delta}  \p_{\mathbf e } f^{\ell+1}(t) \|_{L^\infty(\gamma \setminus \gamma_0 ) }  \le & \sup_\ell \sup_{0 \le t \le T} \| \langle v \rangle^{4 + \delta}  \p_{\mathbf e } f^{\ell+1}(t) \|_\infty
< M_3.
\end{split}
\Ee
From this, we use the same argument to get \eqref{dxEBfinal} in the proof of Lemma \ref{EBW1inftylemma} and obtain
\[
\begin{split}
 \sup_{0 \le t \le T} & \left( \| \p_t E^{\ell+1}(t) \|_\infty +  \| \p_t B^{\ell+1}(t) \|_\infty + \| \nabla_{x } E^{\ell+1}(t) \|_\infty +\| \nabla_{x } B^{\ell+1}(t) \|_\infty   \right)
\\ \le & TC \sup_{1 \le i \le \ell}  \sup_{0 \le  t \le T}    \left( \| \p_t E^{i}(t) \|_\infty +  \| \p_t B^{i}(t) \|_\infty + \| \nabla_{x } E^{i}(t) \|_\infty +\| \nabla_{x } B^{i}(t) \|_\infty   \right)  
\\ & C   \left( \| E_0 \|_{C^2 } + \| B_0 \|_{C_2}  \right) +  C \sup_{0 \le t \le T} \left( \| \langle v \rangle ^{4 + \delta } \nabla_{x_\parallel } f^{\ell +1}(t) \|_\infty  + \| \langle v \rangle^{\ell  + \delta }  \alpha^{n} \p_{x_3 } f^{\ell +1}(t) \|_\infty   \right)  
\\ &  +  C \sup_{0 \le t \le T} \left( \| \langle v \rangle^{4 + \delta } \nabla_{x_\parallel } f^\ell(t) \|_{L^\infty(\gamma \setminus \gamma_0)}  + \| \langle v \rangle^{5 + \delta }  \alpha^{\ell-1} \p_{x_3 } f^\ell(t) \|_{L^\infty(\gamma \setminus \gamma_0)} \right) 
\\ & + C  \sup_{0 \le t \le T} \left( \| \langle v \rangle^{4  + \delta } f^{\ell+1}(t) \|_\infty  + \| E^{\ell +1} (t) \|_\infty + \| B^{\ell +1} (t) \|_\infty \right).
\end{split}
\]
From \eqref{fellboundspec} and \eqref{pefell1bdend}, this gives
\[
\begin{split}
\sup_{\ell} \sup_{0 \le t \le T} & \left( \| \p_t E^{\ell}(t) \|_\infty +  \| \p_t B^{\ell}(t) \|_\infty + \| \nabla_{x } E^{\ell}(t) \|_\infty +\| \nabla_{x } B^{\ell}(t) \|_\infty   \right)
\\ \le & TC \sup_{ \ell}  \sup_{0 \le  t \le T}    \left( \| \p_t E^{\ell}(t) \|_\infty +  \| \p_t B^{\ell}(t) \|_\infty + \| \nabla_{x } E^{\ell}(t) \|_\infty +\| \nabla_{x } B^{\ell}(t) \|_\infty   \right)  
\\ & + C \left( \| E_0 \|_{C^2 } + \| B_0 \|_{C_2}  \right)+ C(M_1 +M_2 +M_3).
\end{split}
\]
Therefore, by choosing $M_4 \gg 1$ and $T \ll 1$, we get
\Be \label{EBptnxbdspec}
\sup_{\ell} \sup_{0 \le t \le T}  \left( \| \p_t E^{\ell}(t) \|_\infty +  \| \p_t B^{\ell}(t) \|_\infty + \| \nabla_{x } E^{\ell}(t) \|_\infty +\| \nabla_{x } B^{\ell}(t) \|_\infty   \right) < M_4.
\Ee
Together with \eqref{pefell1bdend} and \eqref{pefell1pObdend}, we conclude \eqref{flElBldsqbdspec}.

Now, from \eqref{EBptnxbdspec}, we use the same argument as in Lemma \ref{EBelltrace} to conclude that  for any $0 < t <T$, 
\[
E^\ell(t,x ) |_{\p \O } \in L^\infty( \p \O ), \ B^\ell(t,x ) |_{\p \O } \in L^\infty( \p \O ),
\]
and
\[
\sup_{0 \le  t \le T} \left( \| E^\ell (t,x ) \|_{L^\infty( \p \O ) } + \| B^\ell (t,x ) \|_{L^\infty( \p \O ) } \right)  <  \sup_{0 \le  t \le T} \left( \| E^\ell (t,x ) \|_{\infty } + \| B^\ell (t,x ) \|_{\infty } \right) < M_2.
\]
This proves \eqref{EBellpObd}.
\end{proof}

Next, we prove a pointwise convergence result for $f^\ell$.

\begin{lemma} \label{specpointwise}
Suppose $f_0$ satisfies \eqref{f0spec}, $E_0$, $B_0$ satisfy \eqref{E0B0g}, \eqref{E0B0bdd}. Then there exists a function $f$ such that $f^\ell \to f$ pointwise almost everywhere on $(0,T) \times ( \bar \O \times \mathbb R^3 \setminus \gamma_0 ) $.
\end{lemma}

\begin{proof}
Fix any $(t,x,v) \in (0,T) \times ( \bar \O \times \mathbb R^3 \setminus \gamma_0 ) $, then it suffices to show that $\{ f^\ell (t,x,v) \}_{\ell=1}^\infty$ is a Cauchy sequence. Fix $\e > 0$, and let $n > m \ge  N_0$. Note that $f^m - f^n $ satisfies $(f^m - f^n ) |_{t = 0 } = 0 $ and 
\[
(f^m- f^n )(t,x,v) |_{\gamma_-} =  (f^{m-1} - f^{n-1} )(t,x,v_\parallel, - v_3) .
\] 
The equation for $f^m - f^n $ is
\Be \label{fmfneq}
\p_t(f^m- f^n ) + \hat v \cdot \nabla_x (f^m - f^n ) + \mathfrak F^{m-1} \cdot \nabla_v (f^{m} - f^n ) = - ( \mathfrak F^{m-1} - \mathfrak F^{n-1} ) \cdot \nabla_v f^n.
\Ee
Thus, for any $(t,x,v) \in (0,T) \times ( \bar \O \times \mathbb R^3 \setminus \gamma_0 )$, there is a $k \ge 0$ such that $ t^{m-k}_{k}(t,x,v) \le 0 < t^{m- (k-1)}_{k-1}(t,x,v) $, and we have from \eqref{fmfneq}, 
\Be \label{fnmiterate2spec}
\begin{split}
 &   (f^m - f^n)(t,x,v)  
    \\  =  & \int_{t^{m-1}_1}^t  \left(  - ( \mathfrak F^{m-1} - \mathfrak F^{n-1} ) \cdot \nabla_v f^n \right)(s, X^{m-1}(s) , V^{m-1}(s) ) ds
    \\ & +    \sum_{i=1}^{k-2} \int_{t^{m - (i +1)}_{i+1}}^{t^{m -i }_{i} }  \left( - ( \mathfrak F^{m-1-i} - \mathfrak F^{n-1 - i} ) \cdot \nabla_v f^{n-i} \right) ( s, X^{m-1-i} (s), V^{m-1-i} (s)) ds 
    \\ & + \int_0^{ t^{m- (k-1)}_{k-1} }  \left( - ( \mathfrak F^{m-1-(k-1)} - \mathfrak F^{n-1 - (k-1)} ) \cdot \nabla_v f^{n-(k-1)} \right) ( s, X^{m-1-(k-1)} (s), V^{m-1-(k-1)} (s)) ds.
\end{split}
\Ee
Together with \eqref{vfracVspec}, this implies
\Be \label{fmfnspecest1}
 \langle v \rangle^{4 + \delta}  | (f^m- f^n )(t,x,v) |  \le  C_1 \left( \sup_{\ell}  \sup_{0 \le s \le t } \| \langle  v \rangle^{4 + \delta} \nabla_v f^\ell(s) \|_\infty \right) \int_0^t \sup_{1 \le i \le k-1}  \|   \mathfrak F^{m-i} (s) - \mathfrak F^{n-i}(s)  \|_\infty ds
\Ee
where $C_1 =  (1+ T (M_2+g) )^{4 + \delta}$. Note that since $(x,v) \notin \gamma_0$, $\alpha(t,x,v ) > 0$. And from \eqref{tdiffalpha} and \eqref{ellstarbd}, we have
\[
k \le  \frac{T}{c \alpha(t,x,v) \langle v \rangle },
\]
where $c $ depends on $M_2$ and $g$. Now, for some small $\delta' > 0$, we write 
\[
\begin{split}
 \mathfrak F^{m-i}  - \mathfrak F^{n-i} = &  E_{f^{m-i} } - E_{f^{n-i} } + B_{f^{m-i} } - B_{f^{n-i} }
 \\ = &  E_{f^{m-i} -f^{n-i}  }+ B_{f^{m-i} - f^{n-i} } 
 \\ = & E_{ \mathbf 1_{ \{ |v_3| > \delta' \}} (f^{m-i} -f^{n-i} ) } + E_{ \mathbf 1_{ \{ |v_3| \le \delta' \}} (f^{m-i} -f^{n-i} ) } + B_{ \mathbf 1_{ \{ |v_3| > \delta' \} } (f^{m-i} -f^{n-i} ) } + B_{ \mathbf 1_{ \{ |v_3| \le \delta' \}} (f^{m-i} -f^{n-i} ) },
 \end{split}
 \]
where
\Be \label{EBsplit}
\begin{split}
& E_{ \mathbf 1_{ \{ |v_3| > \delta' \} } (f^{m-i} -f^{n-i} ) } = \eqref{Eesttat0pos} + \dots + \eqref{Eest3bdrycontri}, \  \text{ with } f \text{ changes to  } \mathbf 1_{|v_3| > \delta'} (f^{m-i} -f^{n-i} ),
\\ & B_{ \mathbf 1_{ \{ |v_3| > \delta' \}} (f^{m-i} -f^{n-i} ) } = \eqref{Besttat0pos} + \dots + \eqref{Bestbdrycontri}, \ \text{ with } f \text{ changes to  } \mathbf 1_{|v_3| > \delta'} (f^{m-i} -f^{n-i} ),
\\ & E_{ \mathbf 1_{ \{ |v_3| \le \delta' \} } (f^{m-i} -f^{n-i} ) } = \eqref{Eesttat0pos} + \dots + \eqref{Eest3bdrycontri}, \  \text{ with } f \text{ changes to  } \mathbf 1_{|v_3| \le \delta'} (f^{m-i} -f^{n-i} ),
\\ & B_{ \mathbf 1_{ \{| v_3| \le \delta' \}} (f^{m-i} -f^{n-i} ) } = \eqref{Besttat0pos} + \dots + \eqref{Bestbdrycontri}, \ \text{ with } f \text{ changes to  } \mathbf 1_{|v_3| \le \delta'} (f^{m-i} -f^{n-i} ).
\end{split}
\Ee
Now, using the estimate in Lemma \ref{EBlinflemma} and that 
\[
\int_{|v_3 | < \delta' } \frac{1}{\langle v \rangle^{3 + \delta } }  dv \le C ( \delta')^{\delta},
\]
we have
\Be \label{fmfnspecest2}
\|  E_{ \mathbf 1_{ \{ |v_3| \le \delta' \} } (f^{m-i} -f^{n-i} ) }  (s) \|_\infty +  \| B_{ \mathbf 1_{ \{ |v_3| \le \delta' \}} (f^{m-i} -f^{n-i} ) } (s) \|_\infty \le C  ( \delta')^{\delta} \sup_{0 \le s' \le s }  \| \langle v \rangle^{4+\delta}   (f^{m-i} -f^{n-i} ) (s') \|_\infty.
\Ee
And
\Be \label{fmfnspecest3}
\begin{split}
& \|  E_{ \mathbf 1_{ \{ |v_3| > \delta' \} } (f^{m-i} -f^{n-i} ) }  (s) \|_\infty +  \| B_{ \mathbf 1_{ \{ |v_3| > \delta' \}} (f^{m-i} -f^{n-i} ) } (s) \|_\infty 
\\ \le & C   \sup_{0 \le s' \le s }  \| \mathbf 1_{ \{ |v_3| > \delta' \} } \langle v \rangle^{4+\delta}  (f^{m-i} -f^{n-i} ) (s') \|_\infty.
\end{split}
\Ee
So from \eqref{fmfnspecest1}, \eqref{fmfnspecest2}, and \eqref{fmfnspecest3},
\Be
\begin{split}
 \langle v \rangle^{4 + \delta}  | (f^m- f^n )(t,x,v) | \le C C_1 M_3 \int_0^t \left(  \sup_{1 \le i \le k-1} \sup_{0 \le s' \le s }  \| \mathbf 1_{ \{ |v_3| > \delta' \} } \langle v \rangle^{4+\delta}  (f^{m-i} -f^{n-i} ) (s') \|_\infty + M_1 ( \delta')^{\delta} \right) ds.
\end{split}
\Ee
Now, let $j$ be such that $t^{m-i-j}_{j}(s',x,v)  \le 0 < t^{m - i - (j-1) }_{j -1} (s',x,v) $. Then if $|v_3 | > \delta'$, from \eqref{ellstarbd},
\[
j \le \frac{T}{c \alpha(s',x,v) \langle v \rangle } \le \frac{T}{c \delta ' } : = k'. 
\]
So same as \eqref{fmfnspecest1}, this gives
\[
\begin{split}
& \sup_{1 \le i \le k-1} \sup_{0 \le s' \le s }  \| \mathbf 1_{ \{ |v_3| > \delta' \} } \langle v \rangle^{4+\delta}  (f^{m-i} -f^{n-i} ) (s') \|_\infty 
\\ & \le C_1M_3 \int_0^s \sup_{1 \le i' \le k' -1} \sup_{1 \le i \le k-1}  \|   \mathfrak F^{m-i - i'} (s') - \mathfrak F^{n-i - i'}(s')  \|_\infty ds'.
\end{split}
\]
Using the same split \eqref{EBsplit}, like \eqref{fmfnspecest2} and \eqref{fmfnspecest3}, we thus get
\Be \label{fmfnspecest4}
\begin{split}
& \sup_{1 \le i \le k-1} \sup_{0 \le s' \le s }  \| \mathbf 1_{ \{ |v_3| > \delta' \} } \langle v \rangle^{4+\delta}  (f^{m-i} -f^{n-i} ) (s') \|_\infty 
\\ & \le C C_1M_3 \int_0^s \sup_{2 \le i \le k + k' } \sup_{0 \le s'' \le s'}  \left(    \| \mathbf 1_{ \{ |v_3| > \delta' \} } \langle v \rangle^{4+\delta}  (f^{m-i} -f^{n-i} ) (s'') \|_\infty + M_1   ( \delta')^{\delta} \right) ds'.
\end{split}
\Ee
Plug \eqref{fmfnspecest4} into \eqref{fmfnspecest3} yields
\[
\begin{split}
&  \langle v \rangle^{4 + \delta}  | (f^m- f^n )(t,x,v) | 
 \\ \le & (C C_1 M_3)^2 \int_0^t \int_0^s   \sup_{2 \le i \le k + k' } \sup_{0 \le s'' \le s'}  \| \mathbf 1_{ \{ |v_3| > \delta' \} } \langle v \rangle^{4+\delta}  (f^{m-i} -f^{n-i} ) (s'') \|_\infty  ds' ds 
 \\ & + M_1  (\delta')^\delta \left( CC_1M_3 t  +  \frac{ (CC_1M_3 t)^2}{2} \right).
 \end{split}
\]
Iteration of the above gives
\[
\begin{split}
&  \langle v \rangle^{4 + \delta}  | (f^m- f^n )(t,x,v) | 
 \\ \le &  (C C_1 M_3)^l  \frac{t^l}{l!}  \sup_{2 \le i \le k + l k' } \sup_{0 \le s \le t}   \| \mathbf 1_{ \{ |v_3| > \delta' \} } \langle v \rangle^{4+\delta}  (f^{m-i} -f^{n-i} ) (s) \|_\infty    +  M_1  (\delta')^\delta \sum_{i =1}^l    \frac{ (CC_1M_3 t)^i}{i!}.
 \end{split}
\]
Now, by choosing $\delta'$ small enough we have
\[
 M_1  (\delta')^\delta \sum_{i =1}^l    \frac{ (CC_1M_3 t)^i}{i!} <  M_1  (\delta')^\delta e^{  CC_1M_3 t} < \frac{\e}{2}.
\]
And choosing $l$ large enough such that
\[
 (C C_1 M_3)^l  \frac{t^l}{l!}  \sup_{2 \le i \le k + l k' } \sup_{0 \le s \le t}   \| \mathbf 1_{ \{ |v_3| > \delta' \} } \langle v \rangle^{4+\delta}  (f^{m-i} -f^{n-i} ) (s) \|_\infty < \frac{\e}{2}.
\] 
Finally choose $N_0 $ large enough such that $N_0 > k + l k'$, we get for $n,m > N_0$,
\[
\langle v \rangle^{4 + \delta}  | (f^m- f^n )(t,x,v) |  < \e.
\]
Therefore the sequence $ \{   f^\ell(t,x,v)  \}_{\ell=1}^\infty$ is Cauchy, and this proves the lemma.

\end{proof}

\begin{lemma} \label{fEBsollemmaspec}
Suppose $f_0$ satisfies \eqref{f0spec}, $E_0$, $B_0$ satisfy \eqref{E0B0g}, \eqref{E0B0bdd}. Then for $0 < T \ll 1$, there exists functions $(f,E,B)$ with $  \langle v \rangle^{4 +\delta }  f(t,x,v) \in L^\infty( (0, T) ; L^\infty( \bar \O \times \mathbb R^3 ) )  $, and $(E,B) \in L^\infty((0,T) ; L^\infty( \O) \cap  L^\infty( \p \O ) )$, such that 
$(f,E,B)$ is a (weak) solution of the system \eqref{VMfrakF1}--\eqref{rhoJ1}, \eqref{spec}. Moreover,
\Be \label{pfbdlimitspec}
 \| \langle v \rangle^{4 + \delta } \nabla_{x } f(t) \|_\infty  + \|  \langle v \rangle^{4 + \delta }  \nabla_v f(t) \|_\infty < \infty,
\Ee
and
\Be \label{pEBbdlimitspec}
\| \p_t E(t) \|_\infty +  \| \p_t B(t) \|_\infty + \| \nabla_{x } E(t) \|_\infty +\| \nabla_{x } B(t) \|_\infty   < \infty.
\Ee

\end{lemma}

\begin{proof}
From the uniform-in-$\ell$ bound \eqref{fellboundspec}, we can pass the limit up to subsequence if necessary and get the weak$-*$ convergence
\Be \label{fellweakcov}
  \langle v \rangle^{4 + \delta } f^\ell   \overset{\ast}{\rightharpoonup}   \langle v \rangle^{4 + \delta }  f \text{ in } L^\infty((0,T) \times \O \times \mathbb R^3 ) \cap L^\infty((0,T) \times \gamma), 
\Ee
\Be \label{EBellweakcov}
E^\ell   \overset{\ast}{\rightharpoonup} E, \  B^\ell   \overset{\ast}{\rightharpoonup} B   \text{ in }  L^\infty((0,T) \times \O ) \cap L^\infty((0,T) \times \p \O ). \ 
\Ee
for some $(f, E, B)$. Then from \eqref{flElBldsqbdspec} we also have
\Be \label{dEnBncovspec}
\p_t E^\ell  \overset{\ast}{\rightharpoonup} \p_t E , \ \nabla_x E^\ell  \overset{\ast}{\rightharpoonup} \nabla_x E, \  \p_t B^\ell  \overset{\ast}{\rightharpoonup} \p_t B, \  \nabla_x B^\ell  \overset{\ast}{\rightharpoonup} \nabla_x B \text{ in } L^\infty((0,T) \times \O ),
\Ee
and
\Be \label{dfnconvergespec}
\langle v \rangle ^{4 + \delta} \nabla_{x } f^\ell   \overset{\ast}{\rightharpoonup}  \langle v \rangle^{4 + \delta}\nabla_{x } f, \  \langle v \rangle ^{4 + \delta} \nabla_{v } f^\ell   \overset{\ast}{\rightharpoonup}  \langle v \rangle^{4 + \delta}\nabla_{v } f \text{ in } L^\infty((0,T) \times \O \times \mathbb R^3 ).
\Ee

Now it left to show that $(f,E,B)$ is a solution to the system \eqref{VMfrakF1}--\eqref{rhoJ1}, \eqref{spec}. 
Take any $\phi(t,x,v) \in C_c^\infty( [0,T) \times \bar \O \times \mathbb R^3$ with $\text{supp } \phi   \subset \{ [0, T) \times \bar \O \times \mathbb R^3 \} \setminus \{  \{ 0 \} \times \gamma \cup (0,T) \times \gamma_0 \} $, from \eqref{fellseqspec}, we have
\Be \label{weakfellVMspec}
\begin{split}
& \int_{\O \times \mathbb R^3 } f_0 \phi (0) dv dt +  \int_0^T \int_{\O \times \mathbb R^3}   f^\ell  (  \p_t \phi + \hat v \cdot \nabla_x \phi  + \mathfrak F^{\ell-1}   \cdot \nabla_v \phi  ) dv dx dt
\\ = & \int_0^T \int_{\gamma_+} \phi f^\ell \hat v_3 dv dS_x + \int_0^T \int_{\gamma_+} \phi(t,x,v_\parallel, - v_3 ) f^\ell \hat v_3 \  dv dS_x.
\end{split}
\Ee

Because of \eqref{fellweakcov} and \eqref{EBellweakcov}, we have
\Be \label{weakfellVMspec1}
\begin{split}
 & \int_0^T \int_{\O \times \mathbb R^3}  f^\ell( \p_t \phi + \hat v \cdot \nabla_x \phi )  dv dx dt +   \int_0^T \int_{\gamma_+} \phi f^\ell \hat v_3 dv dS_x + \int_0^T \int_{\gamma_+} \phi(t,x,v_\parallel, - v_3 ) f^\ell_\pm \hat v_3 \  dv dS_x 
\\ \to  &  \int_0^T \int_{\O \times \mathbb R^3}  f( \p_t \phi + \hat v \cdot \nabla_x \phi )  dv dx dt +   \int_0^T \int_{\gamma_+} \phi f \hat v_3 dv dS_x + \int_0^T \int_{\gamma_+} \phi(t,x,v_\parallel, - v_3 ) f \hat v_3 \  dv dS_x
\end{split}
\Ee
as $\ell \to \infty$. 
As for the term $\int_0^T \int_{\O \times \mathbb R^3 } f^\ell \mathfrak F^{\ell - 1 } \cdot \nabla_v \phi  dv dx dt$, since
\Be \label{weakfellVMspec2}
\begin{split}
& \int_0^T \int_{\O \times \mathbb R^3 } ( f^\ell \mathfrak F^{\ell - 1 }  -f \mathfrak F) \cdot \nabla_v \phi  dv dx dt 
\\ = &  \int_0^T \int_{\O \times \mathbb R^3 } ( f^\ell - f ) \mathfrak F^{\ell - 1 }  \cdot \nabla_v \phi  dv dx dt  + \int_0^T \int_{\O \times \mathbb R^3 } f (  \mathfrak F^{\ell - 1 } -  \mathfrak F) \cdot \nabla_v \phi  dv dx dt
\end{split}
\Ee
From \eqref{EBellweakcov}, we have $ \int_0^T \int_{\O \times \mathbb R^3 } f (  \mathfrak F^{\ell - 1 } -  \mathfrak F) \cdot \nabla_v \phi  dv dx dt \to 0$ as $\ell \to \infty$. 

Now, let $\text{supp} (\phi) = D$. From Lemma \ref{specpointwise}, $ \mathbf 1_D(t,x,v) f^\ell (t,x,v) $ converges to $\mathbf 1_D(t,x,v) f(t,x,v)$ pointwise almost everywhere. And from \eqref{fellboundspec}, $ | \mathbf 1_D(t,x,v) f^\ell (t,x,v) | \le  \mathbf 1_D(t,x,v) M_1$. Therefore, from the dominated convergence theorem, we have
\[
\int_0^T \int_{\O \times \mathbb R^3} |  \mathbf 1_D  ( f- f^\ell )  | dv dx dt \to 0 \text{ as } \ell \to \infty.
\]
Thus
\Be \label{weakfellVMspec3}
\begin{split}
&  \int_0^T \int_{\O \times \mathbb R^3} ( f- f^\ell ) \mathfrak F^{\ell-1} \cdot \nabla_v \phi dv dx dt 
 \\ \le & \sup_{0 \le t \le T } \| \mathfrak F^{\ell-1} (t) \|_\infty   \sup_{0 \le t \le T } \| \nabla_v \phi (t) \|_\infty \int_0^T \int_{\O \times \mathbb R^3} |  \mathbf 1_D  ( f- f^\ell )  | dv dx dt \to 0 \text{ as } \ell \to \infty.
 \end{split}
\Ee
Put together \eqref{weakfellVMspec}--\eqref{weakfellVMspec3}, we deduce that $(f,E,B)$ satisfy \eqref{weakf}.

Next, using the same argument as in \eqref{Maxwellell}-\eqref{EBellweak2}, we get $(f,E,B)$ satisfy \eqref{Maxweak1} and \eqref{Maxweak2}. 
Therefore, we conclude that $(f,E,B)$ is a (weak) solution of the RVM system \eqref{VMfrakF1}--\eqref{rhoJ1}, with specular BC \eqref{spec}.


Finally, from using the weak lower semi-continuity of the weak-$*$ convergence \eqref{dEnBncovspec}, \eqref{dfnconvergespec}, and the uniform-in-$\ell$ bound \eqref{flElBldsqbdspec}, we conclude \eqref{pfbdlimitspec}, \eqref{pEBbdlimitspec}.
\end{proof}

Lastly, we prove the uniqueness.

\begin{lemma} \label{VMuniqlemmaspec}
Suppose  $(f,E_f, B_f)$ and $(g, E_g, B_g)$ are solutions to the RVM system \eqref{VMfrakF1}--\eqref{rhoJ1}, \eqref{spec} with $f(0) = g(0)$, $E_f(0) = E_g(0)$, $B_f(0) = B_g(0)$, and that 
\[
E_{f}, B_f, E_g, B_g \in W^{1,\infty}((0,T) \times \O ), \ \nabla_x \rho_{f}, \nabla_x J_f,  \p_t J_f , \nabla_x \rho_{g}, \nabla_x J_g,  \p_t J_g \in  L^\infty((0,T); L_{\text{loc}}^p(\O)) \text{ for some } p>1.
\]
And
\Be \label{dvfgbdspec}
\sup_{0 < t < T} \| \langle v \rangle^{4+ \delta} \nabla_v f(t) \|_\infty <\infty, \sup_{0 < t < T} \| \langle v \rangle^{4+ \delta} \nabla_v g(t) \|_\infty <\infty.
\Ee
Then $f = g, E_f = E_g, B_f = B_g$.
\end{lemma}
\begin{proof}
The difference function $f-g $ satisfies
\Be \label{fminusgeqspec}
\begin{split}
(\p_t + \hat v \cdot \nabla_x + \mathfrak F_f \cdot \nabla_v)(f-g) = (\mathfrak F_g - \mathfrak F_f ) \cdot \nabla_v g 
\\ (f-g)(0) = 0, \, (f- g )(t,x,v)|_{\gamma_-  } =     (f - g) (t,x, v_\parallel, -v_3 ),
\end{split} 
\Ee
where
\[
\mathfrak F_f = E_f + E_{\text{ext}}+ \hat v \times ( B_f + B_{\text{ext}}) - g \mathbf e_3 , \, \mathfrak F_g = E_g  + E_{\text{ext}}+ \hat v \times ( B_g + B_{\text{ext} } )- g \mathbf e_3, 
\]
so
\Be \label{mathfrakFfgspec}
\mathfrak F_g - \mathfrak F_f = E_f - E_g + \hat v \times (B_f - B_g ).
\Ee
From Lemma \ref{Maxtowave} we have $E_{f,1} - E_{g,1} , E_{f,2} - E_{g,2}, B_{f,3} - B_{g,3}$ solve the wave equation with the Dirichlet boundary condition \eqref{waveD} in the sense of \eqref{waveD_weak} with \begin{align}
u_0 = 0, \  u_1 = 0 , \ G = -4\pi \p_{x_i} (\rho_f - \rho_g) - 4 \pi \p_t (J_{f,i} - J_{g,i} ), \ g = 0 , \ \ \text{for} \  E_{f,i} - E_{g,i},  i =1,2, \label{E12sol_B} \\
 u_0 = 0, \ u_1 = 0, \   G =  4 \pi (\nabla_x \times (J_f -J_g) )_3, \  g = 0, \ \ \text{for} \ B_{f,3}- B_{g,3},  \label{B3sol_B}
 \end{align} 
respectively. And $E_{f,3} - E_{g,3}, B_{f,1} - B_{g,1}, B_{f,2}- B_{g,2}$ solve the wave equation with the Neumann boundary condition \eqref{waveNeu}  in the sense of \eqref{waveinner} \text{ with }
\begin{align}
u_0 = 0, \  u_1 = 0 , \ G = -4\pi \p_{x_3} ( \rho_f - \rho_g) - 4 \pi \p_t (J_{f,3} - J_{g,3} ) , \ g = - 4\pi (\rho_f - \rho_g), \ \ \text{for} \  E_{f,3} - E_{g,3}, \label{E3sol_A} \\
 u_0 = 0, \ u_1 = 0, \   G =  4 \pi (\nabla_x \times (J_f - J_g) )_i, \  g = (-1)^{i+1} 4 \pi (J_{f,{\underline i}} - J_{g, \underline i } ), \ \ \text{for} \ B_{f,i} - B_{j,i}, \ i=1,2, \label{B12sol_C}
 \end{align} 
respectively. Therefore, from Lemma \ref{wavesol} and Lemma \ref{wavesolD}, we know that $E_f - E_g$  and $B_f - B_g$ would have the form of
\Be \label{EBdiffformspec}
\begin{split}
& E_f - E_g = \eqref{Eesttat0pos} + \dots + \eqref{Eest3bdrycontri}, \  B_f -B_g = \eqref{Besttat0pos} + \dots + \eqref{Bestbdrycontri},
\\ & \text{ with } E_0, B_0 \text{ changes to } 0, \text{ and } f \text{ changes to } f -g.
\end{split}
\Ee 

Now consider the characteristics
\[
\begin{split}
\dot X_f(s;t,x,v) = & \hat V_f(s;t,x,v) ,
\\ \dot V_f(s;t,x,v) = & \mathfrak F_f(s, X_f(s;t,x,v), V_f(s;t,x,v) ) .
\end{split}
\]
Then from \eqref{fminusgeqspec}, same as \eqref{fnmiterate2spec}, we obtain
\Be \label{fnmiterate2finalspec}
\begin{split}
 &   (f - g)(t,x,v)  
    \\  =  & \int_{t^{}_1}^t  \left(   ( \mathfrak F_g - \mathfrak F_f ) \cdot \nabla_v g \right)(s, \dot X(s) , \dot V(s) ) ds
    \\ & +     \sum_{i=1}^{k-2} \int_{t^{}_{i+1}}^{t^{}_{i} }  \left(  ( \mathfrak F_g^{} - \mathfrak F_f^{} ) \cdot \nabla_v g^{} \right) ( s, \cdot X^{} (s), \cdot V^{} (s)) ds 
    \\ & +  \int_0^{ t^{}_{k-1} }  \left(  ( \mathfrak F_g^{} - \mathfrak F_f^{} ) \cdot \nabla_v g^{} \right) ( s, \cdot X^{} (s), \cdot V^{} (s)) ds.
\end{split}
\Ee
So using \eqref{vfracVspec}, \eqref{dvfgbdspec}, we have
\Be \label{fgspecrep}
\begin{split}
\sup_{ 0 \le s \le t } \| \langle v \rangle^{4 + \delta} (f-g)(s) \|_\infty \le & \sup_{0 \le t < T}   \|\langle v \rangle^{4 + \delta} \nabla_v g (t) \|_\infty    \int^t_0 \sup_{0 \le s' \le s } \|  (\mathfrak F_g - \mathfrak F_f )(s') \|_\infty ds
\\ \le & C  \int^t_0 \sup_{0 \le s' \le s } \|  (\mathfrak F_g - \mathfrak F_f )(s') \|_\infty ds .
\end{split}
\Ee
Now, from \eqref{EBdiffform} and the estimate in Lemma \ref{EBlinflemma}, we have
\Be \label{FgFfdiff}
\begin{split}
\sup_{0 \le s' \le s} \|  (\mathfrak F_g - \mathfrak F_f )(s') \|_\infty \le & \sup_{0 \le s' \le s}  \| (E_f - E_g )(s') \|_\infty + \sup_{0 \le s' \le s}  \| (B_f - B_g )(s') \|_\infty
\\ \le & C \sup_{0 \le s' \le s } \| \langle v \rangle^{4 + \delta} (f-g )(s' ) \|_\infty,
\end{split}
\Ee
Therefore from \eqref{fgspecrep} and \eqref{FgFfdiff}, we have
\Be
\sup_{0 \le s \le t } \|  \langle v \rangle^{4 +\delta} (f-g)(s) \|_\infty \le C' \int^t_{0 } \sup_{0 \le s' \le s } \|  \langle v \rangle^{4 +\delta}(f-g )(s' ) \|_\infty  ds.
\Ee
Therefore from Gronwall
\[
\sup_{0 \le s' \le t } \|  \langle v \rangle^{4 +\delta} (f-g)(s') \|_\infty \le   e^{C't}  \|  \langle v \rangle^{4 +\delta} (f-g)(0) \|_\infty = 0.
\]
Therefore we conclude that the solutions to \eqref{VMfrakF1}--\eqref{rhoJ1}, \eqref{spec} is unique.
\end{proof}

We conclude the section by proving Theorem \ref{main3}.
%

\begin{proof}[proof of Theorem \ref{main3}]
Using the sequence $f^\ell, E^\ell , B^\ell$ constructed in \eqref{fellseqspec}, \eqref{ElBlspec}, we have from Lemma \ref{VMuniqlemmaspec} that the limit $(f,E,B)$ is a solution to the RVM system \eqref{VMfrakF1}--\eqref{rhoJ1}, \eqref{diffuseBC}, and it satisfies the regularity estimate \eqref{specfbd}, \eqref{inflowEBreg}. And from Lemma \ref{VMuniqlemmaspec}, we conclude the uniqueness.
\end{proof}

\appendix

\section{
}

In terms of energy. We have from \eqref{Maxwell}, we have
\Be \label{ptEE}
\int_{\O} ( \p_t E \cdot E ) dx  = \int_\O ( c(\nabla_x \times B ) \cdot E - 4 \pi J \cdot E ) dx,
\Ee
and
\Be \label{ptBB}
\int_{\O} ( \p_t B \cdot B ) dx  = \int_\O c( - \nabla_x \times E ) \cdot B dx,
\Ee
Adding \eqref{ptEE} and \eqref{ptBB} we have
\Be
\begin{split}
	& \p_t   \int_\O  \left(\frac{|E|^2}{2} + \frac{|B|^2}{2}  \right)  dx 
	\\ =  & \int_{\O } c \big(  (\p_2 B_3 - \p_3 B_2) E_1 - (\p_1 B_3 - \p_3 B_1) E_2 + ( \p_1 B_2 - \p_2 B_1) E_3
	\\ &  -  (\p_2 E_3 - \p_3 E_2) B_1 + (\p_1 E_3 - \p_3 E_1) B_2 - ( \p_1 E_2 - \p_2 E_1) B_3 \big) dx - \int_\O  (4 \pi J \cdot E )dx.
\end{split}
\Ee
From integration by parts and the perfect conductor boundary condition \eqref{E12B3bc}, 
\[
\begin{split}
	& \int_{\O } \big(  (\p_2 B_3 - \p_3 B_2) E_1 - (\p_1 B_3 - \p_3 B_1) E_2 + ( \p_1 B_2 - \p_2 B_1) E_3
	\\ &  -  (\p_2 E_3 - \p_3 E_2) B_1 + (\p_1 E_3 - \p_3 E_1) B_2 - ( \p_1 E_2 - \p_2 E_1) B_3 \big) dx
	\\ = & \int_\O \left( - \p_3 B_2 E_1 + \p_3 B_1 E_2 + \p_3 E_2 B_1 - \p_3 E_1 B_2 \right) dx
	\\ = & \int_{\p \O } \left( - E_2 B_1 + E_1 B_2 \right) d x_\parallel  = 0
\end{split}
\]
Therefore,
\Be \label{pEEpBB}
\p_t   \int_\O  \left(\frac{|E|^2}{2} + \frac{|B|^2}{2}  \right)  dx  = - \int_\O  (4 \pi J \cdot E )dx.
\Ee
On the other hand, define
\Be
\langle v \rangle_\pm := \sqrt{ m_\pm^2 + |v|^2 / c^2 }.
\Ee
Note that $\hat v_\pm = \frac{v}{\langle v \rangle_\pm } $. Adding up the integration of $\langle v \rangle_\pm \times   \eqref{VMfrakF}_\pm$  for both $f_+$ and $f_-$ gives
\Be \label{vfint}
\begin{split}
	&  \p_t \int_{\O \times \mathbb R^3}  (\langle v \rangle_+ f_+ + \langle v \rangle_- f_-) dvdx 
	\\  = &  - \int_{\O \times \mathbb R^3} ( \langle v \rangle_+ \hat v_+ \cdot \nabla_x f_+ + \langle v \rangle_- \hat v_- \cdot \nabla_x f_-)  dv dx -   \int_{\O \times \mathbb R^3} \left(  \langle v \rangle_+ \mathfrak F_+ \cdot \nabla_v f_+ +  \langle v \rangle_-  \mathfrak F_- \cdot \nabla_v f_- \right) dv dx
	\\ = & \int_{\p \O  \times \mathbb R^3 }    v_3 ( f_{+} +   f_{-} )  dv dS_x + \int_{\O \times \mathbb R^3} \frac{1}{c^2}   (  \hat v_+ \cdot \mathfrak F_+ f_+ + \hat v_- \cdot \mathfrak F_- f_- )   dv dx,
\end{split}
\Ee
where we've used $ \nabla_v \langle v \rangle_\pm = \frac{1}{c^2} \hat v_\pm $, and that $ \nabla_v \cdot \mathfrak F_\pm = 0 $. Now, since
\[
\hat v_\pm \cdot \left( \hat v_\pm \times  B \right)  = 0, 
\]
we have
\Be \label{jFint}
\begin{split}
	& \frac{1}{c^2} \int_{\O \times \mathbb R^3}   [ \hat v_+ \cdot \mathfrak F_+ f_+ +  \hat v_- \cdot \mathfrak F_- f_- ]   dv dx  
	\\  = & \frac{1}{c^2}  \int_{\O \times \mathbb R^3 }  (\hat v_+  ( e_+ (E + E_{\text{ext}})   - m_+ g \mathbf e_3 ) f_+ + \hat v_- ( e_- (E + E_{\text{ext}}) - m_- g \mathbf e_3 ) f_-   )  dv dx
	\\ = & \frac{1}{c^2}  \left( \int_\O ( J \cdot ( E + E_{\text{ext}} )) dx -  g \int_{\O \times \mathbb R^3} ( \hat v_{+,3} m_+ f_+ + \hat v_{-,3} m_- f_- ) dv dx \right).
\end{split}
\Ee
Adding up \eqref{pEEpBB}, \eqref{vfint}, and \eqref{jFint}, we get
\Be
\begin{split}
	& \frac{ \p}{\p t }  \left(   \int_\O  \left(\frac{|E|^2}{2} + \frac{|B|^2}{2}  \right)  dx +   4 \pi c^2 \int_{\O \times \mathbb R^3}   (\langle v \rangle_+ f_+ + \langle v \rangle_- f_-) dvdx  \right) 
	\\ = & 4\pi c^2 \int_{\gamma_- }   v_3  (f_{+} + f_{-} ) dv dS_x +  4\pi c^2 \int_{\gamma_+ }  v_3    (f_{+} + f_{-} ) dv dS_x
	\\ & + 4 \pi \int_{\O} ( J \cdot E_{\text{ext}} ) dx  - 4 \pi g  \int_{\O \times \mathbb R^3}  ( \hat v_{+,3} m_+ f_+ + \hat v_{-,3} m_- f_- ) dv dx.
\end{split}
\Ee
For the specular BC \eqref{spec},
\[
\int_{v_3 < 0} v_3 f_\pm dv  = -  \int_{v_3 > 0 } v_3  f_\pm dv \text{ for } x \in \p \O.
\]
Thus a solution of the system \eqref{VMfrakF}-\eqref{rhoJ} with the specular BC \eqref{spec} has
\Be
\begin{split}
	& \frac{ \p}{\p t }  \left(   \int_\O  \left(\frac{|E|^2}{2} + \frac{|B|^2}{2}  \right)  dx +   4 \pi c^2 \int_{\O \times \mathbb R^3}   (\langle v \rangle_+ f_+ + \langle v \rangle_- f_-) dvdx  \right) 
	\\ = & 4 \pi E_e \int_{\O  \times \mathbb R^3 }( \hat v_{+,3} e_+ f_+ + \hat v_{-,3} e_- f_- ) dv dx  - 4 \pi g  \int_{\O \times \mathbb R^3} ( \hat v_{+,3} m_+ f_+ + \hat v_{-,3} m_- f_- ) dv dx.
\end{split}
\Ee
\hide

\section{}

We finish the section by demonstrating that $\p_3 E_3$, $\p_3 B_1$, $\p_3 B_2$ have a trace at $\p\O$:
\begin{remark}
	Suppose $\p_t E, \nabla_x E, \p_t B, \nabla_x B \in L^\infty((0,T) \times \O)$, and $ \nabla \rho, \p_t J , \nabla_x J \in  L^\infty((0,T); L_{\text{loc}}^p(\O))$ for $p>1$. Suppose 
	\Be\label{divE=rho}
	\nabla_x \cdot E = 4 \pi \rho, \ \p_t E = \nabla_x \times B - 4 \pi J  \ \text{ in the sense of distribution }  \mathcal{D}((0,T) \times \O),
	\Ee	
	and $E_\parallel = (E_1, E_2)$ is a weak solution to 
	\Be
	\begin{split}\label{pde:E_tan}
		\p_t^2 E_\parallel - \Delta E_\parallel, = - 4\pi \nabla_\parallel \rho -4 \pi \p_{t} J_\parallel  \ \ &\text{in }  [0,T] \times \O\\
		E_\parallel =0  \ \ &\text{on } \p\O;
	\end{split}
	\Ee
	$B_3$ is a weak solution to
	\Be
	\begin{split}\label{pde:B3}
		\p_t^2 B_3 - \Delta B_3 = 4 \pi \p_1J_2 - 4\pi \p_2 J_1  \ \ &\text{in }  [0,T] \times \O\\
		B_3 =0  \ \ &\text{on } \p\O;
	\end{split}
	\Ee
	
	Then 
	\Be\label{trace:E_3}
	\p_3 E_3  \ \text{has a trace and } 	\p_3 E_3(t,x) = 4\pi \rho(t,x) \ \ \text{on} \  \  (0,T) \times \p\O \ \text{ a.e.},
	\Ee
	and for $i=1,2$,
	\Be\label{trace:B_tan}
	\p_3 B_i  \ \text{has a trace and } 	\p_3 B_i(t,x) = (-1)^{i+1} 4\pi J_{\underline i } (t,x) \ \ \text{on} \  \  (0,T) \times \p\O \ \text{ a.e.}.
	\Ee
	Here, we have used a convenient notation:
	\Be\label{def_under_i}
	\underline i = \begin{cases} 2, \text{ if } i=1, \\ 1, \text{ if } i=2. \end{cases}
	\Ee
\end{remark}
\begin{proof}
	Step 1. Let $\phi (t) \in C([0,T]) \cap W^{1,1}(0,T)$ which is compactly supported in $(0,T)$, meaning that $\phi(0)= 0= \phi(T)$. We test $\phi (t)$ to \eqref{pde:E_tan}: this is possible since $E_\parallel$ might have a trace so the test function needs not vanish at the boundary $[0,T] \times \p\O$. Then $\bar E_\parallel (x) = \int_0^T E_\parallel (s, x)  \phi(s) \dd s  \in W^{1, \infty}(\O)$ solves
	\Be\label{pde:barE_tan}
	\begin{split}
		- \Delta \bar E_\parallel = \int_0^T \p_t E_\parallel (s,x) \p_t \phi(s) \dd s 
		-4 \pi \int_0^T   \nabla_\parallel \rho \phi(s) \dd s  -4 \pi \int_0^T    \p_t  J_\parallel(s,x ) \phi (s) \dd s   \ \ &\text{in }  \O,\\
		\bar{E}_\parallel =0 \ \ &\text{ on }   \ \p\O.
	\end{split}
	\Ee
	Note that this is a Poisson equation of an $L^p(\O)$-source term with zero Dirichlet boundary condition. From the elliptic regularity theory, we have 
	\Be
	\bar E_\parallel = \int^T_0 E_\parallel (s,x) \phi(s) \dd s   \in W_{\text{loc}}^{2,p}(\O)
	\ \ \text{and } \ \	 \nabla_\parallel \bar E_\parallel  \in W_{\text{loc}}^{1,p}(\O).
	\Ee
	Then by the trace theorem $W_{\text{loc}}^{1,p}(\O) \rightarrow W_{\text{loc}}^{1-1/p,p}(\p\O)$ (\cite{Leoni}, Theorem 18.27, Page 608), we conclude that 
	\Be
	\nabla_\parallel \bar E_\parallel \ \text{has a trace and } \ 	\nabla_\parallel \bar E_\parallel \in W_{\text{loc}}^{1-1/p,p}(\p\O).
	\Ee
	
	On the other hand, we note that $\nabla_\parallel \bar E_\parallel (x) = \int_0^T \nabla_\parallel E_\parallel (s, x)  \phi(s) \dd s$ a.e. in $\O$. Therefore we deduce that 
	\Be
	\int_0^T \nabla_\parallel E_\parallel (s, x)  \phi(s) \dd s \ \text{has a trace and } \ 	\int_0^T \nabla_\parallel E_\parallel (s, x)  \phi(s) \dd s \in W_{\text{loc}}^{1-1/p,p}(\p\O),
	\Ee
	for any	$\phi (t) \in C([0,T]) \cap W^{1,1}(0,T)$ which is compactly supported in $(0,T)$. Since the dual of $W_0^{1,1}((0,T) )$ is $W^{-1,\infty}((0,T))$, we conclude that 
	\Be\label{trace:NE_p}
	\nabla_\parallel E_\parallel
	\in W^{-1, \infty} ((0,T) ; W_{\text{loc}}^{1,p} (\O))
	\ \text{has a trace and } \ \nabla_\parallel E_\parallel  \in W^{-1, \infty} ((0,T); W_{\text{loc}}^{1-1/p,p}(\p\O)).
	\Ee

	On the other hand, since uniformly continuous function can be extended on the boundary, we regard $E$ to be continuous functions in $\bar \O$. For the continuous functions $E_\parallel |_{\p\O}=0$, we derive that 
	\Be\label{bdry:N_pE}
	\nabla_{\parallel} E_\parallel(x_1,x_2,0) =0 \ \ \text{for almost every } (x_1, x_2) \in \R^2.
	\Ee
	\hide

	Since a distributional derivative $\nabla_x \cdot E$ equals $4 \pi \rho$ in $\O$, $\nabla_x \cdot E$ is also continuous function in $\bar \O$.

	On the other hand,  Therefore we conclude that a distributional derivative $\p_3 E_3= \nabla_x \cdot E$ on $\p\O$.

	Moreover, since $E_\parallel=0$ in $\p\O$ now we conclude that 
	\Be\label{bdry:N_pE}
	\nabla_\parallel E_\parallel(t,\cdot ) =0 \ \text{ a.e.  } \ x \in \O \ \text{ almost all } \ t \in (0,T). 
	\Ee\unhide

	Step 2. From \eqref{divE=rho}, \eqref{trace:NE_p}, and $\nabla \rho \in L^\infty((0,T) ; L_{\text{loc}}^p(\O))$, we have 
	\Be
	\p_3 E_3 = - \p_1 E_1 - \p_2 E_2 + 4 \pi \rho\in W^{-1, \infty} ((0,T) ; W_{\text{loc}}^{1,p} (\O)).
	\Ee
	From the trace theorem $W^{-1,\infty} ((0,T) ; W_{\text{loc}}^{1,p} (\O)) \rightarrow W^{-1,\infty} ((0,T) ; W_{\text{loc}}^{1-1/p,p} (\p\O))$ and \eqref{bdry:N_pE}, finally we prove that 
	\Be\label{trace:-1infty}
	\p_3 E_3 (t,x) - 4\pi \rho(t,x)=0 \ \ \text{ in } \ \ W^{-1,\infty} ((0,T) ; W_{\text{loc}}^{1-1/p,p} (\p\O)).
	\Ee
	
	Note that $W^{-1,\infty} (0,T)$ should be considered as a subspace of $\mathcal{D}^\prime(0,T)$. For $g (t)\in W^{-1,\infty} (0,T)$, there exist $\psi_0(t) \in L^\infty$ and $\psi_1(t) \in L^\infty$ such that 
	\Be
	\int^T_0g(s) \phi(s) \dd s = 	\int^T_0\psi_0(s) \phi(s) \dd s
	- 	\int^T_0\psi_1(s)\p_t  \phi(s) \dd s \ \ \text{for all } \phi \in W^{1,1}_0 (0,T)
	.	\Ee
	Therefore \eqref{trace:-1infty} implies that $\p_3 E_3 (t,\cdot) = 4\pi \rho(t,\cdot)$ for almost every $t \in (0,T)$. Therefore we conclude \eqref{trace:E_3}.
	
	Step 3. Next, from \eqref{pde:B3} and using the same argument as in Step 1, we get
	\Be\label{trace:B3_p}
	\nabla_\parallel B_3
	\in W^{-1, \infty} ((0,T) ; W_{\text{loc}}^{1,p} (\O))
	\ \text{has a trace and } \ \nabla_\parallel B_3  \in W^{-1, \infty} ((0,T); L_{\text{loc}}^p(\p\O)).
	\Ee 
	And by regarding $B_3$ to be a continuous function in $\bar \O$ and $B_3 |_{\p \O} = 0$, we have
	\Be\label{bdry:N_B3}
	\nabla_{\parallel} B_3 (x_1,x_2,0) =0 \ \ \text{for almost every } (x_1, x_2) \in \R^2.
	\Ee
	
	Now, for any $\phi (t) \in C([0,T]) \cap W^{1,1}(0,T)$ which is compactly supported in $(0,T)$,
	\Be
	\int_0^T \p_t E_\parallel (s,x) \phi(s) \dd s = \int_0^T E_\parallel (s,x) \p_t \phi(s) \dd s \in W^{1,\infty}(\O ). 
	\Ee	
	Thus	
	\[
	\int_0^T \p_t E_\parallel (s, x)  \phi(s) \dd s \ \text{has a trace and } \ 	\int_0^T \p_t E_\parallel (s, x)  \phi(s) \dd s \in L^\infty(\p\O),
	\]	
	Again, using the dual of $W_0^{1,1}((0,T) )$ is $W^{-1,\infty}((0,T))$, we deduce that
	\Be\label{trace:ptE}
	\p_t E_\parallel
	\in W^{-1, \infty} ((0,T) ; W^{1,\infty} (\O))
	\ \text{has a trace and } \ \p_t E_\parallel  \in W^{-1, \infty} ((0,T); L^{\infty}(\p\O)).
	\Ee 
	On the other hand, we regard $E$ to be continuous function in $\bar \O$ and $E_\parallel |_{\p \O } = 0$, we derive that
	\Be \label{ptEtan0}
	\p_t E_\parallel (x_1, x_2, 0 ) ) =0 \ \ \text{for almost every } (x_1, x_2) \in \R^2.
	\Ee
	From \eqref{divE=rho}, \eqref{trace:B3_p}, \eqref{trace:ptE}, and $\nabla j \in L^\infty((0,T) ; L_{\text{loc}}^p(\O) ) $, we have
	\[
	\begin{split}
		\p_3 B_1 = & \p_t E + \p_1 B_3 + 4 \pi J_2 \in W^{-1, \infty} ((0,T) ; W_{\text{loc}}^{1,p} (\O)),
		\\ \p_3 B_2 =&  - \p_t E + \p_2 B_3 - 4 \pi J_1 \in W^{-1, \infty} ((0,T) ; W_{\text{loc}}^{1,p} (\O)).
	\end{split}
	\]
	From the trace theorem $W^{-1,\infty} ((0,T) ; W_{\text{loc}}^{1,p} (\O)) \rightarrow W^{-1,\infty} ((0,T) ;L_{\text{loc}}^p (\p\O))$ and \eqref{bdry:N_B3}, \eqref{ptEtan0}, finally we prove that 
	\Be\label{trace:-1infty1}
	\p_3 B_1 - 4 \pi J_2=0, \ \text{ and } \p_3 B_2 + 4\pi J_1 = 0  \ \ \text{ in } \ \ W^{-1,\infty} ((0,T) ; L_{\text{loc}}^{p} (\p\O)).
	\Ee
	This implies that $\p_3 B_1 (t,\cdot) =  4\pi J_2(t,\cdot)$, and $\p_3 B_2 (t,\cdot) = - 4\pi J_1(t,\cdot)$ for almost every $t \in (0,T)$. Therefore we conclude \eqref{trace:B_tan}.\end{proof}

\unhide

\end{document}